\documentclass{article}

\RequirePackage{amsthm,amsmath,amsfonts,amssymb}
\RequirePackage[authoryear]{natbib}
\RequirePackage[colorlinks,citecolor=blue,urlcolor=blue]{hyperref}
\RequirePackage{graphicx}
\RequirePackage{algorithm}
\RequirePackage{algpseudocode}
\RequirePackage{multirow}
\RequirePackage{tablefootnote}
\RequirePackage{subcaption}
\usepackage{subcaption}
\usepackage[utf8]{inputenc}

\theoremstyle{plain}

\newtheorem{theorem}{Theorem}[section]
\newtheorem{corollary}{Corollary}
\newtheorem{lemma}{Lemma}[section]
\theoremstyle{definition}
\newtheorem{definition}[theorem]{Definition}
\newtheorem{example}{Example}

\newtheorem{remark}{Remark}

\begin{document}

\title{ANOVA for High-dimensional Non-stationary Time Series}
\author{Yunyi Zhang\\
School of Data Science,\\
The Chinese University of Hong Kong, Shenzhen\\
\texttt{zhangyunyi@cuhk.edu.cn}}
\date{}

\maketitle

\begin{abstract}
Temporal dependence and the resulting autocovariances in time series data can introduce bias into ANOVA test statistics, thereby affecting their size and power. This manuscript accounts for temporal dependence in ANOVA and develops a test statistic suitable for high-dimensional, non-stationary time series. Recognizing that the presence of complex fourth-order cumulants may introduce difficulties in variance estimation of the test statistic, we develop a bootstrap algorithm to conduct hypothesis testing through computer simulations. Theoretical results including the asymptotic distribution of the test statistic under the null hypothesis and the validity of the proposed bootstrap algorithm are established. Numerical studies demonstrate a good finite-sample performance of the proposed test statistic. In addition to the new test procedure, this manuscript derives theoretical results on consistency, Gaussian approximation, and variance estimation for quadratic forms of high-dimensional non-stationary time series, which may be of independent interest to researchers.

\end{abstract}


\section{Introduction} 
\label{section.intro}
Suppose we observe $K$ vector time series $\mathbf{x}_{t,k}\in\mathbf{R}^{d},$ where $k = 1,2,\cdots, K$ and $t = 1,2,\cdots, T_k;$ with $T_k$ denoting the sample size of the $k$-th time series.  Assume that each time series has identical mean $\boldsymbol{\mu}_k = \mathbf{E}[\mathbf{x}_{t,k}]$ for $t = 1,\cdots,T_k.$  This manuscript aims to perform an analysis of variance (ANOVA), which concerns testing the following hypothesis:
\begin{equation}
    H_0: \boldsymbol{\mu}_1 = \boldsymbol{\mu}_2 = \cdots = \boldsymbol{\mu}_K\quad\text{versus}\quad  H_1: \boldsymbol{\mu}_i \neq \boldsymbol{\mu}_j\ \text{for some }i\neq j\in\{1,2,\cdots, K\}.
    \label{eq.ANOVA_hypothesis}
\end{equation}
In particular, when $K = 2,$ the hypothesis in \eqref{eq.ANOVA_hypothesis} reduces to the two-sample test for the equality of two population means.

The two-sample test and the ANOVA problem have been extensively studied in the literature under the assumption of independent observations. For the case $K = 2,$ a classical solution is the Hotelling $T^2$ test introduced in \cite{MR1990662, MR3729293}, which has been widely applied in various fields such as scalp voltage topography analysis \cite{Analysis_of_scalp}, microarray data analysis \cite{10.1093/bioinformatics/bti496},   and proteomic studies \cite{MR2896840}. More recent studies, such as \cite{MR2604697,MR3164870,MR3911118}, have focused on the challenges arising from the high dimensionality of the data, which will be discussed in detail in Section \ref{section.related_literature}.

This manuscript studies ANOVA for high-dimensional time series data. The main distinction between this setting and the setting with independent data lies in the presence of correlations among the observations. Specifically,  when $\mathbf{x}_{t_1,k}$ and $\mathbf{x}_{t_2,k}$ are independent,
the covariance
$$
\mathbf{E}\left[\left(\mathbf{x}_{t_1,k} - \boldsymbol{\mu}_k\right)^\top\left(\mathbf{x}_{t_2,k} - \boldsymbol{\mu}_k\right)\right] = 0.
$$ 
 However, for time series data these covariances are generally nonzero due to the temporal dependence.  Furthermore, test statistics such as that proposed in \cite{MR3911118} consist of the product terms $\mathbf{x}_{t_1,k}^\top\mathbf{x}_{t_2,k}.$  Therefore, the presence of nonzero covariances introduces a non-negligible bias into the expectation of these statistics, which in turn affects the size of the test. 
This phenomenon is further illustrated through numerical examples and calculations in Example \ref{Example.dependency} and Remark \ref{remark.covariance}. 
Consequently, directly applying hypothesis testing procedures developed for independent data to vector time series---without accounting for temporal correlations---may lead to unsatisfactory performance.

\begin{example}
Suppose $K = 2,\ T_1 = T_2 = 100, $ and $d = 200.$ We generate i.i.d. standard normal random variables $\boldsymbol{u}_{t,k},$ where $t = 0,1,\cdots, 100$ and $k = 1,2.$ Define  the following data generating processes:

\begin{enumerate}
    \item \textbf{Independent observations:   }$\mathbf{x}_{t,k} = \boldsymbol{u}_{t,k},$
    \item \textbf{Moving average observations:  } $\mathbf{x}_{t,k} = \boldsymbol{u}_{t,k} + \boldsymbol{u}_{t - 1,k}.$ 
\end{enumerate}
Readers may refer to Chapter 3 of \cite{MR1093459} for an overview of the characteristics of moving average time series. The moving average observations $\mathbf{x}_{t,k}$ are correlated with $\mathbf{x}_{t - 1,k}$ as they share the common innovation term $\boldsymbol{u}_{t - 1,k}.$ Table \ref{table.dependent_error} demonstrates the empirical sizes of several frequently used tests when applied to different data generating processes.  For independent data, all three tests achieve satisfactory sizes (closed to 5\%). However, for moving average observations, the sizes of all tests are substantially inflated.  This inflation is due to temporal correlations in the data, as further discussed in Remark \ref{remark.covariance}.

\begin{table}[htbp]
    \centering
    \caption{Empirical sizes of various hypothesis testing procedures under independent and moving average observations. ``CQ'', ``CLX'', and ``SD,'' respectively, refers to tests mentioned in \cite{MR2604697}, \cite{MR3164870}, and \cite{MR2396970}. Nominal size is 5.0\%, and results is derived through 200 simulations.}
    \begin{tabular}{c c c}
    \hline\hline
    Test procedure     &  Independent &  Moving average\\
    \hline 
    CQ     &   7.5\%  & 100\%\\
    CLX    &   8.5\%  & 88.0\%\\
    SD     &   7.5\%  & 100\%\\
    \hline\hline
    \end{tabular}
    \label{table.dependent_error}
\end{table}
\label{Example.dependency}
\end{example}

\begin{remark}
\label{remark.covariance}
This Remark investigates the underlying cause of the phenomenon in Example \ref{Example.dependency}.  Assume $K = 2$ and consider the test statistic of \cite{MR2604697}. Suppose  $\mathbf{x}_{t,k} = \boldsymbol{\mu}_k + \boldsymbol{\epsilon}_{t,k},$ where $\mathbf{E}[\boldsymbol{\epsilon}_{t,k}] = 0.$ Under this notation, we have 
\begin{equation}
    \mathbf{E} \left[\frac{1}{T_k(T_k - 1)}\sum_{t_1 \neq t_2} \mathbf{x}_{t_1,k}^\top \mathbf{x}_{t_2,k}\right] = \vert\boldsymbol{\mu}_k\vert_2^2 + \frac{2}{T_k(T_k - 1)}\sum_{t_1 = 1}^{T_k}\sum_{t_2 = t_1 + 1}^{T_k}\mathbf{E}\left[\boldsymbol{\epsilon}_{t_1,k}^\top\boldsymbol{\epsilon}_{t_2,k}\right].
    \label{eq.bias_Chen_Qin}
\end{equation}
In the classical setting, $\boldsymbol{\epsilon}_{t,k}$ are mutually independent. Since $t_2 > t_1,$ we have $\mathbf{E}\left[\boldsymbol{\epsilon}_{t_1,k}^\top\boldsymbol{\epsilon}_{t_2,k}\right] = 0, $ and thus the test statistics is unbiased. In contrast, for time series data, we have 
\begin{align*}
\mathbf{E}\left[\boldsymbol{\epsilon}_{t_1,k}^\top\boldsymbol{\epsilon}_{t_2,k}\right]
= \sum_{j = 1}^d\mathbf{E}\left[\boldsymbol{\epsilon}_{t_1,k}^{(j)}\boldsymbol{\epsilon}_{t_2,k}^{(j)}\right],\quad \text{which may have order } O(d).
\end{align*}
Therefore, the bias introduced by autocovariances does not substantially affect the performance of the test statistics when the data dimension $d$ is moderate. However, statisticians need to carefully address these autocovariances when $d$ is large relative to the sample sizes.
\end{remark}

Focusing on ANOVA for high-dimensional non-stationary time series, we modify the test statistics of \cite{MR2604697,MR3911118} and propose a new test statistic that effectively eliminates the bias induced by temporal correlations. In addition, we establish the Gaussian approximation for the distribution of the new test statistic. Since the proposed test statistic involves products of time series data, its variance depends on fourth-order cumulants of data. Consequently, direct estimation of the variance of the test statistic becomes a challenge for statisticians.  To address this issue, we adapt the second-order wild bootstrap algorithm of \cite{MR4829492} to the high-dimensional time series setting and employ it to assist hypothesis testing through computer simulations, so that statisticians avoid complex calculations. The proposed test procedure, together with the adapted bootstrap algorithm, accommodates vector time series whose dimension is comparable to or even exceeds the sample sizes $T_k, k=1,\ldots,K.$ Furthermore, our work does not require the covariance structure of the time series to remain stationary (i.e., satisfying Definition 1.3.2  of \cite{MR1093459}). Since real-life time series often exhibit non-stationary covariance structures for various reasons, as illustrated in \cite{MR3310530, MR4270034}, the proposed test procedure is suitable for a wide range of time series data compared to existing methods.

Beyond its methodological contributions, this manuscript extends the results of \cite{MR4829492} to high-dimensional non-stationary time series, and derives concentration inequalities, Gaussian approximations, and the variance estimation procedure, for a class of quadratic forms. Given the prevalence of quadratic forms in time series analysis, these theoretical results should be of independent interest to researchers.

The remainder of this manuscript is organized as follows. Section \ref{section.related_literature} provides a literature review on high-dimensional ANOVA and time series analysis. Section \ref{section.Setting_Method} develops the ANOVA test statistic for high-dimensional non-stationary vector time series. Furthermore, it introduces the second-order wild bootstrap algorithm to assist hypothesis testing. Section \ref{section.quadratic_form} focuses on a special class of vector time series, referred to as ``$(M,\alpha)$-short-range dependent random vectors,''  and derives theoretical results for their quadratic forms. Section \ref{section.Asymptotic_testing}  establishes the asymptotic distribution of the proposed test statistic and demonstrates the asymptotic validity of the associated second-order wild bootstrap algorithm. Section \ref{section.numerical} demonstrates the finite sample performance of the proposed methods through simulation studies. In addition, it applies the proposed methods to a real-life chickenpox counts dataset. Section \ref{section.conclusion} draws a conclusion. Technical proofs of the theoretical results are postponed to the online supplementary material \cite{Supplement}.

\textbf{Notation: } This paper adopts the standard order notation $O(\cdot)$, $o(\cdot)$, $O_p(\cdot)$, and $o_p(\cdot)$: For two numerical sequences $a_t, b_t$, $t = 1,2,\cdots, $ we say $a_t = O(b_t)$ if there exists a constant $C>0$ such that $\vert a_t\vert\leq C\vert b_t\vert$ for  $\forall t\in\mathbf{N}$; and $a_t = o(b_t)$ if $\lim_{t\to\infty} \frac{a_t}{b_t} = 0$. We say $a_t 	\asymp b_t$ if there exists two constants $0<c\leq C<\infty$ such that $ca_t\leq b_t\leq Ca_t$ for any $t.$ For two random variable sequences $X_t, Y_t$, we say $X_t = O_p(Y_t)$ if for any given $0 < \varepsilon < 1$, there exists a constant $C_\varepsilon > 0$ such that 
$
\mathbf{Pr}(\vert X_t\vert\leq C_\varepsilon\vert Y_t\vert) \geq 1 - \varepsilon
$ for any $t$; and $X_t = o_p(Y_t)$ if $X_t/Y_t\to_p 0$, where the notation $\to_p$ denotes convergence in probability. Readers can refer to \cite{MR2002723} for a detailed introduction.  For a random variable $X\in\mathbf{R}$, define its $m$-norm ($m\geq 1$) as 
$\Vert X\Vert_m = (\mathbf{E}\vert X\vert^m)^{1/m}$. The notation $\wedge$ and $\vee$ respectively represents the minimum and the maximum among two numbers, i.e., 
$a\wedge b = \min(a, b)$ and $a\vee b = \max(a,b)$. In the following sections of the manuscript, we use $C, C_0,C_1,\cdots,$ to represent general constants. Notably, the values of these constants may change from one line to another.

In the remaining parts of this manuscript, bold lowercase (Greek) letters, such as $\mathbf{a} $ or $\boldsymbol{\omega}$, represent vectors; while the bold uppercase (Greek) letters, such as $\mathbf{A}$ and $\boldsymbol\Omega$, represent the matrices. We use the symbol $\top$ to represent matrix transpose. For a vector $\mathbf{a} = (\mathbf{a}^{(1)},\cdots, \mathbf{a}^{(d)})^\top\in\mathbf{R}^d$, define its $m$ norm, $1\leq m < \infty$, as $\vert\mathbf{a}\vert_m = (\sum_{i = 1}^p \vert\mathbf{a}^{(i)}\vert^m)^{1/m}$, and its infinity norm $\vert\mathbf{a}\vert_\infty = \max_{i = 1,\cdots, p}\vert\mathbf{a}^{(i)} \vert$. For a matrix $\mathbf{A}\in \mathbf{R}^{d\times d}$, define its Frobenius norm $\vert\mathbf{A}\vert_F = \sqrt{\sum_{i = 1}^d\sum_{j = 1}^d\mathbf{A}^{(ij)2}}$. For a set $\mathcal{A}$, define $\vert \mathcal{A}\vert$ to be its order, that is, the number of elements inside $\mathcal{A}$. For a vector or a matrix, we use superscripts to represent elements within them.

\section{Related literature}
\label{section.related_literature} 
\textbf{High-dimensional two sample test and ANOVA.} 
Testing the equality of two or more population means is a fundamental topic in statistics literature, and has been extensively studied under the assumption of independent observations. Modern-era datasets often exhibit high dimensionality, where the data dimension $d$ is comparable to---or even exceeds---the sample size $T_k.$  The presence of high dimensionality can fail the classical ANOVA procedures. For example, when $K = 2,$  \cite{MR1399305} showed that the classical Hotelling $T^2$ test statistic was not well-defined for $d > T_1 + T_2 - 2,$ since the pooled sample covariance matrix was not invertible. Even when $d < T_1 + T_2 - 2,$ large dimension $d$ still led to power loss. The literature has proposed various approaches to address the challenge of high dimensionality. The work of \cite{MR2247217, MR2396970, MR2896840, MR2993891, DONG2016127, MR4124345, MR4713924, MR4829486}, among others,  resolved the issue of singular sample covariance matrices by replacing the sample covariance matrix with alternative non-singular matrices, thereby
ensuring that the test statistic was well-defined. \cite{MR1399305, MR2604697, MR3911118, MR4107696} and their references modified the Hotelling $T^2$ test statistic to omit the precision matrix.  \cite{MR3164870, MR3252643} leveraged precision matrices to perform linear transformations of data in two-sample tests.  \cite{MR4124324} introduced a maximum-type test statistics. While Hotelling $T^2$ test and its variations constructed sum-of-squares-type tests,  \cite{MR3551787} introduced a sum-of-powers test statistic, which was better adapted to a wide range of alternative hypotheses. We also mention the works of \cite{MR2985938, MR4589064, MR4627783}, which respectively extended the hypothesis in \eqref{eq.ANOVA_hypothesis} to testing equality of population covariance matrices, distributional equivalence between two populations, and linear constraints in multivariate linear models.

Compared to the rich literature on ANOVA for independent data, relatively little research has been conducted on high-dimensional ANOVA for dependent data, including time series. The studies of \cite{MR3008273, MR3338651, MR3824978, MR4647624}, among others, considered ANOVA for time series. Among them, \cite{MR3824978} investigated ANOVA for stationary time series under the condition $d^{3/2} / \sqrt{T_k} \to 0$. However, to our knowledge,  relatively few studies have been conducted on ANOVA for non-stationary time series---especially in high-dimensional settings where the dimension $d$ is comparable to or exceeds the sample size $T_k$.

\textbf{Analysis of high-dimensional non-staitonary time series. }  Apart from the methodological development of ANOVA for high-dimensional time series, this paper establishes theoretical results---including concentration inequalities, Gaussian approximation, and variance estimation---on quadratic forms of short-range dependent high-dimensional time series. Classical central limit theorems, such as those in Section 1.5 of \cite{MR2002723}, are suitable for linear combinations of data and require the data to have fixed dimensions.  When the data dimension $d$ grows to infinity with respect to the sample size, Gaussian approximations in \cite{MR3161448, MR3350040, MR3693963} approximated the distribution of the maximum of the sample mean vector by that of the maximum of a joint normal random vector. \cite{MR3992401} introduced a Gaussian approximation theorem for quadratic forms of independent data and applied it to analyze Pearson’s $\chi^2$ test statistic.

Deriving distributional results for dependent data can be more challenging. The works of \cite{MR3718156, MR3779697} respectively established Gaussian approximations for the sample mean of high-dimensional stationary and non-stationary time series, while \cite{MR4829492} derived distributional results for quadratic forms of scalar time series. However, to our knowledge, relatively few studies have been conducted on quadratic forms of high-dimensional time series. 
This gap is noteworthy, as the analysis of various important statistics---such as the sample covariance matrix \cite{MR4206676},  the sample precision matrix \cite{MR3161455, MR4134802},  and the spectral density \cite{Chang25042025}, among others---fundamentally depends on distributional results for quadratic forms of vector time series. In this regard, our work should be of independent interest to researchers working on topics beyond ANOVA for vector time series.

In addition to the presence of high dimensionality, advances in data collection, processing, and storage technologies have made it common for statisticians to analyze time series data that exhibit complex temporal dynamics or cover a long time interval. Consequently,  modern vector time series may display non-stationarity, where the marginal distributions or autocovariances evolve over time. Various reasons lead to non-stationarity. For example, FMRI data, seismic signals, and financial time series, such as exchange rates, can display non-stationary temporal dynamics according to \cite{MR2504379, doi:10.1073/pnas.1400181111, MR3310530}. On the other hand, a dataset may span a long time interval. In such case, even for a time series that changes gradually, assuming that its stochastic structure remains invariant over a long time period is unrealistic, as demonstrated in \cite{MR4270034}. The literature offers various tools for analyzing non-stationary time series, including those in \cite{MR3299408, MR3798001, MR4630946} and the references therein.

\section{Setting and Methodology} 
\label{section.Setting_Method}
Focusing on testing the statistical hypothesis \eqref{eq.ANOVA_hypothesis}, this section introduces the test statistic for ANOVA of high-dimensional time series. In addition, it adapts the second-order wild bootstrap algorithm of \cite{MR4829492} to the high-dimensional time series setting,  which facilitates hypothesis testing via simulations.

\subsection{Constructing the test statistic}
Suppose the observed time series data $\mathbf{x}_{t,k}\in\mathbf{R}^d,$ with $k = 1,\cdots,K$ and $t = 1,\cdots, T_k,$ obey the following form:
\begin{equation}
    \mathbf{x}_{t,k} = \boldsymbol{\mu}_k + \boldsymbol{\epsilon}_{t,k},\quad\text{where}\quad \mathbf{E}\left[\boldsymbol{\epsilon}_{t,k}\right] = 0.
    \label{eq.structure_X}
\end{equation}
In other words, the population means of each time series data are equal.
Furthermore, assume that $\mathbf{x}_{t_1,k_1}$ is independent of $\mathbf{x}_{t_2,k_2}$ for any $t_1,t_2$ when $k_1\neq k_2.$ To test the hypothesis in \eqref{eq.ANOVA_hypothesis},  we select two bandwidths $0 < B < B_1 < \min(T_1,\cdots,T_k).$ After that, we construct the following test statistic:
\begin{equation}
    \widehat{R} = \sum_{k = 2}^{K}\widehat{R}_k,
    \label{eq.def_R}
\end{equation}
where 
\begin{equation}
    \begin{aligned}
    \widehat{R}_k & = \frac{1}{V_k\sqrt{d}}\sum_{B\leq \vert t_1 - t_2\vert\leq B_1}^{T_k}\mathbf{x}_{t_1,k}^\top \mathbf{x}_{t_2,k}
    +\frac{1}{V_1\sqrt{d}}\sum_{B\leq \vert t_1 - t_2\vert\leq B_1}^{T_1}\mathbf{x}_{t_1,1}^\top \mathbf{x}_{t_2,1}\\
        & - \frac{2}{T_kT_1\sqrt{d}}\sum_{t_1 = 1}^{T_k}\sum_{t_2 = 1}^{T_1}\mathbf{x}_{t_1,k}^\top\mathbf{x}_{t_2,1},\\
\end{aligned}
\label{eq.def_RK}
\end{equation}
and 
\begin{align*}
    V_k = \sum_{B\leq \vert t_1 - t_2\vert\leq B_1}^{T_k}1 = \left(2T_k - B - B_1\right) \left(B_1-  B+1\right).
\end{align*}
The summation $\sum_{B\leq \vert t_1 - t_2\vert\leq B_1}^{T_k}$ here  ranges over all pairs $(t_1,t_2)$ whose time lag satisfies  $B\leq \vert t_1 - t_2\vert\leq B_1.$

\begin{remark}
    If we choose $B = 1$ and $B_1 > \max(T_1,\cdots, T_K),$ then $V_k =  \sum_{t_1 \neq t_2}^{T_k} 1 = T_k(T_k - 1),$ and the test statistic $\widehat{R}_k$ in \eqref{eq.def_RK} becomes 
    \begin{align*}
        \widehat{R}_k & = \frac{1}{T_k(T_k  - 1)\sqrt{d}}\sum_{t_1\neq t_2}^{T_k}\mathbf{x}_{t_1,k}^\top \mathbf{x}_{t_2,k} 
        + \frac{1}{T_1(T_1 - 1)\sqrt{d}}\mathbf{x}_{t_1,1}^\top \mathbf{x}_{t_2,1}\\
        &- \frac{2}{T_kT_1\sqrt{d}}\sum_{t_1 = 1}^{T_k}\sum_{t_2 = 1}^{T_1}\mathbf{x}_{t_1,k}^\top\mathbf{x}_{t_2,1}.
    \end{align*}
    This form  coincides with the test proposed by \cite{MR2604697, MR3911118}. Therefore, the test statistic in \eqref{eq.def_RK} can be viewed as a generalization of the test in \cite{MR2604697} to time series data. 
\end{remark}

We have demonstrated in Remark \ref{remark.covariance} of Section \ref{section.intro} that temporal dependence may introduce bias into the test statistic, which inflates test sizes. Remark \ref{remark.data_depenence} continues this discussion and demonstrates that, by introducing a bandwidth $B$ and excluding data products $\mathbf{x}_{t_1,k}^\top \mathbf{x}_{t_2,k}$ with $\vert t_1 - t_2\vert < B,$ the bias arising from temporal dependence can be substantially reduced.

\begin{remark}
\label{remark.data_depenence}
    This remark explains the necessity of introducing the bandwidths $B$ and $B_1.$ In our setting,
    $\mathbf{x}_{t_1,k_1}$ is independent of $\mathbf{x}_{t_2,k_2}$ when $k_1\neq k_2.$ From this condition, we have $\mathbf{E}[\mathbf{x}_{t_1,k}^\top\mathbf{x}_{t_2,1}] = \boldsymbol{\mu}_k^\top\boldsymbol{\mu}_1,$ and therefore   
    the expectation of $\widehat{R}_k$ is      
    \begin{align*}
        \mathbf{E}\left[\widehat{R}_k\right] &= \frac{\vert\boldsymbol{\mu}_k  - \boldsymbol{\mu}_1\vert^2_2}{\sqrt{d}} + \frac{1}{V_k\sqrt{d}}\sum_{B\leq \vert t_1 - t_2\vert\leq B_1}^{T_k}\mathbf{E}\left[\boldsymbol{\epsilon}_{t_1,k}^\top\boldsymbol{\epsilon}_{t_2,k}\right]\\
        &+ \frac{1}{V_1\sqrt{d}}\sum_{B\leq \vert t_1 - t_2\vert\leq B_1}^{T_1}\mathbf{E}\left[\boldsymbol{\epsilon}_{t_1,1}^\top\boldsymbol{\epsilon}_{t_2,1}\right].
    \end{align*}
If the innovations $\boldsymbol{\epsilon}_{t,k}$ are mutually independent, then $\mathbf{E}\left[\boldsymbol{\epsilon}_{t_1,k}^\top\boldsymbol{\epsilon}_{t_2,k}\right] = 0$ for $t_1\neq t_2.$ In such case, setting $B = 1$ is sufficient to eliminate the bias arising from data variances, and ensure that the expectation of $\widehat{R}_k$ equals $\frac{\vert\boldsymbol{\mu}_k  - \boldsymbol{\mu}_1\vert^2_2}{\sqrt{d}}.$ Furthermore, under $H_1,$ there exists some $k\neq 1$ such that  $\vert\boldsymbol{\mu}_k  - \boldsymbol{\mu}_1\vert_2 > 0, $ while under $H_0$ all $\vert\boldsymbol{\mu}_k  - \boldsymbol{\mu}_1\vert_2 = 0.$ This distinction guarantees that  $\widehat{R}$ in \eqref{eq.def_R} serves as a valid test statistics for distinguishing between $H_0$ and $H_1.$

The situation becomes more complicated when $\boldsymbol{\epsilon}_{t,k}$ exhibit non-zero covariances.   In this manuscript, we mainly focus on short-range dependent time series. By assuming the $(M,\alpha)$-short-range dependent condition (Definition \ref{def.M_alpha_short_range}) in Section \ref{section.quadratic}, the autocovariances of the  time series  decay polynomially with respect to the time lag: For $t_2 > t_1,$ equation  \eqref{eq.covariance_delta} in the supplementary material implies
\begin{equation}
\begin{aligned}
\left\vert\mathbf{E}\left[\boldsymbol{\epsilon}_{t_1,k}^\top\boldsymbol{\epsilon}_{t_2,k}\right]\right\vert\leq \frac{Cd}{(1  + t_2 - t_1)^\alpha}.
\end{aligned}
\label{eq.covariance}
\end{equation}
Therefore, if we adopt the test statistic of \cite{MR2604697}, i.e., by setting the bandwidth $B = 1$ and taking $B_1 > \max(T_1,\cdots, T_K)$ in \eqref{eq.def_RK}, then 
\begin{align*}
    \left\vert\ 
    \mathbf{E}\left[\widehat{R}_k\right] - \frac{\vert\boldsymbol{\mu}_k  - \boldsymbol{\mu}_1\vert^2_2}{\sqrt{d}}\ 
    \right\vert & \leq \frac{1}{T_k(T_k - 1)\sqrt{d}}\sum_{t_1\neq t_2}^{T_k}\left\vert\ \mathbf{E}\left[\boldsymbol{\epsilon}_{t_1,k}^\top\boldsymbol{\epsilon}_{t_2,k}\right]\ \right\vert\\
    & + \frac{1}{T_1(T_1 - 1)\sqrt{d}}\sum_{t_1\neq t_2}^{T_k}\left\vert\ \mathbf{E}\left[\boldsymbol{\epsilon}_{t_1,1}^\top\boldsymbol{\epsilon}_{t_2,1}\right]\ \right\vert
    \leq C\left(\frac{\sqrt{d}}{T_k} + \frac{\sqrt{d}}{T_1}\right).
\end{align*}
In other words, the bias introduced by data dependence is of order $O\left(\frac{\sqrt{d}}{T_k} + \frac{\sqrt{d}}{T_1}\right).$ In contrast, \cite{MR2604697} showed that, under some regularity conditions, the stochastic error of $\widehat{R}_k$ is of order $O_p\left(\frac{1}{T_k} + \frac{1}{T_1}\right),$ which becomes significantly smaller than the bias if $d \to \infty$ with respect to $T_k, T_1.$ This comparison, along with Example \ref{Example.dependency}, stresses that ignoring temporal dependence in ANOVA can lead to substantial bias, which affects both the size and the power of the test.

We next consider an alternative scenario in which a relatively large bandwidth $B$ is selected. In this case,
\begin{equation}
\begin{aligned}
    \left\vert\ 
    \mathbf{E}\left[\widehat{R}_k\right] - \frac{\vert\boldsymbol{\mu}_k  - \boldsymbol{\mu}_1\vert^2_2}{\sqrt{d}}\ \right\vert
    & \leq \frac{CdT_k}{V_k\sqrt{d}}\sum_{s = B}^{B_1}\frac{1}{(1  + s)^\alpha} + \frac{CdT_1}{V_1\sqrt{d}}\sum_{s = B}^{B_1}\frac{1}{(1  + s)^\alpha}\\
    &\leq \frac{C_1\sqrt{d}}{(B_1 - B)\times  B^{\alpha - 1}}.
\end{aligned}
\label{eq.bias}
\end{equation}
Although the bias from data dependence still remains, its magnitude 
can be  substantially reduced once statisticians use a sufficiently large bandwidth $B.$ Furthermore, as illustrated in Theorem \ref{theorem.consistent_ANOVA}, the stochastic error of $\widehat{R}_k$ is of order $O_p\left(\frac{1}{\sqrt{T_k(B_1 - B)}} +\frac{1}{\sqrt{T_1(B_1  - B)}}\right).$ Hence, with an appropriately chosen bandwidth $B,$ the bias introduced by data dependence becomes asymptotically negligible compared to the stochastic error.

\end{remark}

Apart from the bandwidth $B,$  we introduce the other bandwidth $B_1$ to the  estimator $\widehat{R}$ to mitigate the dependence among the product terms $\mathbf{x}_{t_1,k}^\top\mathbf{x}_{t_2,k},$  where $t_1,t_2 = 1,\cdots,T_k$ and $B\leq \vert t_1 -  t_2\vert\leq B_1.$ Similar to the phenomenon discussed in \cite{MR4829492}, even when the original time series $\mathbf{x}_{t,k}$ is short-range dependent, their products $\mathbf{x}_{t,k}^\top\mathbf{x}_{t - s,k}$ can exhibit long-range dependence when the time lag $s$ is large.  The bandwidth $B_1$ therefore decreases the dependence among the product terms by
disregarding the products $\mathbf{x}_{t_1,k}^\top \mathbf{x}_{t_2,k}$ when the time lag $\vert t_1 - t_2\vert$ exceeds $B_1.$

The selection of bandwidths $B$ and $B_1$ balances the trade-off between bias due to temporal dependence and stochastic error. To further illustrate, from \eqref{eq.bias}, increasing $B$ decreases the magnitude of the bias. On the other hand, Theorem \ref{theorem.consistent_ANOVA} shows that the stochastic error is of order $O_p\left(\frac{1}{\sqrt{T_k(B_1 - B)}} +\frac{1}{\sqrt{T_1(B_1  - B)}}\right),$ which adversely depends on the gap $B_1 - B.$ Since Theorem \ref{theorem.Gaussian_app} requires that $B_1$ remain significantly smaller than the sample size $T_k$  to maintain the Gaussian approximation of the test statistics, selecting a large $B$ inevitably narrows the gap $B_1 - B$, which in turn inflates the stochastic error.

Notice that suboptimal bandwidths $B$ and $B_1$ either increase the bias or enlarge the stochastic error, both of which inflate the overall estimation error.  Algorithm \ref{algorithm.selection} in Section \ref{section.numerical} leverages the work of \cite{MR2380557} and provides a data-driven procedure for selecting $B,B_1$ with the aim of minimizing the total estimation error.

\subsection{Bootstrap assisted hypothesis testing}
Bootstrap methods are powerful tools for analyzing statistics whose asymptotic distributions involve unknown parameters. Since the seminal work of \cite{MR515681}, they have been widely employed in various statistical applications, such as those introduced in \cite{MR868303, MR3750881, MR4298871, MR4441125}, among others. In the context of time series, the presence of complex covariance structures always complicates the variance of the test statistics, making direct estimation of the variance impractical. To address this issue, statisticians have developed various bootstrap algorithms for different setups---such as those in \cite{MR1310224, MR1983228, MR2795613, MR4388918, MR4595470}, among others---enabling hypothesis testing through simulation rather than explicit variance estimations.

Stationarity (weak or strict, as defined in Definitions 1.3.2 and 1.3.3 of \cite{MR1093459}) has become a standard assumption for the validity of classical bootstrap algorithms for time series, such as those in \cite{MR1310224, MR1466304, MR1872222}. However, modern-era time series may violate this assumption for various reasons, as mentioned in \cite{MR4270034}. The literature like \cite{MR2893863, MR4134800} has attempted to relax the stationarity assumptions.  Among them, \cite{MR2656050} introduced the ``dependent wild bootstrap'' for stationary time series, and subsequent works such as \cite{MR3798001, 10.1093/jrsssb/qkad006,MR4829492} adapted this algorithm to settings where the covariance structure changed over time. Apart from the dependent wild bootstrap,  alternative methods, including those in \cite{MR3798001, MR4718536},  have also been developed to address non-stationarity.

Apart from non-stationarity, another critical challenge in our setting is that the test statistic consists of products of time series data $\mathbf{x}_{t_1,k}^\top\mathbf{x}_{t_2,k}$---making its variance depending on fourth-order cumulants. Furthermore, our work allows for the marginal distributions of data to change over time,  these cumulants may also evolve. The original dependent wild bootstrap of \cite{MR2656050} was designed for linear forms of time series and could not capture higher-order moment information, as it weighted time series with normal random variables whose fourth-order cumulants were fixed. Similar limitations applied to the autoregressive sieve bootstrap \cite{MR2893863}. Meanwhile, the validity of other bootstrap algorithms, such as the frequency domain bootstrap \cite{MR4134800}, hinged on the assumption of stationary fourth-order cumulants, which is violated in our setting.

Rather than weighting the raw data with joint normal random variables, the second-order wild bootstrap of \cite{MR4829492} proposed to weight their ``second-order residuals.''  This modification allowed the bootstrap algorithm to capture fourth-order cumulant information, making it suitable for statistics consisting of products of time series.  However, their method was developed only for scalar time series. Building on this idea, Algorithm \ref{algorithm.bootstrap} accommodates their work to high-dimensional time series settings.

\begin{algorithm}
\caption{Second-order wild bootstrap assisted ANOVA}
\label{algorithm.bootstrap}
\begin{algorithmic}[1]
\Require Vector time series data $\mathbf{x}_{t,k}\in\mathbf{R}^d$ for $k = 1,\cdots, K$ and $t = 1,\cdots, T_k,$ $T_k$ here represents the sample size of each time series; bandwidths $B,B_1,H;$ a kernel function $\mathcal{K}(\cdot)$ satisfying Definition \ref{definition.kernel_function} below; the nominal size $\alpha;$ the number of bootstrap replicates $\mathcal{U}.$
\State Derive the test statistics $\widehat{R}$ as in \eqref{eq.def_RK}. After that, calculate  the fitted residuals 
    \begin{align*}
        \widehat{\boldsymbol{\mu}}_k = \frac{1}{T_k}\sum_{t = 1}^{T_k}\mathbf{x}_{t,k},\quad\text{and}\quad 
        \widehat{\boldsymbol{\epsilon}}_{t,k} = \mathbf{x}_{t,k} - \widehat{\boldsymbol{\mu}}_k.
    \end{align*}
\State Derive the ``second-order residuals'' 
    $$
    \widehat{\vartheta}_{t,k} = \sum_{t_2 = (t  - B_1)\vee 1}^{t  - B}\widehat{\boldsymbol{\epsilon}}_{t,k}^\top\widehat{\boldsymbol{\epsilon}}_{t_2,k}\quad \text{for } t = B+1, B+2, \cdots, T_k, \text{where }  k = 1,\cdots, K.
    $$

\For{$u = 1,2,\cdots, \mathcal{U}$}
    \State Generate random variables $\varepsilon^*_{t,k}$ for $k = 1,2,\cdots, K$ and $t = 1,2,\cdots, T_k,$ such that $\varepsilon^*_{1,k},\cdots \varepsilon^*_{T_k,k}$ have joint normal distribution with mean 0 and covariances $\mathrm{Cov}\left(\varepsilon^*_{t_1,k}, \varepsilon^*_{t_2,k}\right) = \mathcal{K}\left(\frac{t_1 - t_2}{H}\right),$ and $\varepsilon^*_{t_1,k_1}$ and $\varepsilon^*_{t_1,k_2}$ are independent for any $k_1\neq k_2$ and any $t_1,t_2.$  
    \State Define 
    $\widehat{A}_{t,k}^* = \widehat{\vartheta}_{t,k} \varepsilon^*_{t,k},$ and calculate
    \begin{equation}
    \begin{aligned}
        \widehat{S}^*_u = 2\sum_{k = 2}^K\frac{\sqrt{\mathcal{T}_\circ(B_1 - B)}}{V_k\sqrt{d}}\sum_{t = B+1}^{T_k}\widehat{A}_{t,k}^* + 2(K - 1)\frac{\sqrt{\mathcal{T}_\circ(B_1 - B)}}{V_1\sqrt{d}}\sum_{t = B+ 1}^{T_k}\widehat{A}_{t,1}^*,
    \end{aligned}
    \label{eq.bootstrap_S_u}
    \end{equation}
where $\mathcal{T}_\circ = \min(T_1,\cdots, T_K).$
\EndFor
\State Calculate the $1 - \alpha$ quantile of $\widehat{S}^*_u$: Sort the bootstrapped statistics to $\widehat{S}^*_{(1)}\leq \widehat{S}^*_{(2)}\leq \cdots \leq \widehat{S}^*_{(\mathcal{U})}$. Choose 
    \begin{align*}
        Q^*_{1 - \alpha} = \widehat{S}^*_{(v)},\quad\text{ where}\quad  v = \min\left\{x = 1,2,\cdots, \mathcal{U}: \frac{x}{\mathcal{U}}\geq 1 - \alpha\right\}.
    \end{align*}
\State Reject $H_0$ if the test statistic $\widehat{R}$ satisfy 
$$
\widehat{R}\geq Q^*_{1 - \alpha}.
$$
\end{algorithmic}
\end{algorithm}

\begin{remark}
The implementation of Algorithm \ref{algorithm.bootstrap} requires a kernel function $\mathcal{K}(\cdot)$ satisfying Definition \ref{definition.kernel_function}. As shown in
Remark \ref{remark.choose_K} of Section \ref{section.quadratic_form}, such kernel functions guarantee that the matrices $\left\{\mathcal{K}\left(\frac{t_1 - t_2}{H}\right)\right\}_{t_1,t_2 = 1,\cdots, T_k}, k = 1,\cdots,K$ are positive semi-definite. This property allows Algorithm  \ref{algorithm.bootstrap} to generate joint normal random variables with the corresponding covariance matrices
$\left\{\mathcal{K}\left(\frac{t_1 - t_2}{H}\right)\right\}_{t_1,t_2 = 1,\cdots, T_k}.$
For practical implementation, a convenient choice that satisfies Definition \ref{definition.kernel_function} is 
$
    \mathcal{K}(x) = \exp\left(-\frac{x^2}{2}\right).
$
\end{remark}

\section{Analysis of quadratic form of vector time series}
\label{section.quadratic_form}
This section introduces a class of non-stationary vector time series, referred to as $(M,\alpha)$-short-range dependent random vectors. We study the asymptotic behaviors of their quadratic forms, establishing a concentration inequality, a Gaussian approximation theorem, and a variance estimation procedure. These results serve as the foundation for Section \ref{section.Asymptotic_testing}, in which we derive the asymptotic distribution of the test statistic $\widehat{R}$ in \eqref{eq.def_R} and establish the asymptotic validity of Algorithm \ref{algorithm.bootstrap}.

Theoretical results on quadratic forms of time series have wide applications in the literature, as the analysis of many widely used estimators and test statistics---such as sample autocovariances, sample autocorrelations, sample covariance and precision matrices \cite{MR4206676, MR3161455, MR4134802}, sample spectral densities \cite{MR4206676, MR3161455, MR4134802, Chang25042025}, and portmanteau test statistics \cite{ad3868f4-d5d8-32fe-9700-07b98018b139}---relies on a solid understanding of quadratic forms. Prior studies, such as \cite{MR362777, MR642724, MR1617055, MR4206676, MR4829492, MR4699549, https://doi.org/10.1111/jtsa.12826}, have analyzed quadratic forms of time series.  However, to our knowledge, existing results either imposed stationarity and additional structural assumptions, such as the linear processes assumption, on observations, neglected the effect of high-dimensionality, or did not derive asymptotic distributions of the quadratic forms. This section simultaneously addresses these challenges and develops distributional results for quadratic forms of high-dimensional, non-stationary time series.

\subsection{($M,\alpha$)-short-range dependent random vectors}
\label{section.quadratic}
Let $\{e_t: t\in\mathbf{Z}\}$ be a sequence of independent (but not necessarily identically distributed) random variables. Suppose the vector time series $\boldsymbol{\epsilon}_t = \left(\boldsymbol{\epsilon}_t^{(1)},\cdots, \boldsymbol{\epsilon}_t^{(d)}\right)^\top\in\mathbf{R}^d$, where $t\in\mathbf{Z},$ satisfies the following form:
\begin{equation}
    \boldsymbol{\epsilon}_t^{(i)} = g_{t,T}^{(i)}\left(\cdots, e_{t - 1}, e_{t}\right),
    \label{eq.physical_dependence}
\end{equation}
the function $g_{t,T}^{(i)}$ in \eqref{eq.physical_dependence} is a measurable function of the $\sigma$-field generated by  $\cdots, e_{t - 1}, e_{t}.$ In other words, $\boldsymbol{\epsilon}_t$ is a function of random variables $\cdots, e_{t - 1}, e_t.$ The subscripts $t$ and $T$ here indicate that the function $g_{t,T}^{(i)}(\cdot)$ may vary both over time $t$ and with the sample size $T.$ With a slight abuse of notation, we omit the sample size $T$ when denoting  $\boldsymbol{\epsilon}_t.$  Since the functions are allowed to evolve with respect to the sample size $T,$  the distributions of  $\boldsymbol{\epsilon}_t$ may also change with the sample size.

The representation \eqref{eq.physical_dependence} is sufficient to capture potential non-stationarity in the data generating process. By allowing an evolving sequence of functions $g_{t,T}^{(i)}, i = 1,2,\cdots, d$ applied to random variables $\cdots, e_t,$ both the marginal distributions of $\boldsymbol{\epsilon}_t^{(i)}$ and their autocovariances may change over time $t,$ leading to non-stationary autocovariance structures.

For any $t\in\mathbf{Z}$, define $e_t^\dagger$ as mutually independent random variables such that $e_{t_1}$ is independent of $e_{t_2}^\dagger$ for any $t_1, t_2\in\mathbf{Z}$, and $e_t^\dagger$ has the same distribution as $e_t$ for any $t.$ Define the random vectors $ \boldsymbol{\epsilon}_t(s) = (\boldsymbol{\epsilon}_t^{(1)}(s),\cdots, \boldsymbol{\epsilon}_t^{(d)}(s))^\top\in\mathbf{R}^d,$ where 
\begin{equation}
    \boldsymbol{\epsilon}_t^{(i)}(s) = 
    \begin{cases}
        g_{t,T}^{(i)}\left(\cdots, e_{t - s - 1}, e_{t - s}^\dagger, e_{t - s + 1},\cdots, e_{t - 1}, e_{t}\right)\quad \text{if } s\geq 0,\\
        \boldsymbol{\epsilon}_t^{(i)}\quad \text{if } s < 0.
    \end{cases}
\label{eq.switch}
\end{equation}
For a given $M > 1$ and any $s\in\mathbf{Z}$, define 
\begin{equation}
    \delta_{s} = \sup_{t\in\mathbf{Z}, \vert\mathbf{a}\vert_2 = 1}
    \left\Vert\ \mathbf{a}^\top \left(\boldsymbol{\epsilon}_t - \boldsymbol{\epsilon}_t(s)\right)\ \right\Vert_M.
    \label{eq.def_deltas}
\end{equation}
According to this definition, $\delta_s = 0$ if $s < 0.$ With a slight abuse of notation, we omit $M$ in $\delta_s$ because $M$ is treated as a fixed number throughout the manuscript; and this abuse of notation should not cause confusion. For any given $M > 1$ and $\alpha > 1,$ we impose the following short-range dependence conditions.

\begin{definition}($(M,\alpha)$-short-range dependence)
\label{def.M_alpha_short_range}
    Suppose random vectors $\boldsymbol{\epsilon}_t, t\in\mathbf{Z}$ satisfy \eqref{eq.physical_dependence}. In addition, assume the following conditions hold true:  
    \begin{enumerate}
        \item $\mathbf{E}\left[\boldsymbol{\epsilon}_t\right]  = 0$ for all $t\in\mathbf{Z},$ and 
        $$
        \sup_{t\in\mathbf{Z}, \vert\mathbf{a}\vert_2 = 1}\left\Vert
        \mathbf{a}^\top \boldsymbol{\epsilon}_t
        \right\Vert_M = O(1).
        $$
    \item  With $\delta_s$ defined as in \eqref{eq.def_deltas}, 
    $$\sup_{s = 0,1,\cdots} (1 + s)^\alpha\sum_{j = s}^\infty  \delta_{s} = O(1).$$
    \end{enumerate}
     We call $\boldsymbol{\epsilon}_t$ an $(M,\alpha)$-short-range dependent random vector process.
\end{definition}

\begin{remark}
    Interest in the form \eqref{eq.physical_dependence} traces back to Wiener's conjecture that, under certain conditions, a stationary process could be expressed as a one-sided function of a sequence of i.i.d. random variables. This conjecture was later proven to be false by \cite{MR2493017}. The work of \cite{MR2172215} introduced the ``physical dependence measure'' to quantify temporal dependence; he also proposed a short-range dependence condition for time series of the form \eqref{eq.physical_dependence}. Later studies, such as \cite{MR3718156, MR3779697, MR4206676, zhang2023statisticalinferencehighdimensionalvector}, among others, extended this concept to vector time series data by imposing short-range dependence conditions either on individual components or on infinite norms. Definition \ref{def.M_alpha_short_range} is also motivated by \cite{MR2172215}. However, it imposes short-range dependence conditions on linear combinations of $\boldsymbol{\epsilon}_t.$ The reason for introducing such conditions is the need to bound the product terms $\boldsymbol{\epsilon}_{t_1}^\top \boldsymbol{\epsilon}_{t_2}$. Similar considerations can be found in condition (C3) of \cite{MR3911118}, where a mixing condition was applied to the components of $\boldsymbol{\epsilon}_{t}.$
\end{remark}

\begin{example}
    Suppose $e_t^{(i)}, t\in\mathbf{Z}, i = 1,\cdots,d$ are mutually independent random variables. Define  $\boldsymbol{\eta}_t = \left(\boldsymbol{\eta}_t^{(1)},\cdots, \boldsymbol{\eta}_t^{(d)}\right)^\top,$ where 
    $
    \boldsymbol{\eta}_t^{(i)} = g_t^{(i)}\left(\cdots, e^{(i)}_{t - 1}, e^{(i)}_t\right), 
    $
    and assume each scalar time series $\boldsymbol{\eta}_t^{(i)}$ satisfies Definition 2.1 of \cite{MR4829492} for all $i.$ Define $\boldsymbol{\epsilon}_t = \mathbf{A}\boldsymbol{\eta}_t,$ where $\mathbf{A}\in\mathbf{R}^{d\times d}$ has bounded singular values. In this setup, $\boldsymbol{\eta}_t$ has independent components, 
    and for any vector $\mathbf{a}\in\mathbf{R}^d$ with $\vert\mathbf{a}\vert_2 =  1,$ from From Theorem 2 of \cite{MR0133849},
    \begin{align*}
        \left\Vert
        \mathbf{a}^\top\boldsymbol{\epsilon}_t
        \right\Vert_{M} = \left\Vert \mathbf{a}\mathbf{A}\boldsymbol{\eta}_t \right\Vert_M\leq C\left\vert\mathbf{A}^\top \mathbf{a}\right\vert_2 = O(1).
    \end{align*}
    For each integer $s\geq 0,$ define $\boldsymbol{\eta}_t^{(i)}(s) = g_t^{(i)}\left(\cdots, e^{(i)}_{t - s - 1}, e^{(i)\dagger}_{t - s}, e^{(i)}_{t - s + 1},\cdots, e_t^{(i)}\right),$ then 
    $
    \boldsymbol{\epsilon}_t(s) = \mathbf{A}\boldsymbol{\eta}_t(s).
    $
    Since $\boldsymbol{\epsilon}_t^{(i)} - \boldsymbol{\epsilon}_t^{(i)}(s)$ and $\boldsymbol{\epsilon}_t^{(j)} - \boldsymbol{\epsilon}_t^{(j)}(s)$ are independent for different $i,j,$ define $\mathbf{b} = \mathbf{A}^\top\mathbf{a},$ we have 
    \begin{align*}
        \left\Vert
        \mathbf{a}^\top\left(\boldsymbol{\epsilon}_t - \boldsymbol{\epsilon}_t(s) \right)
        \right\Vert_M  & = \left\Vert
        \mathbf{a}^\top\mathbf{A}(\boldsymbol{\eta}_t - \boldsymbol{\eta}_t(s) )
        \right\Vert_M\\
        &\leq C\sqrt{\sum_{i = 1}^d \mathbf{b}^{(i)2}\left\Vert
        \boldsymbol{\eta}_t^{(i)} - \boldsymbol{\eta}_t^{(i)}(s) 
        \right\Vert^2_M}\\
        &\leq C\left\vert\mathbf{b}\right\vert_2 \max_{i = 1,\cdots, d}\left\Vert\boldsymbol{\eta}_t^{(i)} - \boldsymbol{\eta}_t^{(i)}(s)\right\Vert_M,
    \end{align*}
    which ensures $\sup_{s = 0,1,\cdots} (1 + s)^\alpha\sum_{j = s}^\infty  \delta_{s} = O(1).$
\label{example.example_Malpha}
\end{example}

\begin{example}
    This example considers a time-varying linear process 
    $$
    \boldsymbol{\epsilon}_t = \sum_{j = 0}^\infty \mathbf{A}_{j,t}\boldsymbol{\eta}_{t - j},
    $$
    where $\boldsymbol{\eta}_{t}\in\mathbf{R}^d$ are mutually independent random vectors with $\mathbf{E}\left[\boldsymbol{\eta}_{t}\right] = 0$ and $\left\Vert\mathbf{a}^\top \boldsymbol{\eta}_{t}\right\Vert_M\leq C\vert \mathbf{a}\vert_2$ for any vector $\mathbf{a}.$ Assume that $\mathbf{A}_{j,t}\in\mathbf{R}^{d\times d}.$ The subscript $t$ means that the coefficient matrices $\mathbf{A}_{j,t}$ are allowed to evolve over $t.$  For any linear combination vectors $\mathbf{a}$ with $\vert \mathbf{a}\vert_2 = 1,$ we have 
    \begin{align*}
        \left\Vert\mathbf{a}^\top\boldsymbol{\epsilon}_t\right\Vert_M\leq \sum_{j = 0}^\infty \left\Vert\mathbf{a}^\top\mathbf{A}_{j,t}\boldsymbol{\eta}_{t - j}\right\Vert_M\leq C\sum_{j = 0}^\infty\left\vert \mathbf{A}_{j,t}^\top\mathbf{a}\right\vert_2\leq C\sum_{j = 0}^\infty\left\vert \mathbf{A}_{j,t}\right\vert_2.
    \end{align*}
   For any integer $s\geq 0,$
    \begin{align*}
        \boldsymbol{\epsilon}_t(s) &= \sum_{j\geq 0, j\neq s}\mathbf{A}_{j,t}\boldsymbol{\eta}_{t - j} + \mathbf{A}_{s,t}\boldsymbol{\eta}_{t - s}^\dagger,
    \end{align*}
    which implies
    \begin{align*}
    \left\Vert
        \mathbf{a}^\top\left(\boldsymbol{\epsilon}_t - \boldsymbol{\epsilon}_t(s) \right)
        \right\Vert_M  &= \Vert 
        \mathbf{a}^\top\mathbf{A}_{s,t}\left(\boldsymbol{\eta}_{t - s} - \boldsymbol{\eta}_{t - s}^\dagger\right)
        \Vert_M
        \leq C\vert
        \mathbf{A}_{s,t}
        \vert_2.
    \end{align*}
    Assume that the matrix two-norm $\vert
        \mathbf{A}_{s,t}
        \vert_2\leq C(1 +s)^{-(\alpha + 1)},$ then  $\boldsymbol{\epsilon}_t, t\in\mathbf{Z}$ satisfy the $(M,\alpha)$-short-range dependent condition.
\end{example}

\subsection{Analysis of quadratic forms of vector time series}
\label{section.theoretical_quadratic}
Suppose $\boldsymbol{\epsilon}_t, t = 1,2,\cdots, T,$ stem from an $(M,\alpha)$-short-range dependent random vector process.  This section establishes theoretical results---including  concentration inequalities, Gaussian approximation, and variance estimation---for the following quadratic form of vector time series:
\begin{equation}
Q = \sum_{t_1 = 1}^T\sum_{t_2 = 1}^T b_{t_1t_2}\left(\boldsymbol{\epsilon}_{t_1}^\top\boldsymbol{\epsilon}_{t_2} - \mathbf{E}\left[\boldsymbol{\epsilon}_{t_1}^\top\boldsymbol{\epsilon}_{t_2}\right]\right),
\label{eq.def_quadratic_form}
\end{equation}
where $b_{t_1t_2}$ are scalar coefficients. Theorem \ref{theorem.consistent_quadratic} begins by introducing the concentration inequality for $Q.$
\begin{theorem}
    Suppose $\boldsymbol{\epsilon}_t, t = 1,2,\cdots, T,$ stem from an $(M,\alpha)$-short-range dependent random vector process with $M > 4, \alpha > 4,$ and $d \asymp T$. Suppose $\vert b_{t_1t_2}\vert\leq 1, t_1,t_2 = 1,2,\cdots, T$ satisfy $b_{t_1t_2} = 0$ if $\vert t_1 - t_2\vert < B$, and the positive integer $B > 1$ satisfies
    $
        T^{\frac{2}{\alpha}} = o(B).$
    Then we have
    \begin{equation}
    \begin{aligned}
        \left\Vert
        Q
        \right\Vert_{M/2}
        = O\left(\sqrt{d\sum_{t_1 = 1}^T\sum_{t_2 = 1}^T  b_{t_1t_2}^2} + \frac{T^{5/2}}{B^\alpha} + \frac{T^{3/2}}{B^{\alpha - 2}}\right).
    \end{aligned}
        \label{eq.whole_formula}
    \end{equation}
    In addition, there exists a constant $C>0$ such that for any integer $\ell\geq B$,
    \begin{equation}
    \begin{aligned}
        & \left\Vert
        \sum_{t_1 = 1}^T\sum_{t_2 = 1}^T b_{t_1t_2}\left(\boldsymbol{\epsilon}_{t_1}^\top \boldsymbol{\epsilon}_{t_2} - \mathbf{E}\left[\boldsymbol{\epsilon}_{t_1}^\top \boldsymbol{\epsilon}_{t_2}\mid\mathcal{F}_{t_1\vee t_2, \ell}\right]\right)
        \right\Vert_{M/2}\\
       &\leq C\sum_{q = 1}^\infty\delta_q\sqrt{d\sum_{\vert t_1 - t_2\vert  = B\vee (\ell + 1 - q)}^{\ell} b^2_{t_1t_2}}\\
       &+  C\sqrt{d\sum_{\vert t_1 - t_2\vert \geq \ell + 1} b^2_{t_1t_2}} + \frac{CdT^{3/2}}{\ell^{\alpha-1}} + \frac{Cd\sqrt{T}}{B^\alpha}.
    \end{aligned}
    \label{eq.truncate_moment}
    \end{equation}
\label{theorem.consistent_quadratic}
\end{theorem}

\begin{remark}
    To further illustrate the moment bounds in \eqref{eq.whole_formula},  assume that $\boldsymbol{\epsilon}_t^{(i)}$ are mutually independent across both $t$ and $i.$ In such case, by Theorem 2 of \cite{MR0133849}, 
    \begin{align*}
        \Vert Q\Vert_{M/2} &= \left\Vert\sum_{j = 1}^d\left(
        \sum_{t_1 = 1}^T\sum_{t_2 = 1}^T b_{t_1t_2}\left(\boldsymbol{\epsilon}_{t_1}^{(j)}\boldsymbol{\epsilon}_{t_2}^{(j)} - \mathbf{E}\left[\boldsymbol{\epsilon}_{t_1}^{(j)}\boldsymbol{\epsilon}_{t_2}^{(j)}\right]\right)
        \right)\right\Vert_{M/2}\\
        &\leq C\sqrt{\sum_{j = 1}^d \left\Vert\sum_{t_1 = 1}^T\sum_{t_2 = 1}^T b_{t_1t_2}\left(\boldsymbol{\epsilon}_{t_1}^{(j)}\boldsymbol{\epsilon}_{t_2}^{(j)} - \mathbf{E}\left[\boldsymbol{\epsilon}_{t_1}^{(j)}\boldsymbol{\epsilon}_{t_2}^{(j)}\right]\right)\right\Vert^2_{M/2}}
        \leq C_1\sqrt{d\sum_{t_1 = 1}^T\sum_{t_2 = 1}^T b_{t_1t_2}^2}.
    \end{align*}
    Therefore, \eqref{eq.whole_formula} attains the oracle bound for independent random variables provided that $T^{5/2} / B^\alpha$ and $T^{3/2} / B^{\alpha - 2}$ are negligible compared to  $\sqrt{d \sum_{t_1 = 1}^T\sum_{t_2 = 1}^T b_{t_1t_2}^2}.$ 

    Remark \ref{remark.dependence_decomposition}  in the online supplement decomposes $Q$ into a 
    martingale sequence with special characteristics, defined in \eqref{eq.b_first_bound}, and the remainder terms, defined in \eqref{eq.second_summ}.  It further shows that the sum of squared coefficients $\sqrt{d\sum_{t_1 = 1}^T\sum_{t_2 = 1}^T b_{t_1t_2}^2}$ is introduced by bounding the moments of the martingale sequence, while the terms $T^{5/2} / B^\alpha$ and $T^{3/2} / B^{\alpha - 2}$ in \eqref{eq.whole_formula} arise from bounding the moments of the remainder terms. Despite the fact that the remainder terms may exhibit complex dependence structures and long-range dependence,  their moments remain small, provided that the bandwidth $B$ is chosen sufficiently large. Accordingly, we adopt a moderately large bandwidth $B$ in $Q$ to mitigate dependence and ensure that the moment of $Q$ approximates that under independence.

    Equation \eqref{eq.whole_formula} also clarifies the impact of data dependence on the choice of bandwidth $B.$ In definition \ref{def.M_alpha_short_range}, larger values of $\alpha$ reflect weaker dependence. When $\alpha$ is close to 1,
    a large $B$ is required to offset the effects of data dependence.
\end{remark}

Our next result establishes the asymptotic normality of the quadratic form $Q.$ Equation \eqref{eq.truncate_moment} shows that the influence of coefficients $b_{t_1t_2}$ differ when approximating $Q$ by quadratic forms of their conditional expectations $\mathbf{E}\left[\boldsymbol{\epsilon}_{t_1}^\top \boldsymbol{\epsilon}_{t_2}\mid\mathcal{F}_{t_1\vee t_2, \ell}\right].$ Specifically, the sum of squares of $b_{t_1t_2}$ substantially contributes to the approximation error when the time lag exceeds $\ell.$ Concerning this, we introduce  an additional bandwidth $B_1$ and set $b_{t_1t_2} = 0$ if  the time lag is larger than $B_1$ to decrease the approximation error.

\begin{theorem}
    Suppose $\boldsymbol{\epsilon}_t, t = 1,\cdots,T$ stem from an $(M,\alpha)$-short-range dependent random vector process with $M >8,\alpha > 4,$ and $d \asymp T.$  Assume that the coefficients $b_{t_1t_2}$ satisfy the following conditions:
\begin{enumerate}
    \item   $\vert b_{t_1t_2}\vert\leq 1$ for $t_1,t_2 = 1,\cdots, T.$
    \item  There exists two integer bandwidths $B < B_1$ such that $ B \asymp T^{\kappa_1}$ and $ B_1 \asymp T^{\kappa_2},$ where $\frac{2}{\alpha} < \kappa_1 < \kappa_2 < 1,$ and    $b_{t_1t_2} = 0$ if $\vert t_1 - t_2\vert < B$ or $\vert t_1  - t_2\vert > B_1.$
    \item  $\sum_{t_1 = 1}^T\sum_{t_2 = 1}^T b_{t_1t_2}^2\asymp T(B_1 - B).$
\end{enumerate}
Suppose the scaling parameter $S_T \asymp \sqrt{Td(B_1 - B)},$ and  there exists a constant $c_1 > 0$ such that 
$$
\mathrm{Var}\left(\frac{Q}{S_T}\right) > c_1
$$
for sufficiently large $T.$
    Then we have 
    \begin{equation}
        \begin{aligned}
            \sup_{x\in\mathbf{R}}\left\vert
            \mathbf{Pr}\left(\frac{Q}{S_T}\leq x\right) - \mathbf{Pr}\left(\xi\leq x\right)
            \right\vert = o(1),
        \end{aligned}
        \label{eq.Gaussian_approx}
    \end{equation}
    where $\xi$ has normal distribution with mean $0$ and variance 
    $$
    \mathrm{Var}(\xi) = \mathrm{Var}\left(
    \frac{Q}{S_T}
    \right).
    $$
\label{theorem.Gaussian_app}
\end{theorem}
Since $\vert b_{t_1t_2}\vert\leq 1$ and $b_{t_1t_2} = 0$ if $\vert t_1 - t_2\vert < B$ or $\vert t_1  - t_2\vert > B_1,$ we have 
\begin{align*}
    \sum_{t_1 = 1}^T\sum_{t_2 = 1}^T b_{t_1t_2}^2\leq \sum_{B\leq \vert t_1 - t_2\vert\leq B_1} 1\leq CT(B_1 - B).
\end{align*}
Therefore, Condition 3 of Theorem \ref{theorem.Gaussian_app} actually requires that the sum of squares of $b_{t_1t_2}$ not be too small. Furthermore, from Theorem \ref{theorem.consistent_quadratic}, and notice that $S_T \asymp \sqrt{Td (B_1 - B)},$ we have
\begin{align*}
    \left\Vert Q\right\Vert_{2}\leq \left\Vert Q\right\Vert_{M/2}\leq C\sqrt{d\sum_{t_1 = 1}^T\sum_{t_2 = 1}^T  b_{t_1t_2}^2} + \frac{CT^{5/2}}{B^\alpha} + \frac{CT^{3/2}}{B^{\alpha - 2}}\leq C_1S_T,
\end{align*}
which implies that the variance of $\frac{Q}{S_T}$ is of order $O(1).$ The assumption $\mathrm{Var}\left(Q  /S_T\right) > c_1$ is thus introduced to ensure that the quadratic form does not degenerate to $0$ asymptotically.

Estimation of the variance $\mathrm{Var}(Q)$ is challenging in our setup because $Q$ involves products of time series data, whose variances and covariances depend on fourth-order cumulants. Furthermore, since we do not assume stationarity of these cumulants, their direct estimation becomes difficult. Theorem \ref{theorem.var_estimation} addresses this issue by presenting a heteroskedasticity and autocorrelation consistent (HAC) estimator for the variance of $Q.$

Since its introduction by \cite{MR890864}, the HAC estimator and its variants \cite{MR1106513,MR2748557, MR3132458, MR3238584}, among others, have become useful tools for estimating variances and covariances in heterogeneous data settings. The work of \cite{zhang2023statisticalinferencehighdimensionalvector,MR4829492} demonstrated that, when raw data are replaced by their products, the HAC estimator consistently estimates the variance of an estimator even in the presence of fourth-order cumulants. Theorem \ref{theorem.var_estimation} employs this idea in constructing the HAC estimator.

The implementation of the HAC estimator requires a kernel function satisfying certain regularity conditions, which are formally stated in Definition \ref{definition.kernel_function}.

\begin{definition}[Kernel function]
   Suppose a function $\mathcal{K}(\cdot):\mathbf{R}\to[0,\infty)$ be symmetric, continuously differentiable, $\mathcal{K}(0) = 1, \int_{\mathbf{R}} \mathcal{K}(x)\mathrm{d}x < \infty,$ and $\mathcal{K}(\cdot)$ is decreasing on $[0,\infty).$ Define the Fourier transformation of $\mathcal{K}$ as 
   $
   \mathcal{F}\mathcal{K}(x) = \int_{\mathbf{R}} \mathcal{K}(t)\exp(-2\pi\mathrm{i} tx)\mathrm{d}t, 
   $
   where $\mathrm{i} = \sqrt{-1}$. We assume $\mathcal{F}\mathcal{K}(x)\geq 0$ for all $x\in\mathbf{R}$ and $\int_{\mathbf{R}}\mathcal{F}\mathcal{K}(x)\mathrm{d}x < \infty$.
   \label{definition.kernel_function}
\end{definition}

\begin{remark}
\label{remark.choose_K}
Definition \ref{definition.kernel_function} imposes stronger conditions than the common assumptions about kernel functions in the literature,  such as those in \cite{MR3694571}.  In particular, we set constraints on the Fourier transform of $\mathcal{K}$ to guarantee that the matrix $\left\{\mathcal{K}\left(\frac{i - j}{H}\right)\right\}_{i,j = 1,\cdots, T}$ is positive semi-definite, where $H$ is a bandwidth defined in Theorem \ref{theorem.var_estimation}. To illustrate, for any vector $\mathbf{x}\in\mathbf{R}^{T},$  the Fourier inversion theorem (Theorem 8.26 of \cite{MR1681462}) yields
\begin{align*}
    \sum_{i = 1}^T\sum_{j = 1}^T \mathbf{x}^{(i)}\mathbf{x}^{(j)} \mathcal{K}\left(\frac{i - j}{H}\right)
    &= \int_{\mathbf{R}}\sum_{i = 1}^T\sum_{j = 1}^T \mathbf{x}^{(i)}\mathbf{x}^{(j)}\mathcal{F}\mathcal{K}(t)\exp\left(2\pi\mathrm{i}t\frac{i - j}{H}\right)\mathrm{d}t\\
    & = \int_{\mathbf{R}}\mathcal{F}\mathcal{K}(t)\left\vert
    \sum_{j = 1}^T\mathbf{x}^{(j)}\exp\left(2\pi\mathrm{i}\frac{j}{H}\right)
    \right\vert^2\mathrm{d}t\geq 0,
\end{align*}
which verifies that the matrix $\left\{\mathcal{K}\left(\tfrac{i - j}{H}\right)\right\}_{i,j=1,\cdots,T}$ is positive semi-definite.
\end{remark}

\begin{theorem}
    Suppose $\mathcal{K}(\cdot)$ is a kernel function satisfying Definition \ref{definition.kernel_function} and $H > 0$ is a bandwidth. For each $t = B + 1,\cdots, T,$ define 
    \begin{equation}
    \vartheta_t = 2\sum_{t_2 = (t - B_1)\vee 1}^{t - B} b^\circ_{tt_2}\boldsymbol{\epsilon}_{t}^\top \boldsymbol{\epsilon}_{t_2},\quad\text{where }  b^\circ_{tt_2} = \frac{\left(b_{tt_2} + b_{t_2t}\right)}{2}, 
    \label{eq.def_vartheta_t}
    \end{equation}
    the coefficients $b_{tt_2}$ coincide with $Q$ in \eqref{eq.def_quadratic_form}, and  $\boldsymbol{\epsilon}_t, t = 1,\cdots,T$ stem from an $(M,\alpha)$-short-range dependent random vector process with $M >8$ and $\alpha > 14.$ Suppose $d\asymp T,$ $B \asymp T^{\kappa_1}$ with $\frac{2}{\alpha} < \kappa_1< 1/6,$ 
    $B_1 \asymp T^{\kappa_2}$ with $\kappa_1\vee\frac{4}{2\alpha - 3} <\kappa_2 < 1/6,$ $H \asymp T^{1/3}\log(T),$ and the coefficients $b_{t_1t_2}$  satisfy conditions 1-3 in Theorem \ref{theorem.Gaussian_app}. Then
    \begin{equation}
    \begin{aligned}
        & \left\Vert\frac{1}{S^2_T}\sum_{t_1 = B + 1}^T\sum_{t_2 = B+1}^T \mathcal{K}\left(\frac{t_1 - t_2}{H}\right)\vartheta_{t_1}\vartheta_{t_2}
         - \mathrm{Var}\left(
    \frac{Q}{S_T}
    \right)\right\Vert_{M/4}\\
        &= O\left(\frac{\log(T)}{T^{1/6 - \kappa_2}} + \frac{1}{T^{\kappa_2 / 2}}\right),
    \end{aligned}
    \label{eq.var_esti}
    \end{equation}
    where $S_T$ coincides with Theorem \ref{theorem.Gaussian_app}.
    \label{theorem.var_estimation}
\end{theorem}

\begin{remark}
    The condition $b_{t_1t_2} = 0$ if $\vert t_1 - t_2\vert < B$  is necessary to maintain the consistency of the HAC estimator in \eqref{eq.var_esti}. Notice that 
    \begin{align*}
        \mathbf{E}\left[\vartheta_t\right] = 2\sum_{t_2 = (t - B_1)\vee 1}^{t - B} b^\circ_{tt_2}\mathbf{E}\left[\boldsymbol{\epsilon}_{t}^\top \boldsymbol{\epsilon}_{t_2}\right] = O\left(\frac{d}{B^{\alpha - 1}}\right),
    \end{align*}
    which in general does not equal $0.$ Moreover, estimating $\mathbf{E}\left[\vartheta_t\right]$ is difficult since no structural or stationary assumptions are imposed on the data covariances.
    After choosing a sufficiently large $B,$ however, such expectation becomes negligible compared to the stochastic order of $\vartheta_t,$ thus eliminating the need to estimate $\mathbf{E}\left[\vartheta_t\right].$
\end{remark}

\section{Asymptotic theory for the testing procedure}
\label{section.Asymptotic_testing}
Section \ref{section.quadratic_form} has established the asymptotic properties of the quadratic form $Q$ in \eqref{eq.def_quadratic_form}.  Building on these results, this section analyzes the test statistic $\widehat{R}$ defined in \eqref{eq.def_RK}.  Specifically, it derives the asymptotic distribution of $\widehat{R}$ under $H_0,$ and establishes the bootstrap consistency. We also analyze the power performance under $H_1. $  The following assumptions are imposed to guarantee the validity of the test procedure.

\textbf{Assumptions: }  1. Suppose $e_{t,k},$ where $t\in\mathbf{Z}$ and $k = 1,2,\cdots, K,$ are mutually independent (but not necessarily identically distributed) random variables. Assume that 
\begin{equation}
\boldsymbol{\epsilon}_{t,k}^{(i)} = g_{t,k}^{(i)}\left(\cdots, e_{t-1, k}, e_{t,k}\right)\quad \text{for any $t\in\mathbf{Z},$ $k = 1,\cdots, K,$ and $ i = 1,\cdots,d.$ }
\end{equation}
Furthermore, for each $k = 1,\cdots, K,$ suppose $\boldsymbol{\epsilon}_{t,k}, t = 1,2,\cdots, T_k,$ stem from an $(M,\alpha)$-short-range dependent random vector process with $M > 8$ and $\alpha > 14.$

2. Define $\mathcal{T}_\circ = \min\left(T_1,\cdots, T_K\right)$ and $\mathcal{T}_\dagger = \max\left(T_1,\cdots, T_K\right).$ Assume that $K = O(1),$ $\mathcal{T}_\circ \asymp \mathcal{T}_\dagger,$ and  $d\asymp \mathcal{T}_\circ.$ Assume $\max_{k = 1,\cdots, K}\vert\boldsymbol{\mu}_k\vert_\infty = O(1).$

3. Suppose $B \asymp \mathcal{T}_\circ^{\kappa_1}$ with $\frac{2}{\alpha} < \kappa_1< 1/6,$ $B_1\asymp \mathcal{T}_\circ^{\kappa_2}$ with $\kappa_1\vee\frac{4}{2\alpha - 3} <\kappa_2 < 1/6,$ 
and $H \asymp \mathcal{T}_\circ^{1/3}\log(\mathcal{T}_\circ).$

4. Suppose there exists a constant $c>0$ such that 
\begin{align*}
    & \sum_{k = 2}^K \mathrm{Var}\left(\frac{\sqrt{\mathcal{T_\circ}(B_1 - B)}}{V_k\sqrt{d}}\sum_{B\leq \vert t_1 - t_2\vert\leq B_1}^{T_k}\left(\boldsymbol{\epsilon}_{t_1,k}^\top\boldsymbol{\epsilon}_{t_2,k} - \mathbf{E}\left[\boldsymbol{\epsilon}_{t_1,k}^\top\boldsymbol{\epsilon}_{t_2,k}\right]\right)\right)\\
    &+ (K - 1)^2 \mathrm{Var}\left(\frac{\sqrt{\mathcal{T}_\circ(B_1 - B)}}{V_1\sqrt{d}}\sum_{B\leq \vert t_1 - t_2\vert\leq B_1}^{T_1}\left(\boldsymbol{\epsilon}_{t_1,k}^\top\boldsymbol{\epsilon}_{t_2,k} - \mathbf{E}\left[\boldsymbol{\epsilon}_{t_1,k}^\top\boldsymbol{\epsilon}_{t_2,k}\right]\right)\right) > c 
\end{align*}
for sufficiently large $\mathcal{T}_\circ.$

Assumption 1 introduces the independent random variables $e_{t,k}$ to preserve the representation in \eqref{eq.physical_dependence} for the time series data, while simultaneously ensuring independence of time series across different populations.  Assumption 2 constraints that the sample sizes of different populations are of the same order. Besides, Assumption 2 allows the dimension $d$ to exceed the largest sample size $\mathcal{T}_\dagger,$ but this work does not accommodate the cases where  $d$ grows substantially faster than  $\mathcal{T}_\dagger.$   Notably, under assumption 2, we have 
$$
\left\vert \boldsymbol{\mu}_k\right\vert_2 = \sqrt{\sum_{i = 1}^d \boldsymbol{\mu}_k^{(i)2}}\leq C\sqrt{d}\asymp \sqrt{\mathcal{T}_\circ}.
$$
Similar to Theorem \ref{theorem.Gaussian_app},  Assumption 4 ensures that the test statistic does not asymptotically degenerate to $0.$

Theorem \ref{theorem.consistent_ANOVA} establishes the consistency of the test statistic and further shows that,  under $H_0,$ its distribution can be approximated by a normal distribution with 0 mean.

\begin{theorem}
    Suppose Assumptions 1-4 hold true, then 
    \begin{equation}
        \left\Vert
        \ 
        \widehat{R} - \frac{1}{\sqrt{d}}\sum_{k = 2}^K\left\vert \boldsymbol{\mu}_k  - \boldsymbol{\mu}_1\right\vert_2^2
        \ 
        \right\Vert_{M/2} = O\left(\frac{1}{\sqrt{\mathcal{T}_\circ}}\right).
        \label{eq.consistency_R}
    \end{equation}
    Moreover, under $H_0,$
    \begin{equation}
        \sup_{x\in\mathbf{R}}\left\vert
        \mathbf{Pr}\left(\sqrt{\mathcal{T}_\circ (B_1- B)}\widehat{R}\leq x\right) - \mathbf{Pr}(\zeta \leq x)
        \right\vert = o(1),
        \label{eq.prob_res}
    \end{equation}
    where $\zeta$ has normal distribution with mean $0$ and variance
    \begin{equation}
    \begin{aligned}
         &\mathrm{Var}(\zeta) = \sum_{k = 2}^K \mathrm{Var}\left(\frac{\sqrt{\mathcal{T_\circ}(B_1 - B)}}{V_k\sqrt{d}}\sum_{B\leq \vert t_1 - t_2\vert\leq B_1}^{T_k}\left(\boldsymbol{\epsilon}_{t_1,k}^\top\boldsymbol{\epsilon}_{t_2,k} - \mathbf{E}\left[\boldsymbol{\epsilon}_{t_1,k}^\top\boldsymbol{\epsilon}_{t_2,k}\right]\right)\right)\\
         & + (K - 1)^2 \mathrm{Var}\left(\frac{\sqrt{\mathcal{T}_\circ(B_1 - B)}}{V_1\sqrt{d}}\sum_{B\leq \vert t_1 - t_2\vert\leq B_1}^{T_1}\left(\boldsymbol{\epsilon}_{t_1,k}^\top\boldsymbol{\epsilon}_{t_2,k} - \mathbf{E}\left[\boldsymbol{\epsilon}_{t_1,k}^\top\boldsymbol{\epsilon}_{t_2,k}\right]\right)\right).
    \end{aligned}
    \label{eq.var_zeta}
    \end{equation}
    \label{theorem.consistent_ANOVA}
\end{theorem}
It is not surprising that the test statistic $\widehat{R}$ follows an asymptotic normal (rather than $\chi^2$) distribution under $H_0,$ and similar results can be found in \cite{MR2604697,MR3992401}. This asymptotic normality arises because the data dimension grows to infinity with $\mathcal{T}_\circ$.

Since the asymptotic distribution of $\widehat{R}$ is known,  statisticians can approximate the quantiles of the test statistic using those of $\zeta$ and construct a rejection region to control the Type-I error, provided its variance is estimated. However, the variance $\mathrm{Var}(\zeta)$ in \eqref{eq.var_zeta} has a complicated structure because $\widehat{R}$ involves products of time series, and the data exhibit temporal dependence and non-stationarity.  Consequently, direct estimation of this variance requires sophisticated calculations. Algorithm \ref{algorithm.bootstrap} addresses this difficulty by implicitly estimating the variance, thereby assisting hypothesis testing through simulation. We prove this claim in Theorem \ref{Theorem.bootstrap}.

\begin{remark}
\label{remark.power_of_test}
Theorem \ref{theorem.consistent_ANOVA} also assists the analysis of power properties under $H_1.$ In Corollary \ref{corollary.power_performance} of the online supplement, we demonstrate that the power converges to 1 if 
$$
1 = o\left(\sum_{k = 2}^K\vert \boldsymbol{\mu}_k  - \boldsymbol{\mu}_1\vert_2^2\right).
$$
This condition is mild in the high-dimensional setting, where the Euclidean norm $\vert\boldsymbol{\mu}_k\vert_2$ often has order $O(\sqrt{d}).$
\end{remark}

We use the notation 
$$
\mathbf{Pr}^*(\cdot) = \mathbf{Pr}\left(\cdot\mid \mathbf{x}_{t,k},\ \text{where } k = 1,\cdots, K\ \text{and } t = 1,\cdots, T_k\right),
$$
which denotes the conditional probability conditional on all observed data. In the literature such as \cite{MR1707286}, this conditional probability is commonly referred to as the ``probability in the bootstrap world.''  According to  \cite{MR1707286}, the consistency of Algorithm \ref{algorithm.bootstrap} can be established once the following claim holds true:
\begin{equation}
        \sup_{x\in\mathbf{R}}\left\vert
        \mathbf{Pr}^*\left(\widehat{S}^*_u\leq x\right) - \mathbf{Pr}\left(\zeta\leq x\right)
        \right\vert = o_p(1),
        \label{eq.consistency_bootstrap}
\end{equation}
where $\widehat{S}^*_u$ is defined in Algorithm \ref{algorithm.bootstrap} and $\zeta$ is the random variable defined in Theorem \ref{theorem.consistent_ANOVA}. This result is established in Theorem \ref{Theorem.bootstrap}.
\begin{theorem}
    Suppose Assumptions 1-4, then Equation \eqref{eq.consistency_bootstrap} holds true.
    \label{Theorem.bootstrap}
\end{theorem}

\begin{remark}
    Since $\varepsilon^*_{t,k}$ in Algorithm \ref{algorithm.bootstrap} follows a joint normal distribution, $\widehat{S}^*_u$ in Algorithm \ref{algorithm.bootstrap}, which is a linear combination of $\varepsilon^*_{t,k},$ is also normally distributed in the bootstrap world. Hence, the validity of Algorithm \ref{algorithm.bootstrap} depends on the conditional variance of $\widehat{S}^*_u$ being close to $\mathrm{Var}(\zeta).$ 
\end{remark}

\section{Numerical experiment}
\label{section.numerical}
This section examines the finite-sample performance of the proposed test statistic as well as the associated bootstrap algorithm under both independent and dependent data settings. In addition, it applies the proposed methods to the weekly chickenpox case counts dataset to assess the temporal dynamics of chickenpox cases across different cities.

\subsection{Bandwidths selection} 
\label{section.bandwidth_selection}
The proposed test procedure and the associated bootstrap algorithm involve three hyperparameters: 
the lower and upper bandwidths of the test statistic, $B,B_1,$ and the bootstrap bandwidth $H.$ The choice of $B$ and $B_1$ is motivated by the proof of Theorem \ref{theorem.consistent_quadratic}, which simultaneously controls the bias due to data dependence and the stochastic error. From Theorem \ref{theorem.consistent_quadratic}, and Lemma \ref{lemma.linear_form} in the online supplement, we know that the statistic $\widehat{Z}_k$ in Algorithm \ref{algorithm.selection} consistently estimates $\vert\boldsymbol{\mu}_k\vert_2^2.$ Accordingly, Algorithm \ref{algorithm.selection} selects a subseries and constructs the statistic $\widehat{Z}_k^\dagger(B,B_1)$ for different bandwidth choices $B$ and $B_1,$ then chooses the bandwidths $B,B_1$ that minimize the discrepancy between $\widehat{Z}_k$ and $\widehat{\mathcal{Z}}_k^\dagger(B,B_1).$

\begin{algorithm}
\caption{Selection of $B$ and $B_1.$}
\label{algorithm.selection}
\begin{algorithmic}[1]
\Require Vector time series data $\mathbf{x}_{t,k}$ for $k = 1,\cdots, K, t = 1,\cdots, T_k,$ pre-chosen bandwidths $\mathcal{B},\mathcal{B}_1.$ The ratio for subsamples $\beta\in(0,1).$  Set $\mathcal{S}_B$ of potential  $B,$ set $\mathcal{S}_{B_1}$ of potential $B_1.$ \Comment{In this manuscript, we set the ratio for subsamples $\beta = 0.3,$ and the pre-chosen bandwidths $\mathcal{B} = 10,$ $\mathcal{B}_1 = 15.$ }

\State  Derive the statistics 
$$
\widehat{Z}_k = \frac{1}{V_k\sqrt{d}}\sum_{\mathcal{B}\leq \vert t_1 - t_2\vert\leq \mathcal{B}_1}^{T_k}\mathbf{x}_{t_1,k}^\top \mathbf{x}_{t_2,k}
$$

\For{$B\in \mathcal{S}_B$}
\For{$B_1\in \mathcal{S}_{B_1}, B_1 > B$}
    \State Define $\mathcal{T}_k^\dagger = \lfloor\beta T_k\rfloor, $ $\mathcal{B}^\dagger = \lfloor \beta \times B\rfloor, \mathcal{B}_1^\dagger = \lfloor \beta \times B\rfloor$ and calculate 
    $$
    \widehat{\mathcal{Z}}_k^\dagger(B,B_1) = \frac{1}{\mathcal{V}^\dagger_k\sqrt{d}}\sum_{\mathcal{B}^\dagger\leq \vert t_1 - t_2\vert\leq \mathcal{B}_1^\dagger}^{\mathcal{T}_k^\dagger}\mathbf{x}_{t_1,k}^\top \mathbf{x}_{t_2,k},
    $$
    where $\mathcal{V}^\dagger_k = \sum_{\mathcal{B}^\dagger\leq \vert t_1 - t_2\vert\leq \mathcal{B}_1^\dagger}^{\mathcal{T}_k^\dagger}1 =  (2\mathcal{T}_k^\dagger - \mathcal{B}^\dagger - \mathcal{B}^\dagger_1)\times (\mathcal{B}^\dagger_1 -  \mathcal{B}^\dagger +1).$
\EndFor
\EndFor
\State Choose the combination $B, B_1$ that minimize 
$$
\sum_{k = 1}^K \vert\widehat{Z}_k  -  \widehat{\mathcal{Z}}_k^\dagger(B,B_1)\vert.
$$
\end{algorithmic}
\end{algorithm}

The selection of the bootstrap bandwidth $H$ is more complicated. The literature has offered several bandwidth selection algorithms under different bootstrap settings, such as those in \cite{MR1366282,MR2041534,MR2380557}. The original dependent wild bootstrap of  \cite{MR2656050} leveraged the work of \cite{MR2041534} for selecting the bandwidth $H.$ Our work adopts the bandwidth selection method of \cite{MR2380557}, which is implemented in the R package ``blocklength'' \footnote{See \url{10.32614/CRAN.package.blocklength}}. Since the method mentioned in \cite{MR2380557} only supports scalar time series, we apply it to products $\widehat{\vartheta}_{t,k} = \sum_{t_2 = (t  - B_1)\vee 1}^{t  - B}\widehat{\boldsymbol{\epsilon}}_{t,k}^\top\widehat{\boldsymbol{\epsilon}}_{t_2,k},$ and concatenate all $\widehat{\vartheta}_{t,k}$ into a single vector before applying the selection algorithm.

\subsection{Simulation data} 
In this section, we construct various types of independent and dependent data and examine the performance of the proposed test statistics and the associated bootstrap algorithm under both the null and alternative hypothesis. We assume 
$$
\mathbf{x}_{t,k} = \boldsymbol{\mu}_k + \boldsymbol{\epsilon}_{t,k},\quad\text{where}\quad\boldsymbol{\mu}_2 = \boldsymbol{\mu}_3 = \cdots = \boldsymbol{\mu}_K = (1,1,\cdots, 1)^\top.
$$
We then set 
\begin{equation}
    \boldsymbol{\mu}_1 = \left(\boldsymbol{\mu}_1^{(1)},\cdots, \boldsymbol{\mu}_1^{(d)}\right)^\top,\quad\text{where}\quad \boldsymbol{\mu}_1^{(k)} = 1.0 + \boldsymbol{\nu}_1^{(k)},
    \label{eq.nu1_add_mu}
\end{equation}
and the random variables $\boldsymbol{\nu}_1^{(k)}$ follow the distributions specified in Tables \ref{table.two_sample_res} and \ref{table.ANOVA_res}. Suppose $\mathbf{e}_{t,k}^{(i)}$ are mutually independent random variables with uniform distribution on $[-1, 1].$ Following the work of \cite{MR4829492}, we define 
$$
\boldsymbol{\epsilon}_{t,k} = \boldsymbol{\Theta}\boldsymbol{\varepsilon}_{t,k}\quad \text{where}\quad  \boldsymbol{\Theta} = \left[
\begin{matrix}
    1 & 0.5 & 0.3 & 0 & 0 & 0 &\cdots & 0\\
    0.5 & 1 & 0.5 & 0.3 & 0 & 0 &\cdots & 0\\
    0.3 & 0.5 & 1 & 0.5 & 0.3 & 0 & \cdots & 0\\
    0 & 0.3 & 0.5 & 1 & 0.5 & 0.3 & \cdots & 0\\
    \vdots & \vdots & \vdots & \vdots & \vdots & \vdots & \cdots & \vdots\\
    0 & 0 & 0 & 0 & 0 & 0 &\cdots & 1
\end{matrix}
\right],
$$
and consider the following types of $\boldsymbol{\varepsilon}_{t,k}:$
\begin{itemize} 
\item \textbf{Spatially independent innovations.} Here we directly set $\boldsymbol{\epsilon}_{t,k}$ (rather than $\boldsymbol{\varepsilon}_{t,k}$) equal to $\mathbf{e}_{t,k}$ to eliminate the spatial and temporal dependence.
    \item \textbf{Independent innovations. } $\boldsymbol{\varepsilon}_{t,k}^{(i)} = \mathbf{e}_{t,k}^{(i)}$ for any $t,k,i.$
    \item \textbf{Autoregressive innovations. } $\boldsymbol{\varepsilon}_{t,k}^{(i)} = 0.7\boldsymbol{\varepsilon}_{t - 1,k}^{(i)} + 0.2\boldsymbol{\varepsilon}_{t - 2,k}^{(i)} + \mathbf{e}_{t,k}^{(i)}.$
    \item \textbf{Moving-average innovations. }  $\boldsymbol{\varepsilon}_{t,k}^{(i)} = \mathbf{e}_{t,k}^{(i)} + 0.6 \mathbf{e}_{t - 2,k}^{(i)} + 0.4 \mathbf{e}_{t - 5,k}^{(i)} + 0.3 \mathbf{e}_{t - 7,k}^{(i)}.$
    \item \textbf{Non-linear innovations.}  $\boldsymbol{\varepsilon}_{t,k}^{(i)} = \sin(\boldsymbol{\varepsilon}_{t - 1,k}^{(i)}) + \mathbf{e}_{t - 1,k}^{(i)}\times \mathbf{e}_{t,k}^{(i)}$
    \item \textbf{Non-stationary innovations. }
    $\boldsymbol{\varepsilon}_{t,k}^{(i)} = \sin(\boldsymbol{\varepsilon}_{t - 1,k}^{(i)}) + 
    \cos(\boldsymbol{\varepsilon}_{t - 4,k}^{(i)})
    + \begin{cases}
        \mathbf{e}_{t,k}^{(i)}\ \text{if } i \bmod 2 = 0,\\
        \mathbf{e}_{t - 2,k}^{(i)}\times \mathbf{e}_{t,k}^{(i)}\ \text{if } i \bmod 2 = 1
    \end{cases}.
    $
\end{itemize}

From these constructions, the ``spatially independent innovations'' exhibit neither spatial (element-wise) nor temporal dependence, while the ``independent innovations'' exhibit spatial dependence, but remain temporally independent. All other types of innovations exhibit both spatial and temporal dependence.  As shown in Example \ref{example.example_Malpha}, these constructions satisfy the $(M,\alpha)$-short-range dependence condition in Definition \ref{def.M_alpha_short_range}.  In the experiments, we set the number of bootstrap replicates ($\mathcal{U}$ in Algorithm \ref{algorithm.bootstrap}) to 100 and 
repeat the procedure 100 times to evaluate the size and power of the test. Numerical results are reported in Tables \ref{table.two_sample_res} and \ref{table.ANOVA_res}, and Figure \ref{figure,all_power}.

We begin with the two-sample test ($K = 2$), where Algorithm \ref{algorithm.bootstrap} is compared against several existing high-dimensional two-sample test procedures, including those proposed in \cite{MR2396970, MR2604697, MR3164870, MR3551787}. According to Table \ref{table.two_sample_res}, Algorithm \ref{algorithm.bootstrap} achieves the nominal size and demonstrates good power performance under both independent and dependent data settings.  In contrast, methods designed for independent data perform well in terms of size and power when the temporal independence condition holds, but either do not adequately control size or exhibit reduced power when temporal dependence is present. This suboptimal performance arises from the bias induced by data dependence and highlights the importance of explicitly accounting for temporal dependence in high-dimensional two-sample tests.

Table \ref{table.ANOVA_res} and Figure \ref{figure,all_power} present the ANOVA results $(K = 6)$ for the simulation data.  Similar to the two-sample test setting, Algorithm \ref{algorithm.bootstrap} maintains the nominal size and demonstrates good power under temporal dependence. Nevertheless, Figure \ref{figure,all_power} highlights that spatial and temporal dependence can negatively affect the power of the test.

\begin{table}[h]
    \centering
    \scriptsize
    \caption{Performance of the proposed estimator $\widehat{R}$ and Algorithm \ref{algorithm.bootstrap} for the two-sample test. ``CQ'' refers to \cite{MR2604697}, ``Cai'' to \cite{MR3164870}, ``SD''  to \cite{MR2396970}, and ``Xu'' to \cite{MR3551787}. The sample sizes considered are $150$ and  $200,$ with dimension $250.$ The number of bootstrap replicates is set to 100,  and the nominal size is 5\%. The test CQ, Cai, and SD are implemented using the R-package ``highmean (see \url{10.32614/CRAN.package.highmean}).''  }
    \begin{tabular}{c c c c c c c c c}
    \hline\hline
    $\boldsymbol{\nu}_1^{(k)}$ & Distance & Innovation Type &     $(B,B_1,H)$&   \multicolumn{5}{c}{Test Type }             \\
    \cline{5-9}
                                                  &        &      &  &    Ours & CQ   & Cai &  SD & Xu\\                                            
     0($H_0$)            & 0   & \multirow{3}{*}{\textbf{Spatial independent}}       & \multirow{3}{*}{(23,31, 86.1)} &  7\%   & 7\% & 0\% & 0\% & 3\% \\ 
     Unif[0,0.5]  &   1.4 & && 100\% & 99\% & 55\% & 70\% & 99\%\\
     Unif[0,1]    &   5.4 & && 100\% & 100\% & 100\% & 100\% & 100\%\\
     \hline
     0 ($H_0$)      &  0  & \multirow{3}{*}{\textbf{Independent}} & \multirow{3}{*}{(10,30, 58.3)}   &   9\%  & 41\% & 6\% & 10\% & 37\%\\ 
    Unif[0,0.5]  &  1.3  & &    &   100\% & 76\% & 36\% & 44\% & 75\%\\ 
    Unif[0, 1]   &  5.3  & &    &   100\% & 98\% & 85\% & 90\% & 97\%\\
    \hline
     0($H_0$)            & 0   & \multirow{3}{*}{\textbf{Autoregressive}} & \multirow{3}{*}{(31,36, 148.3)} & 9\% & 31\% & 4\% & 6\%  & 27\%\\
     Unif[0,0.5]     & 1.4  & & & 39\% & 50\% & 20\% & 22\% & 47\% \\
     Unif[0,1]       & 5.1  & & & 78\% & 64\% & 41\% & 47\% & 63\%\\
     \hline
      0($H_0$)           & 0   & \multirow{3}{*}{\textbf{Moving-average}} & \multirow{3}{*}{(32,45,70.4)} & 2\% & 36\% & 6\% & 12\%  & 28\% \\
    Unif[0,0.5]     & 1.4 & &&97\% & 65\% &  22\% & 33\% & 60\%\\
    Unif[0,1]      &  5.3 & &&100\% & 93\% & 67\% & 75\% & 93\%\\
    \hline 
    0($H_0$)           & 0   & \multirow{3}{*}{\textbf{Non-linear}} & \multirow{3}{*}{(39,46,89.2)} &    1\%  & 28\% & 8\% & 13\% & 26\%\\
    Unif[0,0.5]    & 1.2  & && 35\% &  75\% & 38\% & 48\% & 71\% \\
    Unif[0,1]      &  5.2 & && 100\% & 96\% & 83\% & 88\% & 96\%\\
    \hline
     0($H_0$)            & 0   & \multirow{3}{*}{\textbf{Non-stationary}}       & \multirow{3}{*}{(30, 43, 125.7)} &  5\%   & 38\% & 3\% & 7\%  & 35\%\\ 
     Unif[0,0.5]  &   1.3 & && 79\% & 59\% & 19\% & 21\%  & 55\%\\
     Unif[0,1]    &   5.4 & && 100\% & 82\% & 39\% & 53\% & 78\% \\
    \hline\hline
    \end{tabular}
    
    \vspace{0.2cm}

    \textbf{Remark: }  \textit{In Tables \ref{table.two_sample_res} and \ref{table.ANOVA_res}, $\boldsymbol{\nu}_1^{(k)}$ denotes the distribution of the random variables added to the population mean $\boldsymbol{\mu}_1^{(k)},$ as defined in \eqref{eq.nu1_add_mu}. Therefore, $\boldsymbol{\nu}_1^{(k)} = 0$ corresponds to   $H_0.$ The ``Distance'' column records the term  
$
    \frac{1}{\sqrt{d}}\sum_{k = 2}^K\vert
    \boldsymbol{\mu}_k - \boldsymbol{\mu}_1 
    \vert_2^2.
$}
    \label{table.two_sample_res}
\end{table}

\begin{table}[h]
    \centering
    \caption{Performance of the proposed estimator $\widehat{R}$ and Algorithm \ref{algorithm.bootstrap} for ANOVA. We choose $K = 6, $ sample size $150 + 5 k$ for $k = 1,2,\cdots, K,$ and the dimension $d = 250.$}
    \scriptsize
    \begin{tabular}{c c c c c}
    \hline\hline
    $\boldsymbol{\nu}_1^{(k)}$ & Distance & Innovation Type &    $(B,B_1,H)$&   Size / Power           \\
    \hline
         0 (N)      &  0  & \multirow{3}{*}{\textbf{Spatial independent}} & \multirow{3}{*}{(5,8,13.7)}  &  6\%   \\ 
    Unif[0,0.5]  &  6.5  & &   & 100\%\\ 
    Unif[0, 1]   &  24.9 & &    &  100\%\\
    \hline 
     0 (N)      &  0  & \multirow{3}{*}{\textbf{Independent}} & \multirow{3}{*}{(29,46,102.2)}  &  10\%   \\ 
    Unif[0,0.5]  &  6.8  & &   & 100\%\\ 
    Unif[0, 1]   &  25.8 & &    &  100\%\\
    \hline
    0 (N)      &  0  & \multirow{3}{*}{\textbf{Autoregressive}} & \multirow{3}{*}{(25,56, 95.5)}  &  6\%   \\ 
    Unif[0,0.5]  &  6.4  & &   & 23\%\\ 
    Unif[0, 1]   &  25.8 & &    &  75\%\\
    \hline
    0 (N)      &  0  & \multirow{3}{*}{\textbf{Moving-average}} & \multirow{3}{*}{(31,40, 41.1)}  &  4\%   \\ 
    Unif[0,0.5]  &  6.4  & &   & 99\%\\ 
    Unif[0, 1]   &  28.9 & &    &  99\%\\
    \hline 
    0 (N)      &  0  & \multirow{3}{*}{\textbf{Non-linear}} & \multirow{3}{*}{(29,49, 57.6)}  &  1\%   \\ 
    Unif[0,0.5]  &  6.3  & &   & 57\%\\ 
    Unif[0, 1]   &  26.1 & &    &  100\%\\
    \hline
    0 (N)      &  0  & \multirow{3}{*}{\textbf{Non-stationary}} & \multirow{3}{*}{(30,45, 66.3)}  &  4\%   \\ 
    Unif[0,0.5]  &  6.3  & &   & 86\%\\ 
    Unif[0, 1]   &  26.1 & &    &  100\%\\
    \hline\hline
    \end{tabular}
    \label{table.ANOVA_res}
\end{table}

\begin{figure}
    \centering
    \includegraphics[width = 9cm]{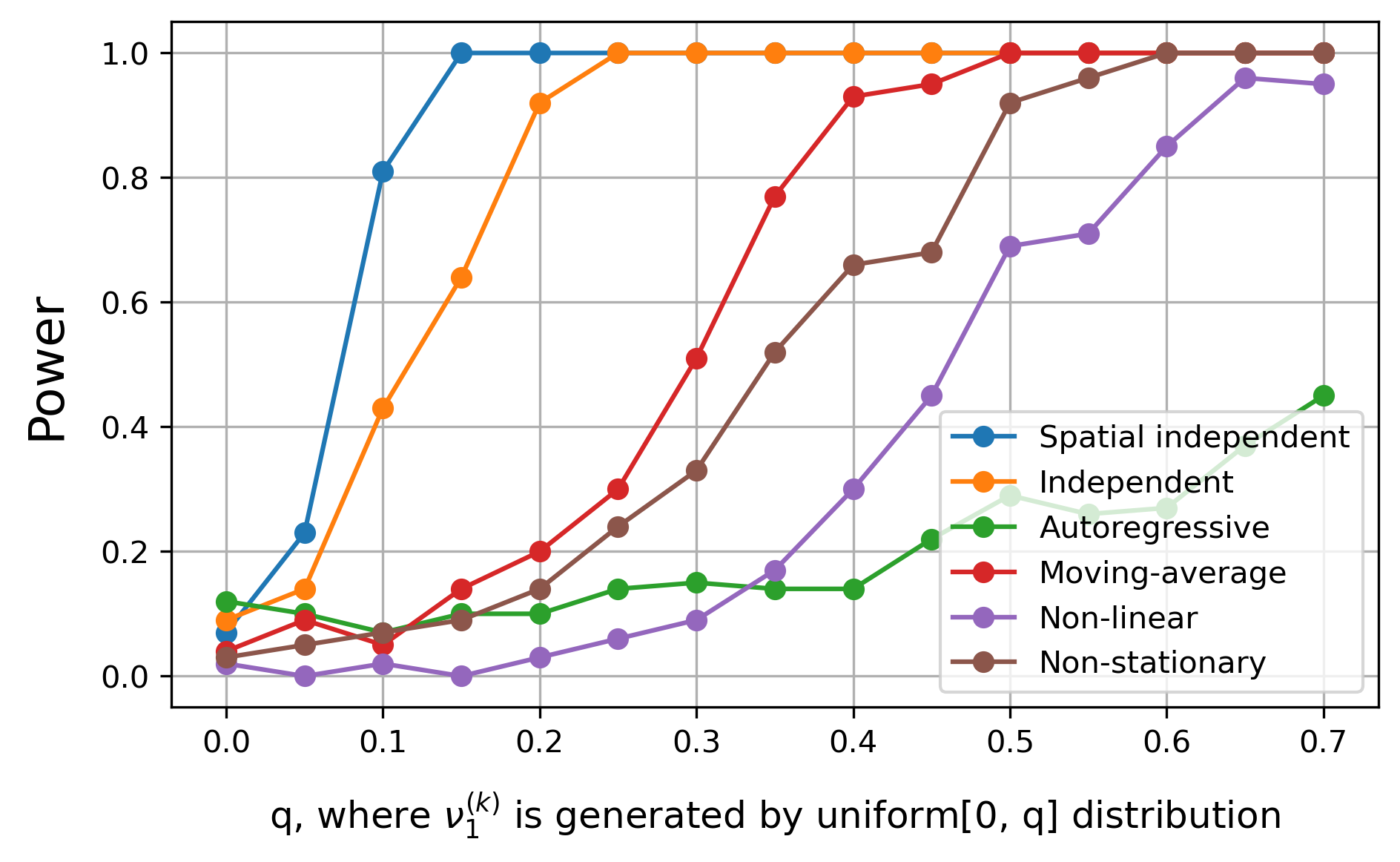}
    \caption{The size and power of Algorithm \ref{algorithm.bootstrap} for different residuals $\boldsymbol{\epsilon}_{t,k}.$ The setting and the bandwidths selection coincide with Table \ref{table.ANOVA_res}. }
    \label{figure,all_power}
\end{figure}

\subsection{Real-life data example}  
\cite{rozemberczki2021chickenpox} collected weekly chickenpox case counts from 20 cities in Hungary,
 covering the year 2005 to 2015.  Figure \ref{fig:sub3} plots the aggregated weekly chickenpox case counts across these cities, revealing a descending pattern. Motivated by this observation, we are interested in the problem of whether the decline is statistically significant. To verify this hypothesis, we divide the data into three-year periods and test the following hypothesis:
$$
H_0:\ \text{The mean chickenpox cases in each city remain stable over time.}
$$

Figures \ref{fig:sub1} and \ref{fig:sub2} plot the weekly chickenpox case counts for Budapest and Baranya, and the sample autocorrelation coefficients of Budapest's weekly case counts. As shown in Figure \ref{fig:sub1},  both cities exhibit similar temporal patterns, indicating strong spatial dependence. Furthermore, the sample autocorrelation coefficients plot in Figure \ref{fig:sub2} reveals a strong temporal dependence in Budapest’s case counts. These two observations provide justification for the setting adopted in this manuscript.

\begin{figure}[htbp]
  \centering
  \begin{subfigure}[b]{0.45\textwidth}
    \includegraphics[width=\textwidth]{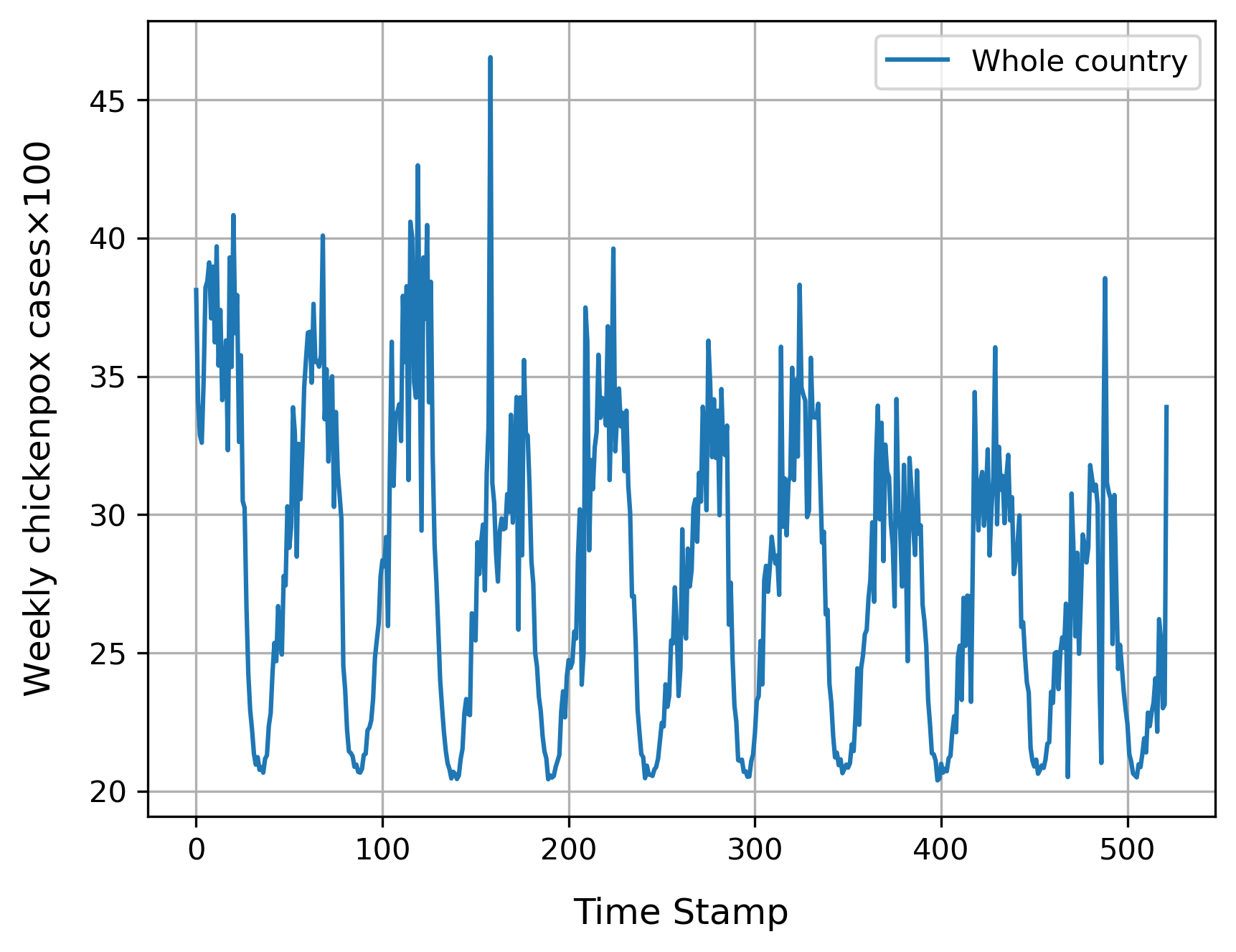}
    \caption{Aggregated weekly chickenpox case counts across these 20 cities.}
    \label{fig:sub3}
\end{subfigure}
\hspace{0.05\textwidth}
  \begin{subfigure}[b]{0.33\textwidth}
    \includegraphics[width=\textwidth]{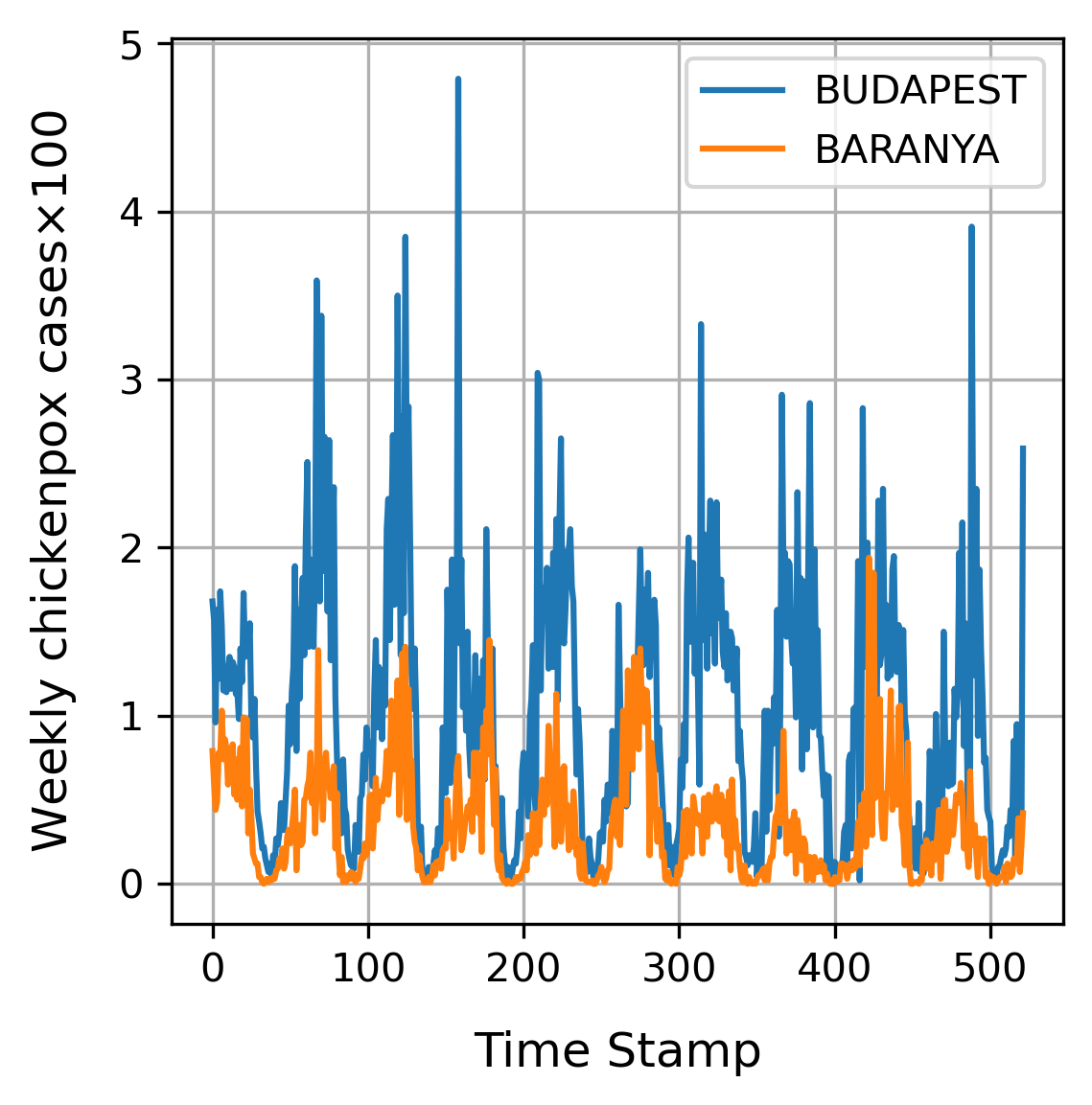}
    \caption{Weekly chickenpox case counts from cities of Budapest and Baranya.}
    \label{fig:sub1}
  \end{subfigure}
  \hfill
  \begin{subfigure}[b]{0.4\textwidth}
    \includegraphics[width=\textwidth]{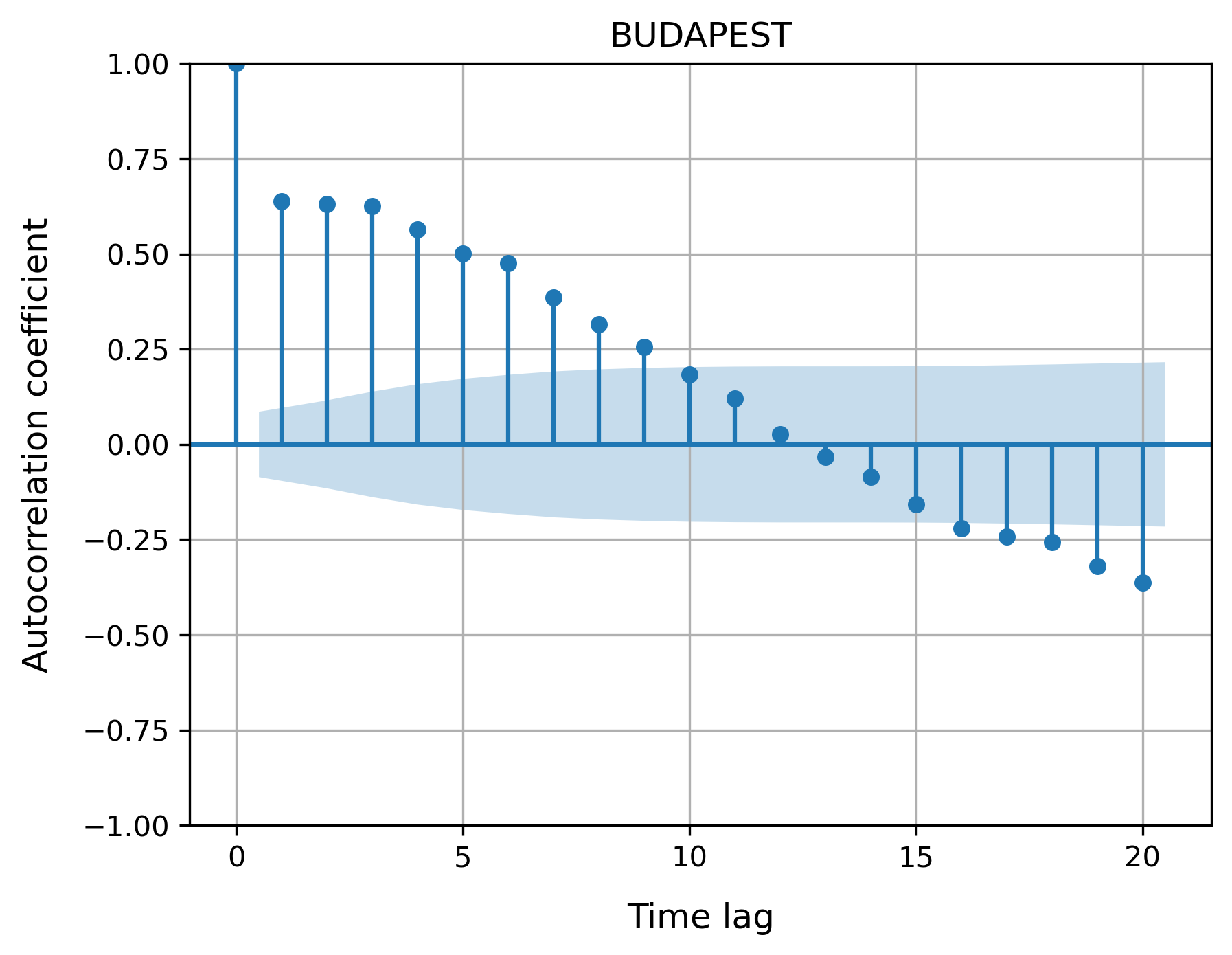}
    \caption{Sample autocorrelation coefficients of Budapest's chickenpox data.}
    \label{fig:sub2}
  \end{subfigure}
  \caption{The weekly chickenpox case counts data and sample autocorrelation coefficients plot.}
  \label{figure.chicken_box}
\end{figure}

Table \ref{chicken_case_res} reports the hypothesis testing results. Since the test statistic exceeds the bootstrapped quantile, we reject $H_0$ and conclude that the population means of chickenpox cases across cities change over time. We attribute part of this change to the introduction of varicella vaccine programs in several European countries \cite{lee2022global}. Although Hungary did not include the varicella vaccine in its national immunization program until 2019, recommendations from healthcare professionals \cite{HUBER20205249} led to an increase in the vaccination population,  contributing to the observed decline in chickenpox infections.

\begin{table}[h]
\caption{Hypothesis testing result. The nominal size is 5\%. }
    \label{chicken_case_res}
    \centering
    \begin{tabular}{c c c c}
    \hline\hline 
     $(B,B_1,H)$    &  Test statistics & Quantile & Decision\\
     \hline 
     (9, 15, 49.3)  &   10.6               &  8.2 & Reject $H_0.$\\
    \hline\hline 
    \end{tabular}
\end{table}

\section{Conclusion}
\label{section.conclusion}
Temporal dependence can introduce bias into classical ANOVA test statistics, leading to unsatisfactory control of size and power. This manuscript introduces an ANOVA test procedure tailored for  high-dimensional non-stationary time series, which generalizes the test statistic of \cite{MR2604697} by excluding product terms $\mathbf{x}_{t_1,k}^\top\mathbf{x}_{t_2, k}$ when $\vert B_1 - B_2\vert$ is either too small or too large.   We demonstrate that this modification effectively reduces bias induced by data dependence. Furthermore, we propose a bootstrap algorithm to assist hypothesis testing through computer simulations.

Apart from methodological developments, this manuscript establishes theoretical results, including concentration inequalities, Gaussian approximation, and variance estimation, on quadratic forms of high-dimensional time series. Leveraging these results, we 
derive the asymptotic distribution of the test statistic under $H_0$ and the asymptotic validity of the bootstrap procedure.   Given the prevalence of time series in practical applications, the proposed method offers a good alternative to classical ANOVA approaches in dependent data settings.




\appendix
\numberwithin{equation}{section}
\numberwithin{lemma}{section}
\numberwithin{figure}{section}
\numberwithin{corollary}{section}
\section{Proofs of theorems in Section \ref{section.theoretical_quadratic}}
\label{section.proof_theorems}
We begin the supplementary material by introducing several notations that will be used frequently in the following sections.  For any $t\in\mathbf{Z}$ and integer $a\geq 0$, define $\mathcal{F}_{t,a}$ as the $\sigma$-field generated by $e_t, e_{t-1},\cdots, e_{t - a}$. By default, if $a < 0$, the conditional expectation $\mathbf{E}\left[\cdot\mid\mathcal{F}_{t,a}\right]$ represents the unconditional expectation $\mathbf{E}[\cdot]$. For any integer $a$,  define
\begin{equation}
\boldsymbol{\gamma}_{t,a} = 
\begin{cases}
\mathbf{E}\left[\boldsymbol{\epsilon}_t\mid\mathcal{F}_{t,a}\right]\ \text{if } a\geq 0,\\
0\ \text{if } a < 0,
\end{cases}
\quad and\quad \boldsymbol{\eta}_{t,a} = \boldsymbol{\epsilon}_t - \boldsymbol{\gamma}_{t,a}.
\label{eq.def_gamma_eta}
\end{equation}
With this definition, for any $0\leq a_1 <  a_2 < \infty$ and any vector $\mathbf{b}\in\mathbf{R}^d$ with  $\vert \mathbf{b}\vert_2 = 1,$ we have  
\begin{align*}
    \left\Vert
    \mathbf{b}^\top\left(\boldsymbol{\gamma}_{t,a_2} - \boldsymbol{\gamma}_{t,a_1}\right)
    \right\Vert_M & \leq \sum_{s = a_1 + 1}^{a_2} \left\Vert
    \mathbf{b}^\top\left(\boldsymbol{\gamma}_{t,s} - \boldsymbol{\gamma}_{t,s - 1}\right)
    \right\Vert_M\\
    & =   \sum_{s = a_1 + 1}^{a_2} \left\Vert
    \mathbf{b}^\top\left(\mathbf{E}[\boldsymbol{\epsilon}_{t}\mid\mathcal{F}_{t,s}] - \mathbf{E}[\boldsymbol{\epsilon}_{t}(s)\mid\mathcal{F}_{t,s}]\right)
    \right\Vert_M\\
    &\leq  \sum_{s = a_1 + 1}^{a_2}\delta_s \leq \frac{C}{(1 + a_1)^\alpha}.
\end{align*}
For any $a\geq 0$, we have 
\begin{align*}
    \left\Vert
    \mathbf{b}^\top\boldsymbol{\eta}_{t,a}
    \right\Vert_M\leq \sum_{s = a + 1}^\infty \left\Vert
    \mathbf{b}^\top\left(\boldsymbol{\gamma}_{t,s} - \boldsymbol{\gamma}_{t,s - 1}\right)
    \right\Vert_M\leq \sum_{s = a + 1}^\infty\delta_s \leq \frac{C}{(1 + a)^\alpha}.
\end{align*}
Lemma \ref{lemma.linear_form} provides concentration inequalities for linear combinations of $(M,\alpha)$-short-range dependent random vectors.  While our main focus is  on quadratic forms, bounding the moments of linear combinations remains essential to the proof of Theorem \ref{theorem.consistent_quadratic}, like those in \eqref{eq.use_linear}.

\begin{lemma}
Suppose $\boldsymbol{\epsilon}_t, t = 1,2,\cdots, T$ stem from $(M,\alpha)$-short-range dependent random vector process with $M > 2, \alpha > 1.$ Then there exists a constant $C>0$ such that for any  vectors
$\boldsymbol{b}_t\in\mathbf{R}^d,  t  =1,\cdots, T,$ and any integer $s\geq 0,$
\begin{equation}
    \left\Vert
    \sum_{t = 1}^T \boldsymbol{b}_t^\top \boldsymbol{\epsilon}_t
    \right\Vert_{M}\leq C\sqrt{\sum_{t = 1}^T \left\vert\mathbf{b}_t\right\vert^2_2}\quad \text{and}\quad \left\Vert\sum_{t = 1}^T \boldsymbol{b}_t^\top \boldsymbol{\eta}_{t,s}\right\Vert_M\leq \frac{C}{(1 + s)^\alpha}\sqrt{\sum_{t = 1}^T \left\vert\mathbf{b}_t\right\vert^2_2}.
    \label{eq.linear_form}
\end{equation}
\label{lemma.linear_form}
\end{lemma} 
\begin{corollary}
Suppose conditions in Lemma \ref{lemma.linear_form} hold true and $t_2 > t_1.$  In such case, $\boldsymbol{\epsilon}_{t_1}$ is independent of $\mathbf{E}\left[\boldsymbol{\epsilon}_{t_2}\mid\mathcal{F}_{t_2, t_2 - t_1 - 1}\right].$
    From Lemma \ref{lemma.linear_form}, we have 
    \begin{equation}
    \begin{aligned}
    \left\vert\mathbf{E}\left[\boldsymbol{\epsilon}_{t_1}^\top\boldsymbol{\epsilon}_{t_2}\right]\right\vert &\leq \sum_{j = 1}^d\left\vert \mathbf{E}\left[\boldsymbol{\epsilon}_{t_1}^{(j)}\boldsymbol{\epsilon}_{t_2}^{(j)}\right]\right\vert\\
        & = \sum_{j = 1}^d\left\vert \mathbf{E}\left[\boldsymbol{\epsilon}_{t_1}^{(j)}\left(\boldsymbol{\epsilon}_{t_2}^{(j)} - \mathbf{E}\left[\boldsymbol{\epsilon}_{t_2}^{(j)}\mid\mathcal{F}_{t_2,t_2 - t_1 - 1}\right]\right)\right]\right\vert\\
        &\leq
        \sum_{j = 1}^d \left\Vert \boldsymbol{\epsilon}_{t_1}^{(j)}\right\Vert_M\left\Vert\boldsymbol{\epsilon}_{t_2}^{(j)} - \mathbf{E}\left[\boldsymbol{\epsilon}_{t_2}^{(j)}\mid\mathcal{F}_{t_2,t_2 - t_1 - 1}\right]\right\Vert_M
        \\
        &\leq C\sum_{j = 1}^d\left\Vert\boldsymbol{\eta}^{(j)}_{t_2, t_2 - t_1 -1}\right\Vert_M
        \leq \frac{C_1d}{(t_2  -t_1)^\alpha}.
    \end{aligned}
    \label{eq.covariance_delta}
    \end{equation}
\end{corollary}

\begin{remark}
   The role of Lemma \ref{lemma.linear_form} is similar to that of Lemma 1 in \cite{10.1093/jrsssb/qkad006}, which established bounds for moments of linear combinations of scalar time series.  In our setting, however, we consider vector time series, so the sum of squares in Lemma \ref{lemma.linear_form} involves the Euclidean norm of the linear combination vectors $\mathbf{b}_t.$
\end{remark}

\begin{proof}[Proof of Lemma \ref{lemma.linear_form}]
    Notice that 
    $$
    \boldsymbol{\epsilon}_t = \mathbf{E}\left[\boldsymbol{\epsilon}_t\mid\mathcal{F}_{t,0}\right] + \sum_{q = 1}^\infty \left(\mathbf{E}\left[\boldsymbol{\epsilon}_t\mid\mathcal{F}_{t,q}\right] - \mathbf{E}\left[\boldsymbol{\epsilon}_t\mid\mathcal{F}_{t,q  - 1}\right]\right) = \sum_{q = 0}^\infty \left(\boldsymbol{\gamma}_{t,q} - \boldsymbol{\gamma}_{t,q - 1}\right),
    $$
    here we recall $\boldsymbol{\gamma}_{t, - 1} = 0.$ Therefore, 
    \begin{align*}
    \left\Vert
    \sum_{t = 1}^T \boldsymbol{b}_t^\top \boldsymbol{\epsilon}_t
    \right\Vert_{M}\leq \sum_{q = 0}^\infty 
    \left\Vert
    \sum_{t = 1}^T \boldsymbol{b}_t^\top \left(\boldsymbol{\gamma}_{t,q} - \boldsymbol{\gamma}_{t,q - 1}\right)
    \right\Vert_M.
    \end{align*}
    For any $s = 1,2,\cdots, T$ and $q\geq 0$, define 
    \begin{align*}
        J_{s, q} =  \sum_{t = T - s + 1}^T  \boldsymbol{b}_t^\top\left(\boldsymbol{\gamma}_{t,q} - \boldsymbol{\gamma}_{t,q - 1}\right)
    \end{align*}
    and $\mathcal{J}_{s,q}$ the $\sigma$-field generated by $e_T, e_{T-1},\cdots, e_{T-s + 1 - q}$. Then $J_{s, q}$ is measurable in $\mathcal{J}_{s,q}$, $\mathcal{J}_{s,q}\subset \mathcal{J}_{s + 1,q}$, and 
    \begin{align*}
        \mathbf{E}\left[(J_{s + 1, q} - J_{s, q})\mid\mathcal{J}_{s,q}\right] &= \mathbf{b}_{T-s}^\top \mathbf{E}\left[(\boldsymbol{\gamma}_{T-s,q} - \boldsymbol{\gamma}_{T-s,q - 1})\mid\mathcal{J}_{s,q}\right]\\
        &= \mathbf{b}_{T-s}^\top(\boldsymbol{\gamma}_{T-s,q - 1} - \boldsymbol{\gamma}_{T-s,q - 1}) = 0,
    \end{align*}
    so $J_{s, q}$ forms a martingale. According to Theorem 1.1 of \cite{MR0400380}, 
    \begin{equation}
    \begin{aligned}
        \left\Vert
        \sum_{t = 1}^T  \boldsymbol{b}_t^\top\left(\boldsymbol{\gamma}_{t,q} - \boldsymbol{\gamma}_{t,q - 1}\right)
        \right\Vert_{M}\leq C\sqrt{\sum_{t = 1}^T \left\Vert\boldsymbol{b}_t^\top\left(\boldsymbol{\gamma}_{t,q} - \boldsymbol{\gamma}_{t,q - 1}\right)\right\Vert^2_M}
        \leq C\delta_q\sqrt{\sum_{t = 1}^T \vert\boldsymbol{b}_t\vert^2_2},
    \end{aligned}
    \label{eq.linear_martingale}
    \end{equation}
    and 
    \begin{equation}
        \begin{aligned}
    \left\Vert
    \sum_{t = 1}^T \boldsymbol{b}_t^\top \boldsymbol{\epsilon}_t
    \right\Vert_{M}\leq C\sqrt{\sum_{t = 1}^T \left\vert\boldsymbol{b}_t\right\vert^2_2}\times \left(\sum_{q = 0}^\infty\delta_q\right)\leq C_1\sqrt{\sum_{t = 1}^T \vert\boldsymbol{b}_t\vert^2_2},
        \end{aligned}
    \end{equation}
which proves the first result in \eqref{eq.linear_form}.

    On the other hand, by definition 
    \begin{align*}
        \boldsymbol{\eta}_{t,s} = \boldsymbol{\epsilon}_t - \boldsymbol{\gamma}_{t,s} = \sum_{q = s + 1}^\infty (\boldsymbol{\gamma}_{t,q} - \boldsymbol{\gamma}_{t,q - 1}),
    \end{align*}
    so from \eqref{eq.linear_martingale},
    \begin{equation}
        \begin{aligned}
            \left\Vert\sum_{t = 1}^T \boldsymbol{b}_t^\top \boldsymbol{\eta}_{t,s}\right\Vert_M & \leq \sum_{q  = s + 1}^\infty \left\Vert\sum_{t = 1}^T \mathbf{b}_t^\top(\boldsymbol{\gamma}_{t,q} - \boldsymbol{\gamma}_{t,q - 1})\right\Vert_{M}\\
            &\leq C\sqrt{\sum_{t = 1}^T \vert\boldsymbol{b}_t\vert^2_2}\times \sum_{q  =s + 1}^\infty \delta_q\leq \frac{C_1}{(1 + s)^\alpha}\sqrt{\sum_{t = 1}^T \vert\boldsymbol{b}_t\vert^2_2},
        \end{aligned}
    \end{equation}
    which proves the second result in \eqref{eq.linear_form}.
\end{proof}

\begin{corollary}
    According to \eqref{eq.linear_form}, for any $s\geq 0,$ 
\begin{equation}
    \begin{aligned}
        \left\Vert\sum_{t = 1}^T \boldsymbol{b}_t^\top \boldsymbol{\gamma}_{t,s}\right\Vert_{M} & = \left\Vert\sum_{t = 1}^T \boldsymbol{b}_t^\top (\boldsymbol{\epsilon}_t - \boldsymbol{\eta}_{t,s})\right\Vert_{M}\\
       & \leq \left\Vert\sum_{t = 1}^T \boldsymbol{b}_t^\top \boldsymbol{\epsilon}_t\right\Vert_{M} + \left\Vert\sum_{t = 1}^T \boldsymbol{b}_t^\top \boldsymbol{\eta}_{t,s}\right\Vert_M\\
       &\leq C\sqrt{\sum_{t = 1}^T \vert\mathbf{b}_t\vert^2_2} + \frac{C}{(1 + s)^\alpha}\sqrt{\sum_{t = 1}^T \vert\mathbf{b}_t\vert^2_2}
       \leq C_1\sqrt{\sum_{t = 1}^T \vert\boldsymbol{b}_t\vert^2_2}.
    \end{aligned}
    \label{eq.linear_gamma}
\end{equation}
\end{corollary}

The proof of Theorem \ref{theorem.consistent_quadratic} is based on the following decomposition: Suppose $t_1 > t_2,$  then 
$
\boldsymbol{\epsilon}_{t_1}^\top\boldsymbol{\epsilon}_{t_2} =  \boldsymbol{\gamma}_{t_1, t_1 - t_2 -  1}^\top\boldsymbol{\epsilon}_{t_2} + \boldsymbol{\eta}^\top_{t_1, t_1 - t_2  - 1}\boldsymbol{\epsilon}_{t_2}.
$
Since  $\boldsymbol{\gamma}_{t_1, t_1 - t_2 -  1}$ is measurable in the $\sigma$-field $\mathcal{F}_{t_1, t_1 - t_2 - 1},$ $\boldsymbol{\gamma}_{t_1, t_1 - t_2 -  1}$ is independent of $\boldsymbol{\epsilon}_{t_2},$ implying that $\mathbf{E}\left[\boldsymbol{\epsilon}_{t_1}^\top\boldsymbol{\epsilon}_{t_2}\right] = \mathbf{E}\left[\boldsymbol{\eta}^\top_{t_1, t_1 - t_2  - 1}\boldsymbol{\epsilon}_{t_2}\right].$ According to Lemma \ref{lemma.linear_form}, 
$$
\left\Vert
\boldsymbol{\eta}^\top_{t_1, t_1 - t_2  - 1}\mathbf{b}
\right\Vert_M\leq \frac{C\vert\mathbf{b}\vert_2}{(t_1 - t_2)^\alpha}\quad \text{for any vector } \mathbf{b}\in\mathbf{R}^d.
$$
Therefore, despite the fact that  $\boldsymbol{\eta}_{t_1, t_1 - t_2  - 1}$ and $\boldsymbol{\epsilon}_{t_2}$ may exhibit complex dependence, the moments of linear combinations of $\boldsymbol{\eta}_{t_1, t_1 - t_2  - 1}$ are not large for sufficiently large $t_1 - t_2,$ and it suffices to study the products $\boldsymbol{\gamma}_{t_1, t_1 - t_2 -  1}^\top\boldsymbol{\epsilon}_{t_2}$ to bound the moments of the quadratic form $Q$ in \eqref{eq.def_quadratic_form}.

\begin{proof}[Proof of Theorem \ref{theorem.consistent_quadratic}]

\textbf{1. The proof of equation \eqref{eq.whole_formula}}

We first prove \eqref{eq.whole_formula}.     
Since $b_{t_1t_2} = 0$ for $\vert t_1 - t_2\vert < B$, 
    \begin{align*}
        \left\Vert Q\right\Vert_{M/2} &= \left\Vert
        \sum_{t_1 = 1}^T\sum_{t_2 = 1}^T b_{t_1t_2}\left(\boldsymbol{\epsilon}_{t_1}^\top \boldsymbol{\epsilon}_{t_2} - \mathbf{E}\left[\boldsymbol{\epsilon}_{t_1}^\top \boldsymbol{\epsilon}_{t_2}\right]\right)
        \right\Vert_{M/2}\\
        &\leq 
        \left\Vert
        \sum_{t_1 = 1}^{T - B}\sum_{t_2 = t_1 + B}^T b_{t_1t_2}\left(\boldsymbol{\epsilon}_{t_1}^\top \boldsymbol{\epsilon}_{t_2} - \mathbf{E}\left[\boldsymbol{\epsilon}_{t_1}^\top \boldsymbol{\epsilon}_{t_2}\right]\right)
        \right\Vert_{M/2}\\
        &+ \left\Vert
        \sum_{t_2 = 1}^{T - B}\sum_{t_1 = t_2 + B}^T b_{t_1t_2}\left(\boldsymbol{\epsilon}_{t_1}^\top \boldsymbol{\epsilon}_{t_2} - \mathbf{E}\left[\boldsymbol{\epsilon}_{t_1}^\top \boldsymbol{\epsilon}_{t_2}\right]\right)
        \right\Vert_{M/2}.
    \end{align*}
    Define $\boldsymbol{\gamma}_{t,a}$ and $\boldsymbol{\eta}_{t,a}$ as in \eqref{eq.def_gamma_eta}. If $t_2 > t_1$, then 
    $\boldsymbol{\gamma}_{t_2, t_2 - t_1 - 1}$ is independent of $\boldsymbol{\epsilon}_{t_1}$, and 
    \begin{align*}
        \mathbf{E}\left[\boldsymbol{\epsilon}_{t_1}^\top \boldsymbol{\gamma}_{t_2, t_2 - t_1 - 1}\right]  = \left(\mathbf{E}\left[\boldsymbol{\epsilon}_{t_1}\right]\right)^\top \left(\mathbf{E}\left[\boldsymbol{\gamma}_{t_2, t_2 - t_1 - 1}\right]\right)  = 0.
    \end{align*}
    Since $\boldsymbol{\epsilon}_{t_2} = \boldsymbol{\gamma}_{t_2, t_2 - t_1 - 1} + \boldsymbol{\eta}_{t_2,t_2  - t_1 - 1},$ we have 
    \begin{align*}    
    \mathbf{E}\left[\boldsymbol{\epsilon}_{t_1}^\top \boldsymbol{\epsilon}_{t_2}\right] = \mathbf{E}\left[\boldsymbol{\epsilon}_{t_1}^\top\boldsymbol{\eta}_{t_2,t_2 - t_1 - 1}\right].
    \end{align*}
    This further implies that 
    \begin{align*}
    \sum_{t_1 = 1}^{T - B}\sum_{t_2 = t_1 + B}^T b_{t_1t_2}\left(\boldsymbol{\epsilon}_{t_1}^\top \boldsymbol{\epsilon}_{t_2} - \mathbf{E}\left[\boldsymbol{\epsilon}_{t_1}^\top \boldsymbol{\epsilon}_{t_2}\right]\right) &= \sum_{t_1 = 1}^{T - B}\sum_{t_2 = t_1 + B}^T b_{t_1t_2}\boldsymbol{\epsilon}_{t_1}^\top\boldsymbol{\gamma}_{t_2, t_2 - t_1 - 1}\\
    &+ \sum_{t_1 = 1}^{T - B}\sum_{t_2 = t_1 + B}^T b_{t_1t_2}\left(\boldsymbol{\epsilon}_{t_1}^\top \boldsymbol{\eta}_{t_2, t_2 - t_1 - 1} - \mathbf{E}\left[\boldsymbol{\epsilon}_{t_1}^\top \boldsymbol{\eta}_{t_2, t_2 - t_1 - 1}\right]\right).
    \end{align*}
    Notice that 
    $$
    \boldsymbol{\epsilon}_t = \mathbf{E}\left[\boldsymbol{\epsilon}_t\mid\mathcal{F}_{t,0}\right] + \sum_{q = 1}^\infty \left(\mathbf{E}\left[\boldsymbol{\epsilon}_t\mid\mathcal{F}_{t,q}\right] - \mathbf{E}\left[\boldsymbol{\epsilon}_t\mid\mathcal{F}_{t,q - 1}\right]\right) =  \sum_{q = 0}^\infty (\boldsymbol{\gamma}_{t, q} - \boldsymbol{\gamma}_{t, q - 1}),
    $$ 
    where $\boldsymbol{\gamma}_{t,-1}= 0.$ We have 
    \begin{align*}
        &\left\Vert
        \sum_{t_1 = 1}^{T - B}\sum_{t_2 = t_1 + B}^T b_{t_1t_2}\left(\boldsymbol{\epsilon}_{t_1}^\top \boldsymbol{\epsilon}_{t_2} - \mathbf{E}\left[\boldsymbol{\epsilon}_{t_1}^\top \boldsymbol{\epsilon}_{t_2}\right]\right)
        \right\Vert_{M/2}\\
        &\leq \left\Vert
        \sum_{t_1 = 1}^{T - B}\sum_{t_2 = t_1 + B}^T b_{t_1t_2}\boldsymbol{\epsilon}_{t_1}^\top\boldsymbol{\gamma}_{t_2, t_2 - t_1 - 1}
        \right\Vert_{M/2}\\
        &+  \left\Vert
        \sum_{t_1 = 1}^{T - B}\sum_{t_2 = t_1 + B}^T b_{t_1t_2}\left(\boldsymbol{\epsilon}_{t_1}^\top\boldsymbol{\eta}_{t_2,t_2 - t_1 - 1} - \mathbf{E}\left[\boldsymbol{\epsilon}_{t_1}^\top\boldsymbol{\eta}_{t_2,t_2 - t_1 - 1}\right]\right)
        \right\Vert_{M/2}\\
        &\leq 
        \sum_{q = 0}^\infty \left\Vert
        \sum_{t_1 = 1}^{T - B}\sum_{t_2 = t_1 + B}^T b_{t_1t_2}\left(\boldsymbol{\gamma}_{t_1,q} - \boldsymbol{\gamma}_{t_1,q - 1}\right)^\top\boldsymbol{\gamma}_{t_2, t_2 - t_1 - 1}
        \right\Vert_{M/2}\\
        &+ \sum_{q = 0}^\infty
        \left\Vert
        \sum_{t_2 = 1 + B}^T\sum_{t_1 = 1}^{t_2 - B} b_{t_1t_2}\left(\mathbf{E}\left[\boldsymbol{\epsilon}_{t_1}^\top\boldsymbol{\eta}_{t_2,t_2 - t_1 - 1}\mid\mathcal{F}_{t_2, q}\right] - \mathbf{E}\left[\boldsymbol{\epsilon}_{t_1}^\top\boldsymbol{\eta}_{t_2,t_2 - t_1 - 1}\mid\mathcal{F}_{t_2, q - 1}\right]\right)
        \right\Vert_{M/2},
    \end{align*}
    where $\mathbf{E}\left[\boldsymbol{\epsilon}_{t_1}^\top\boldsymbol{\eta}_{t_2,t_2 - t_1 - 1}\mid\mathcal{F}_{t_2, - 1}\right] = \mathbf{E}\left[\boldsymbol{\epsilon}_{t_1}^\top\boldsymbol{\eta}_{t_2,t_2 - t_1 - 1}\right].$    
    For any $q\geq 0$ and any $s = 1,2,\cdots, T - B$, define 
    \begin{align*}
        A_{s,q} = \sum_{t_1 = T - B - s  + 1}^{T - B}\sum_{t_2 = t_1 + B}^T b_{t_1t_2}\left(\boldsymbol{\gamma}_{t_1,q} - \boldsymbol{\gamma}_{t_1,q - 1}\right)^\top\boldsymbol{\gamma}_{t_2, t_2 - t_1 - 1},
    \end{align*}
    and $\mathcal{A}_{s,q}$ the $\sigma$-field generated by $e_{T},\cdots, e_{T - B - s + 1  - q}$. Then $\mathcal{A}_{s,q}\subset \mathcal{A}_{ s + 1,q}$, $A_{s,q}$ is measurable in $\mathcal{A}_{s,q}$, and 
    \begin{align*}
        &\mathbf{E}\left[\left(A_{s + 1,q} - A_{s,q}\right)\mid\mathcal{A}_{s,q}\right]\\
        &= \sum_{t_2 = T - s}^T b_{(T - B - s)t_2}\mathbf{E}\left[\left(\boldsymbol{\gamma}_{T - B - s,q} - \boldsymbol{\gamma}_{T - B - s,q - 1}\right)^\top\boldsymbol{\gamma}_{t_2, t_2 - T + B + s - 1}\mid \mathcal{A}_{s,q}\right]\\
        &= \sum_{t_2 = T - s}^T b_{(T - B - s)t_2}\boldsymbol{\gamma}_{t_2, t_2 - T + B + s - 1}^\top
        \mathbf{E}\left[\left(\boldsymbol{\gamma}_{T - B - s,q} - \boldsymbol{\gamma}_{T - B - s,q - 1}\right) \mid  \mathcal{A}_{s,q}\right]
        \\
        &= \sum_{t_2 = T - s}^T b_{(T - B - s)t_2}\boldsymbol{\gamma}_{t_2, t_2 - T + B + s - 1}^\top\left(\boldsymbol{\gamma}_{T - B - s,q - 1} - \boldsymbol{\gamma}_{T - B - s,q - 1}\right) = 0,
    \end{align*}
    so $A_{s,q}$ forms a martingale. According to Theorem 1.1 of  \cite{MR0400380},
    \begin{equation}
    \begin{aligned}
        &\left\Vert
        \sum_{t_1 =  1}^{T - B}\sum_{t_2 = t_1 + B}^T b_{t_1t_2}\left(\boldsymbol{\gamma}_{t_1,q} - \boldsymbol{\gamma}_{t_1,q - 1}\right)^\top\boldsymbol{\gamma}_{t_2, t_2 - t_1 - 1}
        \right\Vert_{M/2}\\
        &\leq C\sqrt{\sum_{t_1 =  1}^{T - B}\left\Vert \left(\boldsymbol{\gamma}_{t_1,q} - \boldsymbol{\gamma}_{t_1,q - 1}\right)^\top\sum_{t_2 = t_1 + B}^T b_{t_1t_2}\boldsymbol{\gamma}_{t_2, t_2 - t_1 - 1}\right\Vert_{M/2}^2}.
    \end{aligned}
    \label{eq.separate_Q}
    \end{equation}
    Since $\boldsymbol{\epsilon}_{t_2}$ is independent of $e_{t_2+1},\cdots, e_T,$ we have 
    $$
    \boldsymbol{\gamma}_{t_2, t_2 - t_1 - 1} = \mathbf{E}\left[\boldsymbol{\epsilon}_{t_2}\mid\mathcal{F}_{t_2, t_2 -t_1 - 1}\right] = \mathbf{E}\left[\boldsymbol{\epsilon}_{t_2}\mid\mathcal{F}_{T, T -t_1 - 1}\right].
    $$ 
    For any given vector $\boldsymbol{\tau}\in\mathbf{R}^d$, from Lemma \ref{lemma.linear_form}, 
    \begin{equation}
    \begin{aligned}
        &\left\Vert
        \boldsymbol{\tau}^\top \sum_{t_2 = t_1 + B}^T b_{t_1t_2}\boldsymbol{\gamma}_{t_2, t_2 - t_1 - 1}
        \right\Vert_{M/2}\\
        &= \left\Vert
        \mathbf{E}\left[\boldsymbol{\tau}^\top \sum_{t_2 = t_1 + B}^T b_{t_1t_2}\boldsymbol{\epsilon}_{t_2}\mid\mathcal{F}_{T, T -t_1 - 1}
        \right]\right\Vert_{M/2}\\
        &\leq 
        \left\Vert
        \boldsymbol{\tau}^\top \sum_{t_2 = t_1 + B}^T b_{t_1t_2}\boldsymbol{\epsilon}_{t_2}
        \right\Vert_{M/2}
        \leq C\left\vert\boldsymbol{\tau}\right\vert_2 \sqrt{\sum_{t_2 = t_1 + B}^T b^2_{t_1t_2}}.
    \end{aligned}
    \label{eq.use_linear}
    \end{equation}
Since $\boldsymbol{\gamma}_{t_1,q} - \boldsymbol{\gamma}_{t_1,q - 1}$ is independent of $\sum_{t_2 = t_1 + B}^T b_{t_1t_2}\boldsymbol{\gamma}_{t_2, t_2 - t_1 - 1}$, from Theorem 1.7 in \cite{MR2002723} and \eqref{eq.use_linear}, for any random vector $\boldsymbol{\tau}\in\mathbf{R}^d,$
\begin{align*}
    &\mathbf{E}\left[\left(\left\vert
    (\boldsymbol{\gamma}_{t_1,q} - \boldsymbol{\gamma}_{t_1,q - 1})^\top\sum_{t_2 = t_1 + B}^T b_{t_1t_2}\boldsymbol{\gamma}_{t_2, t_2 - t_1 - 1}
    \right\vert^{M/2}\right)\mid \boldsymbol{\gamma}_{t_1,q} - \boldsymbol{\gamma}_{t_1,q - 1} = \boldsymbol{\tau}\right]\\
    &=
    \mathbf{E}\left[\left\vert
    \boldsymbol{\tau}^\top\sum_{t_2 = t_1 + B}^T b_{t_1t_2}\boldsymbol{\gamma}_{t_2, t_2 - t_1 - 1}
    \right\vert^{M/2}\right]
    \\
    &\leq
    C\left(\sum_{t_2 = t_1 + B}^Tb^2_{t_1t_2}\right)^{M/4}\left\vert\boldsymbol{\tau}\right\vert_2^{M/2}
     = C\left(\sum_{t_2 = t_1 + B}^Tb^2_{t_1t_2}\right)^{M/4}\left\vert\boldsymbol{\gamma}_{t_1,q} - \boldsymbol{\gamma}_{t_1,q - 1}\right\vert_2^{M/2}.
\end{align*}
This observation implies that
\begin{align*}
    &\left\Vert\ 
    \left(\boldsymbol{\gamma}_{t_1,q} - \boldsymbol{\gamma}_{t_1,q - 1}\right)^\top\sum_{t_2 = t_1 + B}^T b_{t_1t_2}\boldsymbol{\gamma}_{t_2, t_2 - t_1 - 1}
    \ \right\Vert_{M/2}\\
    &\leq C\left\Vert\ \left\vert\boldsymbol{\gamma}_{t_1,q} - \boldsymbol{\gamma}_{t_1,q - 1}\right\vert_2\ \right\Vert_{M/2}\sqrt{\sum_{t_2 = t_1 + B}^Tb^2_{t_1t_2}}\\
    &\leq C\sqrt{\sum_{t_2 = t_1 + B}^Tb^2_{t_1t_2}} \sqrt{\sum_{j = 1}^d \left\Vert
    \boldsymbol{\gamma}_{t_1,q}^{(j)} - \boldsymbol{\gamma}_{t_1,q - 1}^{(j)}
    \right\Vert^2_{M/2}}
    \leq C_1\delta_q \sqrt{d\sum_{t_2 = t_1 + B}^Tb^2_{t_1t_2}}.
\end{align*}
From \eqref{eq.separate_Q},
\begin{equation}
    \begin{aligned}
        &\sum_{q = 0}^\infty \left\Vert
        \sum_{t_1 = 1}^{T - B}\sum_{t_2 = t_1 + B}^T b_{t_1t_2}\left(\boldsymbol{\gamma}_{t_1,q} - \boldsymbol{\gamma}_{t_1,q - 1}\right)^\top\boldsymbol{\gamma}_{t_2, t_2 - t_1 - 1}
        \right\Vert_{M/2}\\
        &\leq C\sum_{q = 0}^\infty\delta_q \sqrt{d \sum_{t_1 =  1}^{T - B}\sum_{t_2 = t_1 + B}^Tb^2_{t_1t_2}}\leq C_1\sqrt{d \sum_{t_1 =  1}^{T - B}\sum_{t_2 = t_1 + B}^Tb^2_{t_1t_2}}.
    \end{aligned}
\label{eq.summary_first_half}
\end{equation}

For any $q = 0,1,\cdots,$ and any $s = 1,2,\cdots, T - B$, define 
\begin{align*}
    D_{q,s}  = \sum_{t_2 = T - s + 1}^T\sum_{t_1 = 1}^{t_2 - B} b_{t_1t_2}\left(\mathbf{E}\left[\boldsymbol{\epsilon}_{t_1}^\top\boldsymbol{\eta}_{t_2,t_2 - t_1 - 1}\mid \mathcal{F}_{t_2, q}\right] - \mathbf{E}\left[\boldsymbol{\epsilon}_{t_1}^\top\boldsymbol{\eta}_{t_2,t_2 - t_1 - 1}\mid\mathcal{F}_{t_2, q - 1}\right]\right)
\end{align*}
and $\mathcal{D}_{q,s}$ the $\sigma$-field generated by $e_{T},e_{T-1},\cdots, e_{T-s+1-q}$. Then $D_{q,s}$ is measurable in $\mathcal{D}_{q,s}$, $\mathcal{D}_{q,s}\subset \mathcal{D}_{q,s + 1}$, and 
\begin{align*}
    & \mathbf{E}\left[\left(D_{q,s + 1} - D_{q,s}\right)\mid\mathcal{D}_{q,s}\right]\\
    &= 
    \sum_{t_1 = 1}^{T - s - B} b_{t_1(T - s)}\mathbf{E}\left[\left(\mathbf{E}\left[\boldsymbol{\epsilon}_{t_1}^\top\boldsymbol{\eta}_{T - s,T - s - t_1 - 1}\mid\mathcal{F}_{T - s, q}\right] - \mathbf{E}\left[\boldsymbol{\epsilon}_{t_1}^\top\boldsymbol{\eta}_{T - s,T - s - t_1 - 1}\mid\mathcal{F}_{T - s, q - 1}\right]\right)
    \mid\mathcal{D}_{q,s}\right]\\
   &= \sum_{t_1 = 1}^{T - s - B} b_{t_1(T - s)}\left(\mathbf{E}\left[\boldsymbol{\epsilon}_{t_1}^\top\boldsymbol{\eta}_{T - s,T - s - t_1 - 1}\mid\mathcal{F}_{T - s, q  - 1}\right] - \mathbf{E}\left[\boldsymbol{\epsilon}_{t_1}^\top\boldsymbol{\eta}_{T - s,T - s - t_1 - 1}\mid\mathcal{F}_{T - s, q - 1}\right]\right) = 0,
\end{align*}
so $D_{q,s}$ forms a martingale. According to Theorem 1.1 of  \cite{MR0400380}, 
\begin{equation}
\begin{aligned}
    &\left\Vert
    \sum_{t_2 = 1 + B}^T\sum_{t_1 = 1}^{t_2 - B} b_{t_1t_2}\left(\mathbf{E}\left[\boldsymbol{\epsilon}_{t_1}^\top\boldsymbol{\eta}_{t_2,t_2 - t_1 - 1}\mid\mathcal{F}_{t_2, q}\right] - \mathbf{E}\left[\boldsymbol{\epsilon}_{t_1}^\top\boldsymbol{\eta}_{t_2,t_2 - t_1 - 1}\mid\mathcal{F}_{t_2, q - 1}\right]\right)
    \right\Vert_{M/2}\\
    &\leq C\sqrt{\sum_{t_2 = 1 + B}^T\left\Vert
    \sum_{t_1 = 1}^{t_2 - B} b_{t_1t_2}\left(\mathbf{E}\left[\boldsymbol{\epsilon}_{t_1}^\top\boldsymbol{\eta}_{t_2,t_2 - t_1 - 1}\mid\mathcal{F}_{t_2, q}\right] - \mathbf{E}\left[\boldsymbol{\epsilon}_{t_1}^\top\boldsymbol{\eta}_{t_2,t_2 - t_1 - 1}\mid\mathcal{F}_{t_2, q - 1}\right]\right)
    \right\Vert^2_{M/2}}.
\end{aligned}
\label{eq.tail_moment_to_first}
\end{equation}
In the following step, we try to bound the norm 
$$
\left\Vert
    \sum_{t_1 = 1}^{t_2 - B} b_{t_1t_2}\left(\mathbf{E}\left[\boldsymbol{\epsilon}_{t_1}^\top\boldsymbol{\eta}_{t_2,t_2 - t_1 - 1}\mid\mathcal{F}_{t_2, q}\right] - \mathbf{E}\left[\boldsymbol{\epsilon}_{t_1}^\top\boldsymbol{\eta}_{t_2,t_2 - t_1 - 1}\mid\mathcal{F}_{t_2, q - 1}\right]\right)
\right\Vert_{M/2}.
$$ 
If $q \leq B$, then 
\begin{align*}
    &\left\Vert
    \sum_{t_1 = 1}^{t_2 - B} b_{t_1t_2}\left(\mathbf{E}\left[\boldsymbol{\epsilon}_{t_1}^\top\boldsymbol{\eta}_{t_2,t_2 - t_1 - 1}\mid\mathcal{F}_{t_2, q}\right] - \mathbf{E}\left[\boldsymbol{\epsilon}_{t_1}^\top\boldsymbol{\eta}_{t_2,t_2 - t_1 - 1}\mid\mathcal{F}_{t_2, q - 1}\right]\right)
    \right\Vert_{M/2}\\
    &\leq \sum_{t_1 = 1}^{t_2 - B}\left\vert b_{t_1t_2}\right\vert \left\Vert \mathbf{E}\left[\boldsymbol{\epsilon}_{t_1}^\top\boldsymbol{\eta}_{t_2,t_2 - t_1 - 1}\mid \mathcal{F}_{t_2, q}\right]\right\Vert_{M/2} + \sum_{t_1 = 1}^{t_2 - B}\vert b_{t_1t_2}\vert\left\Vert \mathbf{E}\left[\boldsymbol{\epsilon}_{t_1}^\top\boldsymbol{\eta}_{t_2,t_2 - t_1 - 1}\mid\mathcal{F}_{t_2, q  - 1}\right]\right\Vert_{M/2}\\
    &\leq 2\sum_{t_1 = 1}^{t_2 - B}\vert b_{t_1t_2}\vert \left\Vert\boldsymbol{\epsilon}_{t_1}^\top\boldsymbol{\eta}_{t_2,t_2 - t_1 - 1}\right\Vert_{M/2}.
\end{align*}
From Lemma \ref{lemma.linear_form},
\begin{equation}
\begin{aligned}
    \left\Vert
    \boldsymbol{\epsilon}_{t_1}^\top\boldsymbol{\eta}_{t_2,t_2 - t_1 - 1}
    \right\Vert_{M/2}  & = 
    \left\Vert
    \sum_{j = 1}^d \boldsymbol{\epsilon}_{t_1}^{(j)}\boldsymbol{\eta}_{t_2,t_2 - t_1 - 1}^{(j)}
    \right\Vert_{M/2}\\
    &\leq \sum_{j = 1}^d\left\Vert\boldsymbol{\epsilon}_{t_1}^{(j)}\right\Vert_M \left\Vert\boldsymbol{\eta}_{t_2,t_2 - t_1 - 1}^{(j)}\right\Vert_M\leq \frac{Cd}{(t_2 - t_1)^\alpha}.
\end{aligned}
\label{eq.eta_epsilon}
\end{equation}
Furthermore, since $\vert b_{t_1t_2}\vert\leq 1,$
\begin{align*}
    \sum_{t_1 = 1}^{t_2 - B}\vert b_{t_1t_2}\vert \left\Vert\boldsymbol{\epsilon}_{t_1}^\top\boldsymbol{\eta}_{t_2,t_2 - t_1 - 1}\right\Vert_{M/2}
    \leq \sum_{t_1 = 1}^{t_2 - B}\frac{Cd}{(t_2 - t_1)^\alpha}\leq \sum_{s = B}^\infty \frac{Cd}{s^\alpha}\leq \frac{C_1d}{B^{\alpha - 1}}
\end{align*}
for a constant $C_1.$ From \eqref{eq.tail_moment_to_first}, if $q\leq B,$ we have 
\begin{equation}
    \begin{aligned}
    &\left\Vert
    \sum_{t_2 = 1 + B}^T\sum_{t_1 = 1}^{t_2 - B} b_{t_1t_2}\left(\mathbf{E}\left[\boldsymbol{\epsilon}_{t_1}^\top\boldsymbol{\eta}_{t_2,t_2 - t_1 - 1}\mid\mathcal{F}_{t_2, q}\right] - \mathbf{E}\left[\boldsymbol{\epsilon}_{t_1}^\top\boldsymbol{\eta}_{t_2,t_2 - t_1 - 1}\mid\mathcal{F}_{t_2, q - 1}\right]\right)
    \right\Vert_{M/2}\\
    &\leq \frac{Cd\sqrt{T}}{B^{\alpha - 1}}.
    \end{aligned}
    \label{eq.q_big}
\end{equation}
On the other hand, if $q > B,$ we perform the following decomposition:
\begin{equation}
\begin{aligned}
    &\left\Vert
    \sum_{t_1 = 1}^{t_2 - B} b_{t_1t_2}\left(\mathbf{E}\left[\boldsymbol{\epsilon}_{t_1}^\top\boldsymbol{\eta}_{t_2,t_2 - t_1 - 1}\mid\mathcal{F}_{t_2, q}\right] - \mathbf{E}\left[\boldsymbol{\epsilon}_{t_1}^\top\boldsymbol{\eta}_{t_2,t_2 - t_1 - 1}\mid\mathcal{F}_{t_2, q - 1}\right]\right)
    \right\Vert_{M/2}\\
    &\leq \left\Vert
    \sum_{t_1 = 1\vee (t_2 - q)}^{t_2 - B} b_{t_1t_2}\left(\mathbf{E}\left[\boldsymbol{\epsilon}_{t_1}^\top\boldsymbol{\eta}_{t_2,t_2 - t_1 - 1}\mid\mathcal{F}_{t_2, q}\right] - \mathbf{E}\left[\boldsymbol{\epsilon}_{t_1}^\top\boldsymbol{\eta}_{t_2,t_2 - t_1 - 1}\mid\mathcal{F}_{t_2, q - 1}\right]\right)
    \right\Vert_{M/2}\\
    &+ \left\Vert
    \sum_{t_1 = 1}^{t_2 - q - 1} b_{t_1t_2}\left(\mathbf{E}\left[\boldsymbol{\epsilon}_{t_1}^\top\boldsymbol{\eta}_{t_2,t_2 - t_1 - 1}\mid\mathcal{F}_{t_2, q}\right] - \mathbf{E}\left[\boldsymbol{\epsilon}_{t_1}^\top\boldsymbol{\eta}_{t_2,t_2 - t_1 - 1}\mid\mathcal{F}_{t_2, q - 1}\right]\right)
    \right\Vert_{M/2}.
\end{aligned}
\label{eq.decomp_summation}
\end{equation}
For the first term, $q\geq t_2 - t_1$, and we rearrange   $\boldsymbol{\eta}_{t_2,t_2 - t_1 - 1}$ to 
\begin{align*}
    \boldsymbol{\eta}_{t_2,t_2 - t_1 - 1} = \boldsymbol{\epsilon}_{t_2} - \boldsymbol{\gamma}_{t_2,t_2 - t_1 - 1} 
    = \boldsymbol{\eta}_{t_2,q - 1} + \boldsymbol{\gamma}_{t_2,q - 1}  - \boldsymbol{\gamma}_{t_2,t_2 - t_1 - 1}.
\end{align*}
Since $q - 1\geq t_2 - t_1 - 1,$ $\boldsymbol{\gamma}_{t_2,q - 1}  - \boldsymbol{\gamma}_{t_2,t_2 - t_1 - 1}$ is measurable in $\mathcal{F}_{t_2, q - 1}$. Furthermore,
\begin{align*}
    &\left\Vert
    \sum_{t_1 = 1\vee (t_2 - q)}^{t_2 - B} b_{t_1t_2}\left(\mathbf{E}\left[\boldsymbol{\epsilon}_{t_1}^\top\boldsymbol{\eta}_{t_2,t_2 - t_1 - 1}\mid\mathcal{F}_{t_2, q}\right] - \mathbf{E}\left[\boldsymbol{\epsilon}_{t_1}^\top\boldsymbol{\eta}_{t_2,t_2 - t_1 - 1}\mid\mathcal{F}_{t_2, q - 1}\right]\right)
    \right\Vert_{M/2}\\
    &\leq \left\Vert
    \sum_{t_1 = 1\vee (t_2 - q)}^{t_2 - B} b_{t_1t_2}
    \left(\boldsymbol{\gamma}_{t_2,q - 1}  - \boldsymbol{\gamma}_{t_2,t_2 - t_1 - 1}\right)^\top
    \left(\mathbf{E}\left[\boldsymbol{\epsilon}_{t_1}\mid\mathcal{F}_{t_1, t_1 + q - t_2}\right] - \mathbf{E}\left[\boldsymbol{\epsilon}_{t_1}\mid\mathcal{F}_{t_1, t_1 + q - t_2- 1}\right]\right)
    \right\Vert_{M/2}\\
    &+ \left\Vert
    \sum_{t_1 = 1\vee (t_2 - q)}^{t_2 - B} b_{t_1t_2}\left(\mathbf{E}\left[\boldsymbol{\epsilon}_{t_1}^\top\boldsymbol{\eta}_{t_2,q - 1}\mid\mathcal{F}_{t_2, q}\right] - \mathbf{E}\left[\boldsymbol{\epsilon}_{t_1}^\top\boldsymbol{\eta}_{t_2,q - 1}\mid\mathcal{F}_{t_2, q - 1}\right]\right)
    \right\Vert_{M/2}\\
    &\leq \sum_{t_1 = 1\vee (t_2 - q)}^{t_2 - B}\vert b_{t_1t_2}\vert\sum_{j = 1}^d\left\Vert
    \boldsymbol{\gamma}_{t_2,q - 1}^{(j)}  - \boldsymbol{\gamma}_{t_2,t_2 - t_1 - 1}^{(j)}
    \right\Vert_M\left\Vert
    \boldsymbol{\gamma}_{t_1, t_1 + q - t_2}^{(j)} - \boldsymbol{\gamma}_{t_1, t_1 + q - t_2 - 1}^{(j)}
    \right\Vert_M\\
    &+ 2\sum_{t_1 = 1\vee (t_2 - q)}^{t_2 - B}\vert b_{t_1t_2}\vert\sum_{j = 1}^d \left\Vert\boldsymbol{\epsilon}_{t_1}^{(j)}\boldsymbol{\eta}_{t_2, q -  1}^{(j)}\right\Vert_{M/2}.
\end{align*}
From Lemma \ref{lemma.linear_form},
\begin{align*}
    \left\Vert
    \boldsymbol{\gamma}_{t_2,q - 1}^{(j)}  - \boldsymbol{\gamma}_{t_2,t_2 - t_1 - 1}^{(j)}
    \right\Vert_M & = \left\Vert\boldsymbol{\eta}_{t_2,t_2 - t_1 - 1}^{(j)} - \boldsymbol{\eta}_{t_2,q - 1}^{(j)}\right\Vert_{M}\\
    &\leq \left\Vert\boldsymbol{\eta}_{t_2,t_2 - t_1 - 1}^{(j)}\right\Vert_M + \left\Vert\boldsymbol{\eta}_{t_2,q - 1}^{(j)}\right\Vert_{M}\leq \frac{C}{(t_2 - t_1)^\alpha},
\end{align*}
and 
\begin{align*}
    \left\Vert
    \boldsymbol{\gamma}_{t_1, t_1 + q - t_2}^{(j)} - \boldsymbol{\gamma}_{t_1, t_1 + q - t_2 - 1}^{(j)}
    \right\Vert_M &= 
    \left\Vert
    \boldsymbol{\eta}_{t_1, t_1 + q - t_2 - 1}^{(j)} - \boldsymbol{\eta}_{t_1, t_1 + q - t_2}^{(j)}
    \right\Vert_M\\
    & \leq \left\Vert
    \boldsymbol{\eta}_{t_1, t_1 + q - t_2 - 1}^{(j)}\right\Vert_M + \left\Vert \boldsymbol{\eta}_{t_1, t_1 + q - t_2}^{(j)}\right\Vert_M
    \leq \frac{C}{(1 + t_1 + q - t_2)^\alpha}.
\end{align*}
Notice that $\vert b_{t_1t_2}\vert\leq 1,$ so 
\begin{align*}
    &\sum_{t_1 = 1\vee (t_2 - q)}^{t_2 - B}\vert b_{t_1t_2}\vert \sum_{j = 1}^d\left\Vert
    \boldsymbol{\gamma}_{t_2,q - 1}^{(j)}  - \boldsymbol{\gamma}_{t_2,t_2 - t_1 - 1}^{(j)}
    \right\Vert_M\left\Vert
    \boldsymbol{\gamma}_{t_1, t_1 + q - t_2}^{(j)} - \boldsymbol{\gamma}_{t_1, t_1 + q - t_2 - 1}^{(j)}
    \right\Vert_M\\
    &\leq Cd \sum_{t_1 = 1\vee (t_2 - q)}^{t_2 - B}\frac{1}{(t_2 - t_1)^\alpha\times (1 + t_1 + q - t_2)^\alpha}\\
    &\leq \frac{Cd}{B^\alpha}\sum_{t_1 = 1\vee (t_2 - q)}^{t_2 - B}\frac{1}{(1 + t_1 + q - t_2)^\alpha}\\
    &\leq 
    \begin{cases}
    \frac{Cd}{B^\alpha}\sum_{t_1 =  t_2  - q}^{t_2 - B}\frac{1}{(1 + t_1 + q - t_2)^\alpha}\leq \frac{C_1d}{B^\alpha}\quad \text{if}\quad t_2 \geq q + 1,\\
    \frac{Cd}{B^\alpha}\sum_{t_1 = 1}^{t_2 - B}\frac{1}{(1 + t_1 + q - t_2)^\alpha}\leq \frac{C_1d}{B^\alpha(2 + q - t_2)^{\alpha - 1}}
    \quad \text{if}\quad t_2 < q + 1.
    \end{cases}
\end{align*}
From Lemma \ref{lemma.linear_form}, and notice that $\vert b_{t_1t_2}\vert\leq 1,$
\begin{align*}
    \sum_{t_1 = 1\vee (t_2 - q)}^{t_2 - B}\vert b_{t_1t_2}\vert\sum_{j = 1}^d \left\Vert\boldsymbol{\epsilon}_{t_1}^{(j)}\boldsymbol{\eta}_{t_2, q -  1}^{(j)}\right\Vert_{M/2}
    &\leq  \sum_{t_1 = 1\vee (t_2 - q)}^{t_2 - B}\sum_{j = 1}^d \left\Vert \boldsymbol{\epsilon}_{t_1}^{(j)}\right\Vert_M\left\Vert \boldsymbol{\eta}_{t_2, q -  1}^{(j)}\right\Vert_M\\
    &\leq \sum_{t_1 = 1\vee (t_2 - q)}^{t_2 - B}\frac{Cd}{q^\alpha}\leq \frac{Cd}{q^{\alpha - 1}}.
\end{align*}
Therefore, if $q > B,$
\begin{equation}
    \begin{aligned}
        &\left\Vert
    \sum_{t_1 = 1\vee (t_2 - q)}^{t_2 - B} b_{t_1t_2}\left(\mathbf{E}\left[\boldsymbol{\epsilon}_{t_1}^\top\boldsymbol{\eta}_{t_2,t_2 - t_1 - 1}\mid\mathcal{F}_{t_2, q}\right] - \mathbf{E}\left[\boldsymbol{\epsilon}_{t_1}^\top\boldsymbol{\eta}_{t_2,t_2 - t_1 - 1}\mid\mathcal{F}_{t_2, q - 1}\right]\right)
    \right\Vert_{M/2}\\
    &\leq 
    \begin{cases}
        \frac{Cd}{B^\alpha} + \frac{Cd}{q^{\alpha -1}}\quad \text{if} \quad  t_2 \geq q + 1,\\
        \frac{Cd}{B^\alpha(2 + q - t_2)^{\alpha - 1}} + \frac{Cd}{q^{\alpha - 1}}\quad \text{if}\quad t_2 < q + 1. 
    \end{cases}
    \end{aligned}
\end{equation}
Now we consider the second term in \eqref{eq.decomp_summation}.   If $t_2\leq q + 1$, then $t_2 - q - 1 < 1$ and the second term equals $0$. Otherwise, if
$t_2 \geq q + 2$, for the second term, from Lemma \ref{lemma.linear_form}, and notice that $\vert b_{t_1t_2}\vert\leq 1,$
\begin{align*}
    &\left\Vert
    \sum_{t_1 = 1}^{t_2 - q - 1} b_{t_1t_2}\left(\mathbf{E}\left[\boldsymbol{\epsilon}_{t_1}^\top\boldsymbol{\eta}_{t_2,t_2 - t_1 - 1}\mid\mathcal{F}_{t_2, q}\right] - \mathbf{E}\left[\boldsymbol{\epsilon}_{t_1}^\top\boldsymbol{\eta}_{t_2,t_2 - t_1 - 1}\mid\mathcal{F}_{t_2, q - 1}\right]\right)
    \right\Vert_{M/2}\\
    &\leq
    \sum_{t_1 = 1}^{t_2 - q - 1}\left\Vert \mathbf{E}\left[\boldsymbol{\epsilon}_{t_1}^\top\boldsymbol{\eta}_{t_2,t_2 - t_1 - 1}\mid\mathcal{F}_{t_2, q}\right]\right\Vert_{M/2} + \sum_{t_1 = 1}^{t_2 - q - 1}\left\Vert \mathbf{E}\left[\boldsymbol{\epsilon}_{t_1}^\top\boldsymbol{\eta}_{t_2,t_2 - t_1 - 1}\mid\mathcal{F}_{t_2, q - 1}\right]\right\Vert_{M/2}
    \\ 
    &\leq 
    2\sum_{t_1 = 1}^{t_2 - q - 1}\left\Vert \boldsymbol{\epsilon}_{t_1}^\top\boldsymbol{\eta}_{t_2,t_2 - t_1 - 1}\right\Vert_{M/2}
    \\
    & \leq 2\sum_{t_1 = 1}^{t_2 - q - 1}\sum_{j = 1}^d \left\Vert \boldsymbol{\epsilon}_{t_1}^{(j)}\right\Vert_M\left\Vert\boldsymbol{\eta}_{t_2,t_2 - t_1 - 1}^{(j)}\right\Vert_{M}\\
    & \leq Cd\sum_{t_1 = 1}^{t_2 - q - 1}\frac{1}{(t_2 - t_1)^\alpha}\leq \frac{C_1d}{(1 + q)^{\alpha - 1}}.
\end{align*}
Combine the aforementioned results, if $q > B$, we have 
\begin{align*}
    &\left\Vert
    \sum_{t_1 = 1}^{t_2 - B} b_{t_1t_2}\left(\mathbf{E}\left[\boldsymbol{\epsilon}_{t_1}^\top\boldsymbol{\eta}_{t_2,t_2 - t_1 - 1}\mid\mathcal{F}_{t_2, q}\right] - \mathbf{E}\left[\boldsymbol{\epsilon}_{t_1}^\top\boldsymbol{\eta}_{t_2,t_2 - t_1 - 1}\mid\mathcal{F}_{t_2, q - 1}\right]\right)
    \right\Vert_{M/2}\\
    &\leq 
    \begin{cases}
        \frac{Cd}{B^\alpha} + \frac{Cd}{q^{\alpha -1}}\quad \text{if} \quad  t_2 \geq q + 1,\\
        \frac{Cd}{B^\alpha(2 + q - t_2)^{\alpha - 1}} + \frac{Cd}{q^{\alpha - 1}}\quad \text{if}\quad t_2 < q + 1. 
    \end{cases}
\end{align*}
From \eqref{eq.tail_moment_to_first}, if $q > B$,
\begin{align*}
    &\left\Vert
    \sum_{t_2 = 1 + B}^T\sum_{t_1 = 1}^{t_2 - B} b_{t_1t_2}\left(\mathbf{E}\left[\boldsymbol{\epsilon}_{t_1}^\top\boldsymbol{\eta}_{t_2,t_2 - t_1 - 1}\mid\mathcal{F}_{t_2, q}\right] - \mathbf{E}\left[\boldsymbol{\epsilon}_{t_1}^\top\boldsymbol{\eta}_{t_2,t_2 - t_1 - 1}\mid\mathcal{F}_{t_2, q - 1}\right]\right)
    \right\Vert_{M/2}\\
    &\leq C\sqrt{
    \sum_{t_2 = 1 + B}^{q\wedge T} \left(\frac{d^2}{q^{2\alpha - 2}} + \frac{d^2}{B^{2\alpha} (2 + q  - t_2)^{2\alpha - 2}}\right)
    } + C\sqrt{\sum_{t_2 = 1 + q}^T\frac{d^2}{B^{2\alpha}} + \frac{d^2}{q^{2\alpha - 2}}}\\
    &\leq \frac{C_1d\sqrt{T}}{q^{\alpha - 1}} + \frac{C_1d}{B^\alpha}\sqrt{\sum_{t_2 = 1 + q}^T 1} + \frac{C_1 d}{B^\alpha}\sqrt{\sum_{t_2 = 1 + B}^{q\wedge T}\frac{1}{(2 + q - t_2)^{2\alpha - 2}}}.
\end{align*}
If $q\leq T - 1,$ then 
\begin{align*}
    \sum_{t_2 = 1 + B}^{q\wedge T}\frac{1}{(2 + q - t_2)^{2\alpha - 2}} = \sum_{t_2 = 1 + B}^{q}\frac{1}{(2 + q - t_2)^{2\alpha - 2}} =  \sum_{s = 2}^{q - B + 1}\frac{1}{s^{2\alpha - 2}}\leq C
\end{align*}
and 
$$
\sum_{t_2 = 1 + q}^T 1 = T - q\leq T,
$$
in such case 
\begin{equation}
\begin{aligned}
    &\left\Vert
    \sum_{t_2 = 1 + B}^T\sum_{t_1 = 1}^{t_2 - B} b_{t_1t_2}\left(\mathbf{E}\left[\boldsymbol{\epsilon}_{t_1}^\top\boldsymbol{\eta}_{t_2,t_2 - t_1 - 1}\mid\mathcal{F}_{t_2, q}\right] - \mathbf{E}\left[\boldsymbol{\epsilon}_{t_1}^\top\boldsymbol{\eta}_{t_2,t_2 - t_1 - 1}\mid\mathcal{F}_{t_2, q - 1}\right]\right)
    \right\Vert_{M/2}\\
    &\leq \frac{Cd\sqrt{T}}{q^{\alpha - 1}} + \frac{Cd\sqrt{T}}{B^\alpha}.
\end{aligned}
\label{eq.fir_half_fir}
\end{equation}
On the other hand, if $q\geq T,$ then 
\begin{align*}
    \sum_{t_2 = 1 + B}^{q\wedge T}\frac{1}{(2 + q - t_2)^{2\alpha - 2}} = \sum_{t_2 = 1 + B}^{T}\frac{1}{(2 + q - t_2)^{2\alpha - 2}} = \sum_{s = 2 +q - T}^{q - B + 1}\frac{1}{s^{2\alpha - 2}}\leq \frac{C}{(2 + q - T)^{2\alpha - 3}},
\end{align*}
and 
$
\sum_{t_2 = 1 + q}^T 1 = 0.
$
In such case 
\begin{equation}
\begin{aligned}
    &\left\Vert
    \sum_{t_2 = 1 + B}^T\sum_{t_1 = 1}^{t_2 - B} b_{t_1t_2}\left(\mathbf{E}\left[\boldsymbol{\epsilon}_{t_1}^\top\boldsymbol{\eta}_{t_2,t_2 - t_1 - 1}\mid\mathcal{F}_{t_2, q}\right] - \mathbf{E}\left[\boldsymbol{\epsilon}_{t_1}^\top\boldsymbol{\eta}_{t_2,t_2 - t_1 - 1}\mid\mathcal{F}_{t_2, q - 1}\right]\right)
    \right\Vert_{M/2}\\
    &\leq \frac{Cd\sqrt{T}}{q^{\alpha - 1}} + \frac{Cd}{B^\alpha}\sqrt{\frac{1}{(2 + q - T)^{2\alpha - 3}}}.
\end{aligned}
\label{eq.fir_half_sec}
\end{equation}

From \eqref{eq.q_big}, \eqref{eq.fir_half_fir}, and \eqref{eq.fir_half_sec},
\begin{equation}
    \begin{aligned} 
    &\sum_{q = 0}^\infty
        \left\Vert
        \sum_{t_2 = 1 + B}^T\sum_{t_1 = 1}^{t_2 - B} b_{t_1t_2}\left(\mathbf{E}\left[\boldsymbol{\epsilon}_{t_1}^\top\boldsymbol{\eta}_{t_2,t_2 - t_1 - 1}\mid\mathcal{F}_{t_2, q}\right] - \mathbf{E}\left[\boldsymbol{\epsilon}_{t_1}^\top\boldsymbol{\eta}_{t_2,t_2 - t_1 - 1}\mid\mathcal{F}_{t_2, q - 1}\right]\right)
        \right\Vert_{M/2}\\
    &\leq 
    \sum_{q = 0}^B \frac{Cd\sqrt{T}}{B^{\alpha - 1}} + \sum_{q = B + 1}^{T - 1} \left(\frac{Cd\sqrt{T}}{q^{\alpha - 1}} + \frac{Cd\sqrt{T}}{B^\alpha}\right) + \sum_{q = T}^\infty\left(\frac{Cd\sqrt{T}}{q^{\alpha - 1}} + \frac{Cd}{B^\alpha}\sqrt{\frac{1}{(2 + q - T)^{2\alpha - 3}}}\right)
    \\
    &\leq
    \frac{C_1d\sqrt{T}}{B^{\alpha - 2}} + \frac{C_1d T^{3/2}}{B^\alpha} + \frac{C_1 d}{B^\alpha}.
    \end{aligned}
    \label{eq.summ_sec_part}
\end{equation}
From \eqref{eq.summary_first_half} and \eqref{eq.summ_sec_part}, and notice that $d\asymp T,$ we have 
\begin{align*}
    \left\Vert
        \sum_{t_1 = 1}^{T - B}\sum_{t_2 = t_1 + B}^T b_{t_1t_2}\left(\boldsymbol{\epsilon}_{t_1}^\top \boldsymbol{\epsilon}_{t_2} - \mathbf{E}\left[\boldsymbol{\epsilon}_{t_1}^\top \boldsymbol{\epsilon}_{t_2}\right]\right)
        \right\Vert_{M/2}
    &\leq C\sqrt{d\sum_{t_1 =  1}^{T - B}\sum_{t_2 = t_1 + B}^Tb^2_{t_1t_2}} + \frac{CT^{5/2}}{B^\alpha} + \frac{CT^{3/2}}{B^{\alpha - 2}}\\
    & \leq C\sqrt{d\sum_{t_1 =  1}^{T}\sum_{t_2 = 1}^T b^2_{t_1t_2}} + \frac{CT^{5/2}}{B^\alpha} + \frac{CT^{3/2}}{B^{\alpha - 2}}.
\end{align*}
Since 
\begin{align*}
    \left\Vert
        \sum_{t_2 = 1}^{T - B}\sum_{t_1 = t_2 + B}^T b_{t_1t_2}\left(\boldsymbol{\epsilon}_{t_1}^\top \boldsymbol{\epsilon}_{t_2} - \mathbf{E}\left[\boldsymbol{\epsilon}_{t_1}^\top \boldsymbol{\epsilon}_{t_2}\right]\right)
        \right\Vert_{M/2} = \left\Vert
        \sum_{t_1 = 1}^{T - B}\sum_{t_2 = t_1 + B}^T b_{t_1t_2}^\dagger\left(\boldsymbol{\epsilon}_{t_2}^\top \boldsymbol{\epsilon}_{t_1} - \mathbf{E}\left[\boldsymbol{\epsilon}_{t_2}^\top \boldsymbol{\epsilon}_{t_1}\right]\right)
        \right\Vert_{M/2},
\end{align*}
where $b_{t_1t_2}^\dagger = b_{t_2t_1}$, the aforementioned results also hold and we have 
\begin{align*}
    \left\Vert
        \sum_{t_2 = 1}^{T - B}\sum_{t_1 = t_2 + B}^T b_{t_1t_2}\left(\boldsymbol{\epsilon}_{t_1}^\top \boldsymbol{\epsilon}_{t_2} - \mathbf{E}\left[\boldsymbol{\epsilon}_{t_1}^\top \boldsymbol{\epsilon}_{t_2}\right]\right)
        \right\Vert_{M/2} \leq C\sqrt{d\sum_{t_1 =  1}^{T}\sum_{t_2 = 1}^T b^2_{t_1t_2}} + \frac{CT^{5/2}}{B^\alpha} + \frac{CT^{3/2}}{B^{\alpha - 2}}.
\end{align*}
This proves eq.\eqref{eq.whole_formula}.

\textbf{2. The proof of equation \eqref{eq.truncate_moment}.}

We then prove equation \eqref{eq.truncate_moment}. Notice that 
\begin{align*}
    &\left\Vert
        \sum_{t_1 = 1}^T\sum_{t_2 = 1}^T b_{t_1t_2}\left(\boldsymbol{\epsilon}_{t_1}^\top \boldsymbol{\epsilon}_{t_2} - \mathbf{E}\left[\boldsymbol{\epsilon}_{t_1}^\top \boldsymbol{\epsilon}_{t_2}\mid\mathcal{F}_{t_1\vee t_2, \ell}\right]\right)
    \right\Vert_{M/2}\\
    &\leq \left\Vert
        \sum_{t_1 = 1}^{T - B}\sum_{t_2 = t_1 + B}^{T} b_{t_1t_2}\left(\boldsymbol{\epsilon}_{t_1}^\top \boldsymbol{\epsilon}_{t_2} - \mathbf{E}\left[\boldsymbol{\epsilon}_{t_1}^\top \boldsymbol{\epsilon}_{t_2}\mid\mathcal{F}_{t_2, \ell}\right]\right)
        \right\Vert_{M/2}\\
    &+ \left\Vert
        \sum_{t_2 = 1}^{T - B}\sum_{t_1 = t_2 + B}^{T} b_{t_1t_2}\left(\boldsymbol{\epsilon}_{t_1}^\top \boldsymbol{\epsilon}_{t_2} - \mathbf{E}\left[\boldsymbol{\epsilon}_{t_1}^\top \boldsymbol{\epsilon}_{t_2}\mid\mathcal{F}_{t_1, \ell}\right]\right)
    \right\Vert_{M/2}.
\end{align*}
Suppose $t_2 > t_1.$ In this case, notice that 
$$\boldsymbol{\epsilon}_{t_2} = \boldsymbol{\gamma}_{t_2, (t_2 - t_1 - 1)\wedge \ell} + \boldsymbol{\eta}_{t_2, (t_2 - t_1 - 1)\wedge \ell},
$$  
where $ \boldsymbol{\gamma}_{t_2, (t_2 - t_1 - 1)\wedge \ell}$ is independent of $\boldsymbol{\epsilon}_{t_1},$ and is measurable in $\mathcal{F}_{t_2,\ell}.$ From this decomposition, we have 
\begin{equation}
\begin{aligned}
    &\left\Vert
        \sum_{t_1 = 1}^{T - B}\sum_{t_2 = t_1 + B}^{T} b_{t_1t_2}\left(\boldsymbol{\epsilon}_{t_1}^\top \boldsymbol{\epsilon}_{t_2} - \mathbf{E}\left[\boldsymbol{\epsilon}_{t_1}^\top \boldsymbol{\epsilon}_{t_2}|\mathcal{F}_{t_2, \ell}\right]\right)
    \right\Vert_{M/2}\\
    &\leq \left\Vert
        \sum_{t_1 = 1}^{T - B}\sum_{t_2 = t_1 + B}^{T} b_{t_1t_2}\boldsymbol{\gamma}_{t_2, (t_2 - t_1 - 1)\wedge \ell}^\top\left(\boldsymbol{\epsilon}_{t_1} - \mathbf{E}\left[\boldsymbol{\epsilon}_{t_1}\mid\mathcal{F}_{t_2, \ell}\right]\right)
        \right\Vert_{M/2}\\
    &+ \left\Vert
        \sum_{t_1 = 1}^{T - B}\sum_{t_2 = t_1 + B}^{T} b_{t_1t_2}\left(\boldsymbol{\epsilon}_{t_1}^\top \boldsymbol{\eta}_{t_2, (t_2 - t_1 - 1)\wedge \ell} - \mathbf{E}\left[\boldsymbol{\epsilon}_{t_1}^\top \boldsymbol{\eta}_{t_2, (t_2 - t_1 - 1)\wedge \ell}\mid\mathcal{F}_{t_2, \ell}\right]\right)
        \right\Vert_{M/2}\\
    &\leq 
        \left\Vert
        \sum_{t_1 = 1}^{T - B}\sum_{t_2 = t_1 + B}^{T\wedge(t_1 + \ell )} b_{t_1t_2}\boldsymbol{\gamma}_{t_2, (t_2 - t_1 - 1)\wedge \ell}^\top\left(\boldsymbol{\epsilon}_{t_1} - \mathbf{E}\left[\boldsymbol{\epsilon}_{t_1}\mid\mathcal{F}_{t_1, t_1 + \ell - t_2}\right]\right)
        \right\Vert_{M/2}\\
    &+\left\Vert
        \sum_{t_1 = 1}^{T - B}\sum_{t_2 =  t_1 + \ell +  1}^{T} b_{t_1t_2}\boldsymbol{\gamma}_{t_2, (t_2 - t_1 - 1)\wedge \ell}^\top\boldsymbol{\epsilon}_{t_1}
        \right\Vert_{M/2}\\
    &+\sum_{q = \ell + 1}^\infty \left\Vert
        \sum_{t_1 = 1}^{T - B}\sum_{t_2 = t_1 + B}^{T} b_{t_1t_2}\left(\mathbf{E}\left[\boldsymbol{\epsilon}_{t_1}^\top \boldsymbol{\eta}_{t_2, (t_2 - t_1 - 1)\wedge \ell}\mid\mathcal{F}_{t_2, q}\right] - \mathbf{E}\left[\boldsymbol{\epsilon}_{t_1}^\top \boldsymbol{\eta}_{t_2, (t_2 - t_1 - 1)\wedge \ell}\mid\mathcal{F}_{t_2,  q - 1}\right]\right)
        \right\Vert_{M/2}
        .
\end{aligned}
\label{eq.separate_3}
\end{equation}
For the first term in \eqref{eq.separate_3}, notice that 
\begin{align*}
    \boldsymbol{\epsilon}_{t_1} - \mathbf{E}\left[\boldsymbol{\epsilon}_{t_1}\mid\mathcal{F}_{t_1, t_1 + \ell - t_2}\right] = 
    \sum_{q = t_1 + \ell - t_2 + 1}^\infty 
    \boldsymbol{\gamma}_{t_1, q} - \boldsymbol{\gamma}_{t_1, q - 1},\quad \text{and}\quad t_2\leq t_1 + \ell.
\end{align*}
Therefore, we have
\begin{align*}
    &\left\Vert
        \sum_{t_1 = 1}^{T - B}\sum_{t_2 = t_1 + B}^{T\wedge (t_1 + \ell)} b_{t_1t_2}\boldsymbol{\gamma}_{t_2, (t_2 - t_1 - 1)\wedge \ell}^\top\left(\boldsymbol{\epsilon}_{t_1} - \mathbf{E}\left[\boldsymbol{\epsilon}_{t_1}\mid\mathcal{F}_{t_1, t_1 + \ell - t_2}\right]\right)
    \right\Vert_{M/2}\\
    & = \left\Vert
        \sum_{t_1 = 1}^{T - B}\sum_{t_2 = t_1 + B}^{T\wedge (t_1 + \ell)} 
        \sum_{q = t_1 + \ell - t_2 + 1}^\infty 
        b_{t_1t_2}\boldsymbol{\gamma}_{t_2, t_2 - t_1 - 1}^\top\left(\boldsymbol{\gamma}_{t_1, q} - \boldsymbol{\gamma}_{t_1, q - 1}\right)
        \right\Vert_{M/2}\\
    &\leq \sum_{q = 1}^\infty\left\Vert
        \sum_{t_1 = 1}^{T - B}\sum_{t_2 = (t_1 + B)\vee (t_1 + \ell + 1 - q)}^{T\wedge (t_1 + \ell)}
        b_{t_1t_2}\boldsymbol{\gamma}_{t_2, t_2 - t_1 - 1}^\top\left(\boldsymbol{\gamma}_{t_1, q} - \boldsymbol{\gamma}_{t_1, q - 1}\right)
        \right\Vert_{M/2}.
\end{align*}
For any given $q\geq 1$ and $s = 1,\cdots, T-B,$ define 
\begin{align*}
    M_{q,s} = 
        \sum_{t_1 = T - B - s + 1}^{T - B}\sum_{t_2 = (t_1 + B)\vee (t_1 + \ell + 1 - q)}^{T\wedge (t_1 + \ell)}
        b_{t_1t_2}\boldsymbol{\gamma}_{t_2, t_2 - t_1 - 1}^\top(\boldsymbol{\gamma}_{t_1, q} - \boldsymbol{\gamma}_{t_1, q - 1})
\end{align*}
and $\mathcal{M}_{q,s}$ the $\sigma$-field generated by $e_T, e_{T-1},\cdots, e_{T - B - s + 1 - q},$ then $ M_{q,s}$ is measurable in $\mathcal{M}_{q,s}$, $\mathcal{M}_{q,s}\subset \mathcal{M}_{q,s + 1}$, and 
\begin{align*}
    &\mathbf{E}\left[\left(M_{q,s + 1} - M_{q,s}\right)\mid\mathcal{M}_{q,s}\right]\\
    &= \sum_{t_2 = (T - s)\vee (T - B - s + \ell + 1 - q)}^{T\wedge (T - B - s + \ell)}
        b_{(T - B - s)t_2}\mathbf{E}\left[\boldsymbol{\gamma}_{t_2, t_2 - T + B + s - 1}^\top(\boldsymbol{\gamma}_{T - B - s, q} - \boldsymbol{\gamma}_{T - B - s, q - 1})\mid \mathcal{M}_{q,s}\right]\\
    &= \sum_{t_2 = (T - s)\vee (T - B - s + \ell + 1 - q)}^{T\wedge (T - B - s + \ell)}
        b_{(T - B - s)t_2}\boldsymbol{\gamma}_{t_2, t_2 - T + B + s - 1}^\top(\boldsymbol{\gamma}_{T - B - s, q - 1} - \boldsymbol{\gamma}_{T - B - s, q - 1}) = 0,
\end{align*}
so $M_{q,s}$ forms a martingale. According to Theorem 1.1 of \cite{MR0400380},
\begin{align*}
   &\left\Vert
        \sum_{t_1 = 1}^{T - B}\sum_{t_2 = (t_1 + B)\vee (t_1 + \ell + 1 - q)}^{T\wedge (t_1 + \ell)}
        b_{t_1t_2}\boldsymbol{\gamma}_{t_2, t_2 - t_1 - 1}^\top\left(\boldsymbol{\gamma}_{t_1, q} - \boldsymbol{\gamma}_{t_1, q - 1}\right)
    \right\Vert_{M/2}\\
    & \leq C\sqrt{\sum_{t_1 = 1}^{T - B}\left\Vert
        \left(\boldsymbol{\gamma}_{t_1, q} - \boldsymbol{\gamma}_{t_1, q - 1}\right)^\top \sum_{t_2 = (t_1 + B)\vee (t_1 + \ell + 1 - q)}^{T\wedge (t_1 + \ell)}
        b_{t_1t_2}\boldsymbol{\gamma}_{t_2, t_2 - t_1 - 1}
    \right\Vert^2_{M/2}}.
\end{align*}
For any $t_2$ such that $(t_1 + B)\vee (t_1 + \ell + 1 - q)\leq t_2\leq T\wedge (t_1 + \ell),$ 
$$
\boldsymbol{\gamma}_{t_2, t_2 - t_1 - 1} = \mathbf{E}\left[\boldsymbol{\epsilon}_{t_2}\mid\mathcal{F}_{t_2, t_2 - t_1 - 1}\right] = \mathbf{E}\left[\boldsymbol{\epsilon}_{t_2}\mid\mathcal{F}_{T, T - t_1 - 1}\right].
$$
From Lemma \ref{lemma.linear_form}, for any vector $\boldsymbol{\tau}\in\mathbf{R}^d,$
\begin{align*}
        &\left\Vert
        \sum_{t_2 = (t_1 + B)\vee (t_1 + \ell + 1 - q)}^{T\wedge (t_1 + \ell)}
        b_{t_1t_2} \boldsymbol{\tau}^\top\boldsymbol{\gamma}_{t_2, t_2 - t_1 - 1}
        \right\Vert_{M/2}\\
     &= \left\Vert
        \sum_{t_2 = (t_1 + B)\vee (t_1 + \ell + 1 - q)}^{T\wedge (t_1 + \ell)}
        b_{t_1t_2} \boldsymbol{\tau}^\top
        \mathbf{E}\left[\boldsymbol{\epsilon}_{t_2}\mid\mathcal{F}_{T, T - t_1 - 1}\right]
        \right\Vert_{M/2}\\
    &=
    \left\Vert\mathbf{E}\left[
        \sum_{t_2 = (t_1 + B)\vee (t_1 + \ell + 1 - q)}^{T\wedge (t_1 + \ell)}
        b_{t_1t_2} \boldsymbol{\tau}^\top
        \boldsymbol{\epsilon}_{t_2}\mid\mathcal{F}_{T, T - t_1 - 1}\right]
        \right\Vert_{M/2}
    \\
    &\leq 
        \left\Vert
        \sum_{t_2 = (t_1 + B)\vee (t_1 + \ell + 1 - q)}^{T\wedge (t_1 + \ell)}
        b_{t_1t_2}\boldsymbol{\tau}^\top\boldsymbol{\epsilon}_{t_2}
        \right\Vert_{M/2}\\
    &\leq C\vert\boldsymbol{\tau}\vert_2 \sqrt{\sum_{t_2 = (t_1 + B)\vee (t_1 + \ell + 1 - q)}^{T\wedge (t_1 + \ell)}
        b_{t_1t_2}^2}.
\end{align*}
Since $\boldsymbol{\gamma}_{t_1, q} - \boldsymbol{\gamma}_{t_1, q - 1}$ is independent of $\sum_{t_2 = (t_1 + B)\vee (t_1 + \ell + 1 - q)}^{T\wedge (t_1 + \ell)}
        b_{t_1t_2}\boldsymbol{\gamma}_{t_2, t_2 - t_1 - 1},$ from Theorem 1.7 in \cite{MR2002723}, for any vector $\boldsymbol{\tau}\in\mathbf{R}^d,$ 
\begin{align*}
    &\mathbf{E}\left[\left\vert\ (\boldsymbol{\gamma}_{t_1, q} - \boldsymbol{\gamma}_{t_1, q - 1})^\top \sum_{t_2 = (t_1 + B)\vee (t_1 + \ell + 1 - q)}^{T\wedge (t_1 + \ell)}
        b_{t_1t_2}\boldsymbol{\gamma}_{t_2, t_2 - t_1 - 1}\ \right\vert^{M/2}\mid\boldsymbol{\gamma}_{t_1, q} - \boldsymbol{\gamma}_{t_1, q - 1} = \boldsymbol{\tau}\right]\\
    & = \mathbf{E}\left[\left\vert\ \boldsymbol{\tau}^\top \sum_{t_2 = (t_1 + B)\vee (t_1 + \ell + 1 - q)}^{T\wedge (t_1 + \ell)}
        b_{t_1t_2}\boldsymbol{\gamma}_{t_2, t_2 - t_1 - 1}\ \right\vert^{M/2}\right]\\
    &\leq C\left(\sum_{t_2 = (t_1 + B)\vee (t_1 + \ell + 1 - q)}^{T\wedge (t_1 + \ell)}
        b_{t_1t_2}^2\right)^{M/4} \vert \boldsymbol{\tau}\vert_2^{M/2}\\
    &= C\left(\sum_{t_2 = (t_1 + B)\vee (t_1 + \ell + 1 - q)}^{T\wedge (t_1 + \ell)}
        b_{t_1t_2}^2\right)^{M/4} \vert \boldsymbol{\gamma}_{t_1, q} - \boldsymbol{\gamma}_{t_1, q - 1}\vert_2^{M/2}
        ,
\end{align*}
which implies that 
\begin{align*}
&\left\Vert
        \left(\boldsymbol{\gamma}_{t_1, q} - \boldsymbol{\gamma}_{t_1, q - 1}\right)^\top \sum_{t_2 = (t_1 + B)\vee (t_1 + \ell + 1 - q)}^{T\wedge (t_1 + \ell)}
        b_{t_1t_2}\boldsymbol{\gamma}_{t_2, t_2 - t_1 - 1}
\right\Vert_{M/2}\\
&\leq 
C\sqrt{\sum_{t_2 = (t_1 + B)\vee (t_1 + \ell + 1 - q)}^{T\wedge (t_1 + \ell)}
        b_{t_1t_2}^2}\left\Vert\ \left\vert \boldsymbol{\gamma}_{t_1, q} - \boldsymbol{\gamma}_{t_1, q - 1}\right\vert_2\ \right\Vert_{M/2}\\
&\leq C\sqrt{\sum_{t_2 = (t_1 + B)\vee (t_1 + \ell + 1 - q)}^{T\wedge (t_1 + \ell)}
        b_{t_1t_2}^2}\sqrt{\sum_{j = 1}^d \left\Vert \boldsymbol{\gamma}_{t_1, q}^{(j)} - \boldsymbol{\gamma}_{t_1, q - 1}^{(j)}\right\Vert^2_{M/2}}\\
& \leq C_1 \delta_q\sqrt{d\sum_{t_2 = (t_1 + B)\vee (t_1 + \ell + 1 - q)}^{T\wedge (t_1 + \ell)}
        b_{t_1t_2}^2}.
\end{align*}
Therefore, 
\begin{align*}
    &\left\Vert
        \sum_{t_1 = 1}^{T - B}\sum_{t_2 = (t_1 + B)\vee (t_1 + \ell + 1 - q)}^{T\wedge (t_1 + \ell)}
        b_{t_1t_2}\boldsymbol{\gamma}_{t_2, t_2 - t_1 - 1}^\top\left(\boldsymbol{\gamma}_{t_1, q} - \boldsymbol{\gamma}_{t_1, q - 1}\right)
        \right\Vert_{M/2}\\
    &\leq C\delta_q\sqrt{d\sum_{t_1 = 1}^{T - B}\sum_{t_2 = (t_1 + B)\vee (t_1 + \ell + 1 - q)}^{T\wedge (t_1 + \ell)}
        b_{t_1t_2}^2},
\end{align*}
and 
\begin{equation}
\label{eq.fir_term_condition}
    \begin{aligned}
        &\left\Vert
        \sum_{t_1 = 1}^{T - B}\sum_{t_2 = t_1 + B}^{T\wedge (t_1 + \ell)} b_{t_1t_2}\boldsymbol{\gamma}_{t_2, (t_2 - t_1 - 1)\wedge \ell}^\top\left(\boldsymbol{\epsilon}_{t_1} - \mathbf{E}\left[\boldsymbol{\epsilon}_{t_1}|\mathcal{F}_{t_2, \ell}\right]\right)
        \right\Vert_{M/2}\\
        &\leq C\sum_{q = 1}^\infty \delta_q\sqrt{d\sum_{t_1 = 1}^{T - B}\sum_{t_2 = (t_1 + B)\vee (t_1 + \ell + 1 - q)}^{T\wedge (t_1 + \ell)}
        b_{t_1t_2}^2}.
    \end{aligned}
\end{equation}
For the second term in \eqref{eq.separate_3},  notice that $t_2\geq t_1  + \ell + 1,$ we have 
\begin{align*}
    &\left\Vert\sum_{t_1 = 1}^{T - B}\sum_{t_2 = t_1 + \ell +  1}^{T} b_{t_1t_2}\boldsymbol{\gamma}_{t_2, (t_2 - t_1 - 1)\wedge \ell}^\top\boldsymbol{\epsilon}_{t_1}\right\Vert_{M/2}\\
    &= \left\Vert\sum_{t_1 = 1}^{T - B}\sum_{t_2 = t_1 + \ell +  1}^{T} b_{t_1t_2}\boldsymbol{\gamma}_{t_2,  \ell}^\top\boldsymbol{\epsilon}_{t_1}\right\Vert_{M/2}\\
    &\leq  \sum_{q = 0}^\infty \left\Vert\sum_{t_1 = 1}^{T - B}\sum_{t_2 = t_1 + \ell +  1}^{T} b_{t_1t_2}\boldsymbol{\gamma}_{t_2,  \ell}^\top\left(\mathbf{E}\left[\boldsymbol{\epsilon}_{t_1}\mid\mathcal{F}_{t_1, q}\right] - \mathbf{E}\left[\boldsymbol{\epsilon}_{t_1}\mid\mathcal{F}_{t_1, q - 1}\right]\right)\right\Vert_{M/2}\\
   & = \sum_{q = 0}^\infty \left\Vert\sum_{t_1 = 1}^{T - B}\sum_{t_2 = t_1 + \ell +  1}^{T} b_{t_1t_2}\boldsymbol{\gamma}_{t_2,  \ell}^\top\left(\boldsymbol{\gamma}_{t_1,q} - \boldsymbol{\gamma}_{t_1,q - 1}\right)\right\Vert_{M/2},
\end{align*}
where $\boldsymbol{\gamma}_{t_1, - 1} = \mathbf{E}\left[\boldsymbol{\epsilon}_{t_1}\mid\mathcal{F}_{t_1, - 1}\right] = \mathbf{E}\left[\boldsymbol{\epsilon}_{t_1}\right] = 0.$ For any $q\geq 0$ and any $s = 1,2,\cdots, T-B,$ define 
\begin{align*}
    G_{q,s} = 
        \sum_{t_1 = T  - B - s + 1}^{T - B}\sum_{t_2 = t_1 + \ell +  1}^{T} b_{t_1t_2}\boldsymbol{\gamma}_{t_2,  \ell}^\top(\boldsymbol{\gamma}_{t_1,q} - \boldsymbol{\gamma}_{t_1,q - 1})
\end{align*}
and $\mathcal{G}_{q,s}$ the $\sigma$-field generated by $e_{T},\cdots, e_{T - B - s + 1 - q}$. Then $G_{q,s}$ is measurable in $\mathcal{G}_{q,s},$ $\mathcal{G}_{q,s}\subset \mathcal{G}_{q,s + 1},$ and
\begin{align*}
    &\mathbf{E}\left[\left(G_{q,s + 1} - G_{q,s}\right)\mid\mathcal{G}_{q,s}\right]\\
   & = 
    \sum_{t_2 = T - B - s + \ell +  1}^{T} b_{(T-B-s)t_2}\mathbf{E}\left[\boldsymbol{\gamma}_{t_2,  \ell}^\top\left(\boldsymbol{\gamma}_{T - B - s, q} - \boldsymbol{\gamma}_{T - B - s, q - 1}\right)\mid\mathcal{G}_{q,s}\right]\\
    &= \sum_{t_2 = T - B - s + \ell +  1}^{T} b_{(T-B-s)t_2}\boldsymbol{\gamma}_{t_2,  \ell}^\top(\boldsymbol{\gamma}_{T - B - s, q - 1} - \boldsymbol{\gamma}_{T - B - s, q - 1})  = 0,
\end{align*}
so $G_{q,s}$ forms a martingale.From Theorem 1.1 of \cite{MR0400380},
\begin{align*}
    &\left\Vert\sum_{t_1 = 1}^{T - B}\sum_{t_2 = t_1 + \ell +  1}^{T} b_{t_1t_2}\boldsymbol{\gamma}_{t_2,  \ell}^\top\left(\boldsymbol{\gamma}_{t_1,q} - \boldsymbol{\gamma}_{t_1,q - 1}\right)\right\Vert_{M/2}\\
    &\leq C\sqrt{\sum_{t_1 = 1}^{T - B}\left\Vert
    (\boldsymbol{\gamma}_{t_1,q} - \boldsymbol{\gamma}_{t_1,q - 1})^\top\sum_{t_2 = t_1 + \ell +  1}^{T} b_{t_1t_2}\boldsymbol{\gamma}_{t_2,  \ell}
    \right\Vert^2_{M/2}}.
\end{align*}
Since $t_2\geq t_1  + \ell + 1$ in the summation,  we have $\boldsymbol{\gamma}_{t_1,q} - \boldsymbol{\gamma}_{t_1,q - 1}$ is independent of $\sum_{t_2 = t_1 + \ell +  1}^{T} b_{t_1t_2}\boldsymbol{\gamma}_{t_2,  \ell}.$ In addition, for any vector $\boldsymbol{\tau}\in\mathbf{R}^d,$ from \eqref{eq.linear_gamma},
\begin{align*}
    \left\Vert
    \sum_{t_2 = t_1 + \ell +  1}^{T} b_{t_1t_2}\boldsymbol{\tau}^\top\boldsymbol{\gamma}_{t_2,  \ell}
    \right\Vert_{M/2}\leq C\vert\boldsymbol{\tau}\vert_2\sqrt{\sum_{t_2 = t_1 + \ell +  1}^{T} b_{t_1t_2}^2}.
\end{align*}
Therefore,  from Theorem 1.7 of \cite{MR2002723}, for any vector $\boldsymbol{\tau}\in\mathbf{R}^d,$
\begin{align*}
    &\mathbf{E}\left[\left\vert (\boldsymbol{\gamma}_{t_1,q} - \boldsymbol{\gamma}_{t_1,q - 1})^\top\sum_{t_2 = t_1 + \ell +  1}^{T} b_{t_1t_2}\boldsymbol{\gamma}_{t_2,  \ell}\right\vert^{M/2}\mid \boldsymbol{\gamma}_{t_1,q} - \boldsymbol{\gamma}_{t_1,q - 1} = \boldsymbol{\tau}\right]\\
    & = \mathbf{E}\left[\left\vert \boldsymbol{\tau}^\top\sum_{t_2 = t_1 + \ell +  1}^{T} b_{t_1t_2}\boldsymbol{\gamma}_{t_2,  \ell}\right\vert^{M/2}\right]\\
    &\leq C\left\vert\boldsymbol{\tau}\right\vert_2^{M/2}\left(\sum_{t_2 = t_1 + \ell +  1}^{T} b_{t_1t_2}^2\right)^{M/4} = C\left\vert\boldsymbol{\gamma}_{t_1,q} - \boldsymbol{\gamma}_{t_1,q - 1}\right\vert_2^{M/2}\left(\sum_{t_2 = t_1 + \ell +  1}^{T} b_{t_1t_2}^2\right)^{M/4},
\end{align*}
and 
\begin{align*}
   &\left\Vert
    (\boldsymbol{\gamma}_{t_1,q} - \boldsymbol{\gamma}_{t_1,q - 1})^\top\sum_{t_2 = t_1 + \ell +  1}^{T} b_{t_1t_2}\boldsymbol{\gamma}_{t_2,  \ell}
    \right\Vert_{M/2}\\
    &\leq  C\sqrt{\sum_{t_2 = t_1 + \ell +  1}^{T} b_{t_1t_2}^2}\left\Vert\ \left\vert\boldsymbol{\gamma}_{t_1,q} - \boldsymbol{\gamma}_{t_1,q - 1}\right\vert_2\ \right\Vert_{M/2}\\
    &\leq C\sqrt{\sum_{t_2 = t_1 + \ell +  1}^{T} b_{t_1t_2}^2}\sqrt{\sum_{j = 1}^d \Vert \boldsymbol{\gamma}_{t_1,q}^{(j)} - \boldsymbol{\gamma}_{t_1,q - 1}^{(j)}\Vert^2_{M/2}}\leq C_1\delta_q\sqrt{d\sum_{t_2 = t_1 + \ell +  1}^{T} b_{t_1t_2}^2}.
\end{align*}
This result implies that 
\begin{align*}
    &\left\Vert\sum_{t_1 = 1}^{T - B}\sum_{t_2 = t_1 + \ell +  1}^{T} b_{t_1t_2}\boldsymbol{\gamma}_{t_2,  \ell}^\top\left(\boldsymbol{\gamma}_{t_1,q} - \boldsymbol{\gamma}_{t_1,q - 1}\right)\right\Vert_{M/2}\\
    &\leq C\delta_q \sqrt{d\sum_{t_1 = 1}^{T - B}\sum_{t_2 = t_1 + \ell +  1}^{T}b^2_{t_1t_2}}.
\end{align*}
Furthermore,
\begin{equation}
\label{eq.sec_term_fir}
    \begin{aligned}
        &\left\Vert\sum_{t_1 = 1}^{T - B}\sum_{t_2 = t_1 + \ell +  1}^{T} b_{t_1t_2}\boldsymbol{\gamma}_{t_2, (t_2 - t_1 - 1)\wedge \ell}^\top\boldsymbol{\epsilon}_{t_1}\right\Vert_{M/2}\\
        &\leq C\sqrt{d\sum_{t_1 = 1}^{T - B}\sum_{t_2 = t_1 + \ell +  1}^{T}b^2_{t_1t_2}} \sum_{q = 0}^\infty \delta_q\\
        &\leq C_1\sqrt{d\sum_{t_1 = 1}^{T - B}\sum_{t_2 = t_1 + \ell +  1}^{T}b^2_{t_1t_2}}.
    \end{aligned}
\end{equation}
For the third term in \eqref{eq.separate_3},  notice that 
\begin{align*}
    &\sum_{t_1 = 1}^{T - B}\sum_{t_2 = t_1 + B}^{T} b_{t_1t_2}\left(\mathbf{E}\left[\boldsymbol{\epsilon}_{t_1}^\top \boldsymbol{\eta}_{t_2, (t_2 - t_1 - 1)\wedge \ell}\mid\mathcal{F}_{t_2, q}\right] - \mathbf{E}\left[\boldsymbol{\epsilon}_{t_1}^\top \boldsymbol{\eta}_{t_2, (t_2 - t_1 - 1)\wedge \ell}\mid\mathcal{F}_{t_2,  q - 1}\right]\right)\\
    &= \sum_{t_2 = 1 + B}^{T}\sum_{t_1 = 1}^{t_2 - B}b_{t_1t_2}\left(\mathbf{E}\left[\boldsymbol{\epsilon}_{t_1}^\top \boldsymbol{\eta}_{t_2, (t_2 - t_1 - 1)\wedge \ell}\mid\mathcal{F}_{t_2, q}\right] - \mathbf{E}\left[\boldsymbol{\epsilon}_{t_1}^\top \boldsymbol{\eta}_{t_2, (t_2 - t_1 - 1)\wedge \ell}\mid\mathcal{F}_{t_2,  q - 1}\right]\right).
\end{align*}
For any given $q\geq \ell + 1$ and any $s = 1,2,\cdots T - B,$ define 
\begin{align*}
    U_{q,s} = \sum_{t_2 = T - s + 1}^{T}\sum_{t_1 = 1}^{t_2 - B}b_{t_1t_2}\left(\mathbf{E}\left[\boldsymbol{\epsilon}_{t_1}^\top \boldsymbol{\eta}_{t_2, (t_2 - t_1 - 1)\wedge \ell}\mid\mathcal{F}_{t_2, q}\right] - \mathbf{E}\left[\boldsymbol{\epsilon}_{t_1}^\top \boldsymbol{\eta}_{t_2, (t_2 - t_1 - 1)\wedge \ell}\mid\mathcal{F}_{t_2,  q - 1}\right]\right)
\end{align*}
and $\mathcal{U}_{q,s}$ the $\sigma$-field generated by $e_T, e_{T - 1},\cdots, e_{T - s + 1 - q}.$ Then $U_{q,s} $ is measurable in $\mathcal{U}_{q,s},$ $\mathcal{U}_{q,s}\subset \mathcal{U}_{q,s + 1},$ and   
\begin{align*}
    &\mathbf{E}\left[\left(U_{q,s + 1} - U_{q,s}\right)\mid\mathcal{U}_{q,s}\right]\\
    &= 
    \sum_{t_1 = 1}^{T - s - B}b_{t_1(T - s)}\mathbf{E}\left[\left(\mathbf{E}\left[\boldsymbol{\epsilon}_{t_1}^\top \boldsymbol{\eta}_{T - s, (T - s - t_1 - 1)\wedge \ell}\mid\mathcal{F}_{T - s, q}\right] - \mathbf{E}\left[\boldsymbol{\epsilon}_{t_1}^\top \boldsymbol{\eta}_{T - s, (T - s - t_1 - 1)\wedge \ell}\mid\mathcal{F}_{T  - s,  q - 1}\right]\right)\mid\mathcal{U}_{q,s}\right]\\
    &= \sum_{t_1 = 1}^{T - s - B}b_{t_1(T - s)}\left(\mathbf{E}\left[\boldsymbol{\epsilon}_{t_1}^\top \boldsymbol{\eta}_{T - s, (T - s - t_1 - 1)\wedge \ell}\mid\mathcal{F}_{T - s, q  - 1}\right] - \mathbf{E}\left[\boldsymbol{\epsilon}_{t_1}^\top \boldsymbol{\eta}_{T - s, (T - s - t_1 - 1)\wedge \ell}\mid\mathcal{F}_{T  - s,  q - 1}\right]\right) = 0,
\end{align*}
so $U_{q,s}$ forms a martingale.  Theorem 1.1 of \cite{MR0400380} implies 
\begin{equation}
\begin{aligned}
    &\left\Vert
    \sum_{t_1 = 1}^{T - B}\sum_{t_2 = t_1 + B}^{T} b_{t_1t_2}\left(\mathbf{E}\left[\boldsymbol{\epsilon}_{t_1}^\top \boldsymbol{\eta}_{t_2, (t_2 - t_1 - 1)\wedge \ell}\mid\mathcal{F}_{t_2, q}\right] - \mathbf{E}\boldsymbol{\epsilon}_{t_1}^\top \boldsymbol{\eta}_{t_2, (t_2 - t_1 - 1)\wedge \ell}|\mathcal{F}_{t_2,  q - 1}\right)
    \right\Vert_{M/2}\\
    &\leq C\sqrt{\sum_{t_2 = 1 + B}^{T}\left\Vert
    \sum_{t_1 = 1}^{t_2 - B}b_{t_1t_2}\left(\mathbf{E}\left[\boldsymbol{\epsilon}_{t_1}^\top \boldsymbol{\eta}_{t_2, (t_2 - t_1 - 1)\wedge \ell}\mid\mathcal{F}_{t_2, q}\right] - \mathbf{E}\left[\boldsymbol{\epsilon}_{t_1}^\top \boldsymbol{\eta}_{t_2, (t_2 - t_1 - 1)\wedge \ell}\mid\mathcal{F}_{t_2,  q - 1}\right]\right)
    \right\Vert^2_{M/2}}.
\end{aligned}
\label{eq.sum_to_square_root_one}
\end{equation}
Since  $(t_2 - t_1 - 1)\wedge \ell\leq \ell$ and $\ell + 1\leq q $,  we have
\begin{align*}
\boldsymbol{\eta}_{t_2, (t_2 - t_1 - 1)\wedge \ell} &= \boldsymbol{\epsilon}_{t_2} - \boldsymbol{\gamma}_{t_2, q - 1} + \boldsymbol{\gamma}_{t_2, q - 1} - \boldsymbol{\gamma}_{t_2, (t_2 - t_1 - 1)\wedge \ell}\\
&= \boldsymbol{\eta}_{t_2, q - 1} + (\boldsymbol{\gamma}_{t_2, q - 1} - \boldsymbol{\gamma}_{t_2, (t_2 - t_1 - 1)\wedge \ell}),
\end{align*}
and $\boldsymbol{\gamma}_{t_2, q - 1} - \boldsymbol{\gamma}_{t_2, (t_2 - t_1 - 1)\wedge \ell}$ is measurable in $\mathcal{F}_{t_2, q - 1}$.  Furthermore,
\begin{align*}
    &\left\Vert
    \sum_{t_1 = 1}^{t_2 - B}b_{t_1t_2}\left(\mathbf{E}\left[\boldsymbol{\epsilon}_{t_1}^\top \boldsymbol{\eta}_{t_2, (t_2 - t_1 - 1)\wedge \ell}\mid\mathcal{F}_{t_2, q}\right] - \mathbf{E}\left[\boldsymbol{\epsilon}_{t_1}^\top \boldsymbol{\eta}_{t_2, (t_2 - t_1 - 1)\wedge \ell}\mid\mathcal{F}_{t_2,  q - 1}\right]\right)
    \right\Vert_{M/2}\\
    &\leq 
    \left\Vert
    \sum_{t_1 = 1}^{t_2 - B}b_{t_1t_2}\left(\mathbf{E}\left[\boldsymbol{\epsilon}_{t_1}^\top \boldsymbol{\eta}_{t_2, q - 1}\mid\mathcal{F}_{t_2, q}\right] - \mathbf{E}\left[\boldsymbol{\epsilon}_{t_1}^\top \boldsymbol{\eta}_{t_2,  q-  1}\mid\mathcal{F}_{t_2,  q - 1}\right]\right)
    \right\Vert_{M/2}\\
    &+ 
    \left\Vert
    \sum_{t_1 = 1}^{t_2 - B}b_{t_1t_2}\left(\boldsymbol{\gamma}_{t_2, q - 1} - \boldsymbol{\gamma}_{t_2, (t_2 - t_1 - 1)\wedge \ell}\right)^\top \left(
    \mathbf{E}\left[\boldsymbol{\epsilon}_{t_1}\mid\mathcal{F}_{t_2, q}\right] - \mathbf{E}\left[\boldsymbol{\epsilon}_{t_1}\mid\mathcal{F}_{t_2,  q - 1}\right]
    \right)
    \right\Vert_{M/2}
    .
\end{align*}
From Lemma \ref{lemma.linear_form}, 
\begin{equation}
\begin{aligned}
    &\left\Vert
    \sum_{t_1 = 1}^{t_2 - B}b_{t_1t_2}\left(\mathbf{E}\left[\boldsymbol{\epsilon}_{t_1}^\top \boldsymbol{\eta}_{t_2, q - 1}\mid\mathcal{F}_{t_2, q}\right] - \mathbf{E}\left[\boldsymbol{\epsilon}_{t_1}^\top \boldsymbol{\eta}_{t_2,  q-  1}|\mathcal{F}_{t_2,  q - 1}\right]\right)
    \right\Vert_{M/2}\\
    &\leq \sum_{t_1 = 1}^{t_2 - B}\vert b_{t_1t_2}\vert\left(\left\Vert \mathbf{E}\left[\boldsymbol{\epsilon}_{t_1}^\top \boldsymbol{\eta}_{t_2, q - 1}\mid \mathcal{F}_{t_2, q}\right]\right\Vert_{M/2} + \left\Vert\mathbf{E}\left[\boldsymbol{\epsilon}_{t_1}^\top \boldsymbol{\eta}_{t_2,  q-  1}\mid\mathcal{F}_{t_2,  q - 1}\right]\right\Vert_{M/2}\right)
    \\ 
    &\leq
    2\sum_{t_1 = 1}^{t_2 - B}\vert b_{t_1t_2}\vert\left\Vert \boldsymbol{\epsilon}_{t_1}^\top \boldsymbol{\eta}_{t_2,  q-  1}\right\Vert_{M/2}
    \\
    &\leq 2\sum_{t_1 = 1}^{t_2 - B}\vert b_{t_1t_2}\vert \sum_{j = 1}^d \left\Vert
    \boldsymbol{\epsilon}_{t_1}^{(j)}\right\Vert_M\left\Vert\boldsymbol{\eta}_{t_2, q - 1}^{(j)}
    \right\Vert_{M}
    \leq \frac{Cd}{q^{\alpha}}\sum_{t_1 = 1}^{t_2 - B}\vert b_{t_1,t_2}\vert\leq \frac{CdT}{q^{\alpha}}.
\end{aligned}
\label{eq.q_s_1}
\end{equation}
Also since $\vert b_{t_1t_2}\vert\leq 1,$
\begin{align*}
    &\left\Vert
    \sum_{t_1 = 1}^{t_2 - B}b_{t_1t_2}\left(\boldsymbol{\gamma}_{t_2, q - 1} - \boldsymbol{\gamma}_{t_2, (t_2 - t_1 - 1)\wedge \ell}\right)^\top \left(
    \mathbf{E}\left[\boldsymbol{\epsilon}_{t_1}\mid\mathcal{F}_{t_2, q}\right] - \mathbf{E}\left[\boldsymbol{\epsilon}_{t_1}\mid\mathcal{F}_{t_2,  q - 1}\right]
    \right)
    \right\Vert_{M/2}\\
    &\leq \sum_{t_1 = 1}^{t_2 - B}\sum_{j=1}^d\vert b_{t_1t_2}\vert \left\Vert \left(\boldsymbol{\gamma}_{t_2, q - 1}^{(j)} - \boldsymbol{\gamma}_{t_2, (t_2 - t_1 - 1)\wedge \ell}^{(j)}\right)\left(
    \mathbf{E}\left[\boldsymbol{\epsilon}_{t_1}^{(j)}\mid\mathcal{F}_{t_2, q}\right] - \mathbf{E}\left[\boldsymbol{\epsilon}_{t_1}^{(j)}\mid\mathcal{F}_{t_2,  q - 1}\right]
    \right)\right\Vert_{M/2}\\
    &\leq \sum_{t_1 = 1}^{t_2 - B}\sum_{j = 1}^d\left\Vert\boldsymbol{\gamma}_{t_2, q - 1}^{(j)} - \boldsymbol{\gamma}_{t_2, (t_2 - t_1 - 1)\wedge \ell}^{(j)}\right\Vert_{M} \left\Vert\mathbf{E}\left[\boldsymbol{\epsilon}_{t_1}^{(j)}\mid\mathcal{F}_{t_2, q}\right] - \mathbf{E}\left[\boldsymbol{\epsilon}_{t_1}^{(j)}\mid\mathcal{F}_{t_2,  q - 1}\right]\right\Vert_M.
\end{align*}
Since $q\geq \ell + 1,$
\begin{align*}
    \left\Vert\boldsymbol{\gamma}_{t_2, q - 1}^{(j)} - \boldsymbol{\gamma}_{t_2, (t_2 - t_1 - 1)\wedge \ell}^{(j)}\right\Vert_{M}\leq \sum_{s = ((t_2 - t_1 - 1)\wedge \ell) + 1}^{q - 1}\delta_{s}\leq \frac{C}{(1 + (t_2 - t_1 - 1)\wedge \ell)^\alpha},
\end{align*}
and 
\begin{align*}
    \left\Vert\mathbf{E}\left[\boldsymbol{\epsilon}_{t_1}^{(j)}\mid \mathcal{F}_{t_2, q}\right] - \mathbf{E}\left[\boldsymbol{\epsilon}_{t_1}^{(j)}\mid\mathcal{F}_{t_2,  q - 1}\right]\right\Vert_M 
    &= 
    \left\Vert
    \mathbf{E}\left[\boldsymbol{\epsilon}_{t_1}^{(j)}\mid\mathcal{F}_{t_1, t_1 + q  - t_2}\right] - \mathbf{E}\left[\boldsymbol{\epsilon}_{t_1}^{(j)}\mid\mathcal{F}_{t_1, t_1 + q  - t_2 - 1}\right]
    \right\Vert_{M}\\
    &\leq \delta_{t_1 + q - t_2},
\end{align*}
we have 
\begin{align*}
    &\left\Vert\boldsymbol{\gamma}_{t_2, q - 1}^{(j)} - \boldsymbol{\gamma}_{t_2, (t_2 - t_1 - 1)\wedge \ell}^{(j)}\right\Vert_{M} \left\Vert\mathbf{E}\left[\boldsymbol{\epsilon}_{t_1}^{(j)}\mid\mathcal{F}_{t_2, q}\right] - \mathbf{E}\left[\boldsymbol{\epsilon}_{t_1}^{(j)}\mid\mathcal{F}_{t_2,  q - 1}\right]\right\Vert_M\\
    &\leq \frac{C\delta_{t_1 + q - t_2}}{(1 + (t_2 - t_1 - 1)\wedge \ell)^\alpha}.
\end{align*}
Therefore,
\begin{align*}
   &\left\Vert
    \sum_{t_1 = 1}^{t_2 - B}b_{t_1t_2}\left(\boldsymbol{\gamma}_{t_2, q - 1} - \boldsymbol{\gamma}_{t_2, (t_2 - t_1 - 1)\wedge \ell}\right)^\top \left(
    \mathbf{E}\left[\boldsymbol{\epsilon}_{t_1}\mid\mathcal{F}_{t_2, q}\right] - \mathbf{E}\left[\boldsymbol{\epsilon}_{t_1}\mid\mathcal{F}_{t_2,  q - 1}\right]
    \right)
    \right\Vert_{M/2}\\
    &\leq Cd \sum_{t_1 = 1}^{t_2 - B}\frac{\delta_{t_1 + q - t_2}}{(1 + (t_2 - t_1 - 1)\wedge \ell)^\alpha}\\
    &\leq \frac{Cd}{(1 + \ell)^\alpha}\sum_{t_1 =  1}^{t_2 - \ell - 1}\delta_{t_1 + q  - t_2} + \sum_{t_1 = 1\vee (t_2 - \ell)}^{t_2  - B}\frac{Cd\times \delta_{t_1 + q  -  t_2}}{(t_2 - t_1)^\alpha}\\
    &\leq\frac{Cd}{(1 + \ell)^\alpha}\sum_{t_1 =  1}^{t_2 - \ell - 1}\delta_{t_1 + q  - t_2} + \frac{Cd}{B^\alpha}\sum_{t_1 = 1\vee (t_2 - \ell)}^{t_2  - B}\delta_{t_1 + q - t_2}.
\end{align*}
Since
\begin{align*}
    \sum_{t_1 =  1}^{t_2 - \ell - 1}\delta_{t_1 + q  - t_2}\leq 
    \begin{cases}
        C\quad \text{if } t_2\geq q + 1,\\
        \frac{C}{(2 + q - t_2)^\alpha}\quad\text{if } t_2 < q + 1,
    \end{cases}\quad \text{and} \quad \sum_{t_1 = 1\vee (t_2 - \ell)}^{t_2  - B}\delta_{t_1 + q - t_2}\leq \frac{C}{(1 + (1+q - t_2)\vee (q - \ell))^\alpha},
\end{align*}
we have 
\begin{equation}
\begin{aligned}
    &\left\Vert
    \sum_{t_1 = 1}^{t_2 - B}b_{t_1t_2}\left(\boldsymbol{\gamma}_{t_2, q - 1} - \boldsymbol{\gamma}_{t_2, (t_2 - t_1 - 1)\wedge \ell}\right)^\top \left(
    \mathbf{E}\left[\boldsymbol{\epsilon}_{t_1}\mid\mathcal{F}_{t_2, q}\right] - \mathbf{E}\left[\boldsymbol{\epsilon}_{t_1}\mid\mathcal{F}_{t_2,  q - 1}\right]
    \right)
    \right\Vert_{M/2}\\
    &\leq \frac{Cd}{B^\alpha\left(1 + (1+q - t_2)\vee (q - \ell)\right)^\alpha} + 
    \begin{cases}
        \frac{Cd}{(1 + \ell)^\alpha}\quad\text{if } t_2\geq q + 1,\\
        \frac{Cd}{(1 + \ell)^\alpha(2 + q - t_2)^\alpha}\quad\text{if } t_2 < q + 1.
    \end{cases}
\end{aligned}
\label{eq.q_s_2}
\end{equation}
From \eqref{eq.q_s_1} and \eqref{eq.q_s_2}, and notice that 
$$
(1 + q - t_2)\vee (q - \ell)\geq q - \ell,\quad\text{and}\quad \ell + 1 \leq q\quad\text{for the third term,}
$$
we have 
\begin{align*}
    &\left\Vert
    \sum_{t_1 = 1}^{t_2 - B}b_{t_1t_2}\left(\mathbf{E}\left[\boldsymbol{\epsilon}_{t_1}^\top \boldsymbol{\eta}_{t_2, (t_2 - t_1 - 1)\wedge \ell}\mid\mathcal{F}_{t_2, q}\right] - \mathbf{E}\left[\boldsymbol{\epsilon}_{t_1}^\top \boldsymbol{\eta}_{t_2, (t_2 - t_1 - 1)\wedge \ell}\mid \mathcal{F}_{t_2,  q - 1}\right]\right)
    \right\Vert_{M/2}\\
    &\leq \frac{CdT}{q^{\alpha}} + \frac{Cd}{B^\alpha(1 + q - \ell)^\alpha} + 
    \begin{cases}
        \frac{Cd}{(1 + \ell)^\alpha}\quad\text{if } t_2\geq q + 1,\\
        \frac{Cd}{(1 + \ell)^\alpha(2 + q - t_2)^\alpha}\quad\text{if } t_2 < q + 1.
    \end{cases}\\
\end{align*}
Furthermore, from \eqref{eq.sum_to_square_root_one},
\begin{align*}
    &\left\Vert
    \sum_{t_1 = 1}^{T - B}\sum_{t_2 = t_1 + B}^{T} b_{t_1t_2}\left(\mathbf{E}\left[\boldsymbol{\epsilon}_{t_1}^\top \boldsymbol{\eta}_{t_2, (t_2 - t_1 - 1)\wedge \ell}\mid\mathcal{F}_{t_2, q}\right] - \mathbf{E}\left[\boldsymbol{\epsilon}_{t_1}^\top \boldsymbol{\eta}_{t_2, (t_2 - t_1 - 1)\wedge \ell}\mid\mathcal{F}_{t_2,  q - 1}\right]\right)
    \right\Vert_{M/2}\\
    &\leq \sqrt{\sum_{t_2 = 1 + B}^{q\wedge T}\left\Vert
    \sum_{t_1 = 1}^{t_2 - B}b_{t_1t_2}\left(\mathbf{E}\left[\boldsymbol{\epsilon}_{t_1}^\top \boldsymbol{\eta}_{t_2, (t_2 - t_1 - 1)\wedge \ell}\mid\mathcal{F}_{t_2, q}\right] - \mathbf{E}\left[\boldsymbol{\epsilon}_{t_1}^\top \boldsymbol{\eta}_{t_2, (t_2 - t_1 - 1)\wedge \ell}\mid\mathcal{F}_{t_2,  q - 1}\right]\right)
    \right\Vert^2_{M/2}}\\
    &+ \sqrt{\sum_{t_2 = 1 + q}^{T}\left\Vert
    \sum_{t_1 = 1}^{t_2 - B}b_{t_1t_2}\left(\mathbf{E}\left[\boldsymbol{\epsilon}_{t_1}^\top \boldsymbol{\eta}_{t_2, (t_2 - t_1 - 1)\wedge \ell}\mid\mathcal{F}_{t_2, q}\right] - \mathbf{E}\left[\boldsymbol{\epsilon}_{t_1}^\top \boldsymbol{\eta}_{t_2, (t_2 - t_1 - 1)\wedge \ell}\mid\mathcal{F}_{t_2,  q - 1}\right]\right)
    \right\Vert^2_{M/2}}\\
    &\leq C\sqrt{\sum_{t_2 = 1 + B}^{q\wedge T}\left(\frac{dT}{q^{\alpha}} + \frac{d}{B^\alpha(1 + q - \ell)^\alpha} + \frac{d}{(1 + \ell)^\alpha(2 + q - t_2)^\alpha}\right)^2}\\
    & + C\sqrt{\sum_{t_2 = 1 + q}^{T}\left(\frac{dT}{q^{\alpha}} + \frac{d}{B^\alpha(1 + q - \ell)^\alpha} + \frac{d}{(1 + \ell)^\alpha}\right)^2}\\
    &\leq \frac{C_1dT\sqrt{T}}{q^\alpha} + \frac{C_1d\sqrt{T}}{B^\alpha(1 + q - \ell)^\alpha} + \frac{C_1d}{(1 + \ell)^\alpha}\sqrt{\sum_{t_2 = 1 + B}^{q\wedge T}\frac{1}{(2 + q - t_2)^{2\alpha}}} + \frac{C_1 d}{(1 + \ell)^\alpha}\sqrt{\sum_{t_2 =  1 + q}^T1}.
\end{align*}
If $q\leq T - 1,$ then
\begin{align*}
    \sum_{t_2 = 1 + B}^{q\wedge T}\frac{1}{(2 + q - t_2)^{2\alpha}} = \sum_{t_2 = 1 + B}^{q}\frac{1}{(2 + q - t_2)^{2\alpha}}\leq C\quad\text{and } \sum_{t_2 =  1 + q}^T1 = T - q.
\end{align*}
If $q \geq T,$ then 
\begin{align*}
    \sum_{t_2 = 1 + B}^{q\wedge T}\frac{1}{(2 + q - t_2)^{2\alpha}} = \sum_{t_2 = 1 + B}^{T}\frac{1}{(2 + q - t_2)^{2\alpha}}\leq \frac{C}{(2 + q  - T)^{2\alpha - 1}} \quad\text{and } \sum_{t_2 =  1 + q}^T1 = 0.
\end{align*}
Therefore,
\begin{align*}
    &\left\Vert
    \sum_{t_1 = 1}^{T - B}\sum_{t_2 = t_1 + B}^{T} b_{t_1t_2}\left(\mathbf{E}\left[\boldsymbol{\epsilon}_{t_1}^\top \boldsymbol{\eta}_{t_2, (t_2 - t_1 - 1)\wedge \ell}\mid\mathcal{F}_{t_2, q}\right] - \mathbf{E}\left[\boldsymbol{\epsilon}_{t_1}^\top \boldsymbol{\eta}_{t_2, (t_2 - t_1 - 1)\wedge \ell}\mid\mathcal{F}_{t_2,  q - 1}\right]\right)
    \right\Vert_{M/2}\\
    &\leq \frac{CdT\sqrt{T}}{q^\alpha} + \frac{Cd\sqrt{T}}{B^\alpha(1 + q - \ell)^\alpha}
    + \begin{cases}
        \frac{Cd(1 + \sqrt{T - q})}{(1 + \ell)^\alpha}\quad \text{if } q\leq T - 1,\\
        \frac{C d}{(1 + \ell)^\alpha}\frac{1}{(2 + q  - T)^{\alpha - 1/2}}\quad\text{if  } q\geq T.
    \end{cases}
\end{align*}
From this observation, notice that $\ell\geq B,$
\begin{equation}
\begin{aligned}
    &\sum_{q = \ell + 1}^\infty \left\Vert
        \sum_{t_1 = 1}^{T - B}\sum_{t_2 = t_1 + B}^{T} b_{t_1t_2}\left(\mathbf{E}\left[\boldsymbol{\epsilon}_{t_1}^\top \boldsymbol{\eta}_{t_2, (t_2 - t_1 - 1)\wedge \ell}\mid\mathcal{F}_{t_2, q}\right] - \mathbf{E}\left[\boldsymbol{\epsilon}_{t_1}^\top \boldsymbol{\eta}_{t_2, (t_2 - t_1 - 1)\wedge \ell}\mid\mathcal{F}_{t_2,  q - 1}\right]\right)
        \right\Vert_{M/2}\\
    &\leq \sum_{q = \ell + 1}^\infty \frac{CdT\sqrt{T}}{q^\alpha} + \sum_{q = \ell + 1}^\infty \frac{Cd\sqrt{T}}{B^\alpha(1 + q - \ell)^\alpha} + \sum_{q = \ell + 1}^{T  - 1}\frac{Cd(1 + \sqrt{T - q})}{(1 + \ell)^\alpha} 
    + \sum_{q = T}^{\infty}\frac{C d}{(1 + \ell)^\alpha}\frac{1}{(2 + q  - T)^{\alpha - 1/2}}\\
    &\leq \frac{C_1dT\sqrt{T}}{(1 + \ell)^{\alpha - 1}} + \frac{C_1d\sqrt{T}}{B^\alpha} + \frac{C_1dT\sqrt{T}}{(1 + \ell)^\alpha} + \frac{C_1d}{(1 + \ell)^\alpha}\\
    &\leq \frac{C_2dT\sqrt{T}}{(1 + \ell)^{\alpha - 1}} + \frac{C_2d\sqrt{T}}{B^\alpha}.
\end{aligned}
\label{eq.third_term_fir}
\end{equation}
From \eqref{eq.fir_term_condition}, \eqref{eq.sec_term_fir}, and \eqref{eq.third_term_fir}, we have 
\begin{equation}
\begin{aligned}
    &\left\Vert
        \sum_{t_1 = 1}^{T - B}\sum_{t_2 = t_1 + B}^{T} b_{t_1t_2}\left(\boldsymbol{\epsilon}_{t_1}^\top \boldsymbol{\epsilon}_{t_2} - \mathbf{E}\left[\boldsymbol{\epsilon}_{t_1}^\top \boldsymbol{\epsilon}_{t_2}|\mathcal{F}_{t_2, \ell}\right]\right)
    \right\Vert_{M/2}\\
    &\leq C\sum_{q = 1}^\infty \delta_q\sqrt{d\sum_{t_1 = 1}^{T - B}\sum_{t_2 = (t_1 + B)\vee (t_1 + \ell + 1 - q)}^{T\wedge (t_1 + \ell)}
        b_{t_1t_2}^2} + C\sqrt{d\sum_{t_1 = 1}^{T - B}\sum_{t_2 = t_1 + \ell +  1}^{T}b^2_{t_1t_2}}\\
    & + \frac{CdT\sqrt{T}}{(1 + \ell)^{\alpha - 1}} + \frac{Cd\sqrt{T}}{B^\alpha}\\
    &\leq C\sum_{q = 1}^\infty \delta_q\sqrt{d\sum_{\vert t_1 - t_2\vert = B\vee (\ell + 1 - q)}^{\ell}
        b_{t_1t_2}^2} + C\sqrt{d\sum_{\vert t_2  - t_1\vert\geq \ell +  1}b^2_{t_1t_2}}\\
    & + \frac{CdT\sqrt{T}}{(1 + \ell)^{\alpha - 1}} + \frac{Cd\sqrt{T}}{B^\alpha}.
\end{aligned}
\end{equation}
On the other hand,  define $b_{t_1t_2}^\dagger = b_{t_2t_1}$, then 
\begin{align*}
    &\left\Vert
        \sum_{t_2 = 1}^{T - B}\sum_{t_1 = t_2 + B}^{T} b_{t_1t_2}\left(\boldsymbol{\epsilon}_{t_1}^\top \boldsymbol{\epsilon}_{t_2} - \mathbf{E}\left[\boldsymbol{\epsilon}_{t_1}^\top \boldsymbol{\epsilon}_{t_2}\mid\mathcal{F}_{t_1, \ell}\right]\right)
    \right\Vert_{M/2}\\
    &= \left\Vert
        \sum_{t_1 = 1}^{T - B}\sum_{t_2 = t_1 + B}^{T} b_{t_1t_2}^\dagger \left(\boldsymbol{\epsilon}_{t_1}^\top \boldsymbol{\epsilon}_{t_2} - \mathbf{E}\left[\boldsymbol{\epsilon}_{t_1}^\top \boldsymbol{\epsilon}_{t_2}\mid\mathcal{F}_{t_2, \ell}\right]\right)
        \right\Vert_{M/2}\\
    &\leq C \sum_{q = 1}^\infty \delta_q\sqrt{d\sum_{t_1 = 1}^{T - B}\sum_{t_2 = (t_1 + B)\vee (t_1 + \ell + 1 - q)}^{T\wedge (t_1 + \ell)}
        b_{t_1t_2}^{\dagger2}} + C\sqrt{d\sum_{t_1 = 1}^{T - B}\sum_{t_2 = t_1 + \ell +  1}^{T}b^{\dagger2}_{t_1t_2}}\\
    & + \frac{CdT\sqrt{T}}{(1 + \ell)^{\alpha - 1}} + \frac{Cd\sqrt{T}}{B^\alpha}\\
    & = C\sum_{q = 1}^\infty \delta_q \sqrt{d\sum_{t_2 = 1}^{T - B}\sum_{t_1 = (t_2 + B)\vee (t_2 + \ell + 1 - q)}^{T\wedge (t_2 + \ell)}
        b_{t_1t_2}^2}
        \\
    &+ C\sqrt{d\sum_{t_2 = 1}^{T - B}\sum_{t_1 = t_2 + \ell +  1}^{T}b^2_{t_1t_2}}
        + \frac{CdT\sqrt{T}}{(1 + \ell)^{\alpha - 1}} + \frac{Cd\sqrt{T}}{B^\alpha}\\
    &\leq C\sum_{q = 1}^\infty \delta_q\sqrt{d\sum_{\vert t_1 - t_2\vert = B\vee (\ell + 1 - q)}^{\ell}
        b_{t_1t_2}^2} + C\sqrt{d\sum_{\vert t_2  - t_1\vert\geq \ell +  1}b^2_{t_1t_2}}\\
    & + \frac{CdT\sqrt{T}}{(1 + \ell)^{\alpha - 1}} + \frac{Cd\sqrt{T}}{B^\alpha}.
\end{align*}
These two inequalities prove \eqref{eq.truncate_moment}.
\end{proof}

\begin{remark}
The proof of equation \eqref{eq.whole_formula} decomposes the quadratic form  $Q$ into a series of martingales 
\begin{equation}
\sum_{t_1 =  1}^{T - B}\sum_{t_2 = t_1 + B}^T b_{t_1t_2}\left(\boldsymbol{\gamma}_{t_1,q} - \boldsymbol{\gamma}_{t_1,q - 1}\right)^\top\boldsymbol{\gamma}_{t_2, t_2 - t_1 - 1},\quad q\geq0,
\label{eq.b_first_bound}
\end{equation}
where $\boldsymbol{\gamma}_{t_2, t_2 - t_1 - 1}$ is independent of $\boldsymbol{\gamma}_{t_1,q} - \boldsymbol{\gamma}_{t_1,q - 1};$ together with remainder terms of the form
\begin{equation}
\sum_{t_2 = 1 + B}^T\sum_{t_1 = 1}^{t_2 - B} b_{t_1t_2}\left(\mathbf{E}\left[\boldsymbol{\epsilon}_{t_1}^\top\boldsymbol{\eta}_{t_2,t_2 - t_1 - 1}\mid\mathcal{F}_{t_2, q}\right] - \mathbf{E}\left[\boldsymbol{\epsilon}_{t_1}^\top\boldsymbol{\eta}_{t_2,t_2 - t_1 - 1}\mid\mathcal{F}_{t_2, q - 1}\right]\right),\quad q\geq 0.
\label{eq.second_summ}
\end{equation}

From \eqref{eq.summary_first_half}, the sum of squares of coefficients $\sqrt{d \sum_{t_1 =  1}^{T - B}\sum_{t_2 = t_1 + B}^Tb^2_{t_1t_2}}$ in \eqref{eq.whole_formula} arises from \eqref{eq.b_first_bound}. The independence between $\boldsymbol{\gamma}_{t_2, t_2 - t_1 - 1}$ and $\boldsymbol{\gamma}_{t_1,q} - \boldsymbol{\gamma}_{t_1,q - 1}$ makes the moment of the summation in \eqref{eq.b_first_bound} behave similarly to those of independent random variables.

From \eqref{eq.summ_sec_part}, the terms $\frac{CT^{5/2}}{B^\alpha}$ and  $\frac{CT^{3/2}}{B^{\alpha - 2}}$ in \eqref{eq.whole_formula} arise from \eqref{eq.second_summ}.  We note that the products $\boldsymbol{\epsilon}_{t_1}^\top\boldsymbol{\eta}_{t_2,t_2 - t_1 - 1}$ may exhibit  complex dependence structures. In particular, if the time lag $a =  t_2 - t_1$ is large, the dependence between $\boldsymbol{\epsilon}_{t_1}^\top\boldsymbol{\eta}_{t_1 + a,a - 1}$ and $\boldsymbol{\epsilon}_{t_2}^\top\boldsymbol{\eta}_{t_2 + a,a - 1}$ can remain strong even when 
$t_1 - t_2$ is large in absolute value.  
Fortunately, after selecting a sufficiently large bandwidth $B,$ Lemma \ref{lemma.linear_form} ensures that the moments of the elements of $\boldsymbol{\eta}_{t_2,t_2 - t_1 - 1}$ remain small.  Hence, when controlling the moments of the summations \eqref{eq.second_summ}, it is unnecessary to analyze the dependence structure of
$\boldsymbol{\epsilon}_{t_1}^\top\boldsymbol{\eta}_{t_2,t_2 - t_1 - 1}$  in detail.
    \label{remark.dependence_decomposition}
\end{remark}

\begin{remark}
    In the context of ANOVA,  our test statistic introduces two thresholds $B < B_1,$ and assumes that $b_{t_1t_2} = 0$ whenever $\vert t_1 - t_2\vert < B$ or $\vert t_1 - t_2\vert > B_1.$ Under this assumption, by choosing $\ell > B_1,$ we have 
$\sum_{\vert t_1 - t_2\vert \geq \ell + 1} b^2_{t_1t_2} = 0, $ and 
$
\sum_{\vert t_1 - t_2\vert  = B\vee (\ell + 1 - q)}^{\ell} b^2_{t_1t_2} = \sum_{\vert t_1 - t_2\vert  = B\vee (\ell + 1 - q)}^{B_1} b^2_{t_1t_2} = 0
$ if $q < \ell + 1 - B_1.$ Therefore,
\begin{align*}
     & \left\Vert
        \sum_{t_1 = 1}^T\sum_{t_2 = 1}^T b_{t_1t_2}\left(\boldsymbol{\epsilon}_{t_1}^\top \boldsymbol{\epsilon}_{t_2} - \mathbf{E}\left[\boldsymbol{\epsilon}_{t_1}^\top \boldsymbol{\epsilon}_{t_2}\mid\mathcal{F}_{t_1\vee t_2, \ell}\right]\right)
        \right\Vert_{M/2}\\
       &\leq C\sum_{q = \ell + 1  -B_1}^\infty\delta_q\sqrt{d\sum_{\vert t_1 - t_2\vert  = B}^{B_1} b^2_{t_1t_2}}\\
       & + \frac{CdT^{3/2}}{(1 + \ell)^{\alpha-1}} + \frac{Cd\sqrt{T}}{B^\alpha}\\
       &\leq \frac{C_1}{(\ell + 1 -  B_1)^\alpha}\sqrt{d\sum_{\vert t_1 - t_2\vert  = B}^{B_1} b^2_{t_1t_2}} + \frac{CdT^{3/2}}{(1 + \ell)^{\alpha-1}} + \frac{Cd\sqrt{T}}{B^\alpha}.
\end{align*}
\end{remark}

The proof of Theorem \ref{theorem.Gaussian_app} leverages the technique used in the proof of Theorem 4 of \cite{MR3992401}. Specifically, it introduces the function $g_{\psi,t}(\cdot)$ in \eqref{eq.def_g_psi} as a continuously differentiable approximation to the indicator function. 
This approximation allows statisticians to use the Taylor expansion to approximate the expectation of $g_{\psi,x}\left(\frac{Q}{S_T}\right)$ and the expectation of $g_{\psi,x}$ applied to a normal random variable. Theorem \ref{theorem.consistent_quadratic} also plays a decisive role in the proof of Theorem \ref{theorem.Gaussian_app}, which enables the use of ``$m$-approximation'' technique---
approximating the original time series by a sequence of ``$m$-dependent'' random variables (see Definition 6.4.3 of \cite{MR1093459}), which is also illustrated in Section C.1 of \cite{MR3718156}. 
 Theorem \ref{theorem.Gaussian_app} is further employed to  bound the moments of the blocks defined in \eqref{eq.big_block} and \eqref{eq.small_block}.

\begin{proof}[Proof of Theorem \ref{theorem.Gaussian_app}]
    Define the functions
    \begin{equation}
    g_0(x) = \left(
    1 - \min\left(1, \max\left(x,0\right)\right)^4
    \right)^4\quad\text{and}\quad g_{\psi,t}(x) = g_0\left(\psi\left(x - t\right)\right).
    \label{eq.def_g_psi}
    \end{equation}
    According to (S2) and (S4)-(S5) of  \cite{MR3992401}, there exists a constant $g_*$ such that
    $$
    \sup_{x,t\in\mathbf{R}}\left\vert g_{\psi,t}^\prime(x)\right\vert\leq g_*\psi,\quad \sup_{x,t\in\mathbf{R}}\left\vert g_{\psi,t}^{\prime\prime}(x)\right\vert\leq g_*\psi^2,\quad\sup_{x,t\in\mathbf{R}}\left\vert g_{\psi,t}^{\prime\prime\prime}(x)\right\vert\leq g_*\psi^3,
    $$
    and 
    \begin{equation}
    \mathbf{1}_{x\leq t} \leq g_{\psi,t}(x)\leq \mathbf{1}_{x\leq t + 1/\psi},
    \label{eq.g_approximation}
    \end{equation}
    where $\mathbf{1}_{x\leq t}$ denotes the indicator function, which equals $1$ if $x\leq t,$ and 0 otherwise.    
    From \eqref{eq.g_approximation},
    \begin{align*}
            \mathbf{Pr}\left(\frac{Q}{S_T}\leq x\right) - \mathbf{Pr}\left(\xi\leq x\right)
            &\leq \mathbf{E}\left[g_{\psi,x}\left(\frac{Q}{S_T}\right)\right] - \mathbf{E}\left[g_{\psi,x - \frac{1}{\psi}}(\xi)\right]\\
            &\leq 
            \left\vert
            \mathbf{E}\left[g_{\psi,x}\left(\frac{Q}{S_T}\right)\right] - \mathbf{E}\left[g_{\psi,x}(\xi)\right]
            \right\vert + 
            \mathbf{Pr}\left(x - \frac{1}{\psi}\leq \xi\leq x + \frac{1}{\psi}\right)
            \\
            &\leq \sup_{x\in\mathbf{R}}\left\vert\mathbf{E}\left[g_{\psi,x}\left(\frac{Q}{S_T}\right)\right] - \mathbf{E}\left[g_{\psi,x}(\xi)\right]
            \right\vert +  \frac{C}{\psi},
    \end{align*}
    and 
    \begin{align*}
            \mathbf{Pr}\left(\frac{Q}{S_T}\leq x\right) - \mathbf{Pr}\left(\xi\leq x\right)
            &\geq  \mathbf{E}\left[g_{\psi,x - \frac{1}{\psi}}\left(\frac{Q}{S_T}\right)\right] - \mathbf{E}\left[g_{\psi,x}(\xi)\right]\\
            &\geq 
            -\left\vert
            \mathbf{E}\left[g_{\psi,x - \frac{1}{\psi}}\left(\frac{Q}{S_T}\right)\right] - \mathbf{E}\left[g_{\psi, x - \frac{1}{\psi}}(\xi)\right]
            \right\vert
            -
            \mathbf{Pr}\left(x - \frac{1}{\psi}\leq \xi\leq x + \frac{1}{\psi}\right)\\
            &\geq 
            -\sup_{x\in\mathbf{R}}\left\vert\mathbf{E}\left[g_{\psi,x}\left(\frac{Q}{S_T}\right)\right]- \mathbf{E}\left[g_{\psi,x}(\xi)\right]
            \right\vert -  \frac{C}{\psi}.
    \end{align*}
    Therefore, we have 
    \begin{equation}
        \begin{aligned}
            \sup_{x\in\mathbf{R}}\left\vert
            \mathbf{Pr}\left(\frac{Q}{S_T}\leq x\right) - \mathbf{Pr}\left(\xi\leq x\right)
            \right\vert
            \leq \frac{C}{\psi} + \sup_{x\in\mathbf{R}}\left\vert\mathbf{E}\left[g_{\psi,x}\left(\frac{Q}{S_T}\right)\right] - \mathbf{E}\left[g_{\psi,x}(\xi)\right]
            \right\vert.
        \end{aligned}
        \label{eq.Prob_to_E}
    \end{equation}
    For any integer $\ell > B_1,$ from \eqref{eq.truncate_moment}, we have 
    \begin{equation}
    \begin{aligned}
        &\left\Vert
        \frac{Q}{S_T} - \frac{1}{S_T}\sum_{t_1 = 1}^T\sum_{t_2=  1}^T b_{t_1t_2}\left(\mathbf{E}\left[\boldsymbol{\epsilon}_{t_1}^\top \boldsymbol{\epsilon}_{t_2}\mid\mathcal{F}_{t_1\vee t_2, \ell}\right] - \mathbf{E}\left[\boldsymbol{\epsilon}_{t_1}^\top \boldsymbol{\epsilon}_{t_2}\right]\right)
        \right\Vert_{M/2}\\
        & = \frac{1}{S_T}\left\Vert
        \sum_{t_1 = 1}^T\sum_{t_2=  1}^T b_{t_1t_2}\left(\boldsymbol{\epsilon}_{t_1}^\top \boldsymbol{\epsilon}_{t_2} - \mathbf{E}\left[\boldsymbol{\epsilon}_{t_1}^\top \boldsymbol{\epsilon}_{t_2}\mid\mathcal{F}_{t_1\vee t_2, \ell}\right]\right)
        \right\Vert_{M/2}\\ 
        &\leq \frac{C}{S_T}\sum_{q =  \ell + 1 - B_1}^\infty\delta_q\sqrt{d\sum_{\vert t_1 - t_2\vert = \ell + 1 - q}^{B_1}b^2_{t_1t_2}} + \frac{CdT^{3/2}}{S_T\ell^{\alpha - 1}} + \frac{Cd\sqrt{T}}{S_TB^\alpha}\\
        &\leq \frac{C_1\sqrt{dT(B_1 - B)}}{S_T(\ell + 1  - B_1)^\alpha} + \frac{CdT^{3/2}}{S_T\ell^{\alpha - 1}} + \frac{Cd\sqrt{T}}{S_TB^\alpha}\\
        &\leq \frac{C_2}{(\ell + 1 - B_1)^\alpha} + \frac{C_2T^{3/2}}{\sqrt{(B_1 -  B)}\ell^{\alpha - 1}} + \frac{C_2T^{1/2}}{\sqrt{(B_1 - B)}B^\alpha}.
    \end{aligned}
    \label{eq.condition_change}
    \end{equation}
Therefore, 
\begin{equation}
\begin{aligned}
    &\sup_{x\in\mathbf{R}}\left\vert
    \mathbf{E}\left[g_{\psi,x}\left(\frac{Q}{S_T}\right)\right]- \mathbf{E}\left[g_{\psi,x}\left(\frac{1}{S_T}\sum_{t_1 = 1}^T\sum_{t_2=  1}^T b_{t_1t_2}\left(\mathbf{E}\left[\boldsymbol{\epsilon}_{t_1}^\top \boldsymbol{\epsilon}_{t_2}\mid\mathcal{F}_{t_1\vee t_2, \ell}\right] - \mathbf{E}\left[\boldsymbol{\epsilon}_{t_1}^\top \boldsymbol{\epsilon}_{t_2}\right]\right)\right)
    \right]\right\vert\\
    &\leq g_*\psi\left\Vert \frac{Q}{S_T} - \frac{1}{S_T}\sum_{t_1 = 1}^T\sum_{t_2=  1}^T b_{t_1t_2}\left(\mathbf{E}\left[\boldsymbol{\epsilon}_{t_1}^\top \boldsymbol{\epsilon}_{t_2}\mid\mathcal{F}_{t_1\vee t_2, \ell}\right] - \mathbf{E}\left[\boldsymbol{\epsilon}_{t_1}^\top \boldsymbol{\epsilon}_{t_2}\right]\right)\right\Vert_{M/2}\\
    &\leq \frac{C\psi}{(\ell + 1 - B_1)^\alpha} + \frac{C\psi T^{3/2}}{\sqrt{(B_1 -  B)}\ell^{\alpha - 1}} + \frac{C\psi T^{1/2}}{\sqrt{(B_1 - B)}B^\alpha}.
\end{aligned}
    \label{eq.E_truncate}
\end{equation}
Notice that $b_{t_1t_2} = 0$ for $\vert t_1  -t_2\vert < B$ or $\vert t_1 - t_2\vert > B_1,$ so we have 
\begin{align*}
    &\sum_{t_1 = 1}^T\sum_{t_2=  1}^T b_{t_1t_2}\left(\mathbf{E}\left[\boldsymbol{\epsilon}_{t_1}^\top \boldsymbol{\epsilon}_{t_2}\mid\mathcal{F}_{t_1\vee t_2, \ell}\right] - \mathbf{E}\left[\boldsymbol{\epsilon}_{t_1}^\top \boldsymbol{\epsilon}_{t_2}\right]\right)\\
    & = \sum_{t_1 = 1}^T\sum_{t_2 = (t_1 - B_1)\vee 1}^{(t_1 - B)} b_{t_1t_2}\left(\mathbf{E}\left[\boldsymbol{\epsilon}_{t_1}^\top \boldsymbol{\epsilon}_{t_2}\mid\mathcal{F}_{t_1, \ell}\right] - \mathbf{E}\left[\boldsymbol{\epsilon}_{t_1}^\top \boldsymbol{\epsilon}_{t_2}\right]\right)\\
    & + \sum_{t_2 = 1}^T\sum_{t_1 = (t_2 - B_1)\vee 1}^{(t_2 - B)} b_{t_1t_2}\left(\mathbf{E}\left[\boldsymbol{\epsilon}_{t_1}^\top \boldsymbol{\epsilon}_{t_2}\mid\mathcal{F}_{t_2, \ell}\right] - \mathbf{E}\left[\boldsymbol{\epsilon}_{t_1}^\top \boldsymbol{\epsilon}_{t_2}\right]\right)\\
    & =  2\sum_{t_1 = 1}^T\sum_{t_2 = (t_1 - B_1)\vee 1}^{(t_1 - B)} b_{t_1t_2}^{\circ}\left(\mathbf{E}\left[\boldsymbol{\epsilon}_{t_1}^\top \boldsymbol{\epsilon}_{t_2}\mid\mathcal{F}_{t_1, \ell}\right] - \mathbf{E}\left[\boldsymbol{\epsilon}_{t_1}^\top \boldsymbol{\epsilon}_{t_2}\right]\right),
\end{align*}
where $b_{t_1t_2}^{\circ} = \left(b_{t_1t_2} + b_{t_2t_1}\right) / 2.$    With this definition $\left\vert b_{t_1t_2}^{\circ} \right\vert\leq 1,$  and $b_{t_1t_2}^{\circ} = 0$ if $\left\vert t_1 - t_2\right\vert < B$ or $\left\vert t_1 - t_2\right\vert > B_1.$ For an integer $v > \ell,$ define the big-block
\begin{equation}
    A_q = 2\sum_{t_1 = (q - 1)(v + \ell) +  1}^{((q - 1)(v + \ell) + v)\wedge T}\sum_{t_2 = (t_1 - B_1)\vee 1}^{(t_1 - B)} b_{t_1t_2}^{\circ}\left(\mathbf{E}\left[\boldsymbol{\epsilon}_{t_1}^\top \boldsymbol{\epsilon}_{t_2}\mid\mathcal{F}_{t_1, \ell}\right] - \mathbf{E}\left[\boldsymbol{\epsilon}_{t_1}^\top \boldsymbol{\epsilon}_{t_2}\right]\right),
    \label{eq.big_block}
\end{equation}
and the small-block 
\begin{equation}
    a_q = 2\sum_{t_1 = (q - 1)(v + \ell) +  v + 1}^{(q(v + \ell))\wedge T}\sum_{t_2 = (t_1 - B_1)\vee 1}^{(t_1 - B)} b_{t_1t_2}^{\circ}\left(\mathbf{E}\left[\boldsymbol{\epsilon}_{t_1}^\top \boldsymbol{\epsilon}_{t_2}\mid\mathcal{F}_{t_1, \ell}\right] - \mathbf{E}\left[\boldsymbol{\epsilon}_{t_1}^\top \boldsymbol{\epsilon}_{t_2}\right]\right).
    \label{eq.small_block}
\end{equation}
Define $R = \lceil\frac{T}{v + \ell}\rceil,$ where $\lceil x\rceil$ represents the smallest integer that is larger than or equal to $x.$ Then
\begin{equation}
\begin{aligned}
    &\sum_{q  =1}^R A_q + \sum_{q = 1}^R a_q \\
    &= 2\sum_{q =  1}^R \sum_{t_1 = (q - 1)(v + \ell) + 1}^{(q(v + \ell))\wedge T}\sum_{t_2 = (t_1 - B_1)\vee 1}^{(t_1 - B)} b_{t_1t_2}^{\circ}\left(\mathbf{E}\left[\boldsymbol{\epsilon}_{t_1}^\top \boldsymbol{\epsilon}_{t_2}\mid\mathcal{F}_{t_1, \ell}\right] - \mathbf{E}\left[\boldsymbol{\epsilon}_{t_1}^\top \boldsymbol{\epsilon}_{t_2}\right]\right)\\
    &= 2\sum_{t_1 = 1}^T\sum_{t_2 = (t_1 - B_1)\vee 1}^{(t_1 - B)} b_{t_1t_2}^{\circ}\left(\mathbf{E}\left[\boldsymbol{\epsilon}_{t_1}^\top \boldsymbol{\epsilon}_{t_2}\mid\mathcal{F}_{t_1, \ell}\right] - \mathbf{E}\left[\boldsymbol{\epsilon}_{t_1}^\top \boldsymbol{\epsilon}_{t_2}\right]\right).
\end{aligned}
\label{eq.a_plus_A}
\end{equation}
Furthermore, $A_q$ are mutually independent, and $a_q$ are mutually independent. From Theorem 2 of \cite{MR0133849},
\begin{align*}
    &\left\Vert
    2\sum_{t_1 = 1}^T\sum_{t_2 = (t_1 - B_1)\vee 1}^{(t_1 - B)} b_{t_1t_2}^{\circ}\left(\mathbf{E}\left[\boldsymbol{\epsilon}_{t_1}^\top \boldsymbol{\epsilon}_{t_2}\mid\mathcal{F}_{t_1, \ell}\right] - \mathbf{E}\left[\boldsymbol{\epsilon}_{t_1}^\top \boldsymbol{\epsilon}_{t_2}\right]\right) - \sum_{q  =1}^R A_q
    \right\Vert_{M/2}\\
    &= 
    \left\Vert
    \sum_{q = 1}^R a_q
    \right\Vert_{M/2}\leq C\sqrt{\sum_{q = 1}^R \left\Vert a_q\right\Vert^2_{M/2}}.
\end{align*}
From Theorem \ref{theorem.consistent_quadratic},
\begin{align*}
    \left\Vert a_q\right\Vert_{M/2} &\leq  2\left\Vert \sum_{t_1 = (q - 1)(v + \ell) +  v + 1}^{(q(v + \ell))\wedge T}\sum_{t_2 = (t_1 - B_1)\vee 1}^{(t_1 - B)} b_{t_1t_2}^{\circ}\left(\boldsymbol{\epsilon}_{t_1}^\top \boldsymbol{\epsilon}_{t_2} - \mathbf{E}\left[\boldsymbol{\epsilon}_{t_1}^\top \boldsymbol{\epsilon}_{t_2}\right]\right)\right\Vert_{M/2}\\
    & + 
    2\left\Vert
    \sum_{t_1 = (q - 1)(v + \ell) +  v + 1}^{(q(v + \ell))\wedge T}\sum_{t_2 = (t_1 - B_1)\vee 1}^{(t_1 - B)} b_{t_1t_2}^{\circ}\left(\boldsymbol{\epsilon}_{t_1}^\top \boldsymbol{\epsilon}_{t_2} -\mathbf{E}\left[\boldsymbol{\epsilon}_{t_1}^\top \boldsymbol{\epsilon}_{t_2}\mid\mathcal{F}_{t_1, \ell}\right]\right)
    \right\Vert_{M/2}\\
    &\leq C\sqrt{d\sum_{t_1 = (q - 1)(v + \ell) +  v + 1}^{(q(v + \ell))\wedge T}\sum_{t_2 = (t_1 - B_1)\vee 1}^{(t_1 - B)} b_{t_1t_2}^{\circ2}} + \frac{CT^{5/2}}{B^\alpha} + \frac{CT^{3/2}}{B^{\alpha - 2}}\\
    & + C\left(\sum_{q = 1}^\infty\delta_q\right)\sqrt{d\sum_{t_1 = (q - 1)(v + \ell) +  v + 1}^{(q(v + \ell))\wedge T}\sum_{t_2 = (t_1 - B_1)\vee 1}^{(t_1 - B)} b_{t_1t_2}^{\circ 2}} + \frac{CdT^{3/2}}{\ell^{\alpha-1}} + \frac{Cd\sqrt{T}}{B^\alpha}\\
    &\leq C_1\sqrt{d\ell (B_1 - B)} + \frac{C_1T^{5/2}}{B^\alpha} + \frac{C_1T^{3/2}}{B^{\alpha - 2}} +\frac{C_1T^{5/2}}{\ell^{\alpha - 1}}.
\end{align*}
Therefore,  we have 
\begin{equation}
\begin{aligned}
&\left\Vert
\frac{1}{S_T}\sum_{q = 1}^R a_q
\right\Vert_{M/2}\\
&\leq \frac{C}{S_T}\left(\sqrt{R d\ell (B_1 - B)} + \frac{\sqrt{R}T^{5/2}}{B^\alpha}
    + \frac{\sqrt{R}T^{3/2}}{B^{\alpha - 2}} +\frac{\sqrt{R}T^{5/2}}{\ell^{\alpha - 1}}
    \right)\\
    &\leq  C_1\sqrt{\frac{\ell}{v}} + \frac{C_1T^2}{B^\alpha\sqrt{v(B_1 - B)}} + \frac{C_1T}{B^{\alpha - 2}\sqrt{v(B_1 - B)}} + \frac{C_1T^2}{\ell^{\alpha - 1}\sqrt{v(B_1-  B)}},
\end{aligned}
    \label{eq.s_q_calculate}
\end{equation}
and 
\begin{equation}
\begin{aligned}
    &\sup_{x\in\mathbf{R}}
    \left\vert
    \mathbf{E}\left[g_{\psi,x}\left(\frac{1}{S_T}\sum_{t_1 = 1}^T\sum_{t_2=  1}^T b_{t_1t_2}\left(\mathbf{E}\left[\boldsymbol{\epsilon}_{t_1}^\top \boldsymbol{\epsilon}_{t_2}\mid\mathcal{F}_{t_1\vee t_2, \ell}\right] - \mathbf{E}\left[\boldsymbol{\epsilon}_{t_1}^\top \boldsymbol{\epsilon}_{t_2}\right]\right)\right)\right]\right.\\
    &\left.- \mathbf{E}\left[g_{\psi,x}\left(\frac{1}{S_T}\sum_{q = 1}^R A_q\right)\right]
    \right\vert\\
    &\leq \frac{C\psi}{S_T}\left\Vert
    \sum_{q = 1}^R a_q
    \right\Vert_{M/2}\\
    &\leq C_1\psi\left(\sqrt{\frac{\ell}{v}} + \frac{T^{2}}{B^\alpha\sqrt{v(B_1 - B)}} + \frac{T}{B^{\alpha - 2}\sqrt{v(B_1 - B)}} + \frac{T^{2}}{\ell^{\alpha - 1}\sqrt{v(B_1 - B)}}\right).
\end{aligned}
\label{eq.truncate_to_Big_block}
\end{equation}
    Define $A^*_q, q = 1,\cdots, R,$ as independent normal random variables such that 
    $$
    \mathbf{E}\left[A^*_q\right] = 0,\quad \mathrm{Var}(A^*_q) = \mathrm{Var}(A_q),
    $$ 
    and $A^*_{q_1}$ and $A_{q_2}$ are independent for different $q_1$ and $q_2.$ Define the summation 
    $$
    T_u = \sum_{q = 1}^{u - 1}A_q + \sum_{q = u + 1}^R A^*_q,\quad  \text{then}\quad T_u + A_u = T_{u + 1} + A^*_{u+1},
    $$
    where $A_u, A^*_u$ are independent of $T_u.$ From Taylor expansion,
    \begin{equation}
    \begin{aligned}
        &\mathbf{E}\left[g_{\psi,x}\left(\frac{T_u + A_u}{S_T}\right)\mid T_u\right]- \mathbf{E}\left[g_{\psi,x}\left(\frac{T_u + A_u^*}{S_T}\right)\mid T_u\right]\\
        &= g_{\psi,x}\left(\frac{T_u}{S_T}\right) + g_{\psi,x}^\prime \left(\frac{T_u}{S_T}\right)\mathbf{E}\left[\frac{A_u}{S_T}\right] + \frac{1}{2}g_{\psi,x}^{\prime\prime} \left(\frac{T_u}{S_T}\right)\mathbf{E}\left(\frac{A_u}{S_T}\right)^2
        + \frac{1}{6}\mathbf{E}\left[g_{\psi,x}^{\prime\prime\prime} (\zeta_1)\left(\frac{A_u}{S_T}\right)^3\mid T_u\right]\\
        &- g_{\psi,x}\left(\frac{T_u}{S_T}\right) - g_{\psi,x}^\prime \left(\frac{T_u}{S_T}\right)\mathbf{E}\left[\frac{A_u^*}{S_T}\right] - \frac{1}{2}g_{\psi,x}^{\prime\prime} \left(\frac{T_u}{S_T}\right)\mathbf{E}\left(\frac{A_u^*}{S_T}\right)^2
        - \frac{1}{6}\mathbf{E}\left[g_{\psi,x}^{\prime\prime\prime}(\zeta_2)\left(\frac{A_u^*}{S_T}\right)^3\mid T_u\right],
    \end{aligned}
    \label{eq.Taylor_expansion}
    \end{equation}
    where $\zeta_1,\zeta_2$ are two random variables. Since 
    $$\mathbf{E}\left[A_u\right] = \mathbf{E}\left[A_u^*\right] = 0,\quad \mathrm{Var}\left(A_u\right) = \mathrm{Var}\left(A_u^{*}\right),
    $$ 
    and $A_u^{*}$ has normal distribution, we have 
    \begin{equation}
        \begin{aligned}
            &\left\vert
            \mathbf{E}\left[g_{\psi,x}\left(\frac{T_u + A_u}{S_T}\right)\mid T_u\right]- \mathbf{E}\left[g_{\psi,x}\left(\frac{T_u + A_u^*}{S_T}\right)\mid T_u\right]
            \right\vert\\
            &\leq 
            \left\vert
            \frac{1}{6}\mathbf{E}\left[g_{\psi,x}^{\prime\prime\prime} (\zeta_1)\left(\frac{A_u}{S_T}\right)^3\mid T_u\right]
            \right\vert +\left\vert \frac{1}{6}\mathbf{E}\left[g_{\psi,x}^{\prime\prime\prime}(\zeta_2)\left(\frac{A_u^*}{S_T}\right)^3\mid T_u\right]\right\vert\\
            &\leq \frac{C\psi^3}{S^3_T}\left(
            \left\Vert A_u\right\Vert_{M/2}^3 + \left\Vert A_u^*\right\Vert_{M/2}^3
            \right)\leq \frac{C_1\psi^3 \left\Vert A_u\right\Vert_{M/2}^3}{S^3_T}.
        \end{aligned}
        \label{eq.g_psi_approx}
    \end{equation}
From Theorem \ref{theorem.consistent_quadratic},
\begin{equation}
    \begin{aligned}
        &\left\Vert
        \frac{A_u}{S_T}
        \right\Vert_{M/2}\\
        &\leq \frac{2}{S_T}\left\Vert
        \sum_{t_1 = (q - 1)(v + \ell) +  1}^{((q - 1)(v + \ell) + v)\wedge T}\sum_{t_2 = (t_1 - B_1)\vee 1}^{(t_1 - B)} b_{t_1t_2}^{\circ}\left(\boldsymbol{\epsilon}_{t_1}^\top \boldsymbol{\epsilon}_{t_2} - \mathbf{E}\left[\boldsymbol{\epsilon}_{t_1}^\top \boldsymbol{\epsilon}_{t_2}\right]\right)
        \right\Vert_{M/2}\\
        &+ \frac{2}{S_T}\left\Vert
        \sum_{t_1 = (q - 1)(v + \ell) +  1}^{((q - 1)(v + \ell) + v)\wedge T}\sum_{t_2 = (t_1 - B_1)\vee 1}^{(t_1 - B)} b_{t_1t_2}^{\circ}\left(\mathbf{E}\left[\boldsymbol{\epsilon}_{t_1}^\top \boldsymbol{\epsilon}_{t_2}\mid\mathcal{F}_{t_1, \ell}\right] - \boldsymbol{\epsilon}_{t_1}^\top \boldsymbol{\epsilon}_{t_2}\right)
        \right\Vert_{M/2}\\
        &\leq \frac{C}{S_T}\left(\sqrt{d\sum_{t_1 = (q - 1)(v + \ell) +  1}^{((q - 1)(v + \ell) + v)\wedge T}\sum_{t_2 = (t_1 - B_1)\vee 1}^{(t_1 - B)} b_{t_1t_2}^{\circ2}} + \frac{T^{5/2}}{B^\alpha} + \frac{T^{3/2}}{B^{\alpha - 2}}\right)\\
        & + \frac{C}{S_T}\left(
        \left(\sum_{q = 1}^\infty\delta_q\right)\sqrt{d\sum_{t_1 = (q - 1)(v + \ell) +  1}^{((q - 1)(v + \ell) + v)\wedge T}\sum_{t_2 = (t_1 - B_1)\vee 1}^{(t_1 - B)} b_{t_1t_2}^{\circ2}} + \frac{dT^{3/2}}{\ell^{\alpha - 1}} + \frac{d\sqrt{T}}{B^\alpha}
        \right)\\
        &\leq \frac{C_1}{S_T}\left(
        \sqrt{dv(B_1 - B)} + \frac{T^{5/2}}{B^\alpha} + \frac{T^{3/2}}{B^{\alpha - 2}} + \frac{T^{5/2}}{\ell^{\alpha - 1}}
        \right)\\
        &\leq C_2\sqrt{\frac{v}{T}} + \frac{C_2T^{3/2}}{B^\alpha\sqrt{(B_1  - B)}} + \frac{C_2T^{1/2}}{B^{\alpha - 2}\sqrt{(B_1-  B)}} + \frac{C_2T^{3/2}}{\ell^{\alpha - 1}\sqrt{(B_1-  B)}}.
    \end{aligned}
    \label{eq.Big_S}
\end{equation}
This implies that 
\begin{equation}
    \begin{aligned}
        &\sup_{x\in\mathbf{R}}\left\vert
        \mathbf{E}\left[
        g_{\psi,x}\left(\frac{\sum_{q = 1}^R A_q}{S_T}\right) 
        \right] - \mathbf{E}\left[g_{\psi,x}\left(\frac{\sum_{q = 1}^R A^*_u}{S_T}\right)\right]
        \right\vert\\
        & =  \sup_{x\in\mathbf{R}}\left\vert
        \mathbf{E}\left[
        g_{\psi,x}\left(\frac{T_{R} + A_R}{S_T}\right) 
        \right] - \mathbf{E}\left[g_{\psi,x}\left(\frac{T_1 + A_1^*}{S_T}\right)\right]
        \right\vert\\
        &\leq \sum_{u = 1}^R\sup_{x\in\mathbf{R}}\left\vert
        \mathbf{E}\left[
        g_{\psi,x}\left(\frac{T_{u} + A_u}{S_T}\right) 
        \right] - \mathbf{E}\left[g_{\psi,x}\left(\frac{T_u + A_u^*}{S_T}\right)\right]
        \right\vert\\
        &\leq \sum_{u = 1}^R \frac{C\psi^3 \left\Vert A_u\right\Vert_{M/2}^3}{S^3_T}\\
        &\leq C_1\psi^3 R \left(\frac{v^{3/2}}{T^{3/2}} + \frac{T^{9/2}}{B^{3\alpha}(B_1  - B)^{3/2}} + \frac{T^{3/2}}{B^{3\alpha - 6}(B_1 - B)^{3/2}} + \frac{T^{9/2}}{\ell^{3\alpha - 3}(B_1  -B)^{3/2}}\right)\\
        &\leq C_2\psi^3\left(\sqrt{\frac{v}{T}} + \frac{T^{11/2}}{B^{3\alpha}v(B_1  - B)^{3/2}} + \frac{T^{5/2}}{B^{3\alpha - 6}v(B_1 - B)^{3/2}} + \frac{T^{11/2}}{\ell^{3\alpha - 3}v(B_1  -B)^{3/2}}\right).
    \end{aligned}
    \label{eq.big_block_gaussian}
\end{equation}
    Finally, notice that  $\frac{1}{S_T}\sum_{q = 1}^RA_q^*$ has normal distribution with mean $0$ and variance 
    $$
    \mathrm{Var}\left(
    \frac{1}{S_T}\sum_{q = 1}^RA_q^*
    \right) = 
    \frac{1}{S^2_T}\sum_{q = 1}^R \mathrm{Var}\left(A^*_q\right) = \frac{1}{S^2_T}\sum_{q = 1}^R \mathrm{Var}(A_q) = \frac{1}{S^2_T}\left\Vert\sum_{q = 1}^R A_q\right\Vert_2^2,
    $$
    we have 
    \begin{equation}
    \begin{aligned}
        \mathbf{E}\left[g_{\psi,x}\left(\frac{1}{S_T}\sum_{q = 1}^RA_q^*\right)\right] = \mathbf{E}\left[g_{\psi,x}\left(
        \frac{Z}{S_T}\left\Vert\sum_{q = 1}^R A_q\right\Vert_2
        \right)\right],
    \end{aligned}
    \label{eq.change_to_Z}
    \end{equation}
    where $Z$ is a normal random with standard normal distribution. Similarly, since $\xi$ has normal distribution with $\mathbf{E}[\xi] = 0$ and $\mathrm{Var}(\xi) = \mathrm{Var}\left(Q / S_T\right),$ we have 
    \begin{align*}
        \mathbf{E}\left[
        g_{\psi,x}\left(\xi\right)
        \right] = \mathbf{E}\left[
        g_{\psi,x}\left(
        \frac{Z}{S_T}\left\Vert
        \sum_{t_1 = 1}^T\sum_{t_2 = 1}^T b_{t_1t_2}\left(\boldsymbol{\epsilon}_{t_1}^\top \boldsymbol{\epsilon}_{t_2} - \mathbf{E}\left[\boldsymbol{\epsilon}_{t_1}^\top \boldsymbol{\epsilon}_{t_2}\right]\right)
        \right\Vert_2
        \right)
        \right],
    \end{align*}
where $Z$ is the standard normal random variable that coincides with that in \eqref{eq.change_to_Z}. Therefore, 
\begin{align*}
    &\left\vert
    \mathbf{E}\left[g_{\psi,x}\left(\frac{1}{S_T}\sum_{q = 1}^RA_q^*\right)\right] - \mathbf{E}\left[
        g_{\psi,x}\left(\xi\right)
        \right]
    \right\vert\\
    & = \left\vert
    \mathbf{E}\left[g_{\psi,x}\left(
        \frac{Z}{S_T}\left\Vert\sum_{q = 1}^R A_q\right\Vert_2
        \right)\right] - \mathbf{E}\left[
        g_{\psi,x}\left(
        \frac{Z}{S_T}\left\Vert
        \sum_{t_1 = 1}^T\sum_{t_2 = 1}^T b_{t_1t_2}\left(\boldsymbol{\epsilon}_{t_1}^\top \boldsymbol{\epsilon}_{t_2} - \mathbf{E}\left[\boldsymbol{\epsilon}_{t_1}^\top \boldsymbol{\epsilon}_{t_2}\right]\right)
        \right\Vert_2
        \right)
        \right]
    \right\vert\\
    &\leq \frac{C\psi}{S_T}\left\vert\ 
    \left\Vert\sum_{q = 1}^R A_q\right\Vert_2 - \left\Vert
        \sum_{t_1 = 1}^T\sum_{t_2 = 1}^T b_{t_1t_2}\left(\boldsymbol{\epsilon}_{t_1}^\top \boldsymbol{\epsilon}_{t_2} - \mathbf{E}\left[\boldsymbol{\epsilon}_{t_1}^\top \boldsymbol{\epsilon}_{t_2}\right]\right)
        \right\Vert_2\ 
    \right\vert\\
    &.\leq \frac{C\psi}{S_T}
    \left\Vert\sum_{q = 1}^R A_q - 
        \sum_{t_1 = 1}^T\sum_{t_2 = 1}^T b_{t_1t_2}\left(\boldsymbol{\epsilon}_{t_1}^\top \boldsymbol{\epsilon}_{t_2} - \mathbf{E}\left[\boldsymbol{\epsilon}_{t_1}^\top \boldsymbol{\epsilon}_{t_2}\right]\right)
        \right\Vert_2.
\end{align*}
    From \eqref{eq.a_plus_A},
    $$\sum_{q = 1}^R A_q + \sum_{q = 1}^R a_q = \sum_{t_1 = 1}^T\sum_{t_2=  1}^T b_{t_1t_2}\left(\mathbf{E}\left[\boldsymbol{\epsilon}_{t_1}^\top \boldsymbol{\epsilon}_{t_2}\mid\mathcal{F}_{t_1\vee t_2, \ell}\right] - \mathbf{E}\left[\boldsymbol{\epsilon}_{t_1}^\top \boldsymbol{\epsilon}_{t_2}\right]\right),
    $$ so from \eqref{eq.condition_change} and \eqref{eq.s_q_calculate},
    \begin{align*}
    &\frac{1}{S_T}
    \left\Vert\sum_{q = 1}^R A_q - 
        \sum_{t_1 = 1}^T\sum_{t_2 = 1}^T b_{t_1t_2}\left(\boldsymbol{\epsilon}_{t_1}^\top \boldsymbol{\epsilon}_{t_2} - \mathbf{E}\left[\boldsymbol{\epsilon}_{t_1}^\top \boldsymbol{\epsilon}_{t_2}\right]\right)\right\Vert_2\\
    &\leq \left\Vert\frac{\sum_{q = 1}^R a_q}{S_T}\right\Vert_2 + \frac{1}{S_T}\left\Vert
    \sum_{t_1 = 1}^T\sum_{t_2 = 1}^T b_{t_1t_2}\left(\boldsymbol{\epsilon}_{t_1}^\top \boldsymbol{\epsilon}_{t_2} - \mathbf{E}\left[\boldsymbol{\epsilon}_{t_1}^\top \boldsymbol{\epsilon}_{t_2}\mid\mathcal{F}_{t_1\vee t_2, \ell}\right]\right)
    \right\Vert_2\\
    &\leq  C\sqrt{\frac{\ell}{v}} + \frac{CT^2}{B^\alpha\sqrt{v(B_1 - B)}} + \frac{CT}{B^{\alpha - 2}\sqrt{v(B_1 - B)}} + \frac{CT^2}{\ell^{\alpha - 1}\sqrt{v(B_1-  B)}}\\
    & +  \frac{C}{(\ell + 1 - B_1)^\alpha} + \frac{CT^{3/2}}{\sqrt{(B_1 -  B)}\ell^{\alpha - 1}} + \frac{CT^{1/2}}{\sqrt{(B_1 - B)}B^\alpha}.
    \end{align*}
    This implies that 
    \begin{equation}
        \begin{aligned}
            &\left\vert
    \mathbf{E}\left[g_{\psi,x}\left(\frac{1}{S_T}\sum_{q = 1}^RA_q^*\right)\right] - \mathbf{E}\left[
        g_{\psi,x}\left(\xi\right)
        \right]
    \right\vert\\
    &\leq  C\psi\sqrt{\frac{\ell}{v}} + \frac{C\psi T^2}{B^\alpha\sqrt{v(B_1 - B)}} + \frac{C\psi T}{B^{\alpha - 2}\sqrt{v(B_1 - B)}} + \frac{C\psi T^2}{\ell^{\alpha - 1}\sqrt{v(B_1-  B)}}\\
    & +  \frac{C\psi}{(\ell + 1 - B_1)^\alpha} + \frac{C\psi T^{3/2}}{\sqrt{(B_1 -  B)}\ell^{\alpha - 1}} + \frac{C\psi T^{1/2}}{\sqrt{(B_1 - B)}B^\alpha}
        \end{aligned}
        \label{eq.gaussian_xi}
    \end{equation}
    We choose the parameters 
    $$\psi = \log(T),\quad \ell = 
    \lfloor
    T^{\frac{3\kappa_2 + 1}{4}}
    \rfloor,\quad\text{and}\quad v = \lfloor T^{\frac{\kappa_2 + 1}{2}}\rfloor,$$
    where $\lfloor x\rfloor$ denotes the largest integer that is smaller than or equal to $x.$
    In such case 
    $$
    \sqrt{\frac{\ell}{v}} \asymp T^{\frac{\kappa_2 - 1}{8}},\quad \sqrt{\frac{v}{T}} \asymp T^{\frac{\kappa_2 - 1}{4}},\quad\text{and}\quad \frac{1}{(1 + \ell - B_1)^\alpha} \asymp T^{-\frac{\alpha(1 + 3\kappa_2)}{4}}.
    $$
    Furthermore, 
    \begin{align*}
    &\frac{T^{3/2}}{\sqrt{(B_1 - B)}\ell^{\alpha - 1}}  = O\left(T^{\frac{7 - \alpha}{4} - \frac{(3\alpha - 1)\kappa_2}{4}}\right),\quad  \frac{T^{1/2}}{\sqrt{(B_1 - B)}B^\alpha}= O\left(T^{\frac{1 - \kappa_2}{2} - \alpha\kappa_1}\right),\\
    &\frac{T^2}{\ell^{\alpha - 1}\sqrt{v(B_1 - B)}} = O\left(T^{\frac{8 - \alpha}{4} - \frac{3\alpha\kappa_2}{4}}\right),\quad \frac{T}{B^{\alpha - 2}\sqrt{v(B_1-  B)}} = O\left(T^{\frac{3(1 - \kappa_2)}{4} - (\alpha - 2)\kappa_1}\right),\\
    &\frac{T^2}{B^\alpha \sqrt{v(B_1-  B)}} = O\left(T^{\frac{7 - 3\kappa_2}{4} - \alpha\kappa_1}\right),\quad
    \frac{T^{11/2}}{B^{3\alpha}v(B_1  - B)^{3/2}} = O\left(T^{5 - 2\kappa_2 - 3\alpha\kappa_1}\right),\\ 
    &\frac{T^{11/2}}{\ell^{3\alpha - 3}v(B_1  -B)^{3/2}} = O\left(T^{\frac{23 - 3\alpha}{4} - \frac{(9\alpha - 1)\kappa_2}{4}}\right),\quad \frac{T^{5/2}}{B^{3\alpha - 6}v(B_1 - B)^{3/2}} = O\left(T^{2 - 2\kappa_2 - (3\alpha - 6)\kappa_1}\right).
    \end{align*} 
    Since $\kappa_1 > \frac{2}{\alpha}$ and $\alpha > 4,$ we have 
    \begin{align*}
    &\frac{1 - \kappa_2}{2} - \alpha\kappa_1 < \frac{1}{2} - 2 = -\frac{3}{2},\quad  
    \frac{3(1 - \kappa_2)}{4} - (\alpha - 2)\kappa_1 < \frac{3}{4} - 1 = -\frac{1}{4},\\ 
    &\frac{7 - 3\kappa_2}{4} - \alpha\kappa_1 < \frac{7}{4} - 2 = -\frac{1}{4},\quad
    5 - 2\kappa_2 - 3\alpha\kappa_1 < 5 - 6 = -1,\\ 
    &2 - 2\kappa_2 - (3\alpha - 6)\kappa_1 < 2 - 3 = -1.
    \end{align*}
    On the other hand, $\frac{2}{\alpha} < \kappa_2 < 1,$  so 
    \begin{align*}
    &7 - \alpha- (3\alpha - 1)\kappa_2 < 1  - \alpha + \kappa_2 < -2,\\ 
    &8 - \alpha - 3\alpha\kappa_2 < 2 -\alpha < -2,\\
    &23 - 3\alpha - (9\alpha - 1)\kappa_2 < 5 - 3\alpha + \kappa_2 < -6.
    \end{align*}
    Therefore, from \eqref{eq.Prob_to_E}, \eqref{eq.E_truncate}, \eqref{eq.truncate_to_Big_block}, \eqref{eq.big_block_gaussian}, and \eqref{eq.gaussian_xi}, we prove \eqref{eq.Gaussian_approx}.
\end{proof}

The proof of Theorem \ref{theorem.var_estimation} relies on balancing the bias and variance of the variance estimator. However, compared to the setting of the classical Newey--West estimator, such as (5) in \cite{MR890864}, the variance of the quadratic form $Q$ is influenced by fourth-order cumulants of time series data. Our work therefore weights products of the second-order residuals $\vartheta_t$ defined in \eqref{eq.def_vartheta_t} instead of the original time series data to construct the variance estimator.

Another notable trick here is that our variance estimator does not incorporate bias arising from data covariances. Specifically, since $b_{t_1t_2} = 0$ if $\vert t_1 - t_2\vert < B$ or $\vert t_1  -t_2\vert > B_1,$ we have 
\begin{equation}
\begin{aligned}
    Q & = \sum_{t_1 = 1}^T\sum_{t_2 = 1}^T b_{t_1t_2}\left(\boldsymbol{\epsilon}^\top_{t_1}\boldsymbol{\epsilon}_{t_2} - \mathbf{E}\left[\boldsymbol{\epsilon}^\top_{t_1}\boldsymbol{\epsilon}_{t_2}\right]\right)\\
      & = \sum_{t_1 = 2}^T\sum_{t_2  = 1}^{t_1 - 1}b_{t_1t_2}\left(\boldsymbol{\epsilon}^\top_{t_1}\boldsymbol{\epsilon}_{t_2} - \mathbf{E}\left[\boldsymbol{\epsilon}^\top_{t_1}\boldsymbol{\epsilon}_{t_2}\right]\right) + \sum_{t_2  = 2}^T\sum_{t_1 = 1}^{t_2 - 1}b_{t_1t_2}\left(\boldsymbol{\epsilon}^\top_{t_1}\boldsymbol{\epsilon}_{t_2} - \mathbf{E}\left[\boldsymbol{\epsilon}^\top_{t_1}\boldsymbol{\epsilon}_{t_2}\right]\right)\\
      & = 2\sum_{t_1 = B+1}^T\sum_{t_2  = 1\vee (t_1 - B_1)}^{t_1 - B}b_{t_1t_2}^\circ\left(\boldsymbol{\epsilon}^\top_{t_1}\boldsymbol{\epsilon}_{t_2} - \mathbf{E}\left[\boldsymbol{\epsilon}^\top_{t_1}\boldsymbol{\epsilon}_{t_2}\right]\right),
\end{aligned}
\label{eq.Q_to_one_side}
\end{equation}
where $b_{t_1t_2}^\circ = (b_{t_1t_2} + b_{t_2t_1}) / 2.$ The variance of $Q$ is then given by 
\begin{equation}
    \begin{aligned}
        &\mathrm{Var}\left( \frac{Q}{S_T}\right) = \frac{4}{S_T^2}\mathbf{E}\left[
        \left(\sum_{t_1 = B+1}^T\sum_{t_2  = 1\vee (t_1 - B_1)}^{t_1 - B}b_{t_1t_2}^\circ\left(\boldsymbol{\epsilon}^\top_{t_1}\boldsymbol{\epsilon}_{t_2} - \mathbf{E}\left[\boldsymbol{\epsilon}^\top_{t_1}\boldsymbol{\epsilon}_{t_2}\right]\right)\right)^2
        \right] = {\Gamma} - {\Theta},\\
       &\text{where}\quad  {\Gamma} = \frac{4}{S^2_T}\sum_{t_1 = B+1}^T\sum_{t_2  = 1\vee (t_1 - B_1)}^{t_1 - B}\sum_{t_3 = B+1}^T\sum_{t_4  = 1\vee (t_3 - B_1)}^{t_3 - B}b_{t_1t_2}^\circ b_{t_3t_4}^\circ\mathbf{E}\left[\boldsymbol{\epsilon}^\top_{t_1}\boldsymbol{\epsilon}_{t_2}\boldsymbol{\epsilon}^\top_{t_3}\boldsymbol{\epsilon}_{t_4}\right],\\
    &\text{and}\quad {\Theta} = \frac{4}{S_T^2}\sum_{t_1 = B+1}^T\sum_{t_2  = 1\vee (t_1 - B_1)}^{t_1 - B}\sum_{t_3 = B+1}^T\sum_{t_4  = 1\vee (t_3 - B_1)}^{t_3 - B}b_{t_1t_2}^\circ b_{t_3t_4}^\circ\mathbf{E}\left[\boldsymbol{\epsilon}^\top_{t_1}\boldsymbol{\epsilon}_{t_2}\right]\mathbf{E}\left[\boldsymbol{\epsilon}^\top_{t_3}\boldsymbol{\epsilon}_{t_4}\right].
    \end{aligned}
    \label{eq.decomposition_var_Q}
\end{equation}
The following proof establishes that the proposed estimator is consistent for $\Gamma,$ and that $\Theta$ is negligible compared to $\Gamma.$ This phenomenon arises because $b_{t_1t_2} = 0$ if $\vert t_1 - t_2\vert < B,$ and is crucial in our setup---the time series data are not assumed to be stationary, implying that $\mathbf{E}\left[\boldsymbol{\epsilon}^\top_{t_1}\boldsymbol{\epsilon}_{t_2}\right]$ can vary with respect to different time lag $\vert  t_1 - t_2\vert.$ Consequently, estimating each individual covariance term $\mathbf{E}\left[\boldsymbol{\epsilon}^\top_{t_1}\boldsymbol{\epsilon}_{t_2}\right]$ would in general be infeasible without imposing further structural assumptions. Fortunately, since $\Theta$ is negligible compared to $\Gamma,$ estimating $\mathbf{E}\left[\boldsymbol{\epsilon}^\top_{t_1}\boldsymbol{\epsilon}_{t_2}\right]$ becomes unnecessary.

\begin{proof}[Proof of Theorem \ref{theorem.var_estimation}]
The proof is separated into two parts: that $\Theta$ is small and that $\Gamma$ can be estimated from our estimator. From \eqref{eq.eta_epsilon} and use the same notation as in \eqref{eq.def_gamma_eta}, for any $t_2 > t_1,$
\begin{align*}
    \left\vert\mathbf{E}\left[
    \boldsymbol{\epsilon}^\top_{t_1}\boldsymbol{\epsilon}_{t_2}
    \right]\right\vert &= \left\vert\mathbf{E}\left[\boldsymbol{\epsilon}^\top_{t_1}\left(\boldsymbol{\epsilon}_{t_2} - \mathbf{E}\left[\boldsymbol{\epsilon}_{t_2}\mid\mathcal{F}_{t_2,t_2 - t_1 - 1}\right]\right)\right]\right\vert\\
    & = \left\vert\mathbf{E}\left[\boldsymbol{\epsilon}^\top_{t_1}\boldsymbol{\eta}_{t_2,t_2 - t_1 - 1}\right]\right\vert\leq \frac{Cd}{(t_2 - t_1)^\alpha}.
\end{align*}
Notice that $\vert b_{t_1t_2}\vert\leq 1,$ so $\vert b_{t_1t_2}^\circ\vert\leq \left(\vert b_{t_1t_2}\vert + \vert b_{t_2t_1}\vert\right) / 2\leq 1,$ and 
\begin{align*}
    \left\vert
    \sum_{t_1 = B+1}^T\sum_{t_2  = 1\vee (t_1 - B_1)}^{t_1 - B}b_{t_1t_2}^\circ\mathbf{E}\left[\boldsymbol{\epsilon}^\top_{t_1}\boldsymbol{\epsilon}_{t_2}\right]
    \right\vert &\leq  \sum_{t_1 = B+1}^T\sum_{t_2  = 1\vee (t_1 - B_1)}^{t_1 - B}\left\vert \mathbf{E}\left[\boldsymbol{\epsilon}^\top_{t_1}\boldsymbol{\epsilon}_{t_2}\right]\right\vert\\
    &\leq Cd\sum_{t_1 = B+1}^T\sum_{t_2  = 1\vee (t_1 - B_1)}^{t_1 - B}\frac{1}{(t_1 - t_2)^\alpha }\\
    &\leq C_1 d\sum_{t_1 = B+1}^T\frac{1}{B^{\alpha - 1}}\leq \frac{C_2T^2}{B^{\alpha - 1}}.
\end{align*}
This implies that 
\begin{equation}
    \begin{aligned}
        \left\vert\Theta\right\vert\leq \frac{C}{S^2_T}\frac{T^4}{B^{2\alpha - 2}} \leq \frac{C_1T^4}{Td(B_1 - B)B^{2\alpha - 2}}\leq \frac{C_2T^2}{(B_1 - B)B^{2\alpha - 2}}.
    \end{aligned}
    \label{eq.abs_Theta}
\end{equation}
From the assumptions in Theorem \ref{theorem.var_estimation}, 
\begin{align*}
    \frac{T^2}{(B_1 - B)B^{2\alpha - 2}} = O\left( T^{2 - (2\alpha - 2)\kappa_1 - \kappa_2}\right) = O(1/T).
\end{align*}
Therefore, the bias $\left\vert\Theta\right\vert$ is negligible compared to $\Gamma,$ so its estimation is unnecessary.

We then prove the consistency of the estimation of $\Gamma.$  From \eqref{eq.def_vartheta_t},
\begin{align*}
    \Gamma &= \frac{4}{S^2_T}\sum_{t_1 = B+1}^T\sum_{t_3 = B+1}^T \mathbf{E}
    \left[
    \left(\sum_{t_2  = 1\vee (t_1 - B_1)}^{t_1 - B}b^\circ_{t_1t_2}\boldsymbol{\epsilon}^\top_{t_1}\boldsymbol{\epsilon}_{t_2}\right)\left(\sum_{t_4  = 1\vee (t_3 - B_1)}^{t_3 - B}b^\circ_{t_3t_4}\boldsymbol{\epsilon}_{t_3}^\top\boldsymbol{\epsilon}_{t_4}\right)
    \right]\\
    & = \frac{1}{S^2_T}\sum_{t_1 = B+1}^T\sum_{t_3 = B+1}^T \mathbf{E}\left[\vartheta_{t_1}\vartheta_{t_3}\right],
\end{align*}
so 
\begin{align*}
    &\left\vert\frac{1}{S^2_T}\sum_{t_1 = B + 1}^T\sum_{t_2 = B+1}^T \mathcal{K}\left(\frac{t_1 - t_2}{H}\right)\vartheta_{t_1}\vartheta_{t_2}  - \Gamma\right\vert\\
    &\leq \left\vert \frac{1}{S^2_T}\sum_{t_1 = B + 1}^T\sum_{t_2 = B+1}^T \mathcal{K}\left(\frac{t_1 - t_2}{H}\right)\left(\vartheta_{t_1}\vartheta_{t_2} - \mathbf{E}\left[\vartheta_{t_1}\vartheta_{t_2}\right]\right)\right\vert\\
    &+ \left\vert\frac{1}{S^2_T}\sum_{t_1 = B + 1}^T\sum_{t_2 = B+1}^T\left(\mathcal{K}\left(\frac{t_1 - t_2}{H}\right) - 1\right)\mathbf{E}\left[\vartheta_{t_1}\vartheta_{t_2}\right]\right\vert.
\end{align*}
Define 
\begin{align*}
    \iota_t = \vartheta_{t} - \mathbf{E}\left[\vartheta_{t}\right] = 2\sum_{t_2 = (t - B_1)\vee 1}^{t - B} b^\circ_{tt_2}\left(\boldsymbol{\epsilon}_{t}^\top \boldsymbol{\epsilon}_{t_2} - \mathbf{E}\left[\boldsymbol{\epsilon}_{t}^\top \boldsymbol{\epsilon}_{t_2}\right]\right),
\end{align*}
then $\mathbf{E}\left[ \iota_t\right] = 0,$ and 
\begin{align*}
     &\left\vert\frac{1}{S^2_T}\sum_{t_1 = B + 1}^T\sum_{t_2 = B+1}^T \mathcal{K}\left(\frac{t_1 - t_2}{H}\right)\left(\vartheta_{t_1}\vartheta_{t_2} - \mathbf{E}\left[\vartheta_{t_1}\vartheta_{t_2}\right]\right)\right\vert\\
     &\leq  \left\vert\frac{1}{S^2_T}\sum_{t_1 = B + 1}^T\sum_{t_2 = B+1}^T \mathcal{K}\left(\frac{t_1 - t_2}{H}\right)\left(\iota_{t_1}\iota_{t_2} - \mathbf{E}\left[\iota_{t_1}\iota_{t_2}\right]\right)\right\vert\\
     & + \frac{1}{S^2_T}\left\vert\sum_{t_1 = B + 1}^T\sum_{t_2 = B+1}^T \mathcal{K}\left(\frac{t_1 - t_2}{H}\right)\iota_{t_1}\mathbf{E}\left[\vartheta_{t_2}\right]\right\vert\\
     & + \frac{1}{S^2_T}\left\vert\sum_{t_1 = B + 1}^T\sum_{t_2 = B+1}^T \mathcal{K}\left(\frac{t_1 - t_2}{H}\right)\iota_{t_2}\mathbf{E}\left[\vartheta_{t_1}\right]\right\vert,
\end{align*}
and
\begin{equation}
\begin{aligned}
    &\left\vert\frac{1}{S^2_T}\sum_{t_1 = B + 1}^T\sum_{t_2 = B+1}^T\left(\mathcal{K}\left(\frac{t_1 - t_2}{H}\right) - 1\right)\mathbf{E}\left[\vartheta_{t_1}\vartheta_{t_2}\right]\right\vert\\
    &\leq \left\vert\frac{1}{S^2_T}\sum_{t_1 = B + 1}^T\sum_{t_2 = B+1}^T\left(\mathcal{K}\left(\frac{t_1 - t_2}{H}\right) - 1\right)\mathbf{E}\left[\iota_{t_1}\iota_{t_2}\right]\right\vert\\
    &+ \frac{1}{S^2_T}\sum_{t_1 = B + 1}^T\sum_{t_2 = B+1}^T\left( 1 - \mathcal{K}\left(\frac{t_1 - t_2}{H}\right)\right)\left\vert\mathbf{E}\left[\vartheta_{t_1}\right]\right\vert\times \left\vert\mathbf{E}\left[\vartheta_{t_2}\right]\right\vert.
\end{aligned}
\label{eq.bias_decomposition}
\end{equation}
From Theorem \ref{theorem.consistent_quadratic},
\begin{align*}
    \left\Vert
    \iota_t
    \right\Vert_{M/2} &= 2\left\Vert
    \sum_{t_2 = (t - B_1)\vee 1}^{t - B} b^\circ_{tt_2}\left(\boldsymbol{\epsilon}_{t}^\top \boldsymbol{\epsilon}_{t_2} - \mathbf{E}\left[\boldsymbol{\epsilon}_{t}^\top \boldsymbol{\epsilon}_{t_2}\right]\right)
    \right\Vert_{M/2}\\
    &\leq C\sqrt{d\sum_{t_2 = (t - B_1)\vee 1}^{t - B}b^{\circ2}_{tt_2}} + \frac{CT^{5/2}}{B^\alpha} + \frac{CT^{3/2}}{B^{\alpha - 2}}\\
    &\leq C_1\sqrt{d(B_1-  B)} + \frac{CT^{5/2}}{B^\alpha} + \frac{CT^{3/2}}{B^{\alpha - 2}}.
\end{align*}
Since $B\asymp T^{\kappa_1},$  $\alpha\kappa_1 > 2,$ $\kappa_1 < 1/6,$ and $\alpha > 14, $ we have 
\begin{equation}
\left\Vert\iota_t\right\Vert_{M/2}\leq C\sqrt{d(B_1-  B)}.
\label{eq.moment_iota}
\end{equation}
On the other hand, from \eqref{eq.covariance_delta}, 
\begin{equation}
    \begin{aligned}
        \left\vert
    \mathbf{E}\left[\vartheta_{t}\right]
    \right\vert & \leq 2\sum_{t_2 = (t -  B_1)\vee 1}^{t - B}\vert b_{tt_2}^\circ\vert\left\vert \mathbf{E}\left[\boldsymbol{\epsilon}_{t}^\top \boldsymbol{\epsilon}_{t_2}\right]\right\vert\\
    &\leq Cd\sum_{t_2 = (t - B_1)\vee 1}^{t - B}\frac{1}{(t - t_2)^\alpha}\leq \frac{C_1d}{B^{\alpha - 1}}.
    \label{eq.bias_theta}
    \end{aligned}
\end{equation}
For any $a  \geq  B_1,$ from \eqref{eq.truncate_moment},
\begin{equation}
\begin{aligned}
    \left\Vert
    \iota_t - \mathbf{E}
    \left[
    \iota_t\mid \mathcal{F}_{t,a}
    \right]
    \right\Vert_{M/2} &= 2\left\Vert
    \sum_{t_2  = (t - B_1)\vee 1}^{t - B}b^\circ_{tt_2}\left(\boldsymbol{\epsilon}_t^\top\boldsymbol{\epsilon}_{t_2} - \mathbf{E}\left[\boldsymbol{\epsilon}_t^\top\boldsymbol{\epsilon}_{t_2}\mid \mathcal{F}_{t,a}\right]\right)
    \right\Vert_{M/2}\\
    &\leq C\sum_{q = a + 1 - B_1}^\infty\delta_q\sqrt{d\sum_{t_2  = (t - B_1)\vee 1}^{t - B}b^{\circ 2}_{tt_2}} +\frac{CdT^{3/2}}{a^{\alpha - 1}} + \frac{Cd\sqrt{T}}{B^\alpha}\\
    &\leq \frac{C_1\sqrt{d(B_1  -B)}}{(a + 1 - B_1)^\alpha} + \frac{CT^{5/2}}{a^{\alpha - 1}} + \frac{CT^{3/2}}{B^\alpha}.
\end{aligned}
\label{eq.truncate_iota}
\end{equation}
This implies that 
\begin{align*}
    &\left\vert\frac{1}{S^2_T}\sum_{t_1 = B + 1}^T\sum_{t_2 = B+1}^T\left(\mathcal{K}\left(\frac{t_1 - t_2}{H}\right) - 1\right)\mathbf{E}\left[\iota_{t_1}\iota_{t_2}\right]\right\vert\\
    &\leq \frac{2}{S_T^2}\sum_{q = 0}^{T-B-1}\left(1 - \mathcal{K}\left(\frac{q}{H}\right)\right)\sum_{t_1 = B+1}^{T- q}\left\vert \mathbf{E}\left[\iota_{t_1}\iota_{t_1+q}\right]\right\vert\\
    &\leq \frac{2}{S^2_T}\sum_{q = 0}^{B_1}\left(1 - \mathcal{K}\left(\frac{q}{H}\right)\right)\sum_{t_1 = B + 1}^{T - q}\left\Vert \iota_{t_1}\right\Vert_{M/2}\left\Vert \iota_{t_1 + q}\right\Vert_{M/2}\\
    & + \frac{2}{S^2_T}\sum_{q = B_1 + 1}^{H}\left(1 - \mathcal{K}\left(\frac{q}{H}\right)\right)\sum_{t_1 = B+1}^{T  - q}\left\vert
    \mathbf{E}\left[\iota_{t_1}\left(\iota_{t_1 + q} - \mathbf{E}\left[\iota_{t_1 + q}\mid\mathcal{F}_{t_1+q, q - 1}\right]\right)\right]
    \right\vert\\
    & +\frac{2}{S^2_T}\sum_{q = H+1}^{T - B - 1}\sum_{t_1 = B+1}^{T  - q}\left\vert
    \mathbf{E}\left[\iota_{t_1}\left(\iota_{t_1 + q} - \mathbf{E}\left[\iota_{t_1 + q}\mid\mathcal{F}_{t_1+q, q - 1}\right]\right)\right]
    \right\vert.
\end{align*}    
From Definition \ref{definition.kernel_function}, we have $1 - \mathcal{K}(x) = \mathcal{K}(0) - \mathcal{K}(x)\leq Cx$ for any $x\in[0,1],$ so from \eqref{eq.moment_iota},
\begin{align*}
    &\frac{1}{S^2_T}\sum_{q = 0}^{B_1}\left(1 - \mathcal{K}\left(\frac{q}{H}\right)\right)\sum_{t_1 = B + 1}^{T - q}\left\Vert \iota_{t_1}\right\Vert_{M/2}\left\Vert \iota_{t_1 + q}\right\Vert_{M/2}\\
    &\leq \frac{CTd(B_1 - B)}{Td(B_1 - B)}\sum_{q = 0}^{B_1}\left(1 - \mathcal{K}\left(\frac{q}{H}\right)\right)\\
    & = C\sum_{q = 0}^{B_1}\left(1 - \mathcal{K}\left(\frac{q}{H}\right)\right) \leq \sum_{q = 0}^{B_1} \frac{C_1q}{H}\leq \frac{C_2 B^2_1}{H}.
\end{align*}
Similarly, from  \eqref{eq.truncate_iota},
\begin{align*}
    &\frac{1}{S^2_T}\sum_{q = B_1 + 1}^{H}\left(1 - \mathcal{K}\left(\frac{q}{H}\right)\right)\sum_{t_1 = B+1}^{T  - q}\left\vert
    \mathbf{E}\left[\iota_{t_1}\left(\iota_{t_1 + q} - \mathbf{E}\left[\iota_{t_1 + q}\mid\mathcal{F}_{t_1+q, q - 1}\right]\right)\right]
    \right\vert\\
    &\leq \frac{C}{Td(B_1 - B)}\sum_{q = B_1 + 1}^H \frac{q}{H}\sum_{t_1 = B+1}^{T  - q}\left\Vert
    \iota_{t_1}
    \right\Vert_{M/2}\left\Vert
    \iota_{t_1 + q} - \mathbf{E}\left[\iota_{t_1 + q}\mid\mathcal{F}_{t_1+q, q - 1}\right]
    \right\Vert_{M/2}\\
    &\leq \frac{C_1}{H}\sum_{q = B_1 + 1}^H\left(\frac{q}{(q - B_1)^\alpha} + \frac{qT^{5/2}}{(q - 1)^{\alpha - 1}\sqrt{d(B_1 - B)}} + \frac{qT^{3/2}}{B^\alpha\sqrt{d(B_1- B)}}\right)\\
    &=
    \frac{C_1}{H}\sum_{q = B_1 + 1}^H\left(\frac{1}{(q - B_1)^{\alpha - 1}} + \frac{B_1}{(q - B_1)^\alpha} + \frac{T^{5/2}}{(q - 1)^{\alpha - 2}\sqrt{d(B_1 - B)}} + \frac{T^{5/2}}{(q - 1)^{\alpha - 1}\sqrt{d(B_1 - B)}}\right.\\
    &\left.+\frac{qT^{3/2}}{B^\alpha\sqrt{d(B_1- B)}}
    \right)
    \\
    &\leq \frac{C_2B_1}{H} + \frac{C_2T^{2}}{HB_1^{\alpha - 3}\sqrt{(B_1 - B)}} + \frac{C_2TH}{B^\alpha\sqrt{(B_1 - B)}},
\end{align*}
and
\begin{align*}
    &\frac{1}{S^2_T}\sum_{q = H+1}^{T - B - 1}\sum_{t_1 = B+1}^{T  - q}\left\vert
    \mathbf{E}\left[\iota_{t_1}\left(\iota_{t_1 + q} - \mathbf{E}\left[\iota_{t_1 + q}\mid\mathcal{F}_{t_1+q, q - 1}\right]\right)\right]
    \right\vert\\
    &\leq \frac{C}{Td(B_1 - B)}\sum_{q = H+1}^{T - B - 1}\sum_{t_1 = B+1}^{T  - q}\left\Vert\iota_{t_1}\right\Vert_{M/2}\left\Vert\iota_{t_1 + q} - \mathbf{E}\left[\iota_{t_1 + q}\mid\mathcal{F}_{t_1+q, q - 1}\right]\right\Vert_{M/2}\\
    &\leq C_1\sum_{q = H+1}^{T - B - 1}\left(
    \frac{1}{(q - B_1)^\alpha} + \frac{T^{5/2}}{(q - 1)^{\alpha - 1}\sqrt{d(B_1 - B)}} + \frac{T^{3/2}}{B^\alpha\sqrt{d(B_1 - B)}}
    \right)\\
    &\leq \frac{C_2}{(H - B_1)^{\alpha - 1}} + \frac{C_2T^2}{H^{\alpha - 2}\sqrt{(B_1 - B)}} + \frac{C_2T^2}{B^\alpha\sqrt{(B_1  - B)}}.
\end{align*}
Therefore, notice that $B_1\asymp T^{\kappa_2}$ with $\kappa_1 < \kappa_2 < 1/6$ and $H\asymp T^{1/3}\log(T),$ we have
\begin{equation}
    \begin{aligned}
        &\left\vert\frac{1}{S^2_T}\sum_{t_1 = B + 1}^T\sum_{t_2 = B+1}^T\left(\mathcal{K}\left(\frac{t_1 - t_2}{H}\right) - 1\right)\mathbf{E}\left[\iota_{t_1}\iota_{t_2}\right]\right\vert\\
        &\leq \frac{CB^2_1}{H} + \frac{CT^2}{HB^{\alpha - 3}_1\sqrt{(B_1 - B)}} + \frac{CT^2}{B^\alpha\sqrt{(B_1 - B)}}\\
        &\leq \frac{C_1}{T^{\frac{1}{3} - 2\kappa_2}\log(T)} + \frac{C_1T^{5/3}}{\log(T)T^{\kappa_2(\alpha - \frac{5}{2})}} + \frac{C_1}{
        T^{\kappa_2 / 2}}.
    \end{aligned}
     \label{eq.var_bias}
\end{equation}
From \eqref{eq.bias_theta},
\begin{align*}
    &\frac{1}{S^2_T}\sum_{t_1 = B + 1}^T\sum_{t_2 = B+1}^T\left( 1 - \mathcal{K}\left(\frac{t_1 - t_2}{H}\right)\right)\left\vert\mathbf{E}\left[\vartheta_{t_1}\right]\right\vert\times \left\vert\mathbf{E}\left[\vartheta_{t_2}\right]\right\vert\\
    &\leq \frac{1}{S_T^2}\left(\sum_{t = B + 1}^T\left\vert\mathbf{E}\left[\vartheta_{t}\right]\right\vert\right)^2
    \leq \frac{C}{Td(B_1 - B)}\frac{T^2d^2}{B^{2\alpha - 2}}\leq \frac{CTd}{(B_1 - B)B^{2\alpha - 2}}.
\end{align*}
Since $B\asymp T^{\kappa_1}$ with $\frac{2}{\alpha}<\kappa_1 < \frac{1}{6}$ and $d\asymp T,$ we have 
\begin{align*}
    \frac{Td}{(B_1 - B)B^{2\alpha - 2}}\leq \frac{CT^2B^2}{(B_1 - B)T^4}\leq C_1T^{-5/3}.
\end{align*}
Therefore, from \eqref{eq.bias_decomposition}, we have 
\begin{equation}
    \begin{aligned}
        &\left\vert\frac{1}{S^2_T}\sum_{t_1 = B + 1}^T\sum_{t_2 = B+1}^T\left(\mathcal{K}\left(\frac{t_1 - t_2}{H}\right) - 1\right)\mathbf{E}\left[\vartheta_{t_1}\vartheta_{t_2}\right]\right\vert\\
        &\leq \frac{C}{T^{\frac{1}{3} - 2\kappa_2}\log(T)} + \frac{CT^{5/3}}{\log(T)T^{\kappa_2(\alpha - \frac{5}{2})}} + \frac{C}{T^{\kappa_2 / 2}} + \frac{CTd}{(B_1 - B)B^{2\alpha - 2}}\\
        &\leq \frac{C_1}{T^{\frac{1}{3} - 2\kappa_2}\log(T)} + \frac{C_1T^{5/3}}{\log(T)T^{\kappa_2(\alpha - \frac{5}{2})}} + \frac{C_1}{T^{\kappa_2 / 2}}.
    \end{aligned}
    \label{eq.bias_diff}
\end{equation}
On the other hand,
\begin{align*}
    &\left\Vert\frac{1}{S^2_T}\sum_{t_1 = B + 1}^T\sum_{t_2 = B+1}^T \mathcal{K}\left(\frac{t_1 - t_2}{H}\right)\left(\iota_{t_1}\iota_{t_2} - \mathbf{E}\left[\iota_{t_1}\iota_{t_2}\right]\right)\right\Vert_{M/4}\\
    &\leq \frac{2}{S^2_T}\sum_{q = 0}^{T -  B - 1}\mathcal{K}\left(\frac{q}{H}\right)\left\Vert\sum_{t_1 = B + 1}^{T - q}\left(\iota_{t_1}\iota_{t_1+q} - \mathbf{E}\left[\iota_{t_1}\iota_{t_1+q}\right]\right)\right\Vert_{M/4}.
\end{align*}
    For any $t$ and $s$, define 
    \begin{equation}
    \omega_{t,s} = 
    \begin{cases}
    \mathbf{E}\left[\iota_{t}\mid \mathcal{F}_{t,s}\right]\ \text{if } s\geq 0,\\
    0\ \text{if } s < 0,
    \end{cases}\quad\text{and}\quad \varphi_{t,s} = \iota_{t} - \omega_{t,s}.
    \label{eq.def_omega_psi}
    \end{equation}
From this definition,   $\omega_{t_1 + q,q -1}$ is independent of $\iota_{t_1},$ and 
$$
\mathbf{E}\left[\omega_{t_1 + q,q -1}\iota_{t_1}\right] = \mathbf{E}\left[\omega_{t_1 + q,q -1}\right]\mathbf{E}\left[\iota_{t_1}\right] = 0,\quad\text{therefore}\quad \mathbf{E}\left[\iota_{t_1}\iota_{t_1+q}\right] = \mathbf{E}\left[\iota_{t_1}\varphi_{t_1 + q, q - 1}\right].
$$
Define 
$$
\mathbf{E}\left[\iota_{t_1}\varphi_{t_1 + q,q - 1}\mid \mathcal{F}_{t_1 + q, - 1}\right]  =\mathbf{E}\left[\iota_{t_1}\varphi_{t_1 + q,q - 1}\right].
$$
For any $t_1,$
\begin{align*}
    \iota_{t_1}
    = \sum_{q = 0}^T \left(\omega_{t_1, q} - \omega_{t_1, q - 1}\right) +  \iota_{t_1} - \omega_{t_1, T} = \sum_{q = 0}^T (\omega_{t_1, q} - \omega_{t_1, q - 1}) + \varphi_{t_1,T},
\end{align*}
and 
\begin{align*}
    &\iota_{t_1}\varphi_{t_1 + s,s - 1} - \mathbf{E}\left[\iota_{t_1}\varphi_{t_1 + s,s - 1}\right]\\
    &= \sum_{q = 0}^T\left(\mathbf{E}\left[\iota_{t_1}\varphi_{t_1 + s,s - 1}\mid\mathcal{F}_{t_1 + s, q}\right] - \mathbf{E}\left[\iota_{t_1}\varphi_{t_1 + s,s - 1}\mid\mathcal{F}_{t_1 + s, q - 1}\right]\right)\\
    &+ \left(\iota_{t_1}\varphi_{t_1 + s,s - 1} - \mathbf{E}\left[\iota_{t_1}\varphi_{t_1 + s,s - 1}\mid\mathcal{F}_{t_1 + s, T}\right]\right).
\end{align*}
Therefore, for any $q = 0,1,\cdots, T - B - 1,$
\begin{equation}
    \begin{aligned}
        &\left\Vert
        \sum_{t_1 = B + 1}^{T - q}\left(\iota_{t_1}\iota_{t_1+q} - \mathbf{E}\left[\iota_{t_1}\iota_{t_1+q}\right]\right)
        \right\Vert_{M/4}\\
        &\leq \left\Vert
        \sum_{t_1 = B + 1}^{T - q}\iota_{t_1}\omega_{t_1+q, q - 1}
        \right\Vert_{M/4} +\left\Vert
        \sum_{t_1 = B + 1}^{T - q}\left(\iota_{t_1}\varphi_{t_1+q, q - 1} - \mathbf{E}\left[\iota_{t_1}\varphi_{t_1+q, q - 1}\right]\right)
        \right\Vert_{M/4}\\
        &\leq \sum_{l = 0}^T\left\Vert
        \sum_{t_1 = B + 1}^{T - q}(\omega_{t_1, l} - \omega_{t_1, l - 1})\omega_{t_1+q, q - 1}
        \right\Vert_{M/4} + \left\Vert \sum_{t_1 = B + 1}^{T - q}\varphi_{t_1,T}\omega_{t_1+q, q - 1}\right\Vert_{M/4}\\
        & + \sum_{l = 0}^T\left\Vert
        \sum_{t_1 = B + 1}^{T - q}\left(
        \mathbf{E}\left[\iota_{t_1}\varphi_{t_1 + q,q - 1}\mid\mathcal{F}_{t_1 + q, l}\right] - \mathbf{E}\left[\iota_{t_1}\varphi_{t_1 + q,q - 1}\mid\mathcal{F}_{t_1 + q, l - 1}\right]\right)
        \right\Vert_{M/4}\\
        & + \left\Vert
        \sum_{t_1 = B + 1}^{T - q}\left(\iota_{t_1}\varphi_{t_1 + q,q - 1} - \mathbf{E}\left[\iota_{t_1}\varphi_{t_1 + q,q - 1}\mid\mathcal{F}_{t_1 + q, T}\right]\right)
        \right\Vert_{M/4}.
    \end{aligned}
    \label{eq.separate_terms}
\end{equation}
For any given $l = 0,1,\cdots,T, q = 0,1,\cdots, T - B - 1,$ and $z = 1,\cdots, T - q - B,$ define 
\begin{align*}
    T_{q,l,z} = \sum_{t_1 = T - q - z + 1}^{T - q}(\omega_{t_1, l} - \omega_{t_1, l - 1})\omega_{t_1+q, q - 1},
\end{align*}
and $\mathcal{T}_{q,l,z}$ the $\sigma$-field generated by $e_T,e_{T-1},\cdots,e_{T - q -z + 1 - l},$ then 
$\mathcal{T}_{q,l,z}\subset \mathcal{T}_{q,l,z + 1},$  $T_{q,l,z}$ is measurable in $\mathcal{T}_{q,l,z},$ and 
\begin{align*}
    \mathbf{E}\left[\left(T_{q,l,z + 1} - T_{q,l,z}\right)\mid \mathcal{T}_{q,l,z}\right] 
    &= \mathbf{E}\left[(\omega_{T - q - z, l} - \omega_{T - q - z, l - 1})\omega_{T - z , q - 1}\mid \mathcal{T}_{q,l,z}\right]\\
    & = \omega_{T - z , q - 1}\left(\omega_{T - q - z, l - 1} - \omega_{T - q - z, l - 1}\right) = 0,
\end{align*}
so $T_{q,l,z}$ forms a martingale. From Theorem 1.1 of \cite{MR0400380},
\begin{equation}
    \begin{aligned}
        &\left\Vert
        \sum_{t_1 = B + 1}^{T - q}(\omega_{t_1, l} - \omega_{t_1, l - 1})\omega_{t_1+q, q - 1}
        \right\Vert_{M/4}\\
        &\leq C\sqrt{\sum_{t_1 = B + 1}^{T - q}\left\Vert (\omega_{t_1, l} - \omega_{t_1, l - 1})\omega_{t_1+q, q - 1}\right\Vert^2_{M/4}}\\
        &\leq C\sqrt{\sum_{t_1 = B + 1}^{T - q}\left\Vert\omega_{t_1, l} - \omega_{t_1, l - 1}\right\Vert_{M/2}^2\left\Vert\iota_{t_1+q}\right\Vert^2_{M/2}}\\
        &\leq C_1\sqrt{d(B_1 - B)}\sqrt{\sum_{t_1 = B + 1}^{T - q}\left\Vert\omega_{t_1, l} - \omega_{t_1, l - 1}\right\Vert_{M/2}^2}.
    \end{aligned}
\end{equation}
If $l\leq B_1,$ then 
\begin{align*}
    \left\Vert \omega_{t_1, l} - \omega_{t_1, l - 1}\right\Vert_{M/2} &\leq  \left\Vert
    \omega_{t_1, l}
    \right\Vert_{M/2} + \left\Vert \omega_{t_1, l - 1}\right\Vert_{M/2}\\
    &\leq 2\left\Vert \iota_{t_1}\right\Vert_{M/2}\leq C\sqrt{d(B_1 -  B)},
\end{align*}
making 
\begin{align*}
    \left\Vert
        \sum_{t_1 = B + 1}^{T - q}(\omega_{t_1, l} - \omega_{t_1, l - 1})\omega_{t_1+q, q - 1}
        \right\Vert_{M/4}
    \leq C\sqrt{T}d(B_1 - B).
\end{align*}
If $l \geq B_1 + 1,$ from \eqref{eq.truncate_iota}, 
\begin{align*}
    \left\Vert \omega_{t_1, l} - \omega_{t_1, l - 1}\right\Vert_{M/2} &= \left\Vert \varphi_{t_1,l} - \varphi_{t_1,l - 1}\right\Vert_{M/2}\\
    &\leq \left\Vert \varphi_{t_1,l}\right\Vert_{M/2} + \left\Vert\varphi_{t_1,l - 1}\right\Vert_{M/2}\\
    &\leq \frac{C\sqrt{d(B_1-  B)}}{(l - B_1)^\alpha} + \frac{CT^{5/2}}{(l - 1)^{\alpha - 1}} + \frac{CT^{3/2}}{B^\alpha},
\end{align*}
making
\begin{align*}
    &\left\Vert
        \sum_{t_1 = B + 1}^{T - q}(\omega_{t_1, l} - \omega_{t_1, l - 1})\omega_{t_1+q, q - 1}
        \right\Vert_{M/4}\\
    &\leq \frac{C\sqrt{T}d(B_1 - B)}{(l - B_1)^\alpha} + \frac{CT^{7/2}\sqrt{(B_1  - B)}}{(l - 1)^{\alpha - 1}} + \frac{CT^{5/2}\sqrt{(B_1  - B)}}{B^\alpha}.
\end{align*}
Therefore,
\begin{equation}
    \begin{aligned}
    &\frac{1}{S^2_T}\sum_{l = 0}^T\left\Vert
        \sum_{t_1 = B + 1}^{T - q}(\omega_{t_1, l} - \omega_{t_1, l - 1})\omega_{t_1+q, q - 1}
        \right\Vert_{M/4}\\
    &\leq \frac{C}{S^2_T}\sum_{l = 0}^{B_1}\sqrt{T}d(B_1 - B)\\
    &+ \frac{C}{S^2_T}\sum_{l = B_1 + 1}^T\left(
    \frac{\sqrt{T}d(B_1 - B)}{(l - B_1)^\alpha} + \frac{T^{7/2}\sqrt{(B_1  - B)}}{(l - 1)^{\alpha - 1}} + \frac{T^{5/2}\sqrt{(B_1  - B)}}{B^\alpha}
    \right)\\
    &\leq \frac{C_1}{S^2_T}B_1\sqrt{T}d(B_1  - B) +\frac{C}{S^2_T}\sqrt{T}d(B_1 - B)\sum_{l = B_1 + 1}^\infty\frac{1}{(l - B_1)^\alpha}\\
    &+ \frac{C}{S^2_T}T^{7/2}\sqrt{(B_1 - B)}\sum_{l = B_1 + 1}^T\frac{1}{(l - 1)^{\alpha - 1}} + \frac{CT^{7/2}\sqrt{(B_1 - B)}}{S^2_TB^\alpha}\\
    &\leq \frac{C_2 B_1}{\sqrt{T}} +\frac{C_2}{\sqrt{T}} + \frac{C_2T^{3/2}}{B_1^{\alpha - 3/2}} + \frac{C_2T^{3/2}}{B^\alpha\sqrt{(B_1-  B)}}. 
    \end{aligned}
\end{equation}
Since $B_1\asymp T^{\kappa_2}$ with $\kappa_2 > \frac{4}{2\alpha - 3},$ we have 
\begin{align*}
    \frac{T^{3/2}}{B_1^{\alpha - 3/2}}\leq \frac{CB_1}{\sqrt{T}}\quad\text{and}\quad \frac{C_2T^{3/2}}{B^\alpha\sqrt{(B_1-  B)}}\leq \frac{CB_1}{\sqrt{T}}
\end{align*}
for a constant $C.$ Therefore, we have 
\begin{equation}
    \begin{aligned}
        \frac{1}{S^2_T}\sum_{l = 0}^T\left\Vert
        \sum_{t_1 = B + 1}^{T - q}(\omega_{t_1, l} - \omega_{t_1, l - 1})\omega_{t_1+q, q - 1}
        \right\Vert_{M/4}\leq \frac{CB_1}{\sqrt{T}}.
    \end{aligned}
    \label{eq.fir_part_omega}
\end{equation}

Besides,
\begin{equation}
\begin{aligned}
    &\frac{1}{S^2_T}\left\Vert \sum_{t_1 = B + 1}^{T - q}\varphi_{t_1,T}\omega_{t_1+q, q - 1}\right\Vert_{M/4}\\
    &\leq \frac{1}{S^2_T}\sum_{t_1 = B+1}^{T - q}\left\Vert \varphi_{t_1,T}\right\Vert_{M/2}\left\Vert \omega_{t_1+q, q - 1}\right\Vert_{M/2}\\
    &\leq\frac{C}{\sqrt{d(B_1 - B)}}\left(\frac{\sqrt{d(B_1  -B)}}{(T - B_1)^\alpha} + \frac{CT^{5/2}}{(T-1)^{\alpha - 1}} + \frac{T^{3/2}}{B^\alpha}\right)\\
    &\leq \frac{C_1}{T^\alpha} + \frac{C_1T^2}{T^{\alpha - 1}\sqrt{(B_1-  B)}} + \frac{C_1T}{B^\alpha\sqrt{(B_1 - B)}}.
\end{aligned}
\label{eq.sec_part_omega}
\end{equation}
For any given $l = 0,1,\cdots, T$ any $q = 0,1,\cdots, T-B-1,$ and any $z =  1,2,\cdots, T-q-B,$ define 
\begin{align*}
        N_{q,l,z} = \sum_{t_1 = T-q-z+1}^{T - q}\left(
        \mathbf{E}\left[\iota_{t_1}\varphi_{t_1 + q,q - 1}\mid\mathcal{F}_{t_1 + q, l}\right] - \mathbf{E}\left[\iota_{t_1}\varphi_{t_1 + q,q - 1}\mid\mathcal{F}_{t_1 + q, l - 1}\right]\right)
\end{align*}
and $\mathcal{N}_{q,l,z}$ the $\sigma$-field generated by $e_T,e_{T-1},\cdots,e_{T-z+1-l},$ then $N_{q,l,z}$ is measurable in $\mathcal{N}_{q,l,z},$ $\mathcal{N}_{q,l,z}\subset \mathcal{N}_{q,l,z+1},$ and 
\begin{align*}
    &\mathbf{E}\left[
    N_{q,l,z+1} - N_{q,l,z} \mid \mathcal{N}_{q,l,z}
    \right]\\
    & = \mathbf{E}\left[\left(\mathbf{E}\left[\iota_{T-q-z}\varphi_{T-z,q - 1}\mid\mathcal{F}_{T-z, l}\right] - \mathbf{E}\left[\iota_{T-q-z}\varphi_{T-z,q - 1}\mid\mathcal{F}_{T-z, l - 1}\right]\right)\mid \mathcal{N}_{q,l,z}\right]\\
     & = \mathbf{E}\left[\iota_{T-q-z}\varphi_{T-z,q - 1}\mid\mathcal{F}_{T-z, l - 1}\right] - \mathbf{E}\left[\iota_{T-q-z}\varphi_{T-z,q - 1}\mid \mathcal{F}_{T-z, l - 1}\right] =0,
\end{align*}
so $N_{q,l,z}$ forms a martingale, and from Theorem 1.1 of \cite{MR0400380}, 
\begin{align*}
    &\left\Vert
    \sum_{t_1 = B + 1}^{T - q}\left(
        \mathbf{E}\left[\iota_{t_1}\varphi_{t_1 + q,q - 1}\mid\mathcal{F}_{t_1 + q, l}\right] - \mathbf{E}\left[\iota_{t_1}\varphi_{t_1 + q,q - 1}\mid\mathcal{F}_{t_1 + q, l - 1}\right]\right)
    \right\Vert_{M/4}\\
    &\leq C\sqrt{ \sum_{t_1 = B + 1}^{T - q}\left\Vert \mathbf{E}\left[\iota_{t_1}\varphi_{t_1 + q,q - 1}\mid\mathcal{F}_{t_1 + q, l}\right] - \mathbf{E}\left[\iota_{t_1}\varphi_{t_1 + q,q - 1}\mid\mathcal{F}_{t_1 + q, l - 1}\right]\right\Vert^2_{M/4}}.
\end{align*}
If $l\leq q,$ then 
\begin{align*}
    &\left\Vert \mathbf{E}\left[\iota_{t_1}\varphi_{t_1 + q,q - 1}\mid\mathcal{F}_{t_1 + q, l}\right] - \mathbf{E}\left[\iota_{t_1}\varphi_{t_1 + q,q - 1}\mid\mathcal{F}_{t_1 + q, l - 1}\right]\right\Vert_{M/4}\\
    &\leq \left\Vert\mathbf{E}\left[\iota_{t_1}\varphi_{t_1 + q,q - 1}\mid\mathcal{F}_{t_1 + q, l}\right]\right\Vert_{M/4} + \left\Vert \mathbf{E}\left[\iota_{t_1}\varphi_{t_1 + q,q - 1}\mid\mathcal{F}_{t_1 + q, l}\right]\right\Vert_{M/4}\\
    &\leq 2\left\Vert \iota_{t_1}\right\Vert_{M/2}\left\Vert \varphi_{t_1 + q,q - 1}\right\Vert_{M/2}\\
    &\leq C\sqrt{d(B_1 - B)}\left\Vert  \varphi_{t_1 + q,q - 1}\right\Vert_{M/2},
\end{align*}
making 
\begin{align*}
    &\left\Vert
    \sum_{t_1 = B + 1}^{T - q}\left(
        \mathbf{E}\left[\iota_{t_1}\varphi_{t_1 + q,q - 1}\mid\mathcal{F}_{t_1 + q, l}\right] - \mathbf{E}\left[\iota_{t_1}\varphi_{t_1 + q,q - 1}\mid\mathcal{F}_{t_1 + q, l - 1}\right]\right)
    \right\Vert_{M/4}\\
    &\leq C\sqrt{d(B_1 - B)}\sqrt{\sum_{t_1 = B + 1}^{T - q}\left\Vert  \varphi_{t_1 + q,q - 1}\right\Vert_{M/2}^2}.
\end{align*}
Otherwise if $l\geq q + 1,$ we notice that 
\begin{align*}
    \varphi_{t_1 + q,q - 1} & = \iota_{t_1 + q} - \omega_{t_1 + q, q - 1}\\
    &= \iota_{t_1 + q} - \omega_{t_1 + q, l - 1} + 
    \omega_{t_1 + q, l - 1} - 
    \omega_{t_1 + q, q - 1}= \varphi_{t_1 + q,l - 1} + \left(\omega_{t_1 + q, l - 1} - 
    \omega_{t_1 + q, q - 1}\right).
\end{align*}
Since $q\leq  l - 1,$ $\omega_{t_1 + q, l - 1} - \omega_{t_1 + q, q - 1}$ is measurable in $\mathcal{F}_{t_1 + q, l - 1},$ and 
\begin{align*}
    &\left\Vert \mathbf{E}\left[\iota_{t_1}\varphi_{t_1 + q,q - 1}\mid\mathcal{F}_{t_1 + q, l}\right] - \mathbf{E}\left[\iota_{t_1}\varphi_{t_1 + q,q - 1}\mid\mathcal{F}_{t_1 + q, l - 1}\right]\right\Vert_{M/4}\\
    &\leq \left\Vert \mathbf{E}\left[\iota_{t_1}\varphi_{t_1 + q,l - 1}\mid\mathcal{F}_{t_1 + q, l}\right] - \mathbf{E}\left[\iota_{t_1}\varphi_{t_1 + q,l - 1}\mid\mathcal{F}_{t_1 + q, l - 1}\right]\right\Vert_{M/4}\\
    &+ \left\Vert\left(\omega_{t_1 + q, l - 1} - 
    \omega_{t_1 + q, q - 1}\right)\left(\mathbf{E}\left[\iota_{t_1}\mid\mathcal{F}_{t_1, l - q}\right] - \mathbf{E}\left[\iota_{t_1}\mid\mathcal{F}_{t_1, l - 1 - q}\right]\right)\right\Vert_{M/4}\\
    &\leq 2\left\Vert \iota_{t_1}\right\Vert_{M/2}\left\Vert \varphi_{t_1 + q,l - 1}\right\Vert_{M/2}\\
    & +\left\Vert \omega_{t_1 + q, l - 1} - 
    \omega_{t_1 + q, q - 1}\right\Vert_{M/2}\left\Vert\omega_{t_1, l - q} - \omega_{t_1, l - q - 1}\right\Vert_{M/2},
\end{align*}
making 
\begin{align*}
    &\left\Vert
    \sum_{t_1 = B + 1}^{T - q}\left(
        \mathbf{E}\left[\iota_{t_1}\varphi_{t_1 + q,q - 1}\mid\mathcal{F}_{t_1 + q, l}\right] - \mathbf{E}\left[\iota_{t_1}\varphi_{t_1 + q,q - 1}\mid\mathcal{F}_{t_1 + q, l - 1}\right]\right)
    \right\Vert_{M/4}\\
    &\leq C\sqrt{d(B_1 - B)}\sqrt{\sum_{t_1 = B+1}^{T-q}\left\Vert \varphi_{t_1 + q,l - 1}\right\Vert^2_{M/2}}\\
    & + C\sqrt{\sum_{t_1 = B+1}^{T-q}\left\Vert \omega_{t_1 + q, l - 1} - 
    \omega_{t_1 + q, q - 1}\right\Vert_{M/2}^2\left\Vert\omega_{t_1, l - q} - \omega_{t_1, l - q - 1}\right\Vert_{M/2}^2}.
\end{align*}
If $l \geq q + 1 + B_1,$ then 
\begin{align*}
    \left\Vert\omega_{t_1, l - q} - \omega_{t_1, l - q - 1}\right\Vert_{M/2} &= \left\Vert\varphi_{t_1, l - q} - \varphi_{t_1, l - q - 1}\right\Vert_{M/2}\\
    &\leq \left\Vert \varphi_{t_1, l - q}\right\Vert_{M/2} + \left\Vert \varphi_{t_1, l - q - 1}\right\Vert_{M/2}\\
    &\leq \frac{C\sqrt{d(B_1 - B)}}{(l - q - B_1)^\alpha} + \frac{CT^{5/2}}{(l - q - 1)^{\alpha - 1}} +\frac{CT^{3/2}}{B^\alpha}.
\end{align*}
Otherwise if $l \leq q + B_1,$
\begin{align*}
    \left\Vert\omega_{t_1, l - q} - \omega_{t_1, l - q - 1}\right\Vert_{M/2}\leq 2\left\Vert\iota_{t_1}\right\Vert_{M/2}\leq C\sqrt{d(B_1 - B)}.
\end{align*}
Therefore,
\begin{align*}
    &\frac{1}{S^2_T}\sum_{l = 0}^T\left\Vert
        \sum_{t_1 = B + 1}^{T - q}\left(
        \mathbf{E}\left[\iota_{t_1}\varphi_{t_1 + q,q - 1}\mid\mathcal{F}_{t_1 + q, l}\right] - \mathbf{E}\left[\iota_{t_1}\varphi_{t_1 + q,q - 1}\mid\mathcal{F}_{t_1 + q, l - 1}\right]\right)
        \right\Vert_{M/4}\\
    &\leq \frac{C\sqrt{d(B_1 - B)}}{S^2_T}\left(\sum_{l = 0}^q \sqrt{\sum_{t_1 = B + 1}^{T - q}\left\Vert  \varphi_{t_1 + q,q - 1}\right\Vert_{M/2}^2} + \sum_{l = q + 1}^T\sqrt{\sum_{t_1 = B+1}^{T-q}\left\Vert \varphi_{t_1 + q,l - 1}\right\Vert^2_{M/2}}\right)\\
    & +\frac{1}{S^2_T}\sum_{l = q + 1}^T\sqrt{\sum_{t_1 = B+1}^{T-q}\left\Vert \omega_{t_1 + q, l - 1} - 
    \omega_{t_1 + q, q - 1}\right\Vert_{M/2}^2\left\Vert\omega_{t_1, l - q} - \omega_{t_1, l - q - 1}\right\Vert_{M/2}^2}\\
    &\leq \frac{C_1}{T\sqrt{d(B_1 - B)}}\left(\sum_{l = 0}^q \sqrt{\sum_{t_1 = B + 1}^{T - q}\left\Vert  \varphi_{t_1 + q,q - 1}\right\Vert_{M/2}^2} + \sum_{l = q + 1}^T\sqrt{\sum_{t_1 = B+1}^{T-q}\left\Vert \varphi_{t_1 + q,l - 1}\right\Vert^2_{M/2}}\right)\\
    & + \frac{C_1}{T\sqrt{d(B_1 - B)}}\sum_{l = q+1}^{(q + B_1)\wedge T}\sqrt{\sum_{t_1 = B+1}^{T-q}\left\Vert \omega_{t_1 + q, l - 1} - 
    \omega_{t_1 + q, q - 1}\right\Vert_{M/2}^2}\\
    &+ \frac{C_1}{Td(B_1 - B)}\sum_{l = q + B_1 + 1}^T\left(
    \frac{\sqrt{d(B_1 - B)}}{(l - q - B_1)^\alpha} + \frac{T^{5/2}}{(l - q - 1)^{\alpha - 1}} +\frac{T^{3/2}}{B^\alpha}
    \right)\sqrt{\sum_{t_1 = B+1}^{T-q}\left\Vert \omega_{t_1 + q, l - 1} - 
    \omega_{t_1 + q, q - 1}\right\Vert_{M/2}^2}.
\end{align*}
If $q\leq B_1,$ then  
$$\left\Vert  \varphi_{t_1 + q,q - 1}\right\Vert_{M/2}\leq 2\left\Vert\iota_{t_1 + q}\right\Vert_{M/2}\leq C\sqrt{d(B_1 - B)},
$$ 
notice that $\kappa_2 > \frac{4}{2\alpha - 3}$ and $\kappa_1 > 2/\alpha,$ so 
\begin{align*}
    &\frac{1}{T\sqrt{d(B_1 - B)}}\left(\sum_{l = 0}^q \sqrt{\sum_{t_1 = B + 1}^{T - q}\left\Vert  \varphi_{t_1 + q,q - 1}\right\Vert_{M/2}^2} + \sum_{l = q + 1}^T\sqrt{\sum_{t_1 = B+1}^{T-q}\left\Vert \varphi_{t_1 + q,l - 1}\right\Vert^2_{M/2}}\right)\\
    &\leq \frac{C}{T\sqrt{d(B_1 - B)}}\sum_{l = 0}^q \sqrt{T\left(\sqrt{d(B_1 - B)}\right)^2} + \frac{C}{T\sqrt{d(B_1 - B)}}\sum_{l = q + 1}^{B_1}\sqrt{T\left(\sqrt{d(B_1 - B)}\right)^2}\\
    &+ \frac{C}{T\sqrt{d(B_1 - B)}}\sum_{l = B_1 + 1}^T\sqrt{T\left(\frac{\sqrt{d(B_1  -B)}}{(l - B_1)^\alpha} + \frac{CT^{5/2}}{(l - 1)^{\alpha - 1}} + \frac{CT^{3/2}}{B^\alpha}\right)^2}\\
    &\leq \frac{C_1B_1}{\sqrt{T}} + \frac{C_1}{\sqrt{T}}\sum_{l = B_1 + 1}^T\frac{1}{(l - B_1)^\alpha} + \frac{C_1T^2}{\sqrt{d(B_1 - B)}}\sum_{l = B_1 + 1}^T\frac{1}{(l - 1)^{\alpha - 1}} + \frac{C_1T^2}{B^\alpha\sqrt{d(B_1 - B)}}\\
    &\leq \frac{C_2B_1}{\sqrt{T}} + \frac{C_2T^2}{\sqrt{d(B_1 - B)}B_1^{\alpha - 2}} + \frac{C_1T^2}{B^\alpha\sqrt{d(B_1 - B)}}\leq \frac{C_3B_1}{\sqrt{T}}.
\end{align*}
If $q \geq B_1 + 1, $ then 
\begin{align*}
    \left\Vert  \varphi_{t_1 + q,q - 1}\right\Vert_{M/2}\leq \frac{C\sqrt{d(B_1 - B)}}{(q - B_1)^\alpha} + \frac{CT^{5/2}}{(q - 1)^{\alpha - 1}} + \frac{CT^{3/2}}{B^\alpha},
\end{align*}
and 
\begin{align*}
    &\frac{1}{T\sqrt{d(B_1 - B)}}\left(\sum_{l = 0}^q \sqrt{\sum_{t_1 = B + 1}^{T - q}\left\Vert  \varphi_{t_1 + q,q - 1}\right\Vert_{M/2}^2} + \sum_{l = q + 1}^T\sqrt{\sum_{t_1 = B+1}^{T-q}\left\Vert \varphi_{t_1 + q,l - 1}\right\Vert^2_{M/2}}\right)\\
    &\leq \frac{C q}{\sqrt{Td(B_1 - B)}}\left(
    \frac{\sqrt{d(B_1 - B)}}{(q - B_1)^\alpha} + \frac{T^{5/2}}{(q - 1)^{\alpha -  1}} + \frac{T^{3/2}}{B^\alpha}
    \right)\\
    &+ \frac{C}{\sqrt{Td(B_1 - B)}}\sum_{l = q + 1}^T\left(\frac{\sqrt{d(B_1 - B)}}{(l - B_1)^\alpha} + \frac{T^{5/2}}{(l - 1)^{\alpha - 1}} + \frac{T^{3/2}}{B^\alpha}\right)\\
    &\leq \frac{Cq}{\sqrt{T}(q - B_1)^\alpha} + \frac{CT^2q}{\sqrt{d(B_1-  B)}(q - 1)^{\alpha - 1}} + \frac{CqT}{\sqrt{d(B_1 - B)}B^\alpha}\\
    & + \frac{C_1}{\sqrt{T}(q + 1 - B_1)^{\alpha - 1}}+ \frac{C_1T^2}{\sqrt{d(B_1 - B)}q^{\alpha - 2}} +  \frac{C_1T^2}{\sqrt{d(B_1 - B)}B^\alpha}.
\end{align*}
On the other hand, since 
\begin{align*}
\left\Vert \omega_{t_1 + q, l - 1} - 
    \omega_{t_1 + q, q - 1}\right\Vert_{M/2} &= 
    \left\Vert
    \varphi_{t_1 + q, q - 1} - \varphi_{t_1 + q, l - 1}
    \right\Vert_{M/2}\\
    &\leq \left\Vert \varphi_{t_1 + q, q - 1}\right\Vert_{M/2} 
    + \left\Vert \varphi_{t_1 + q, l - 1}\right\Vert_{M/2}, 
\end{align*}
and 
\begin{align*}
    \left\Vert \omega_{t_1 + q, l - 1} - 
    \omega_{t_1 + q, q - 1}\right\Vert_{M/2} &\leq \left\Vert \omega_{t_1 + q, l - 1} \right\Vert_{M/2}
    + \left\Vert
    \omega_{t_1 + q, q - 1}\right\Vert_{M/2}\\
    &\leq 2\left\Vert\iota_{t_1 + q}\right\Vert_{M/2}\leq C\sqrt{d(B_1 - B)}.
\end{align*}
If $q\leq B_1,$ then 
\begin{align*}
    &\frac{1}{T\sqrt{d(B_1 - B)}}\sum_{l = q+1}^{(q + B_1)\wedge T}\sqrt{\sum_{t_1 = B+1}^{T-q}\left\Vert \omega_{t_1 + q, l - 1} - 
    \omega_{t_1 + q, q - 1}\right\Vert_{M/2}^2}\\
    &\leq \frac{CB_1}{\sqrt{Td(B_1 - B)}}\sqrt{d(B_1 - B)}\leq \frac{CB_1}{\sqrt{T}},
\end{align*}
and
\begin{align*}
    &\frac{1}{Td(B_1 - B)}\sum_{l = q + B_1 + 1}^T\left(
    \frac{\sqrt{d(B_1 - B)}}{(l - q - B_1)^\alpha} + \frac{T^{5/2}}{(l - q - 1)^{\alpha - 1}} +\frac{T^{3/2}}{B^\alpha}
    \right)\sqrt{\sum_{t_1 = B+1}^{T-q}\left\Vert \omega_{t_1 + q, l - 1} - 
    \omega_{t_1 + q, q - 1}\right\Vert_{M/2}^2}\\
    &\leq \frac{C}{\sqrt{Td(B_1 - B)}}\sum_{l = q + B_1 + 1}^T\left(
    \frac{\sqrt{d(B_1 - B)}}{(l - q - B_1)^\alpha} + \frac{T^{5/2}}{(l - q - 1)^{\alpha - 1}} +\frac{T^{3/2}}{B^\alpha}
    \right)\\
    &\leq \frac{C_1}{\sqrt{T}}+ \frac{C_1T^2}{\sqrt{d(B_1 - B)}B_1^{\alpha - 2}} + \frac{C_1T^2}{B^\alpha\sqrt{d(B_1 - B)}}.
\end{align*}
If $q\geq B_1 +1,$ then 
\begin{align*}
    &\frac{1}{T\sqrt{d(B_1 - B)}}\sum_{l = q+1}^{(q + B_1)\wedge T}\sqrt{\sum_{t_1 = B+1}^{T-q}\left\Vert \omega_{t_1 + q, l - 1} - 
    \omega_{t_1 + q, q - 1}\right\Vert_{M/2}^2}\\
    &\leq \frac{CB_1}{\sqrt{Td(B_1 - B)}}\left(\frac{\sqrt{d(B_1 - B)}}{(q - B_1)^\alpha} + \frac{T^{5/2}}{(q - 1)^{\alpha - 1}} + \frac{T^{3/2}}{B^\alpha}\right)\\
    &\leq \frac{CB_1}{\sqrt{T}(q  - B_1)^\alpha} +\frac{CB_1T^2}{\sqrt{d(B_1 - B)}(q - 1)^{\alpha - 1}} 
    + \frac{CB_1T}{\sqrt{d(B_1 - B)}B^\alpha}. 
\end{align*}
Since $\kappa_2 > \frac{4}{2\alpha - 3},$ we also have 
\begin{align*}
    &\frac{1}{Td(B_1 - B)}\sum_{l = q + B_1 + 1}^T\left(
    \frac{\sqrt{d(B_1 - B)}}{(l - q - B_1)^\alpha} + \frac{T^{5/2}}{(l - q - 1)^{\alpha - 1}} +\frac{T^{3/2}}{B^\alpha}
    \right)\sqrt{\sum_{t_1 = B+1}^{T-q}\left\Vert \omega_{t_1 + q, l - 1} - 
    \omega_{t_1 + q, q - 1}\right\Vert_{M/2}^2}\\
    &\leq \frac{C}{\sqrt{T}d(B_1 - B)}\left(\sum_{l = q + B_1 + 1}^T\left(
    \frac{\sqrt{d(B_1 - B)}}{(l - q - B_1)^\alpha} + \frac{T^{5/2}}{(l - q - 1)^{\alpha - 1}} +\frac{T^{3/2}}{B^\alpha}
    \right)\right)\\
    &\times\left(\frac{\sqrt{d(B_1 - B)}}{(q - B_1)^\alpha} + \frac{T^{5/2}}{(q - 1)^{\alpha - 1}} + \frac{T^{3/2}}{B^\alpha}\right)\\
    &\leq \left(\frac{C}{\sqrt{Td(B_1 - B)}(q - B_1)^\alpha} + \frac{CT^2}{d(B_1 -B)(q - 1)^{\alpha - 1}} + \frac{CT}{dB^\alpha(B_1 -  B)}\right)\\
    &\times\left(\sqrt{d(B_1 - B)} +\frac{T^{5/2}}{B_1^{\alpha - 2}} + \frac{T^{5/2}}{B^\alpha}\right)\\
    &\leq \frac{C_1}{\sqrt{T}(q - B_1)^\alpha} + \frac{CT^2}{\sqrt{d(B_1 -B)}(q - 1)^{\alpha - 1}} + \frac{CT}{B^\alpha\sqrt{d(B_1 - B)}}.
\end{align*}
We remark that the last inequality holds because 
$$
\frac{T^{5/2}}{B_1^{\alpha - 2}}\leq C\sqrt{d(B_1 - B)}\quad\text{and}\quad \frac{T^{5/2}}{B^\alpha}\leq C\sqrt{d(B_1 - B)}
$$
for a constant $C.$

From these observations, if $q\leq B_1,$ for $\kappa_2 > \frac{4}{2\alpha - 3},$ we have  
\begin{equation}
    \begin{aligned}
        &\frac{1}{S^2_T}\sum_{l = 0}^T\left\Vert
        \sum_{t_1 = B + 1}^{T - q}\left(
        \mathbf{E}\left[\iota_{t_1}\varphi_{t_1 + q,q - 1}\mid\mathcal{F}_{t_1 + q, l}\right] - \mathbf{E}\left[\iota_{t_1}\varphi_{t_1 + q,q - 1}\mid\mathcal{F}_{t_1 + q, l - 1}\right]\right)
        \right\Vert_{M/4}\\
       &\leq \frac{CB_1}{\sqrt{T}} + \frac{CT^2}{\sqrt{d(B_1 - B)}B_1^{\alpha - 2}} + \frac{CT^2}{B^\alpha\sqrt{d(B_1 - B)}}\leq \frac{C_1B_1}{\sqrt{T}}.
    \end{aligned}
    \label{eq.thi_part_omega}
\end{equation}
On the other hand, if $B_1 + 1\leq q \leq T - B - 1,$ we have 
\begin{equation}
    \begin{aligned}
        &\frac{1}{S^2_T}\sum_{l = 0}^T\left\Vert
        \sum_{t_1 = B + 1}^{T - q}\left(
        \mathbf{E}\left[\iota_{t_1}\varphi_{t_1 + q,q - 1}\mid\mathcal{F}_{t_1 + q, l}\right] - \mathbf{E}\left[\iota_{t_1}\varphi_{t_1 + q,q - 1}\mid\mathcal{F}_{t_1 + q, l - 1}\right]\right)
        \right\Vert_{M/4}\\
        &\leq \frac{CB_1}{\sqrt{T}(q - B_1)^\alpha} +\frac{C}{\sqrt{T}(q - B_1)^{\alpha - 1}} + \frac{CT^2}{\sqrt{d(B_1 - B)}(q-1)^{\alpha - 2}} \\
        &+ \frac{CB_1T^2}{\sqrt{d(B_1 - B)}(q-1)^{\alpha - 1}} + \frac{CT^2}{\sqrt{d(B_1 -  B)}B^\alpha}.
    \end{aligned}
    \label{eq.thi_part_omega_greater}
\end{equation}
Finally,  since 
\begin{align*}
    \varphi_{t_1 + q,q - 1} = \iota_{t_1+q} - \omega_{t_1 + q, q - 1} =  \varphi_{t_1 + q, T} + \omega_{t_1 + q, T} - \omega_{t_1 + q, q - 1},
\end{align*}
and $\omega_{t_1 + q, T} - \omega_{t_1 + q, q - 1}$ is measurable in $\mathcal{F}_{t_1 + q, T}.$ Therefore, 
\begin{align*}
    &\frac{1}{S^2_T}\left\Vert
        \sum_{t_1 = B + 1}^{T - q}\left(\iota_{t_1}\varphi_{t_1 + q,q - 1} - \mathbf{E}\left[\iota_{t_1}\varphi_{t_1 + q,q - 1}\mid\mathcal{F}_{t_1 + q, T}\right]\right)
        \right\Vert_{M/4}\\
    &\leq \frac{1}{S^2_T}\left\Vert
        \sum_{t_1 = B + 1}^{T - q}\left(\iota_{t_1}\varphi_{t_1 + q,T} - \mathbf{E}\left[\iota_{t_1}\varphi_{t_1 + q,T}\mid\mathcal{F}_{t_1 + q, T}\right]\right)
        \right\Vert_{M/4}\\
    & + \frac{1}{S^2_T}\left\Vert
    \sum_{t_1 = B + 1}^{T - q}\left(\omega_{t_1 + q, T} - \omega_{t_1 + q, q - 1}\right)
    \left(\iota_{t_1} - \mathbf{E}\left[\iota_{t_1}\mid \mathcal{F}_{t_1, T - q}\right]\right)
    \right\Vert_{M/4}\\
    &\leq \frac{2}{S^2_T}\sum_{t_1 = B + 1}^{T - q}\left\Vert \iota_{t_1}\varphi_{t_1 + q,T}\right\Vert_{M/4}
    + \frac{1}{S^2_T}\sum_{t_1 = B + 1}^{T - q}\left\Vert
    \left(\omega_{t_1 + q, T} - \omega_{t_1 + q, q - 1}\right)
    \left(\iota_{t_1} - \mathbf{E}\left[\iota_{t_1}\mid \mathcal{F}_{t_1, T - q}\right]\right)
    \right\Vert_{M/4}\\
    &\leq \frac{2}{S^2_T}\sum_{t_1 = B + 1}^{T - q}\left\Vert \iota_{t_1}\right\Vert_{M/2}\left\Vert\varphi_{t_1 + q,T}\right\Vert_{M/2} + \frac{1}{S^2_T}\sum_{t_1 = B + 1}^{T - q}\left\Vert \omega_{t_1 + q, T} - \omega_{t_1 + q, q - 1}\right\Vert_{M/2}\left\Vert 
    \varphi_{t_1,T-q}\right\Vert_{M/2}.
\end{align*}
Notice that
\begin{align*}
    &\frac{1}{S^2_T}\sum_{t_1 = B + 1}^{T - q}\left\Vert \iota_{t_1}\right\Vert_{M/2}\left\Vert\varphi_{t_1 + q,T}\right\Vert_{M/2}\\
    &\leq \frac{CT}{Td(B_1  -B)}\sqrt{d(B_1 - B)}\left(
    \frac{\sqrt{d(B_1  -B)}}{(T + 1 - B_1)^\alpha} + \frac{T^{5/2}}{T^{\alpha - 1}} + \frac{T^{3/2}}{B^\alpha}
    \right)\\
    & = \frac{C}{(T + 1 - B_1)^\alpha } + \frac{CT^{5/2}}{T^{\alpha - 1}\sqrt{d(B_1 - B)}} + \frac{CT^{3/2}}{B^\alpha\sqrt{d(B_1 - B)}}.
\end{align*}
If $q\leq T / 2,$ then $T-q\geq T/2 >B_1$  for sufficiently large $T,$ and 
\begin{align*}
    &\frac{1}{S^2_T}\sum_{t_1 = B + 1}^{T - q}\left\Vert \omega_{t_1 + q, T} - \omega_{t_1 + q, q - 1}\right\Vert_{M/2}\left\Vert \varphi_{t_1,T-q}\right\Vert_{M/2}\\
    &\leq \frac{CT}{Td(B_1 - B)}\sqrt{d(B_1 -B)}\left(
    \frac{\sqrt{d(B_1  -B)}}{(T-q + 1 - B_1)^\alpha} + \frac{T^{5/2}}{(T-q)^{\alpha - 1}} + \frac{T^{3/2}}{B^\alpha}\right)\\
    &\leq \frac{C_1}{T^\alpha} + \frac{C_1T^{5/2}}{T^{\alpha - 1}\sqrt{d(B_1  -B)}} + \frac{CT^{3/2}}{B^\alpha\sqrt{d(B_1 - B)}}.
\end{align*}
Otherwise if $q > T/2,$ then 
\begin{align*}
    &\frac{1}{S^2_T}\sum_{t_1 = B + 1}^{T - q}\left\Vert \omega_{t_1 + q, T} - \omega_{t_1 + q, q - 1}\right\Vert_{M/2}\left\Vert \varphi_{t_1,T-q}\right\Vert_{M/2}\\
    &\leq \frac{CT}{Td(B_1 - B)}\left(\frac{\sqrt{d(B_1  -B)}}{(q - B_1)^\alpha} + \frac{T^{5/2}}{(q-1)^{\alpha - 1}} + \frac{T^{3/2}}{B^\alpha}\right)\sqrt{d(B_1 - B)}\\
    &= \frac{C}{(q - B_1)^\alpha} + \frac{CT^{5/2}}{\sqrt{d(B_1 - B)}(q-1)^{\alpha - 1}} +\frac{CT^{3/2}}{B^\alpha\sqrt{d(B_1  -B)}}\\
    &\leq \frac{C_1}{T^\alpha} + \frac{C_1T^{5/2}}{\sqrt{d(B_1 - B)}T^{\alpha - 1}} + \frac{CT^{3/2}}{B^\alpha\sqrt{d(B_1  -B)}}.
\end{align*}
Therefore, we have for any $q,$
\begin{equation}
    \begin{aligned}
        &\frac{1}{S^2_T}\left\Vert
        \sum_{t_1 = B + 1}^{T - q}\left(\iota_{t_1}\varphi_{t_1 + q,q - 1} - \mathbf{E}\left[\iota_{t_1}\varphi_{t_1 + q,q - 1}\mid\mathcal{F}_{t_1 + q, T}\right]\right)
        \right\Vert_{M/4}\\
        &\leq \frac{C}{(T + 1 - B_1)^\alpha } + \frac{CT^{5/2}}{T^{\alpha - 1}\sqrt{d(B_1 - B)}} + \frac{CT^{3/2}}{B^\alpha\sqrt{d(B_1 - B)}}\\
        & + \frac{C}{T^\alpha} + \frac{CT^{5/2}}{T^{\alpha - 1}\sqrt{d(B_1  -B)}} + \frac{CT^{3/2}}{B^\alpha\sqrt{d(B_1 - B)}}\\
        &\leq \frac{C_1}{T^\alpha} +\frac{C_1T^{5/2}}{T^{\alpha - 1}\sqrt{d(B_1  -B)}} + \frac{C_1T^{3/2}}{B^\alpha\sqrt{d(B_1 - B)}}.
    \end{aligned}
    \label{eq.fourth_part_omega}
\end{equation}
If $q\leq B_1,$ from \eqref{eq.fir_part_omega}, \eqref{eq.sec_part_omega}, \eqref{eq.thi_part_omega},  \eqref{eq.fourth_part_omega},
\begin{equation}
    \begin{aligned}
        &\frac{1}{S^2_T}\left\Vert
        \sum_{t_1 = B + 1}^{T - q}\left(\iota_{t_1}\iota_{t_1+q} - \mathbf{E}\left[\iota_{t_1}\iota_{t_1+q}\right]\right)
        \right\Vert_{M/4}
        &\leq \frac{CB_1}{\sqrt{T}}.
    \end{aligned}
\end{equation}
On the other hand, if $B_1 + 1\leq q,$ from \eqref{eq.fir_part_omega}, \eqref{eq.sec_part_omega}, \eqref{eq.thi_part_omega_greater}, \eqref{eq.fourth_part_omega},
\begin{equation}
    \begin{aligned}
        &\frac{1}{S^2_T}\left\Vert
        \sum_{t_1 = B + 1}^{T - q}\left(\iota_{t_1}\iota_{t_1+q} - \mathbf{E}\left[\iota_{t_1}\iota_{t_1+q}\right]\right)
        \right\Vert_{M/4}\\
        &\leq \frac{C B_1}{\sqrt{T}} + \frac{CB_1}{\sqrt{T}(q - B_1)^\alpha} +\frac{C}{\sqrt{T}(q - B_1)^{\alpha - 1}} \\
        &+ \frac{CT^2}{\sqrt{d(B_1 - B)}(q-1)^{\alpha - 2}}
        + \frac{CB_1T^2}{\sqrt{d(B_1 - B)}(q-1)^{\alpha - 1}}\\
        &\leq \frac{C_1 B_1}{\sqrt{T}}.
    \end{aligned}
\end{equation}
Therefore, we have 
\begin{equation}
    \begin{aligned}
        &\left\Vert\frac{1}{S^2_T}\sum_{t_1 = B + 1}^T\sum_{t_2 = B+1}^T \mathcal{K}\left(\frac{t_1 - t_2}{H}\right)\left(\iota_{t_1}\iota_{t_2} - \mathbf{E}\left[\iota_{t_1}\iota_{t_2}\right]\right)\right\Vert_{M/4}\\
        &\leq \frac{CB_1}{\sqrt{T}}\sum_{q = 0}^{T-B-1} \mathcal{K}\left(\frac{q}{H}\right) 
        \leq \frac{C_1B_1H}{\sqrt{T}} .
    \end{aligned}
    \label{eq.var_diff}
\end{equation}
From \eqref{eq.moment_iota} and \eqref{eq.bias_theta}, 
\begin{align*}
    &\frac{1}{S_T^2}\left\Vert\sum_{t_1 = B + 1}^T\sum_{t_2 = B+1}^T \mathcal{K}\left(\frac{t_1 - t_2}{H}\right)\iota_{t_1}\mathbf{E}\left[\vartheta_{t_2}\right]\right\Vert_{M/2}\\
    &\leq \frac{1}{S_T^2}\sum_{t_1 = B + 1}^T\sum_{t_2 = B+1}^T \mathcal{K}\left(\frac{t_1 - t_2}{H}\right)\left\Vert \iota_{t_1}\right\Vert_{M/2}\left\vert \mathbf{E}\left[\vartheta_{t_2}\right]\right\vert\\
    &\leq \frac{C}{Td(B_1 - B)}\sum_{t_1 = B + 1}^T\sum_{t_2 = B+1}^T \mathcal{K}\left(\frac{t_1 - t_2}{H}\right)\sqrt{d(B_1 - B)}\frac{d}{B^{\alpha - 1}}\\
    &\leq \frac{C_1TH}{\sqrt{d(B_1 - B)}B^{\alpha - 1}},
\end{align*}
and 
\begin{align*}
    &\frac{1}{S^2_T}\left\Vert\sum_{t_1 = B + 1}^T\sum_{t_2 = B+1}^T \mathcal{K}\left(\frac{t_1 - t_2}{H}\right)\iota_{t_2}\mathbf{E}\left[\vartheta_{t_1}\right]\right\Vert_{M/2}\\
    &\leq \frac{1}{S^2_T}\sum_{t_1 = B + 1}^T\sum_{t_2 = B+1}^T\mathcal{K}\left(\frac{t_1 - t_2}{H}\right)\left\Vert \iota_{t_2}\right\Vert_{M/2}\left\vert \mathbf{E}\left[\vartheta_{t_1}\right]\right\vert\\
&\leq \frac{CTH}{\sqrt{d(B_1 - B)}B^{\alpha - 1}}.
\end{align*}
Therefore, we have 
\begin{equation}
    \begin{aligned}
        &\left\Vert\frac{1}{S^2_T}\sum_{t_1 = B + 1}^T\sum_{t_2 = B+1}^T \mathcal{K}\left(\frac{t_1 - t_2}{H}\right)\left(\vartheta_{t_1}\vartheta_{t_2} - \mathbf{E}\left[\vartheta_{t_1}\vartheta_{t_2}\right]\right)\right\Vert_{M/4}\\
        &\leq \frac{CB_1H}{\sqrt{T}}  + \frac{CTH}{\sqrt{d(B_1 - B)}B^{\alpha - 1}}\leq \frac{CB_1H}{\sqrt{T}}.
    \end{aligned}
    \label{eq.moment_theta_2}
\end{equation}
From \eqref{eq.decomposition_var_Q}, \eqref{eq.abs_Theta}, \eqref{eq.bias_diff}, and \eqref{eq.moment_theta_2}, since $\kappa_2 < 1/6$ and $\frac{1}{6} > \kappa_1> \frac{2}{\alpha},$
\begin{equation}
    \begin{aligned}
        &\left\Vert\frac{1}{S^2_T}\sum_{t_1 = B + 1}^T\sum_{t_2 = B+1}^T \mathcal{K}\left(\frac{t_1 - t_2}{H}\right)\vartheta_{t_1}\vartheta_{t_2}  - \mathrm{Var}\left(\frac{Q}{S_T}\right)\right\Vert_{M/4}\\
        &\leq
        \left\Vert\frac{1}{S^2_T}\sum_{t_1 = B + 1}^T\sum_{t_2 = B+1}^T \mathcal{K}\left(\frac{t_1 - t_2}{H}\right)\vartheta_{t_1}\vartheta_{t_2}  - \Gamma\right\Vert_{M/4} + \left\vert\Theta\right\vert
        \\
        &\leq \frac{CB_1H}{\sqrt{T}} + \frac{C}{T^{\frac{1}{3} - 2\kappa_2}\log(T)} + \frac{CT^{5/3}}{\log(T)T^{\kappa_2(\alpha - \frac{5}{2})}} + \frac{C}{T^{\kappa_2 / 2}}\\
        &  + \frac{CT^2}{(B_1 - B)B^{2\alpha - 2}}\\
        &\leq \frac{C_1B_1H}{\sqrt{T}}
        + \frac{C}{T^{\frac{1}{3} - 2\kappa_2}\log(T)}  + \frac{C}{T^{\kappa_2 / 2}}.
    \end{aligned}
\end{equation}
Since 
\begin{align*}
    &\frac{HB_1}{\sqrt{T}} = O\left(\frac{\log(T)}{T^{1/6 - \kappa_2}}\right),
\end{align*}
we have 
\begin{equation}
    \begin{aligned}
        &\left\Vert\frac{1}{S^2_T}\sum_{t_1 = B + 1}^T\sum_{t_2 = B+1}^T \mathcal{K}\left(\frac{t_1 - t_2}{H}\right)\vartheta_{t_1}\vartheta_{t_2}  - \mathrm{Var}\left(\frac{Q}{S_T}\right)\right\Vert_{M/4}\\
        &= O\left(\frac{\log(T)}{T^{1/6 - \kappa_2}} + \frac{1}{T^{\kappa_2 / 2}}\right),
    \end{aligned}
\end{equation}
which proves \eqref{eq.var_esti}
\end{proof}

\section{Proofs of theoretical results in Section \ref{section.Asymptotic_testing}}
In the proofs of the results in Section \ref{section.Asymptotic_testing}, we introduce additional notation for the filters: For any $k = 1,2,\cdots,K,$ $t\in\mathbf{Z},$ and $\ell= 0,1,2,\cdots,$ define $\mathcal{F}_{t,\ell}^{(k)}$ as the $\sigma$-field generated by $e_{t,k}, e_{t-1,k},\cdots,e_{t - \ell,k}.$ This new notation is introduced to reflect the fact that the data in this section are generated from $K$ different populations.

The proof of Theorem \ref{theorem.consistent_ANOVA} leverages Theorem \ref{theorem.consistent_quadratic} to establish both the consistency and the asymptotic distributional results for the test statistics $\widehat{R}.$

\begin{proof}[Proof of Theorem \ref{theorem.consistent_ANOVA}]

\textbf{1. The proof of equation \eqref{eq.consistency_R}.}  From \eqref{eq.def_RK}, 
\begin{align*}
    \widehat{R}  - \frac{1}{\sqrt{d}}\sum_{k = 2}^K\left\vert \boldsymbol{\mu}_k  - \boldsymbol{\mu}_1\right\vert_2^2 = \sum_{k = 2}^K \left(\widehat{R}_k  - \frac{1}{\sqrt{d}}\vert \boldsymbol{\mu}_k  - \boldsymbol{\mu}_1\vert_2^2\right),
\end{align*}
and 
\begin{equation}
\begin{aligned}
    &\widehat{R}_k - \frac{1}{\sqrt{d}}\left\vert
    \boldsymbol{\mu}_k  - \boldsymbol{\mu}_1
    \right\vert_2^2\\
    &= \frac{1}{V_k\sqrt{d}}\sum_{B\leq \vert t_1 - t_2\vert\leq B_1}^{T_k}\boldsymbol{\epsilon}_{t_1,k}^\top \boldsymbol{\epsilon}_{t_2,k} + \frac{2}{V_k\sqrt{d}}\sum_{B\leq \vert t_1 - t_2\vert\leq B_1}^{T_k}\boldsymbol{\mu}_k^\top\boldsymbol{\epsilon}_{t_1,k}\\
    & + \frac{1}{V_1\sqrt{d}}\sum_{B\leq \vert t_1 - t_2\vert\leq B_1}^{T_1}\boldsymbol{\epsilon}_{t_1,1}^\top \boldsymbol{\epsilon}_{t_2,1} + \frac{2}{V_1\sqrt{d}}\sum_{B\leq \vert t_1 - t_2\vert\leq B_1}^{T_1}\boldsymbol{\mu}_1^\top\boldsymbol{\epsilon}_{t_1,1}\\
    & - \frac{2}{T_kT_1\sqrt{d}}\sum_{t_1 = 1}^{T_k}\sum_{t_2 = 1}^{T_1}\boldsymbol{\epsilon}_{t_1,k}^\top\boldsymbol{\epsilon}_{t_2,1} - \frac{2}{T_k\sqrt{d}}\sum_{t_1 = 1}^{T_k}\boldsymbol{\mu}_1^\top\boldsymbol{\epsilon}_{t_1,k} - \frac{2}{T_1\sqrt{d}}\sum_{t_1 = 1}^{T_1}\boldsymbol{\mu}_k^\top\boldsymbol{\epsilon}_{t_1,1}.
\end{aligned}
\label{eq.expand_R_k}
\end{equation}
Notice that 
\begin{equation}
\begin{aligned}
    &\left\Vert
    \frac{1}{V_k\sqrt{d}}\sum_{B\leq \vert t_1 - t_2\vert\leq B_1}^{T_k}\boldsymbol{\epsilon}_{t_1,k}^\top \boldsymbol{\epsilon}_{t_2,k}
    \right\Vert_{M/2}\\
    &\leq \left\Vert
    \frac{1}{V_k\sqrt{d}}\sum_{B\leq \vert t_1 - t_2\vert\leq B_1}^{T_k}\left(\boldsymbol{\epsilon}_{t_1,k}^\top \boldsymbol{\epsilon}_{t_2,k} - \mathbf{E}\left[\boldsymbol{\epsilon}_{t_1,k}^\top \boldsymbol{\epsilon}_{t_2,k}\right]\right)
    \right\Vert_{M/2}\\
    & + \frac{1}{V_k\sqrt{d}}\sum_{B\leq \vert t_1 - t_2\vert\leq B_1}^{T_k}\left\vert \mathbf{E}\left[\boldsymbol{\epsilon}_{t_1,k}^\top \boldsymbol{\epsilon}_{t_2,k}\right]\right\vert.
\end{aligned}
\end{equation}
From Theorem \ref{theorem.consistent_quadratic}, for any $k = 1,\cdots, K,$
\begin{equation}
\begin{aligned}
    &\left\Vert
\frac{1}{V_k\sqrt{d}}\sum_{B\leq \vert t_1 -  t_2\vert\leq B_1}^{T_k}\left(\boldsymbol{\epsilon}_{t_1,k}^\top\boldsymbol{\epsilon}_{t_2,k} - \mathbf{E}\left[\boldsymbol{\epsilon}_{t_1,k}^\top\boldsymbol{\epsilon}_{t_2,k}\right]\right)
    \right\Vert_{M/2}\\
    &\leq \frac{C}{V_k\sqrt{d}}\left(\sqrt{d\times \sum_{B\leq \vert t_1 -  t_2\vert\leq B_1}^{T_k} 1^2} + \frac{T^{5/2}}{B^\alpha} + \frac{T^{3/2}}{B^{\alpha - 2}}\right) = O\left(\frac{1}{\sqrt{V_k}} \right).
\end{aligned}
\label{eq.V_k_sqrt}
\end{equation}
From \eqref{eq.covariance},
\begin{equation}
\begin{aligned}
    \left\vert
    \frac{1}{V_k\sqrt{d}}\sum_{B\leq \vert t_1 -  t_2\vert\leq B_1}^{T_k}\mathbf{E}\left[\boldsymbol{\epsilon}_{t_1,k}^\top\boldsymbol{\epsilon}_{t_2,k}\right]
    \right\vert &\leq \frac{1}{V_k\sqrt{d}}\sum_{B\leq \vert t_1 -  t_2\vert\leq B_1}^{T_k}\sum_{j = 1}^d \left\vert
    \mathbf{E}\left[\boldsymbol{\epsilon}_{t_1,k}^{(j)}\boldsymbol{\epsilon}_{t_2,k}^{(j)}\right]
    \right\vert\\
    &\leq \frac{Cd}{V_k\sqrt{d}}\sum_{B\leq \vert t_1 -  t_2\vert\leq B_1}^{T_k} \frac{1}{(1 + \vert t_1 - t_2\vert)^\alpha}\\
    &\leq \frac{C_1dT_k}{V_k\sqrt{d} B^{\alpha - 1}} = O\left(\frac{B}{B_1 \mathcal{T}_\circ^{3/2}}\right).
\end{aligned}
\label{eq.V_k_bias}
\end{equation}
From Lemma \ref{lemma.linear_form}, 
\begin{align*}
    \left\Vert\frac{1}{V_k\sqrt{d}}\sum_{B\leq \vert t_1 -  t_2\vert\leq B_1}^{T_k}\boldsymbol{\mu}_k^\top\boldsymbol{\epsilon}_{t_1,k} \right\Vert_{M / 2} 
    &= \frac{1}{V_k\sqrt{d}}\left\Vert
    \sum_{t_1 =  1}^{T_k}\boldsymbol{\epsilon}_{t_1,k}^\top\left(\boldsymbol{\mu}_k\sum_{B\leq \vert t_2 -t_1\vert\leq B_1}1\right)
    \right\Vert_{M / 2}\\
    &\leq \frac{C}{V_k\sqrt{d}}\sqrt{\sum_{t_1 =  1}^{T_k}\left\vert\boldsymbol{\mu}_k\right\vert_2^2\times \left(\sum_{B\leq \vert t_2 -t_1\vert\leq B_1}1\right)^2}.
\end{align*}
For any given $t_1 = 1,2,\cdots, T_k,$ 
\begin{align*}
    \sum_{B\leq \vert t_2 -t_1\vert\leq B_1}1 = \sum_{t_2 = t_1 + B}^{(t_1 + B_1)\wedge T_k}1 + \sum_{t_2 = (t_1  -B_1)\vee 1}^{t_1 - B}1\asymp (B_1 - B),
\end{align*}
so 
\begin{equation}
    \left\Vert\frac{1}{V_k\sqrt{d}}\sum_{B\leq \vert t_1 -  t_2\vert\leq B_1}^{T_k}\boldsymbol{\mu}_k^\top\boldsymbol{\epsilon}_{t_1,k}\right\Vert_{M / 2} 
    \leq \frac{C\vert\boldsymbol{\mu}_k\vert_2\sqrt{T_k}\times (B_1  - B)}{V_k\sqrt{d}} = O\left(\frac{1}{\sqrt{\mathcal{T}_\circ}}\right).
\label{eq.V_k_sqrt_d}
\end{equation}
From Lemma \ref{lemma.linear_form}, 
\begin{align*}
    \left\Vert
    \frac{1}{T_k\sqrt{d}}\sum_{t_1 = 1}^{T_k}\boldsymbol{\mu}_1^\top\boldsymbol{\epsilon}_{t_1,k}
    \right\Vert_{M}\leq \frac{C}{T_k\sqrt{d}}\sqrt{\sum_{t_1 = 1}^{T_k}\left\vert\boldsymbol{\mu}_1\right\vert^2_2}\leq \frac{C_1}{\sqrt{\mathcal{T}_\circ}},
\end{align*}
and 
\begin{align*}
    \left\Vert\frac{1}{T_1\sqrt{d}}\sum_{t_1 = 1}^{T_1}\boldsymbol{\mu}_k^\top\boldsymbol{\epsilon}_{t_1,1}\right\Vert_{M}\leq \frac{C}{T_1\sqrt{d}}\sqrt{\sum_{t_1 = 1}^{T_1}\left\vert\boldsymbol{\mu}_k\right\vert^2_2}\leq \frac{C_1}{\sqrt{\mathcal{T}_\circ}}.
\end{align*}
Define $\overline{\boldsymbol{\epsilon}}_{1} =  \frac{1}{T_1}\sum_{t = 1}^{T_1}\boldsymbol{\epsilon}_{t,1},$ then 
\begin{align*}
    \left\Vert
    \frac{1}{T_1T_k\sqrt{d}}\sum_{t_1 =1}^{T_k}\sum_{t_2 = 1}^{T_1}\boldsymbol{\epsilon}_{t_1, k}^\top \boldsymbol{\epsilon}_{t_2, 1}
    \right\Vert_{M/2} = \left\Vert\frac{1}{T_k\sqrt{d}}\sum_{t_1 =1}^{T_k}\overline{\boldsymbol{\epsilon}}_{1}^\top\boldsymbol{\epsilon}_{t_1, k}\right\Vert_{M/2}.
\end{align*}
Since $\overline{\boldsymbol{\epsilon}}_{1}$ is independent of $\boldsymbol{\epsilon}_{t_1, k}$ for $k\geq 2,$ from Lemma \ref{lemma.linear_form} and Theorem 1.7 in \cite{MR2002723}, for any vector $\boldsymbol{\tau}\in\mathbf{R}^d,$
\begin{align*}
    \mathbf{E}
    \left[
    \left\vert 
    \frac{1}{T_k\sqrt{d}}\sum_{t_1 =1}^{T_k}\overline{\boldsymbol{\epsilon}}_{1}^\top\boldsymbol{\epsilon}_{t_1, k}
    \right\vert^{M/2}\mid \overline{\boldsymbol{\epsilon}}_{1} = \boldsymbol{\tau}
    \right] & = 
    \mathbf{E}\left[\left\vert 
    \frac{1}{T_k\sqrt{d}}\sum_{t_1 =1}^{T_k}\boldsymbol{\tau}^\top\boldsymbol{\epsilon}_{t_1, k}
    \right\vert^{M/2}\right]
    \\
    &\leq
    \frac{C\left\vert \boldsymbol{\tau}\right\vert_2^{M/2}}{\left(\sqrt{T_kd}\right)^{M/2}}
    = \frac{C\left\vert\overline{\boldsymbol{\epsilon}}_{1}\right\vert_2^{M/2}}{\left(\sqrt{T_kd}\right)^{M/2}},
\end{align*}
so 
\begin{align*}
    \left\Vert\frac{1}{T_k\sqrt{d}}\sum_{t_1 =1}^{T_k}\overline{\boldsymbol{\epsilon}}_{1}^\top\boldsymbol{\epsilon}_{t_1, k} \right\Vert_{M/2}\leq \frac{C\left\Vert\ 
    \vert \overline{\boldsymbol{\epsilon}}_{1}\vert_2\ 
    \right 
    \Vert_{M/2}}{\sqrt{T_k d}}\leq \frac{C}{\sqrt{T_k d}}\sqrt{\sum_{j = 1}^d \left\Vert\overline{\boldsymbol{\epsilon}}_{1}^{(j)}\right\Vert_{M/2}^2}.
\end{align*}
From Lemma \ref{lemma.linear_form}, 
$$
\left\Vert\overline{\boldsymbol{\epsilon}}_{1}^{(j)}\right\Vert_{M/2} = \frac{1}{T_1}\left\Vert\sum_{t = 1}^{T_1}\boldsymbol{\epsilon}^{(j)}_{t, 1}\right\Vert
\leq \frac{C}{T_1}\sqrt{\sum_{t = 1}^{T_1}1^2} = \frac{C}{\sqrt{T_1}},
$$ 
so 
\begin{equation}
\begin{aligned}
    \left\Vert
    \frac{1}{T_1T_k\sqrt{d}}\sum_{t_1 =1}^{T_k}\sum_{t_2 = 1}^{T_1}\boldsymbol{\epsilon}_{t_1, k}^\top \boldsymbol{\epsilon}_{t_2, 1}
    \right\Vert_{M/2}\leq \frac{C}{\sqrt{T_1T_k}} = O\left(\frac{1}{\mathcal{T}_\circ}\right).
\end{aligned}
\label{eq.T_kT_1}
\end{equation}
From \eqref{eq.V_k_sqrt}, \eqref{eq.V_k_bias}, \eqref{eq.V_k_sqrt_d}, and \eqref{eq.T_kT_1}, we have 
\begin{equation}
   \left\Vert
    \ 
    \widehat{R}_k  - \frac{1}{\sqrt{d}}\left\vert \boldsymbol{\mu}_k  - \boldsymbol{\mu}_1\right \vert_2^2
    \ 
    \right\Vert_{M/2}\leq \frac{C}{\sqrt{\mathcal{T}_\circ}},
\end{equation}
which proves \eqref{eq.consistency_R}.

\textbf{2. The proof of equation \eqref{eq.prob_res}.} From \eqref{eq.Prob_to_E},
\begin{equation}
\begin{aligned}
    &\sup_{x\in\mathbf{R}}\left\vert
        \mathbf{Pr}\left(\sqrt{\mathcal{T}_\circ(B_1- B)}\widehat{R}\leq x\right) - \mathbf{Pr}\left(\zeta \leq x\right)
    \right\vert\\
    &\leq \frac{C}{\psi} + \sup_{x\in\mathbf{R}}\left\vert\mathbf{E}\left[g_{\psi,x}\left(\sqrt{\mathcal{T}_\circ(B_1- B)}\widehat{R}\right)\right] - \mathbf{E}\left[g_{\psi,x}(\zeta)\right]
            \right\vert.
\end{aligned}
\label{eq.prob_to_g}
\end{equation}
Under $H_0,$ we have $\boldsymbol{\mu}_1 = \cdots = \boldsymbol{\mu}_K.$ From \eqref{eq.expand_R_k}, for any $k = 2,3,\cdots, K,$
\begin{align*}
    &\left\Vert
    \widehat{R}_k - \frac{1}{V_k\sqrt{d}}\sum_{B\leq \vert t_1  - t_2\vert\leq B_1}^{T_k}\left(\boldsymbol{\epsilon}_{t_1,k}^\top\boldsymbol{\epsilon}_{t_2, k} - \mathbf{E}\left[\boldsymbol{\epsilon}_{t_1,k}^\top\boldsymbol{\epsilon}_{t_2, k}\right]\right)\right. \\
    &\left.- \frac{1}{V_1\sqrt{d}}\sum_{B\leq \vert t_1 - t_2\vert\leq B_1}^{T_1}\left(\boldsymbol{\epsilon}_{t_1,1}^\top\boldsymbol{\epsilon}_{t_2, 1} - \mathbf{E}\left[\boldsymbol{\epsilon}_{t_1,1}^\top\boldsymbol{\epsilon}_{t_2, 1}\right]\right)
    \right\Vert_{M/2}\\
    &\leq \frac{1}{V_k\sqrt{d}}\sum_{B\leq \vert t_1 - t_2\vert\leq B_1}^{T_k}\left\vert
    \mathbf{E}\left[\boldsymbol{\epsilon}_{t_1,k}^\top\boldsymbol{\epsilon}_{t_2,k}\right]
    \right\vert
    + \frac{1}{V_1\sqrt{d}}\sum_{B\leq \vert t_1 - t_2\vert\leq B_1}^{T_k}\left\vert
    \mathbf{E}\left[\boldsymbol{\epsilon}^\top_{t_1,1}\boldsymbol{\epsilon}_{t_2,1}\right]
    \right\vert\\
    & + 
    2\left\Vert
    \frac{1}{V_k\sqrt{d}}\sum_{B\leq\vert t_1  -t_2\vert\leq B_1}^{T_k}\boldsymbol{\mu}_1^\top\boldsymbol{\epsilon}_{t_1,k} - \frac{1}{T_k\sqrt{d}}\sum_{t_1 = 1}^{T_k}\boldsymbol{\mu}_1^\top\boldsymbol{\epsilon}_{t_1,k}
    \right\Vert_{M/2}\\
    & + 
    2\left\Vert
    \frac{1}{V_1\sqrt{d}}\sum_{B\leq \vert t_1 - t_2\vert\leq B_1}^{T_1}\boldsymbol{\mu}_1^\top\boldsymbol{\epsilon}_{t_1,1} - \frac{1}{T_1\sqrt{d}}\sum_{t_1 = 1}^{T_1}\boldsymbol{\mu}_1^\top\boldsymbol{\epsilon}_{t_1,1}
    \right\Vert_{M/2}\\
    & + \frac{2}{T_kT_1\sqrt{d}}\left\Vert\sum_{t_1 = 1}^{T_k}\sum_{t_2 = 1}^{T_1}\boldsymbol{\epsilon}_{t_1,k}^\top\boldsymbol{\epsilon}_{t_2,1}\right\Vert_{M/2}.
\end{align*}
For any $k = 1,2,\cdots, K,$ from Lemma \ref{lemma.linear_form},
\begin{align*}
    &\left\Vert\frac{1}{V_k\sqrt{d}}\sum_{B\leq \vert t_1 -  t_2\vert\leq B_1}^{T_k}\boldsymbol{\mu}_1^\top\boldsymbol{\epsilon}_{t_1,k} - \frac{1}{T_k\sqrt{d}}\sum_{t_1 = 1}^{T_k} \boldsymbol{\mu}_1^\top \boldsymbol{\epsilon}_{t_1,k}\right\Vert_{M/2}\\
    &= 
    \left\Vert
    \sum_{t_1 = 1}^{T_k}\boldsymbol{\mu}_1^\top\boldsymbol{\epsilon}_{t_1,k}\left(\frac{1}{V_k\sqrt{d}}\sum_{t_2 = t_1 + B}^{(t_1 + B_1)\wedge T_k}1 + \frac{1}{V_k\sqrt{d}}\sum_{t_2 = (t_1 -  B_1)\vee 1}^{t_1 - B}1  - \frac{1}{T_k\sqrt{d}}\right)
    \right\Vert_{M/2}\\
    &\leq \frac{C\vert\boldsymbol{\mu}_1\vert_2}{\sqrt{d}}\sqrt{\sum_{t_1 = 1}^{T_k}\left(\frac{1}{V_k}\sum_{t_2 = t_1 + B}^{(t_1 + B_1)\wedge T_k}1 + \frac{1}{V_k}\sum_{t_2 = (t_1 -  B_1)\vee 1}^{t_1 - B}1  - \frac{1}{T_k}\right)^2}.
\end{align*}
If $1 + B_1\leq t_1\leq T_k - B_1,$ then 
\begin{align*}
    \frac{1}{V_k}\sum_{t_2 = t_1 + B}^{(t_1 + B_1)\wedge T_k}1 + \frac{1}{V_k}\sum_{t_2 = (t_1 -  B_1)\vee 1}^{t_1 - B}1  - \frac{1}{T_k}
    &= \frac{1}{V_k}\sum_{t_2 = t_1 + B}^{t_1 + B_1}1 + \frac{1}{V_k}\sum_{t_2 = t_1 -  B_1}^{t_1 - B}1  - \frac{1}{T_k}\\
    &= \frac{2(B_1 - B + 1)}{(2T_k - B - B_1)\times (B_1-  B+1)} - \frac{1}{T_k}\\
    &= \frac{B + B_1}{T_k\times(2T_k - B - B_1)}\asymp \frac{B_1}{T^2_k}.
\end{align*}
If $t_1\leq B_1$ or $t_1 \geq T_k - B_1 + 1,$ then for sufficiently large $\mathcal{T}_\circ,$
\begin{align*}
    \frac{1}{V_k}\sum_{t_2 = t_1 + B}^{(t_1 + B_1)\wedge T_k}1 + \frac{1}{V_k}\sum_{t_2 = (t_1 -  B_1)\vee 1}^{t_1 - B}1  - \frac{1}{T_k}
    &\leq \frac{1}{V_k}\sum_{t_2 = t_1 + B}^{t_1 + B_1}1 + \frac{1}{V_k}\sum_{t_2 = t_1 -  B_1}^{t_1 - B}1  - \frac{1}{T_k}\\
    &= \frac{B + B_1}{T_k\times(2T_k - B - B_1)}\leq \frac{2B_1}{T^2_k},\\
     \text{and } \frac{1}{V_k}\sum_{t_2 = t_1 + B}^{(t_1 + B_1)\wedge T_k}1 + \frac{1}{V_k}\sum_{t_2 = (t_1 -  B_1)\vee 1}^{t_1 - B}1  - \frac{1}{T_k}\geq -\frac{1}{T_k}.
\end{align*}
Therefore, for sufficiently large $\mathcal{T}_\circ,$
\begin{align*}
    \left\vert\frac{1}{V_k}\sum_{t_2 = t_1 + B}^{(t_1 + B_1)\wedge T_k}1 + \frac{1}{V_k}\sum_{t_2 = (t_1 -  B_1)\vee 1}^{t_1 - B}1  - \frac{1}{T_k}\right\vert\leq \frac{2B_1}{T^2_k}\vee \frac{1}{T_k} = \frac{1}{T_k},
\end{align*}
and 
\begin{align*}
    &\sum_{t_1 = 1}^{T_k}\left(\frac{1}{V_k}\sum_{t_2 = t_1 + B}^{(t_1 + B_1)\wedge T_k}1 + \frac{1}{V_k}\sum_{t_2 = (t_1 -  B_1)\vee 1}^{t_1 - B}1  - \frac{1}{T_k}\right)^2\\
    &\leq \sum_{t_1 =  1}^{B_1} \frac{1}{T^2_k} + \sum_{t_1  = B_1 + 1}^{T_k - B_1}\frac{CB^2_1}{T^4_k} + \sum_{t_1 = T_k - B_1 + 1}^{T_k}  \frac{1}{T^2_k}\leq \frac{C_1B_1}{T^2_k},
\end{align*}
which implies 
\begin{equation}
    \left\Vert\frac{1}{V_k\sqrt{d}}\sum_{B\leq \vert t_1 -  t_2\vert\leq B_1}^{T_k}\boldsymbol{\mu}_1^\top\boldsymbol{\epsilon}_{t_1,k} - \frac{1}{T_k\sqrt{d}}\sum_{t_1 = 1}^{T_k} \boldsymbol{\mu}_1^\top \boldsymbol{\epsilon}_{t_1,k}\right\Vert_{M/2}
    \leq \frac{C\sqrt{B_1}}{\mathcal{T}_\circ}.    
    \label{eq.difference}
\end{equation}
From \eqref{eq.V_k_bias}, \eqref{eq.T_kT_1}, and \eqref{eq.difference}, 
we have 
\begin{align*}
    &\left\Vert
    \sqrt{\mathcal{T}_\circ(B_1  - B)}\widehat{R}_k - \frac{\sqrt{\mathcal{T}_\circ(B_1  - B)}}{V_k\sqrt{d}}\sum_{B\leq \vert t_1  - t_2\vert\leq B_1}^{T_k}\left(\boldsymbol{\epsilon}_{t_1,k}^\top\boldsymbol{\epsilon}_{t_2, k} - \mathbf{E}\left[\boldsymbol{\epsilon}_{t_1,k}^\top\boldsymbol{\epsilon}_{t_2, k}\right]\right)\right.\\
    &\left.- \frac{\sqrt{\mathcal{T}_\circ(B_1  - B)}}{V_1\sqrt{d}}\sum_{B\leq \vert t_1 - t_2\vert\leq B_1}^{T_1}\left(\boldsymbol{\epsilon}_{t_1,1}^\top\boldsymbol{\epsilon}_{t_2, 1} - \mathbf{E}\left[\boldsymbol{\epsilon}_{t_1,1}^\top\boldsymbol{\epsilon}_{t_2, 1}\right]\right)
    \right\Vert_{M/2}\\
    &\leq \frac{C\sqrt{\mathcal{T}_\circ  B_1}B}{B_1 \mathcal{T}_\circ^{3/2}} + \frac{C\sqrt{\mathcal{T}_\circ  B_1}}{\mathcal{T}_\circ} + \frac{CB_1 \sqrt{\mathcal{T}_\circ}}{\mathcal{T}_\circ} = O\left(\frac{B_1}{\sqrt{\mathcal{T}_\circ}}\right).
\end{align*}
For $\widehat{R} = \sum_{k = 2}^K \widehat{R}_k,$ we have 
    \begin{align*}
        &\sup_{x\in\mathbf{R}}\left\vert
        \mathbf{E}\left[g_{\psi,x}\left(\sqrt{\mathcal{T}_\circ(B_1 - B)}\widehat{R}\right)\right]
        \right.\\
        & - \mathbf{E}\left[g_{\psi,x}\left(\sum_{k = 2}^K\frac{\sqrt{\mathcal{T}_\circ(B_1  - B)}}{V_k\sqrt{d}}\sum_{B\leq \vert t_1  - t_2\vert\leq B_1}^{T_k}\left(\boldsymbol{\epsilon}_{t_1,k}^\top\boldsymbol{\epsilon}_{t_2, k} - \mathbf{E}\left[\boldsymbol{\epsilon}_{t_1,k}^\top\boldsymbol{\epsilon}_{t_2, k}\right]\right)\right.\right.\\
        &\left.\left.\left. + \left( K - 1\right)\frac{\sqrt{\mathcal{T}_\circ(B_1 - B)}}{V_1\sqrt{d}}\sum_{B\leq \vert t_1 - t_2\vert\leq B_1}^{T_1}\left(\boldsymbol{\epsilon}_{t_1,1}^\top\boldsymbol{\epsilon}_{t_2, 1} - \mathbf{E}\boldsymbol{\epsilon}_{t_1,1}^\top\boldsymbol{\epsilon}_{t_2, 1}\right)\right)\right]\right\vert\\
        &\leq C\psi\left\Vert
        \sqrt{\mathcal{T}_\circ(B_1 - B)}\widehat{R}\right.\\
        &- \sum_{k = 2}^K\frac{\sqrt{\mathcal{T}_\circ(B_1  - B)}}{V_k\sqrt{d}}\sum_{B\leq \vert t_1  - t_2\vert\leq B_1}^{T_k}\left(\boldsymbol{\epsilon}_{t_1,k}^\top\boldsymbol{\epsilon}_{t_2, k} - \mathbf{E}\left[\boldsymbol{\epsilon}_{t_1,k}^\top\boldsymbol{\epsilon}_{t_2, k}\right]\right)\\
        & \left. - (K-1)\frac{\sqrt{\mathcal{T}_\circ(B_1 - B)}}{V_1\sqrt{d}}\sum_{B\leq \vert t_1 - t_2\vert\leq B_1}^{T_1}\left(\boldsymbol{\epsilon}_{t_1,1}^\top\boldsymbol{\epsilon}_{t_2, 1} - \mathbf{E}\left[\boldsymbol{\epsilon}_{t_1,1}^\top\boldsymbol{\epsilon}_{t_2, 1}\right]\right)
        \right\Vert_{M/2}\leq \frac{C\psi B_1}{\sqrt{\mathcal{T}_\circ}},
    \end{align*}
making 
\begin{equation}
    \begin{aligned}
        &\sup_{x\in\mathbf{R}}\left\vert
        \mathbf{E}\left[g_{\psi,x}\left(\sqrt{\mathcal{T}_\circ(B_1 - B)}\widehat{R}\right)\right]
        \right.\\
        & - \mathbf{E}\left[g_{\psi,x}\left(\sum_{k = 2}^K\frac{\sqrt{\mathcal{T}_\circ(B_1  - B)}}{V_k\sqrt{d}}\sum_{B\leq \vert t_1  - t_2\vert\leq B_1}^{T_k}\left(\boldsymbol{\epsilon}_{t_1,k}^\top\boldsymbol{\epsilon}_{t_2, k} - \mathbf{E}\left[\boldsymbol{\epsilon}_{t_1,k}^\top\boldsymbol{\epsilon}_{t_2, k}\right]\right)\right.\right.\\
        &\left.\left.\left. + \left( K - 1\right)\frac{\sqrt{\mathcal{T}_\circ(B_1 - B)}}{V_1\sqrt{d}}\sum_{B\leq \vert t_1 - t_2\vert\leq B_1}^{T_1}\left(\boldsymbol{\epsilon}_{t_1,1}^\top\boldsymbol{\epsilon}_{t_2, 1} - \mathbf{E}\left[\boldsymbol{\epsilon}_{t_1,1}^\top\boldsymbol{\epsilon}_{t_2, 1}\right]\right)\right)\right]\right\vert\\
        &= O\left(\frac{\psi B_1}{\sqrt{\mathcal{T}_\circ}}\right).
    \end{aligned}
    \label{eq.g_to_sum_epsilon}
\end{equation}
Choose integer $\ell =  \lfloor \mathcal{T}_\circ^{\frac{3}{8}}\rfloor,$ for sufficiently large $\mathcal{T}_\circ,$ we have $\ell > B_1,$ and from Theorem \ref{theorem.consistent_quadratic},
\begin{equation}
\begin{aligned}
&\left\Vert\left(\frac{\sqrt{\mathcal{T}_\circ(B_1  - B)}}{V_k\sqrt{d}}\sum_{B\leq \vert t_1  - t_2\vert\leq B_1}^{T_k}\left(\boldsymbol{\epsilon}_{t_1,k}^\top\boldsymbol{\epsilon}_{t_2, k} - \mathbf{E}\left[\boldsymbol{\epsilon}_{t_1,k}^\top\boldsymbol{\epsilon}_{t_2, k}\right]\right)\right)\right.\\
&\left. - \left(
\frac{\sqrt{\mathcal{T}_\circ(B_1  - B)}}{V_k\sqrt{d}}\sum_{B\leq \vert t_1  - t_2\vert\leq B_1}^{T_k}\left(\mathbf{E}\left[\boldsymbol{\epsilon}_{t_1,k}^\top\boldsymbol{\epsilon}_{t_2, k}\mid\mathcal{F}_{t_1\vee t_2,\ell}^{(k)}\right] - \mathbf{E}\left[\boldsymbol{\epsilon}_{t_1,k}^\top\boldsymbol{\epsilon}_{t_2, k}\right]\right)
\right)\right\Vert_{M/2}\\
& = 
\frac{\sqrt{\mathcal{T}_\circ(B_1  - B)}}{V_k\sqrt{d}}
\left\Vert
\sum_{B\leq \vert t_1  - t_2\vert\leq B_1}^{T_k}\left(\boldsymbol{\epsilon}_{t_1,k}^\top\boldsymbol{\epsilon}_{t_2, k} - \mathbf{E}\left[\boldsymbol{\epsilon}_{t_1,k}^\top\boldsymbol{\epsilon}_{t_2, k}\mid\mathcal{F}_{t_1\vee t_2,\ell}^{(k)}\right]\right)
\right\Vert_{M/2}\\
&\leq \frac{C}{(\ell + 1 - B_1)^\alpha} + \frac{C\mathcal{T}_\circ^{3/2}}{\ell^{\alpha - 1}\sqrt{(B_1 - B)}} 
+ \frac{C\sqrt{\mathcal{T}_\circ}}{B^\alpha\sqrt{(B_1 - B)}} = O\left(\frac{1}{\mathcal{T}_\circ^{3/2}\sqrt{(B_1 - B)}}\right).
\end{aligned}
\label{eq.diff_cut_tail}
\end{equation}
Therefore, 
\begin{equation}
    \begin{aligned}
        &\sup_{x\in\mathbf{R}}\left\vert
        \mathbf{E}\left[g_{\psi,x}\left(\sum_{k = 2}^K\frac{\sqrt{\mathcal{T}_\circ(B_1  - B)}}{V_k\sqrt{d}}\sum_{B\leq \vert t_1  - t_2\vert\leq B_1}^{T_k}\left(\boldsymbol{\epsilon}_{t_1,k}^\top\boldsymbol{\epsilon}_{t_2, k} - \mathbf{E}\left[\boldsymbol{\epsilon}_{t_1,k}^\top\boldsymbol{\epsilon}_{t_2, k}\right]\right)\right.\right.\right.\\
        &\left.\left.\left. + \left( K - 1\right)\frac{\sqrt{\mathcal{T}_\circ(B_1 - B)}}{V_1\sqrt{d}}\sum_{B\leq \vert t_1 - t_2\vert\leq B_1}^{T_1}\left(\boldsymbol{\epsilon}_{t_1,1}^\top\boldsymbol{\epsilon}_{t_2, 1} - \mathbf{E}\left[\boldsymbol{\epsilon}_{t_1,1}^\top\boldsymbol{\epsilon}_{t_2, 1}\right]\right)\right)\right]\right.\\
        & - \mathbf{E}\left[
        g_{\psi,x}\left(
        \sum_{k = 2}^K\frac{\sqrt{\mathcal{T}_\circ(B_1  - B)}}{V_k\sqrt{d}}\sum_{B\leq \vert t_1  - t_2\vert\leq B_1}^{T_k}\left(\mathbf{E}\left[\boldsymbol{\epsilon}_{t_1,k}^\top\boldsymbol{\epsilon}_{t_2, k}\mid\mathcal{F}_{t_1\vee t_2,\ell}^{(k)}\right] - \mathbf{E}\left[\boldsymbol{\epsilon}_{t_1,k}^\top\boldsymbol{\epsilon}_{t_2, k}\right]\right)\right.\right.\\
        & + \left.\left.\left.\left( K - 1\right)\frac{\sqrt{\mathcal{T}_\circ(B_1 - B)}}{V_1\sqrt{d}}\sum_{B\leq \vert t_1 - t_2\vert\leq B_1}^{T_1}\left(\mathbf{E}\left[\boldsymbol{\epsilon}_{t_1,1}^\top\boldsymbol{\epsilon}_{t_2, 1}\mid\mathcal{F}_{t_1\vee t_2,\ell}^{(1)}\right] - \mathbf{E}\left[\boldsymbol{\epsilon}_{t_1,1}^\top\boldsymbol{\epsilon}_{t_2, 1}\right]\right)\right)\right]
        \right\vert\\
        &\leq C\psi\sum_{k = 2}^K \frac{\sqrt{\mathcal{T}_\circ(B_1  - B)}}{V_k\sqrt{d}}\left\Vert 
        \sum_{B\leq \vert t_1  - t_2\vert\leq B_1}^{T_k}\left(\boldsymbol{\epsilon}_{t_1,k}^\top\boldsymbol{\epsilon}_{t_2, k} - \mathbf{E}\left[\boldsymbol{\epsilon}_{t_1,k}^\top\boldsymbol{\epsilon}_{t_2, k}\mid\mathcal{F}_{t_1\vee t_2,\ell}^{(k)}\right]\right)
        \right\Vert_{M/2}\\
        & + C\psi (K-1) \frac{\sqrt{\mathcal{T}_\circ(B_1 - B)}}{V_1\sqrt{d}}\left\Vert
        \sum_{B\leq \vert t_1  - t_2\vert\leq B_1}^{T_k}\left(\boldsymbol{\epsilon}_{t_1,1}^\top\boldsymbol{\epsilon}_{t_2, 1} - \mathbf{E}\left[\boldsymbol{\epsilon}_{t_1,1}^\top\boldsymbol{\epsilon}_{t_2, 1}\mid\mathcal{F}_{t_1\vee t_2,\ell}^{(1)}\right]\right)
        \right\Vert_{M/2}\\
        & = O\left(\frac{\psi}{\mathcal{T}_\circ^{3/2}\sqrt{(B_1 - B)}}\right).
    \end{aligned}
    \label{eq.e_to_truncate}
\end{equation}
Choose $ v = \lfloor \mathcal{T}_\circ^{\frac{7}{12}}\rfloor,$ and choose the big-blocks
\begin{align*}
    A_{q,k} = 2\sum_{t_1 = (q - 1)\times (v + \ell) + 1}^{((q - 1)\times (v + \ell) + v)\wedge T_k}\sum_{t_2 = (t_1 - B_1)\vee 1}^{(t_1 - B)}\left(\mathbf{E}\left[\boldsymbol{\epsilon}_{t_1,k}^\top\boldsymbol{\epsilon}_{t_2,k}\mid\mathcal{F}_{t_1,\ell}^{(k)}\right] - \mathbf{E}\left[\boldsymbol{\epsilon}_{t_1,k}^\top\boldsymbol{\epsilon}_{t_2,k}\right]\right),
\end{align*}
the small-blocks
\begin{align*}
    a_{q,k} = 2\sum_{t_1 = (q - 1)\times (v + \ell) + v + 1}^{(q\times (v + \ell))\wedge T_k}\sum_{t_2 = (t_1 - B_1)\vee 1}^{(t_1 - B)}\left(\mathbf{E}\left[\boldsymbol{\epsilon}_{t_1,k}^\top\boldsymbol{\epsilon}_{t_2,k}\mid\mathcal{F}_{t_1,\ell}^{(k)}\right] - \mathbf{E}\left[\boldsymbol{\epsilon}_{t_1,k}^\top\boldsymbol{\epsilon}_{t_2,k}\right]\right)
\end{align*}
for $q = 1,2,\cdots, R = \lceil\frac{\mathcal{T}_\dagger}{v + \ell}\rceil.$ Then 
\begin{align*}
    &\sum_{B\leq \vert t_1 - t_2\vert\leq B_1}^{T_k}\left(\mathbf{E}\left[\boldsymbol{\epsilon}_{t_1,k}^\top\boldsymbol{\epsilon}_{t_2,k}\mid \mathcal{F}_{t_1\vee t_2, \ell}^{(k)}\right] - \mathbf{E}\left[\boldsymbol{\epsilon}_{t_1,k}^\top\boldsymbol{\epsilon}_{t_2,k}\right]\right)\\
    &= 2\sum_{t_1 = B + 1}^{T_k}\sum_{t_2 = (t_1 - B_1)\vee 1}^{t_1 - B}\left(\mathbf{E}\left[\boldsymbol{\epsilon}_{t_1,k}^\top\boldsymbol{\epsilon}_{t_2,k}\mid\mathcal{F}_{t_1, \ell}^{(k)}\right] - \mathbf{E}\left[\boldsymbol{\epsilon}_{t_1,k}^\top\boldsymbol{\epsilon}_{t_2,k}\right]\right)\\
    &= \sum_{q =  1}^R A_{q,k} + \sum_{q=  1}^R a_{q,k},
\end{align*}
and $ A_{q,k}$ are mutually independent with respect to different $q$ or $k$; and $a_{q,k}$ are mutually independent with respect to different $q$ or $k.$ Since $a_{q,k}$ are mutually independent, from Theorem 2 of \cite{MR0133849},
\begin{align*}
    &\frac{\sqrt{\mathcal{T}_\circ(B_1 - B)}}{V_k\sqrt{d}}\left\Vert
    \sum_{B\leq \vert t_1  - t_2\vert\leq B_1}^{T_k}\left(\mathbf{E}\left[\boldsymbol{\epsilon}_{t_1,k}^\top\boldsymbol{\epsilon}_{t_2,k}\mid\mathcal{F}_{t_1\vee t_2,\ell}^{(k)}\right] - \mathbf{E}\left[\boldsymbol{\epsilon}_{t_1,k}^\top\boldsymbol{\epsilon}_{t_2,k}\right]\right) - \sum_{q =  1}^R A_{q,k}
    \right\Vert_{M/2}\\
    & = \frac{\sqrt{\mathcal{T}_\circ(B_1 - B)}}{V_k\sqrt{d}}\left\Vert
    \sum_{q=  1}^R a_{q,k}
    \right\Vert_{M/2}\leq \frac{C\sqrt{\mathcal{T}_\circ(B_1 - B)}}{V_k\sqrt{d}}\sqrt{\sum_{q = 1}^R\left\Vert a_{q,k}\right\Vert^2_{M/2}}.
\end{align*}
From Theorem \ref{theorem.consistent_quadratic}, for any $q,k,$
\begin{align*}
    \left\Vert
    a_{q,k}
    \right\Vert_{M/2} &\leq 2\left\Vert
    \sum_{t_1 = (q - 1)\times (v + \ell) + v + 1}^{(q\times (v + \ell))\wedge T_k}\sum_{t_2 = (t_1 - B_1)\vee 1}^{(t_1 - B)}\left(\boldsymbol{\epsilon}_{t_1,k}^\top\boldsymbol{\epsilon}_{t_2,k} - \mathbf{E}\left[\boldsymbol{\epsilon}_{t_1,k}^\top\boldsymbol{\epsilon}_{t_2,k}\right]\right)
    \right\Vert_{M/2}\\
    &+ 2\left\Vert
    \sum_{t_1 = (q - 1)\times (v + \ell) + v + 1}^{(q\times (v + \ell))\wedge T_k}\sum_{t_2 = (t_1 - B_1)\vee 1}^{(t_1 - B)}\left(\boldsymbol{\epsilon}_{t_1,k}^\top\boldsymbol{\epsilon}_{t_2,k} - \mathbf{E}\left[\boldsymbol{\epsilon}_{t_1,k}^\top\boldsymbol{\epsilon}_{t_2,k}\mid\mathcal{F}_{t_1,\ell}^{(k)}\right]\right)
    \right\Vert_{M/2}\\
    &\leq C\sqrt{d\sum_{t_1 = (q - 1)\times (v + \ell) + v + 1}^{(q\times (v + \ell))\wedge T_k}\sum_{t_2 = (t_1 - B_1)\vee 1}^{(t_1 - B)}1^2} + \frac{CT^{5/2}}{B^\alpha} + \frac{CT^{3/2}}{B^{\alpha - 2}}\\
    & + \frac{C\sqrt{d\sum_{t_1 = (q - 1)\times (v + \ell) + v + 1}^{(q\times (v + \ell))\wedge T_k}\sum_{t_2 = (t_1 - B_1)\vee 1}^{(t_1 - B)}1^2}}{(\ell + 1 - B_1)^\alpha} + \frac{CdT^{3/2}}{\ell^{\alpha - 1}} + \frac{Cd\sqrt{T}}{B^\alpha}\\
    &\leq C_1\sqrt{d\ell \left(B_1 - B\right)}.
\end{align*}
Therefore,
\begin{equation}
\begin{aligned}
    &\frac{\sqrt{\mathcal{T}_\circ(B_1 - B)}}{V_k\sqrt{d}}\left\Vert
    \sum_{B\leq \vert t_1  - t_2\vert\leq B_1}^{T_k}\left(\mathbf{E}\left[\boldsymbol{\epsilon}_{t_1,k}^\top\boldsymbol{\epsilon}_{t_2,k}\mid\mathcal{F}_{t_1\vee t_2,\ell}^{(k)}\right] - \mathbf{E}\left[\boldsymbol{\epsilon}_{t_1,k}^\top\boldsymbol{\epsilon}_{t_2,k}\right]\right) - \sum_{q =  1}^R A_{q,k}
    \right\Vert_{M/2}\\
    &\leq \frac{C\sqrt{\mathcal{T}_\circ(B_1 - B)}}{\mathcal{T}_\circ(B_1 - B)\sqrt{d}}\sqrt{\frac{\mathcal{T}_\dagger}{v}}\sqrt{d\ell(B_1 - B)} = O\left(\sqrt{\frac{\ell}{v}}\right),
\end{aligned}
\label{eq.small_block}
\end{equation}
and we have 
\begin{equation}
    \begin{aligned}
        &\sup_{x\in\mathbf{R}}\left\vert\mathbf{E}\left[
        g_{\psi,x}\left(
        \sum_{k = 2}^K\frac{\sqrt{\mathcal{T}_\circ(B_1  - B)}}{V_k\sqrt{d}}\sum_{B\leq \vert t_1  - t_2\vert\leq B_1}^{T_k}\left(\mathbf{E}\left[\boldsymbol{\epsilon}_{t_1,k}^\top\boldsymbol{\epsilon}_{t_2, k}\mid\mathcal{F}_{t_1\vee t_2,\ell}^{(k)}\right] - \mathbf{E}\left[\boldsymbol{\epsilon}_{t_1,k}^\top\boldsymbol{\epsilon}_{t_2, k}\right]\right)\right.\right.\right.\\
        & + \left.\left.\left( K - 1\right)\frac{\sqrt{\mathcal{T}_\circ(B_1 - B)}}{V_1\sqrt{d}}\sum_{B\leq \vert t_1 - t_2\vert\leq B_1}^{T_1}\left(\mathbf{E}\left[\boldsymbol{\epsilon}_{t_1,1}^\top\boldsymbol{\epsilon}_{t_2, 1}\mid\mathcal{F}_{t_1\vee t_2,\ell}^{(1)}\right] - \mathbf{E}\left[\boldsymbol{\epsilon}_{t_1,1}^\top\boldsymbol{\epsilon}_{t_2, 1}\right]\right)\right)\right]\\
        &\left.-\mathbf{E}\left[g_{\psi,x}\left(\sum_{k = 2}^K\frac{\sqrt{\mathcal{T}_\circ(B_1  - B)}}{V_k\sqrt{d}}\sum_{q = 1}^R A_{q,k} + (K - 1)\frac{\sqrt{\mathcal{T}_\circ(B_1 - B)}}{V_1\sqrt{d}}\sum_{q = 1}^R A_{q,1}\right)\right]\right\vert\\
        &= O\left(\psi\sqrt{\frac{\ell}{v}}\right).
    \end{aligned}
    \label{eq.truncate_to_big_sum}
\end{equation}
Define $A^*_{q,k}, q = 1,\cdots, R, k = 1,\cdots, K, $ as independent normal random variables such that 
$$
\mathbf{E}\left[A^*_{q,k}\right] = 0,\quad \mathrm{Var}(A^*_{q,k}) = \mathrm{Var}(A_{q,k}),
$$
and $A^*_{q_1, k_1}$ and $A_{q_2,k_2}$ are independent for any  $q_1,q_2,k_1,k_2$. Define the summation 
\begin{align*}
    T_u &= \sum_{q = 1}^{u - 1}\left(\sum_{k = 2}^K\frac{\sqrt{\mathcal{T}_\circ(B_1 - B)}}{V_k\sqrt{d}}A_{q,k} + (K-1)\frac{\sqrt{\mathcal{T}_\circ(B_1 - B)}}{V_1\sqrt{d}}A_{q,1}\right)\\
&+ \sum_{q = u+1}^{R}\left(\sum_{k = 2}^K\frac{\sqrt{\mathcal{T}_\circ(B_1 - B)}}{V_k\sqrt{d}}A_{q,k}^* + (K-1)\frac{\sqrt{\mathcal{T}_\circ(B_1 - B)}}{V_1\sqrt{d}}A_{q,1}^*\right)
\end{align*}
and 
\begin{align*}
    R_u = \sum_{k = 2}^K\frac{\sqrt{\mathcal{T}_\circ(B_1 - B)}}{V_k\sqrt{d}}A_{u,k} + (K-1)\frac{\sqrt{\mathcal{T}_\circ(B_1 - B)}}{V_1\sqrt{d}}A_{u,1},\\
    R^*_u = \sum_{k = 2}^K\frac{\sqrt{\mathcal{T}_\circ(B_1 - B)}}{V_k\sqrt{d}}A_{u,k}^* + (K-1)\frac{\sqrt{\mathcal{T}_\circ(B_1 - B)}}{V_1\sqrt{d}}A_{u,1}^*,
\end{align*}
then $R^*_u, R_u$ are independent of $T_u,$ and 
$
T_u + R_u = T_{u+1} + R^*_{u+1}.
$
From Theorem \ref{theorem.consistent_quadratic},
\begin{align*}
    \left\Vert
     A_{q,k}
    \right\Vert_{M/2} & \leq 2\left\Vert
    \sum_{t_1 = (q - 1)\times (v + \ell) + 1}^{((q - 1)\times (v + \ell) + v)\wedge T_k}\sum_{t_2 = (t_1 - B_1)\vee 1}^{(t_1 - B)}\left(\boldsymbol{\epsilon}_{t_1,k}^\top\boldsymbol{\epsilon}_{t_2,k} - \mathbf{E}\left[\boldsymbol{\epsilon}_{t_1,k}^\top\boldsymbol{\epsilon}_{t_2,k}\right]\right)
    \right\Vert_{M/2}\\
    & + 2\left\Vert\sum_{t_1 = (q - 1)\times (v + \ell) + 1}^{((q - 1)\times (v + \ell) + v)\wedge T_k}\sum_{t_2 = (t_1 - B_1)\vee 1}^{(t_1 - B)}\left(\boldsymbol{\epsilon}_{t_1,k}^\top\boldsymbol{\epsilon}_{t_2,k} - \mathbf{E}\left[\boldsymbol{\epsilon}_{t_1,k}^\top\boldsymbol{\epsilon}_{t_2,k}\mid\mathcal{F}_{t_1, \ell}^{(k)}\right]\right)\right\Vert_{M/2}\\
    &\leq C\sqrt{d\sum_{t_1 = (q - 1)\times (v + \ell) + 1}^{((q - 1)\times (v + \ell) + v)\wedge T_k}\sum_{t_2 = (t_1 - B_1)\vee 1}^{(t_1 - B)}1^2} + \frac{CT^{5/2}}{B^\alpha} + \frac{CT^{3/2}}{B^{\alpha - 2}} \\
    & + \frac{C\sqrt{d}}{(\ell + 1 - B_1)^\alpha}\sqrt{\sum_{t_1 = (q - 1)\times (v + \ell) + 1}^{((q - 1)\times (v + \ell) + v)\wedge T_k}\sum_{t_2 = (t_1 - B_1)\vee 1}^{(t_1 - B)}1^2} + \frac{CdT^{3/2}}{\ell^{\alpha - 1}} + \frac{Cd\sqrt{T}}{B^\alpha}\\
    &\leq C_1\sqrt{dv(B_1 - B)}.
\end{align*}
Since $A^*_{q, k}$ is normal random variable, we have 
\begin{align*}
    \left\Vert
    A_{q,k}^*
    \right\Vert_{M/2}\leq C\left\Vert A_{q,k}\right\Vert_2\leq C_1\sqrt{dv (B_1 - B)}.
\end{align*}
Therefore, 
\begin{align*}
    \left\Vert
    R_u
    \right\Vert_{M/2}&\leq \sum_{k = 2}^K\frac{\sqrt{\mathcal{T}_\circ(B_1 - B)}}{V_k\sqrt{d}}\left\Vert A_{u,k}\right\Vert_{M/2} + (K-1)\frac{\sqrt{\mathcal{T}_\circ(B_1 - B)}}{V_1\sqrt{d}}\left\Vert A_{u,1}\right\Vert_{M/2}\\
    &\leq C\sqrt{\frac{v}{\mathcal{T}_\circ}},
\end{align*}
and 
\begin{align*}
    \Vert
    R^*_u
    \Vert_{M/2} &\leq \sum_{k = 2}^K\frac{\sqrt{\mathcal{T}_\circ(B_1 - B)}}{V_k\sqrt{d}}\Vert A_{u,k}^*\Vert_{M/2} + \frac{(K-1)\times \sqrt{\mathcal{T}_\circ(B_1 - B)}}{V_1\sqrt{d}}\Vert A_{u,1}^*\Vert_{M/2}\\
    &\leq C\sqrt{\frac{v}{\mathcal{T}_\circ}}.
\end{align*}
From \eqref{eq.Taylor_expansion},
\begin{align*}
    &\sup_{x\in\mathbf{R}}\left\vert\mathbf{E}\left[
    g_{\psi,x}\left(
    T_u + R_u
    \right)\mid T_u
    \right] - \mathbf{E}\left[g_{\psi,x}\left(T_u + R^*_u\right)\mid T_u\right]\right\vert\\
    &\leq C\psi^3\left(\left\Vert R_u\right\Vert^3_{M/2} + \left\Vert R_u^*\right\Vert^3_{M/2}\right)
    \leq C_1\psi^3\left(\frac{v}{\mathcal{T}_\circ}\right)^{3/2},
\end{align*}
and 
\begin{equation}
\begin{aligned}
    &\sup_{x\in\mathbf{R}}\left\vert
    \mathbf{E}\left[g_{\psi,x}\left(\sum_{k = 2}^K\frac{\sqrt{\mathcal{T}_\circ(B_1 - B)}}{V_k\sqrt{d}}\sum_{q = 1}^R A_{q,k} + (K - 1)\frac{\sqrt{\mathcal{T}_\circ(B_1 - B)}}{V_1\sqrt{d}}\sum_{q = 1}^R A_{q,1}\right)\right]\right.\\
    &\left.- \mathbf{E}\left[g_{\psi,x}\left(\sum_{k = 2}^K\frac{\sqrt{\mathcal{T}_\circ(B_1 - B)}}{V_k\sqrt{d}}\sum_{q = 1}^R A_{q,k}^* + (K - 1)\frac{\sqrt{\mathcal{T}_\circ (B_1 - B)}}{V_1\sqrt{d}}\sum_{q = 1}^R A_{q,1}^*\right)\right]
    \right\vert\\
    &\leq \sum_{u = 1}^R\sup_{x\in\mathbf{R}}\left\vert
    \mathbf{E}\left[g_{\psi,x}(T_u + R_u)\right] -  \mathbf{E}\left[g_{\psi,x}(T_u + R_u^*)\right]
    \right\vert\leq C\psi^3\sqrt{\frac{v}{\mathcal{T}_\circ}}.
\end{aligned}
\label{eq.A_to_A_star}
\end{equation}
Finally, the term 
$$
\sum_{k = 2}^K\frac{\sqrt{\mathcal{T}_\circ (B_1 - B)}}{V_k\sqrt{d}}\sum_{q = 1}^R A_{q,k}^* + (K - 1)\frac{\sqrt{\mathcal{T}_\circ (B_1 - B)}}{V_1\sqrt{d}}\sum_{q = 1}^R A_{q,1}^*
$$
has normal distribution with mean $0$ and variance
\begin{align*}
    &\mathrm{Var}\left(
    \sum_{k = 2}^K\frac{\sqrt{\mathcal{T}_\circ (B_1 - B)}}{V_k\sqrt{d}}\sum_{q = 1}^R A_{q,k}^* + (K - 1)\frac{\sqrt{\mathcal{T}_\circ (B_1 - B)}}{V_1\sqrt{d}}\sum_{q = 1}^R A_{q,1}^*
    \right)\\
    &= \sum_{k = 2}^K\sum_{q = 1}^R \frac{\mathcal{T}_\circ (B_1 - B)}{V^2_k d} \mathrm{Var}\left(A_{q,k}\right) + (K - 1)^2 \frac{\mathcal{T}_\circ(B_1 - B)}{V^2_1 d}\sum_{q =  1}^R \mathrm{Var}\left(A_{q,1}\right)\\
    &= \left\Vert
    \sum_{k = 2}^K\frac{\sqrt{\mathcal{T}_\circ (B_1 - B)}}{V_k\sqrt{d}}\sum_{q = 1}^R A_{q,k} + (K - 1)\frac{\sqrt{\mathcal{T}_\circ (B_1 - B)}}{V_1\sqrt{d}}\sum_{q = 1}^R A_{q,1}
    \right\Vert^2_2.
\end{align*}
Since 
\begin{align*}
    &\left\vert\ \left\Vert
    \sum_{k = 2}^K\frac{\sqrt{\mathcal{T}_\circ (B_1 - B)}}{V_k\sqrt{d}}\sum_{q = 1}^R A_{q,k} + (K - 1)\frac{\sqrt{\mathcal{T}_\circ (B_1 - B)}}{V_1\sqrt{d}}\sum_{q = 1}^R A_{q,1}
    \right\Vert_2\right.\\
    &- 
    \left\Vert
    \sum_{k = 2}^K\frac{\sqrt{\mathcal{T}_\circ (B_1 - B)}}{V_k\sqrt{d}}\sum_{B\leq \vert t_1 - t_2\vert\leq B_1}^{T_k}\left(\boldsymbol{\epsilon}_{t_1,k}^\top\boldsymbol{\epsilon}_{t_2,k} - \mathbf{E}\left[\boldsymbol{\epsilon}_{t_1,k}^\top\boldsymbol{\epsilon}_{t_2,k}\right]\right)\right.\\
    &\left.\left.+ (K - 1)\frac{\sqrt{\mathcal{T}_\circ(B_1 - B)}}{V_1\sqrt{d}}\sum_{B\leq \vert t_1 - t_2\vert\leq B_1}^{T_1}\left(\boldsymbol{\epsilon}_{t_1,1}^\top\boldsymbol{\epsilon}_{t_2,1} - \mathbf{E}\left[\boldsymbol{\epsilon}_{t_1,1}^\top\boldsymbol{\epsilon}_{t_2,1}\right]\right)
    \right\Vert_{2}\right\vert\\
   & \leq \sum_{k = 2}^K
    \frac{\sqrt{\mathcal{T}_\circ (B_1 - B)}}{V_k\sqrt{d}}
    \left\Vert\ \sum_{B\leq \vert t_1 - t_2\vert\leq B_1}^{T_k}\left(\boldsymbol{\epsilon}_{t_1,k}^\top\boldsymbol{\epsilon}_{t_2,k} - \mathbf{E}\left[\boldsymbol{\epsilon}_{t_1,k}^\top\boldsymbol{\epsilon}_{t_2,k}\right]\right) - \sum_{q = 1}^R A_{q,k}\right\Vert_{M/2}\\
    &+ (K - 1)\frac{\sqrt{\mathcal{T}_\circ(B_1 - B)}}{V_1\sqrt{d}}\left\Vert
    \sum_{B\leq \vert t_1 -  t_2\vert\leq B_1}^{T_1}\left(\boldsymbol{\epsilon}_{t_1,1}^\top\boldsymbol{\epsilon}_{t_2,1} - \mathbf{E}\left[\boldsymbol{\epsilon}_{t_1,1}^\top\boldsymbol{\epsilon}_{t_2,1}\right]\right) - \sum_{q = 1}^R A_{q,1}
    \right\Vert_{M/2},
\end{align*}
and for any $k = 1,\cdots,K,$ from \eqref{eq.diff_cut_tail} and \eqref{eq.small_block}, 
\begin{align*}
    &\frac{\sqrt{\mathcal{T}_\circ (B_1 - B)}}{V_k\sqrt{d}}\left\Vert\ \sum_{B\leq \vert t_1 - t_2\vert\leq B_1}^{T_k}\left(\boldsymbol{\epsilon}_{t_1,k}^\top\boldsymbol{\epsilon}_{t_2,k} - \mathbf{E}\left[\boldsymbol{\epsilon}_{t_1,k}^\top\boldsymbol{\epsilon}_{t_2,k}\right]\right) - \sum_{q = 1}^R A_{q,k}\right\Vert_{M/2}\\
    &\leq 
    \frac{\sqrt{\mathcal{T}_\circ (B_1 - B)}}{V_k\sqrt{d}}\left\Vert
    \sum_{q = 1}^R a_{q,k}
    \right\Vert_{M/2}  + 
    \frac{\sqrt{\mathcal{T}_\circ (B_1 - B)}}{V_k\sqrt{d}}\left\Vert
    \sum_{B\leq \vert t_1 - t_2\vert\leq B_1}^{T_k}\left(\boldsymbol{\epsilon}_{t_1,k}^\top\boldsymbol{\epsilon}_{t_2,k} - \mathbf{E}\left[\boldsymbol{\epsilon}_{t_1,k}^\top\boldsymbol{\epsilon}_{t_2,k}\mid\mathcal{F}_{t_1\vee t_t,\ell}^{(k)}\right]\right)
    \right\Vert_{M/2}\\
    &\leq C\sqrt{\frac{\ell}{v}} + \frac{C}{\mathcal{T}_\circ^{3/2}\sqrt{(B_1 - B)}},
\end{align*}
making 
\begin{equation}
\begin{aligned}
    &\left\vert\ \left\Vert
    \sum_{k = 2}^K\frac{\sqrt{\mathcal{T}_\circ (B_1 - B)}}{V_k\sqrt{d}}\sum_{q = 1}^R A_{q,k} + (K - 1)\frac{\sqrt{\mathcal{T}_\circ (B_1 - B)}}{V_1\sqrt{d}}\sum_{q = 1}^R A_{q,1}
    \right\Vert_2\right.\\
    &- 
    \left\Vert
    \sum_{k = 2}^K\frac{\sqrt{\mathcal{T}_\circ (B_1 - B)}}{V_k\sqrt{d}}\sum_{B\leq \vert t_1 - t_2\vert\leq B_1}^{T_k}\left(\boldsymbol{\epsilon}_{t_1,k}^\top\boldsymbol{\epsilon}_{t_2,k} - \mathbf{E}\left[\boldsymbol{\epsilon}_{t_1,k}^\top\boldsymbol{\epsilon}_{t_2,k}\right]\right)\right.\\
    &\left.\left.+ (K - 1)\frac{\sqrt{\mathcal{T}_\circ(B_1 - B)}}{V_1\sqrt{d}}\sum_{B\leq \vert t_1 - t_2\vert\leq B_1}^{T_1}\left(\boldsymbol{\epsilon}_{t_1,1}^\top\boldsymbol{\epsilon}_{t_2,1} - \mathbf{E}\left[\boldsymbol{\epsilon}_{t_1,1}^\top\boldsymbol{\epsilon}_{t_2,1}\right]\right)
    \right\Vert_{2}\right\vert\\
    & = O\left(\sqrt{\frac{\ell}{v}} + \frac{1}{\mathcal{T}_\circ^{3/2}\sqrt{(B_1 - B)}}\right).
\end{aligned}
\label{eq.two_norm_diff}
\end{equation}
Define $Z$ as a standard normal random variable (normal random variable with mean $0$ and variance $1$), then 
\begin{equation}
\begin{aligned}
    &\sup_{x\in\mathbf{R}}\left\vert
    \mathbf{E}\left[g_{\psi,x}\left(\sum_{k = 2}^K\frac{\sqrt{\mathcal{T}(B_1 - B)}}{V_k\sqrt{d}}\sum_{q = 1}^R A_{q,k}^* + (K - 1)\frac{\sqrt{\mathcal{T}(B_1 - B)}}{V_1\sqrt{d}}\sum_{q = 1}^R A_{q,1}^*\right)\right] - \mathbf{E}\left[g_{\psi,x}(\zeta)\right]
    \right\vert\\
    &= \sup_{x\in\mathbf{R}}\left\vert
    \mathbf{E}\left[g_{\psi,x}\left(\left\Vert
    \sum_{k = 2}^K\frac{\sqrt{\mathcal{T}_\circ (B_1 - B)}}{V_k\sqrt{d}}\sum_{q = 1}^R A_{q,k} + (K - 1)\frac{\sqrt{\mathcal{T}_\circ (B_1 - B)}}{V_1\sqrt{d}}\sum_{q = 1}^R A_{q,1}
    \right\Vert_2 Z\right)\right]\right.\\
    &- \mathbf{E}\left[g_{\psi,x}\left(\left\Vert
    \sum_{k = 2}^K\frac{\sqrt{\mathcal{T}_\circ (B_1 - B)}}{V_k\sqrt{d}}\sum_{B\leq \vert t_1 - t_2\vert\leq B_1}^{T_k}\left(\boldsymbol{\epsilon}_{t_1,k}^\top\boldsymbol{\epsilon}_{t_2,k} - \mathbf{E}\left[\boldsymbol{\epsilon}_{t_1,k}^\top\boldsymbol{\epsilon}_{t_2,k}\right]\right)\right.\right.\right.\\
    &\left.\left.\left.\left.+ (K - 1)\frac{\sqrt{\mathcal{T}_\circ(B_1 - B)}}{V_1\sqrt{d}}\sum_{B\leq \vert t_1 - t_2\vert\leq B_1}^{T_1}\left(\boldsymbol{\epsilon}_{t_1,1}^\top\boldsymbol{\epsilon}_{t_2,1} - \mathbf{E}\left[\boldsymbol{\epsilon}_{t_1,1}^\top\boldsymbol{\epsilon}_{t_2,1}\right]\right)
    \right\Vert_{2} Z\right)\right]
    \right\vert\\
    & = O\left(\psi\sqrt{\frac{\ell}{v}} + \frac{\psi}{\mathcal{T}_\circ^{3/2}\sqrt{(B_1 - B)}}\right)
\end{aligned}
\label{eq.A_star_to_zeta}
\end{equation}
Choose $\psi = \log^3(T).$  From \eqref{eq.prob_to_g}, \eqref{eq.g_to_sum_epsilon}, \eqref{eq.e_to_truncate}, \eqref{eq.truncate_to_big_sum},  \eqref{eq.A_to_A_star}, \eqref{eq.A_star_to_zeta}, we have 
\begin{equation}
    \begin{aligned}
        \sup_{x\in\mathbf{R}}\left\vert
        \mathbf{Pr}\left(\sqrt{\mathcal{T}_\circ(B_1- B)}\widehat{R}\leq x\right) - \mathbf{Pr}\left(\zeta \leq x\right)
    \right\vert\\
     = O\left(\frac{1}{\psi} + \frac{\psi B_1}{\sqrt{\mathcal{T}_\circ}} + \frac{\psi}{\mathcal{T}_\circ^{3/2}\sqrt{(B_1 - B)}} + \psi\sqrt{\frac{\ell}{v}} + \psi^3\sqrt{\frac{v}{\mathcal{T}_\circ}}\right) = o(1),
    \end{aligned}
\end{equation}
which proves \eqref{eq.prob_res}.
\end{proof}

\begin{corollary}[Further discussion of Remark \ref{remark.power_of_test}]
\label{corollary.power_performance}
     With any given threshold $\mathcal{\gamma} > 0,$ we have 
    \begin{align*}
        & \mathbf{Pr}\left(
        \sqrt{\mathcal{T}_\circ (B_1- B)}\widehat{R} \leq \mathcal{\gamma}
        \right)\\
        & =  \mathbf{Pr}\left(
        \widehat{R} - \frac{1}{\sqrt{d}}\sum_{k = 2}^K\vert \boldsymbol{\mu}_k  - \boldsymbol{\mu}_1\vert_2^2\leq \frac{\mathcal{\gamma}}{\sqrt{\mathcal{T}_\circ (B_1- B)}}- \frac{1}{\sqrt{d}}\sum_{k = 2}^K\vert \boldsymbol{\mu}_k  - \boldsymbol{\mu}_1\vert_2^2
        \right).
    \end{align*}
    We notice that $d\asymp \mathcal{T}_\circ,$ if the condition 
    \begin{align*}
        1 = o\left(\sum_{k = 2}^K\vert \boldsymbol{\mu}_k  - \boldsymbol{\mu}_1\vert_2^2\right)
    \end{align*}
    holds true, then for sufficiently large $\mathcal{T}_\circ,$
    \begin{align*}
        \frac{\mathcal{\gamma}}{\sqrt{\mathcal{T}_\circ (B_1- B)}} - \frac{1}{\sqrt{d}}\sum_{k = 2}^K\vert \boldsymbol{\mu}_k  - \boldsymbol{\mu}_1\vert_2^2 < 0,
    \end{align*}
    and  we have 
    \begin{align*}
        & \mathbf{Pr}\left(
        \sqrt{\mathcal{T}_\circ (B_1- B)}\widehat{R} \leq \mathcal{\gamma}
        \right)\\
        &\leq \mathbf{Pr}\left(
        \left\vert\ \widehat{R} - \frac{1}{\sqrt{d}}\sum_{k = 2}^K\vert \boldsymbol{\mu}_k  - \boldsymbol{\mu}_1\vert_2^2\ \right\vert\geq \left\vert\ \frac{\mathcal{\gamma}}{\sqrt{\mathcal{T}_\circ (B_1- B)}}- \frac{1}{\sqrt{d}}\sum_{k = 2}^K\vert \boldsymbol{\mu}_k  - \boldsymbol{\mu}_1\vert_2^2\ \right\vert
        \right)\\
        &\leq \frac{\mathbf{E}\left\vert\ \widehat{R} - \frac{1}{\sqrt{d}}\sum_{k = 2}^K\vert \boldsymbol{\mu}_k  - \boldsymbol{\mu}_1\vert_2^2\ \right\vert^{M/2}}{\left\vert\ \frac{\mathcal{\gamma}}{\sqrt{\mathcal{T}_\circ (B_1- B)}}- \frac{1}{\sqrt{d}}\sum_{k = 2}^K\vert \boldsymbol{\mu}_k  - \boldsymbol{\mu}_1\vert_2^2\ \right\vert^{M/2}}\\
        & \leq \frac{C}{\left(\sqrt{\mathcal{T}_\circ}\times \left(\frac{1}{\sqrt{d}}\sum_{k = 2}^K\vert \boldsymbol{\mu}_k  - \boldsymbol{\mu}_1\vert_2^2 - \frac{\mathcal{\gamma}}{\sqrt{\mathcal{T}_\circ (B_1- B)}}\right)\right)^{M/2}},
    \end{align*}
    which tends to 0 as $\mathcal{T}_\circ$ tends to infinity, making the power of the test tending to 1. 
\end{corollary}

Since the normal distribution has a bounded density, the validity of Algorithm \ref{algorithm.bootstrap} depends on the consistent estimation of the variance of $\widehat{R},$ as established in the following proof.

\begin{proof}[Proof of Theorem \ref{Theorem.bootstrap}]
Define the conditional mean and variance as follows: 
$$\mathbf{E}^*\left[\cdot\right] = \mathbf{E}\left[\cdot\mid\mathbf{x}_{t,k}, k = 1,\cdots, K, t = 1,\cdots, T_k\right],\quad \text{and}\quad \mathrm{Var}^*(x) = \mathbf{E}^*\left[x^2\right] - (\mathbf{E}^*\left[x\right])^2.
$$
We also call such mean and variance ``the expectation and the variance in the bootstrap world.'' From Algorithm \ref{algorithm.bootstrap}, $\widehat{S}^*_b$ has normal distribution in the bootstrap world, with mean
\begin{align*}
    \mathbf{E}^*\left[\widehat{S}^*_b\right] &=  2\sum_{k = 2}^K\frac{\sqrt{\mathcal{T}_\circ (B_1 - B)}}{V_k\sqrt{d}}\sum_{t = B+1}^{T_k}
    \widehat{\vartheta}_{t,k}\mathbf{E}\left[\varepsilon^*_{t,k}\right]\\
    &+ 2(K - 1)\frac{\sqrt{\mathcal{T}_\circ (B_1 - B)}}{V_1\sqrt{d}}\sum_{t = B+1}^{T_k}\widehat{\vartheta}_{t,1}\mathbf{E}\left[\varepsilon^*_{t,1}\right] = 0,
\end{align*}
and variance 
\begin{align*}
    \mathrm{Var}^*\left(
    \widehat{S}^*_b
    \right) &= 4\sum_{k = 2}^K \frac{\mathcal{T}_\circ (B_1 - B)}{V_k^2d} 
    \mathrm{Var}^*\left(
    \sum_{t = B+1}^{T_k}
    \widehat{\vartheta}_{t,k}\varepsilon^*_{t,k}
    \right)\\
    &+ 4(K - 1)^2\frac{\mathcal{T}_\circ (B_1 - B)}{V_1^2d} \mathrm{Var}^*\left(\sum_{t = B+1}^{T_k}\widehat{\vartheta}_{t,1}\varepsilon^*_{t,1}\right).
\end{align*}
For any $k = 1,2,\cdots, K,$ from Step 4 of Algorithm \ref{algorithm.bootstrap},
\begin{align*}
     \mathrm{Var}^*\left(
    \sum_{t = B+1}^{T_k}
    \widehat{\vartheta}_{t,k}\varepsilon^*_{t,k}
    \right) = \sum_{t_1 = B+1}^{T_k}\sum_{t_2 = B+1}^{T_k} \mathcal{K}\left(\frac{t_1 - t_2}{H}\right)\widehat{\vartheta}_{t_1,k}\widehat{\vartheta}_{t_2,k}.
\end{align*}
Since 
\begin{align*}
&\left\Vert
    \frac{4\mathcal{T}_\circ(B_1 - B)}{V^2_k d}\sum_{t_1 = B+1}^{T_k}\sum_{t_2 = B+1}^{T_k}\widehat{\vartheta}_{t_1,k}\widehat{\vartheta}_{t_2,k} \mathcal{K}\left(\frac{t_1 - t_2}{H}\right)\right.\\
&\left.- \mathrm{Var}\left(\frac{\sqrt{\mathcal{T}_\circ(B_1-  B)}}{V_k\sqrt{d}}\sum_{B\leq \vert t_1 - t_2\vert\leq B_1}^{T_k}\left(\boldsymbol{\epsilon}_{t_1,k}^\top\boldsymbol{\epsilon}_{t_2,k} - \mathbf{E}\left[\boldsymbol{\epsilon}_{t_1,k}^\top\boldsymbol{\epsilon}_{t_2,k}\right]\right)\right)
\right\Vert_{M/4}\\
&\leq  \frac{4\mathcal{T}_\circ (B_1 - B)}{V^2_k d}\left\Vert
    \sum_{t_1 = B+1}^{T_k}\sum_{t_2 = B+1}^{T_k}\left(\widehat{\vartheta}_{t_1,k}\widehat{\vartheta}_{t_2,k} - \vartheta_{t_1,k}\vartheta_{t_2,k}\right)\mathcal{K}\left(\frac{t_1 - t_2}{H}\right)
    \right\Vert_{M/4}\\
&+ \left\Vert
    \frac{4\mathcal{T}_\circ (B_1 - B)}{V^2_k d}\sum_{t_1 = B+1}^{T_k}\sum_{t_2 = B+1}^{T_k}\vartheta_{t_1,k}\vartheta_{t_2,k}\mathcal{K}\left(\frac{t_1 - t_2}{H}\right)\right.\\
&\left.- \mathrm{Var}\left(\frac{\sqrt{\mathcal{T}_\circ(B_1-  B)}}{V_k\sqrt{d}}\sum_{B\leq \vert t_1 - t_2\vert\leq B_1}^{T_k}\left(\boldsymbol{\epsilon}_{t_1,k}^\top\boldsymbol{\epsilon}_{t_2,k} - \mathbf{E}\boldsymbol{\epsilon}_{t_1,k}^\top\boldsymbol{\epsilon}_{t_2,k}\right)\right)
    \right\Vert_{M/4},
\end{align*}
where we define 
$$
\vartheta_{t,k} = \sum_{t_2 = (t  - B_1)\vee 1}^{t  - B}\boldsymbol{\epsilon}_{t,k}^\top\boldsymbol{\epsilon}_{t_2,k}.
$$
From Theorem \ref{theorem.var_estimation}, 
\begin{align*}
    &\left\Vert
    \frac{4\mathcal{T}_\circ(B_1 - B)}{V^2_k d}\sum_{t_1 = B+1}^{T_k}\sum_{t_2 = B+1}^{T_k}\vartheta_{t_1,k}\vartheta_{t_2,k}\mathcal{K}\left(\frac{t_1 - t_2}{H}\right)\right.\\
    &\left.- \mathrm{Var}\left(\frac{\sqrt{\mathcal{T}_\circ(B_1-  B)}}{V_k\sqrt{d}}\sum_{B\leq \vert t_1 - t_2\vert\leq B_1}^{T_k}\left(\boldsymbol{\epsilon}_{t_1,k}^\top\boldsymbol{\epsilon}_{t_2,k} - \mathbf{E}\left[\boldsymbol{\epsilon}_{t_1,k}^\top\boldsymbol{\epsilon}_{t_2,k}\right]\right)\right)
    \right\Vert_{M/4}\\
    &= O\left(\frac{\log(\mathcal{T}_\circ)}{\mathcal{T}_\circ^{1/6 - \kappa_2}} + \frac{1}{\mathcal{T}_\circ^{\kappa_2 / 2}}\right).
\end{align*}
On the other hand, 
\begin{align*}
    &\left\Vert
    \sum_{t_1 = B+1}^{T_k}\sum_{t_2 = B+1}^{T_k}\left(\widehat{\vartheta}_{t_1,k}\widehat{\vartheta}_{t_2,k} - \vartheta_{t_1,k}\vartheta_{t_2,k}\right)\mathcal{K}\left(\frac{t_1 - t_2}{H}\right)
    \right\Vert_{M/4}\\
    &\leq 2\sum_{u = 0}^{T_k - B - 1}\mathcal{K}\left(\frac{u}{H}\right)\left\Vert\sum_{t_1 = B+1}^{T_k - u}\left(\widehat{\vartheta}_{t_1,k}\widehat{\vartheta}_{t_1 + u,k} - \vartheta_{t_1,k}\vartheta_{t_1 + u,k}\right)\right\Vert_{M/4}\\
    &\leq 2\sum_{u = 0}^{T_k - B - 1}\mathcal{K}\left(\frac{u}{H}\right)\left\Vert\sum_{t_1 = B+1}^{T_k - u}\vartheta_{t_1,k}\left(\widehat{\vartheta}_{t_1 + u,k} - \vartheta_{t_1 + u,k}\right)\right\Vert_{M/4}\\
    &+ 2\sum_{u = 0}^{T_k - B - 1}\mathcal{K}\left(\frac{u}{H}\right)\left\Vert\sum_{t_1 = B+1}^{T_k - u}\vartheta_{t_1 + u,k}\left(\widehat{\vartheta}_{t_1,k} - \vartheta_{t_1,k}\right)\right\Vert_{M/4}\\
    &+ 2\sum_{u = 0}^{T_k - B - 1}\mathcal{K}\left(\frac{u}{H}\right)\left\Vert\sum_{t_1 = B+1}^{T_k - u}\left(\widehat{\vartheta}_{t_1 + u,k} - \vartheta_{t_1 + u, k}\right)\left(\widehat{\vartheta}_{t_1,k} - \vartheta_{t_1,k}\right)\right\Vert_{M/4}.
\end{align*}
From \eqref{eq.moment_iota} and \eqref{eq.bias_theta},
\begin{equation}
\begin{aligned}
    \left\Vert\vartheta_{t,k}\right\Vert_{M/2}&\leq 
    \left\Vert
    \sum_{t_2 = (t  - B_1)\vee 1}^{t  - B}\left(\boldsymbol{\epsilon}_{t,k}^\top\boldsymbol{\epsilon}_{t_2,k} - \mathbf{E}\left[\boldsymbol{\epsilon}_{t,k}^\top\boldsymbol{\epsilon}_{t_2,k}\right]\right)
    \right\Vert_{M/2} + \left\vert\sum_{t_2 = (t  - B_1)\vee 1}^{t  - B}\mathbf{E}\left[\boldsymbol{\epsilon}_{t,k}^\top\boldsymbol{\epsilon}_{t_2,k}\right]\right\vert\\
    &\leq C\sqrt{d(B_1 - B)} + \frac{Cd}{B^{\alpha - 1}}\leq C_1\sqrt{\mathcal{T}_\circ(B_1 - B)}.
\end{aligned}
\label{eq.size_vartheta}
\end{equation}
Define 
$$
\overline{\boldsymbol{\epsilon}}_k = \frac{1}{T_k}\sum_{t = 1}^{T_k}\boldsymbol{\epsilon}_{t,k},
$$ 
then 
\begin{align*}
    \left\Vert
    \widehat{\vartheta}_{t,k} - \vartheta_{t, k}
    \right\Vert_{M/2} &= \left\Vert
    \sum_{t_2 = (t - B_1)\vee 1}^{t - B}\left(\left(\boldsymbol{\epsilon}_{t,k} - \overline{\boldsymbol{\epsilon}}_k\right)^\top\left(\boldsymbol{\epsilon}_{t_2,k} - \overline{\boldsymbol{\epsilon}}_k\right) - \boldsymbol{\epsilon}_{t,k}^\top \boldsymbol{\epsilon}_{t_2,k}\right)
    \right\Vert_{M/2}\\
    &\leq 
    \left\Vert
    \overline{\boldsymbol{\epsilon}}_k^\top\sum_{t_2 = (t - B_1)\vee 1}^{t - B}\boldsymbol{\epsilon}_{t_2,k}
    \right\Vert_{M/2} + (B_1 - B + 1)\left\Vert
    \boldsymbol{\epsilon}_{t,k}^\top\overline{\boldsymbol{\epsilon}}_k
    \right\Vert_{M/2}\\
    & + (B_1 - B + 1)\left\Vert\overline{\boldsymbol{\epsilon}}_k^\top\overline{\boldsymbol{\epsilon}}_k\right\Vert_{M/2}.
\end{align*}
Notice that 
\begin{align*}
    \left\Vert
    \overline{\boldsymbol{\epsilon}}_k^\top\sum_{t_2 = (t - B_1)\vee 1}^{t - B}\boldsymbol{\epsilon}_{t_2,k}
    \right\Vert_{M/2} &= \frac{1}{T_k}
    \left\Vert
    \sum_{t_1 = 1}^{T_k}\sum_{t_2 = (t - B_1)\vee 1}^{t - B}\boldsymbol{\epsilon}_{t_1,k}^\top\boldsymbol{\epsilon}_{t_2,k}
    \right\Vert_{M/2}\\
    &\leq \frac{1}{T_k}
    \left\Vert
    \sum_{t_2 = (t - B_1)\vee 1}^{t - B}\sum_{t_1 = 1}^{t_2 - B - 1}\boldsymbol{\epsilon}_{t_1,k}^\top\boldsymbol{\epsilon}_{t_2,k} - \mathbf{E}\left[\boldsymbol{\epsilon}_{t_1,k}^\top\boldsymbol{\epsilon}_{t_2,k}\right]
    \right\Vert_{M/2}\\
    &+ \frac{1}{T_k}\sum_{t_2 = (t - B_1)\vee 1}^{t - B}\sum_{t_1 = 1}^{t_2 - B - 1}\left\vert \mathbf{E}\left[\boldsymbol{\epsilon}_{t_1,k}^\top\boldsymbol{\epsilon}_{t_2,k}\right]\right\vert\\
    & + \frac{1}{T_k}
    \left\Vert
    \sum_{t_2 = (t - B_1)\vee 1}^{t - B}\sum_{t_1 =  (t_2 - B)\vee 1}^{(t_2 + B)\wedge T_k}\boldsymbol{\epsilon}_{t_1,k}^\top\boldsymbol{\epsilon}_{t_2,k}
    \right\Vert_{M/2}\\
    & + \frac{1}{T_k}
    \left\Vert
    \sum_{t_2 = (t - B_1)\vee 1}^{t - B}\sum_{t_1 = t_2 + B + 1}^{T_k}
    \boldsymbol{\epsilon}_{t_1,k}^\top\boldsymbol{\epsilon}_{t_2,k} - \mathbf{E}\left[\boldsymbol{\epsilon}_{t_1,k}^\top\boldsymbol{\epsilon}_{t_2,k}\right]
    \right\Vert_{M/2}\\
    & + \frac{1}{T_k}\sum_{t_2 = (t - B_1)\vee 1}^{t - B}\sum_{t_1 = t_2 +B + 1}^{T_k}\left\vert \mathbf{E}\left[\boldsymbol{\epsilon}_{t_1,k}^\top\boldsymbol{\epsilon}_{t_2,k}\right]\right\vert.
\end{align*}
From Theorem \ref{theorem.consistent_quadratic}, 
\begin{align*}
    &\frac{1}{T_k}
    \left\Vert
    \sum_{t_2 = (t - B_1)\vee 1}^{t - B}\sum_{t_1 = 1}^{t_2 - B - 1}\boldsymbol{\epsilon}_{t_1,k}^\top\boldsymbol{\epsilon}_{t_2,k} - \mathbf{E}\left[\boldsymbol{\epsilon}_{t_1,k}^\top\boldsymbol{\epsilon}_{t_2,k}\right]
    \right\Vert_{M/2}\\
    &\leq \frac{C}{T_k}\sqrt{d\sum_{t_2 = (t - B_1)\vee 1}^{t - B}\sum_{t_1 = 1}^{t_2 - B - 1}1^2} + \frac{CT_k^{3/2}}{B^\alpha} + \frac{CT_k^{1/2}}{B^{\alpha - 2}}\\
    &\leq C_1\sqrt{(B_1 -  B)}.
\end{align*}
From \eqref{eq.covariance},
\begin{align*}
   \frac{1}{T_k}\sum_{t_2 = (t - B_1)\vee 1}^{t - B}\sum_{t_1 = 1}^{t_2 - B - 1}\left\vert \mathbf{E}\left[\boldsymbol{\epsilon}_{t_1,k}^\top\boldsymbol{\epsilon}_{t_2,k}\right]\right\vert&\leq  \frac{Cd}{T_k}\sum_{t_2 = (t - B_1)\vee 1}^{t - B}\sum_{t_1 = 1}^{t_2 - B - 1}\frac{1}{(1  + t_2 - t_1)^\alpha}\\
   &\leq \frac{C_1(B_1 - B)}{B^{\alpha - 1}}.
\end{align*}
For any $p < q$ and $k = 1,\cdots,K,$ notice that 
\begin{align*}
    \left\Vert
    \boldsymbol{\epsilon}_{p,k}^\top\boldsymbol{\epsilon}_{q,k}
    \right\Vert_{M/2} &\leq\left\Vert\boldsymbol{\epsilon}_{p,k}^\top\mathbf{E}\left[\boldsymbol{\epsilon}_{q,k}\mid\mathcal{F}^{(k)}_{q,q - p - 1}\right]
    \right\Vert_{M/2}\\
    & + \left\Vert\boldsymbol{\epsilon}_{p,k}^\top\left(\boldsymbol{\epsilon}_{q,k} - \mathbf{E}\left[\boldsymbol{\epsilon}_{q,k}\mid\mathcal{F}^{(k)}_{q,q - p - 1}\right]\right)\right\Vert_{M/2}.
\end{align*}
Since $\mathbf{E}\left[\boldsymbol{\epsilon}_{q,k}\mid\mathcal{F}^{(k)}_{q,q - p - 1}\right]$ is independent of $\boldsymbol{\epsilon}_{p,k},$ from Definition \ref{def.M_alpha_short_range} and Theorem 1.7 of  \cite{MR2002723}, for any $\boldsymbol{\tau}\in\mathbf{R}^d,$ 
\begin{align*}
    \mathbf{E}\left[\left\vert \boldsymbol{\epsilon}_{p,k}^\top\mathbf{E}\left[\boldsymbol{\epsilon}_{q,k}\mid\mathcal{F}^{(k)}_{q,q - p - 1}\right]\right\vert^{M/2}\mid \mathbf{E}\left[\boldsymbol{\epsilon}_{q,k}\mid\mathcal{F}^{(k)}_{q,q - p - 1}\right] = \boldsymbol{\tau}\right] &= \mathbf{E}\left[\left\vert \boldsymbol{\epsilon}_{p,k}^\top\boldsymbol{\tau}\right\vert^{M/2}\right]\\
    &\leq C\left\vert \boldsymbol{\tau}\right\vert_2^{M/2}\\
    & = C\left\vert \mathbf{E}\left[\boldsymbol{\epsilon}_{q,k}\mid\mathcal{F}^{(k)}_{q,q - p - 1}\right]\right\vert_2^{M/2},
\end{align*}
making 
\begin{align*}
    \left\Vert\boldsymbol{\epsilon}_{p,k}^\top\mathbf{E}\left[\boldsymbol{\epsilon}_{q,k}\mid\mathcal{F}^{(k)}_{q,q - p - 1}\right]
    \right\Vert_{M/2}&\leq C\left\Vert\ \left\vert \mathbf{E}\left[\boldsymbol{\epsilon}_{q,k}\mid\mathcal{F}^{(k)}_{q,q - p - 1}\right]\right\vert_2\ \right\Vert_{M/2}\\
    &\leq C\sqrt{\sum_{j = 1}^d \left\Vert \mathbf{E}\left[\boldsymbol{\epsilon}_{q,k}^{(j)}\mid\mathcal{F}^{(k)}_{q,q - p - 1}\right]\right\Vert_{M/2}^2}\leq C_1\sqrt{d}.
\end{align*}
On the other hand,
\begin{align*}
    \left\Vert\boldsymbol{\epsilon}_{p,k}^\top\left(\boldsymbol{\epsilon}_{q,k} - \mathbf{E}\left[\boldsymbol{\epsilon}_{q,k}\mid\mathcal{F}^{(k)}_{q,q - p - 1}\right]\right)\right\Vert_{M/2} &\leq \sum_{j = 1}^d\left\Vert\boldsymbol{\epsilon}_{p,k}^{(j)}\left(\boldsymbol{\epsilon}_{q,k}^{(j)} - \mathbf{E}\left[\boldsymbol{\epsilon}_{q,k}^{(j)}\mid\mathcal{F}^{(k)}_{q,q - p - 1}\right]\right)\right\Vert_{M/2}\\
    &\leq \sum_{j = 1}^d \left\Vert \boldsymbol{\epsilon}_{p,k}^{(j)}\right\Vert_M\left\Vert  \boldsymbol{\epsilon}_{q,k}^{(j)} - \mathbf{E}\left[\boldsymbol{\epsilon}_{q,k}^{(j)}\mid\mathcal{F}^{(k)}_{q,q - p - 1}\right]\right\Vert_M\\
    &\leq \frac{Cd}{(q - p)^\alpha}.
\end{align*}
If $p = q,$ then 
\begin{align*}
    \left\Vert
    \boldsymbol{\epsilon}_{p,k}^\top\boldsymbol{\epsilon}_{p,k}
    \right\Vert_{M/2}\leq \sum_{j = 1}^d\left\Vert
    \boldsymbol{\epsilon}_{p,k}^{(j)}\boldsymbol{\epsilon}_{p,k}^{(j)}
    \right\Vert_{M/2}\leq Cd.
\end{align*}
Therefore, for any $p,q,$
\begin{equation}
\begin{aligned}
    \left\Vert
    \boldsymbol{\epsilon}_{p,k}^\top\boldsymbol{\epsilon}_{q,k}
    \right\Vert_{M/2}\leq C\sqrt{d} + \frac{Cd}{(1 + \vert q - p\vert )^\alpha},
\end{aligned}
\label{eq.moment}
\end{equation}
and 
\begin{align*}
    \frac{1}{T_k}
    \left\Vert
    \sum_{t_2 = (t - B_1)\vee 1}^{t - B}\sum_{t_1 = (t_2 - B)\vee 1}^{(t_2 + B)\wedge T_k}\boldsymbol{\epsilon}_{t_1,k}^\top\boldsymbol{\epsilon}_{t_2,k}
    \right\Vert_{M/2} &\leq \frac{1}{T_k}\sum_{t_2 = (t - B_1)\vee 1}^{t - B}\sum_{t_1 = (t_2 - B)\vee 1}^{(t_2 + B)\wedge T_k}\left\Vert \boldsymbol{\epsilon}_{t_1,k}^\top\boldsymbol{\epsilon}_{t_2,k}\right\Vert_{M/2}\\
    &\leq \frac{CB(B_1 - B)\sqrt{d}}{T_k}\\
    &+ \frac{Cd}{T_k}\sum_{t_2 = (t - B_1)\vee 1}^{t - B}\sum_{t_1 = (t_2 - B)\vee 1}^{(t_2 + B)\wedge T_k}\frac{1}{(1 + \vert t_1 - t_2\vert )^\alpha}\\
    &\leq C_1\left(B_1 - B\right).
\end{align*}
From Theorem \ref{theorem.consistent_quadratic},
\begin{align*}
    &\frac{1}{T_k}
    \left\Vert
    \sum_{t_2 = (t - B_1)\vee 1}^{t - B}\sum_{t_1 = t_2 + B + 1}^{T_k}
    \boldsymbol{\epsilon}_{t_1,k}^\top\boldsymbol{\epsilon}_{t_2,k} - \mathbf{E}\left[\boldsymbol{\epsilon}_{t_1,k}^\top\boldsymbol{\epsilon}_{t_2,k}\right]
    \right\Vert_{M/2}\\
    & \leq \frac{C}{T_k}\sqrt{d\sum_{t_2 = (t - B_1)\vee 1}^{t - B}\sum_{t_1 = t_2 + B + 1}^{T_k}1^2} + \frac{CT_k^{3/2}}{B^\alpha} + \frac{CT_k^{1/2}}{B^{\alpha - 2}}\\
    &\leq C_1\sqrt{(B_1 - B)}.
\end{align*}
From \eqref{eq.covariance},
\begin{align*}
    \frac{1}{T_k}\sum_{t_2 = (t - B_1)\vee 1}^{t - B}\sum_{t_1 = t_2 + B + 1}^{T_k}\left\vert \mathbf{E}\left[\boldsymbol{\epsilon}_{t_1,k}^\top\boldsymbol{\epsilon}_{t_2,k}\right]\right\vert
    &\leq \frac{Cd}{T_k}\sum_{t_2 = (t - B_1)\vee 1}^{t - B}\sum_{t_1 = t_2 + B + 1}^{T_k}\frac{1}{(1 + t_1 - t_2)^\alpha}\\
    &\leq \frac{Cd(B_1 - B)}{T_k B^{\alpha - 1}} \leq \frac{C_1(B_1 - B)}{B^{\alpha - 1}}.
\end{align*}
From these observations, we have
\begin{equation}
\begin{aligned}
    \left\Vert
    \overline{\boldsymbol{\epsilon}}_k^\top\sum_{t_2 = (t - B_1)\vee 1}^{t - B}\boldsymbol{\epsilon}_{t_2,k}
    \right\Vert_{M/2}\leq C\sqrt{(B_1 - B)} + \frac{CB_1}{B^{\alpha - 1}} + C(B_1 - B) \leq C_1B_1.
\end{aligned}
\label{eq.delta_first_term}
\end{equation}
Notice that 
\begin{align*}
    \left\Vert
    \boldsymbol{\epsilon}_{t,k}^\top\overline{\boldsymbol{\epsilon}}_k
    \right\Vert_{M/2} &\leq \frac{1}{T_k}\left\Vert \sum_{t_1 = 1}^{t - B-1}\boldsymbol{\epsilon}_{t,k}^\top\boldsymbol{\epsilon}_{t_1,k} - \mathbf{E}\left[\boldsymbol{\epsilon}_{t,k}^\top\boldsymbol{\epsilon}_{t_1,k}\right]\right\Vert_{M/2}
     + \frac{1}{T_k}\sum_{t_1 = 1}^{t - B-1}\left\vert\mathbf{E}\left[\boldsymbol{\epsilon}_{t,k}^\top\boldsymbol{\epsilon}_{t_1,k}\right]\right\vert\\
     & + \frac{1}{T_k}\left\Vert \sum_{t_1 = (t - B)\vee 1}^{(t + B)\wedge T_k}\boldsymbol{\epsilon}_{t,k}^\top\boldsymbol{\epsilon}_{t_1,k}\right\Vert_{M/2}
     + \frac{1}{T_k}\left\Vert \sum_{t_1 = t + B +1}^{T_k}\boldsymbol{\epsilon}_{t,k}^\top\boldsymbol{\epsilon}_{t_1,k} - \mathbf{E}\left[\boldsymbol{\epsilon}_{t,k}^\top\boldsymbol{\epsilon}_{t_1,k}\right]\right\Vert_{M/2}\\
     & + \frac{1}{T_k}\sum_{t_1 = t + B +1}^{T_k}\left\vert \mathbf{E}\left[\boldsymbol{\epsilon}_{t,k}^\top\boldsymbol{\epsilon}_{t_1,k}\right]\right\vert.
\end{align*}
From Theorem \ref{theorem.consistent_quadratic},
\begin{align*}
    \frac{1}{T_k}\left\Vert \sum_{t_1 = 1}^{t - B-1}\boldsymbol{\epsilon}_{t,k}^\top\boldsymbol{\epsilon}_{t_1,k} - \mathbf{E}\left[\boldsymbol{\epsilon}_{t,k}^\top\boldsymbol{\epsilon}_{t_1,k}\right]\right\Vert_{M/2}&\leq \frac{C}{T_k}\sqrt{d\sum_{t_1 = 1}^{t - B-1} 1^2} + \frac{CT_k^{3/2}}{B^\alpha} + \frac{CT_k^{1/2}}{B^{\alpha - 2}}\\
    &\leq C_1,
\end{align*}
and 
\begin{align*}
    \frac{1}{T_k}\left\Vert \sum_{t_1 = t+B+1}^{T_k}\boldsymbol{\epsilon}_{t,k}^\top\boldsymbol{\epsilon}_{t_1,k} - \mathbf{E}\left[\boldsymbol{\epsilon}_{t,k}^\top\boldsymbol{\epsilon}_{t_1,k}\right]\right\Vert_{M/2}&\leq \frac{C}{T_k}\sqrt{d\sum_{t_1 = t+B+1}^{T_k}1^2} + \frac{CT_k^{3/2}}{B^\alpha} + \frac{CT_k^{1/2}}{B^{\alpha - 2}}\\
    &\leq C_1.
\end{align*}
From \eqref{eq.covariance}, 
\begin{align*}
    &\frac{1}{T_k}\sum_{t_1 = 1}^{t - B - 1}\left\vert \mathbf{E}\left[\boldsymbol{\epsilon}_{t,k}^\top\boldsymbol{\epsilon}_{t_1,k}\right]\right\vert 
    + \frac{1}{T_k}\sum_{t_1 = t+B+1}^{T_k}\left\vert \mathbf{E}\left[\boldsymbol{\epsilon}_{t,k}^\top\boldsymbol{\epsilon}_{t_1,k}\right]\right\vert\\
    &\leq \frac{Cd}{T_k}\sum_{t_1 = 1}^{t - B - 1}\frac{1}{(1 + t - t_1)^\alpha} + \frac{Cd}{T_k}\sum_{t_1 = t + B +1}^{T_k}\frac{1}{(1 + t_1 - t)^\alpha}\\
    &\leq \frac{C_1d}{T_kB^{\alpha - 1}} \leq \frac{C_2}{B^{\alpha - 1}}.
\end{align*}
From \eqref{eq.moment},
\begin{align*}
    \frac{1}{T_k}\left\Vert \sum_{t_1 = (t - B)\vee 1}^{(t + B)\wedge T_k}\boldsymbol{\epsilon}_{t,k}^\top\boldsymbol{\epsilon}_{t_1,k}\right\Vert_{M/2} &\leq \frac{1}{T_k}\sum_{t_1 = (t - B)\vee 1 }^{(t + B)\wedge T_k}\left\Vert \boldsymbol{\epsilon}_{t,k}^\top\boldsymbol{\epsilon}_{t_1,k}\right\Vert_{M/2}\\
    &\leq \frac{CB\sqrt{d}}{T_k} + \frac{Cd}{T_k}\sum_{t_1 = t - B}^{t + B}\frac{1}{(1 + \vert t_1 - t\vert)^\alpha}\leq C_1.
\end{align*}
Therefore,
\begin{equation}
\begin{aligned}
    \left\Vert
    \boldsymbol{\epsilon}_{t,k}^\top\overline{\boldsymbol{\epsilon}}_k
    \right\Vert_{M/2}\leq C.
\end{aligned}
\label{eq.delta_second_term}
\end{equation}
Notice that 
\begin{align*}
\left\Vert\overline{\boldsymbol{\epsilon}}_k^\top\overline{\boldsymbol{\epsilon}}_k\right\Vert_{M/2} &= \frac{1}{T^2_k}
    \left\Vert
    \sum_{t_1 = 1}^{T_k}\sum_{t_2 = 1}^{T_k}\boldsymbol{\epsilon}_{t_1,k}^\top\boldsymbol{\epsilon}_{t_2,k}
    \right\Vert_{M/2}\\
    &\leq \frac{1}{T_k^2}\left\Vert
    \sum_{t_1 = 1}^{T_k}\sum_{t_2 = 1}^{t_1 - B - 1}\left(\boldsymbol{\epsilon}_{t_1,k}^\top\boldsymbol{\epsilon}_{t_2,k} - \mathbf{E}\left[\boldsymbol{\epsilon}_{t_1,k}^\top\boldsymbol{\epsilon}_{t_2,k}\right]\right)
    \right\Vert_{M/2}\\
    &+ \frac{1}{T_k^2}\sum_{t_1 = 1}^{T_k}\sum_{t_2 = 1}^{t_1 - B - 1}\left\vert \mathbf{E}\left[\boldsymbol{\epsilon}_{t_1,k}^\top\boldsymbol{\epsilon}_{t_2,k}\right]\right\vert
    + \frac{1}{T_k^2}\left\Vert
    \sum_{t_1 = 1}^{T_k}\sum_{t_2 = (t_1 - B)\vee 1}^{(t_1 + B)\wedge T_k}\boldsymbol{\epsilon}_{t_1,k}^\top\boldsymbol{\epsilon}_{t_2,k}
    \right\Vert_{M/2}\\
    & + \frac{1}{T_k^2}\left\Vert
    \sum_{t_1 = 1}^{T_k}\sum_{t_2 = t_1+B+1}^{T_k}\left(\boldsymbol{\epsilon}_{t_1,k}^\top\boldsymbol{\epsilon}_{t_2,k} - \mathbf{E}\left[\boldsymbol{\epsilon}_{t_1,k}^\top\boldsymbol{\epsilon}_{t_2,k}\right]\right)
    \right\Vert_{M/2}\\
    & +\frac{1}{T_k^2}\sum_{t_1 = 1}^{T_k}\sum_{t_2 = t_1+B+1}^{T_k}\left\vert\mathbf{E}\left[\boldsymbol{\epsilon}_{t_1,k}^\top\boldsymbol{\epsilon}_{t_2,k}\right]\right\vert.
\end{align*}
From Theorem \ref{theorem.consistent_quadratic},
\begin{align*}
    &\frac{1}{T_k^2}\left\Vert
    \sum_{t_1 = 1}^{T_k}\sum_{t_2 = 1}^{t_1 - B - 1}\left(\boldsymbol{\epsilon}_{t_1,k}^\top\boldsymbol{\epsilon}_{t_2,k} - \mathbf{E}\left[\boldsymbol{\epsilon}_{t_1,k}^\top\boldsymbol{\epsilon}_{t_2,k}\right]\right)
    \right\Vert_{M/2} + \frac{1}{T_k^2}\left\Vert
    \sum_{t_1 = 1}^{T_k}\sum_{t_2 = t_1+B+1}^{T_k}\left(\boldsymbol{\epsilon}_{t_1,k}^\top\boldsymbol{\epsilon}_{t_2,k} - \mathbf{E}\left[\boldsymbol{\epsilon}_{t_1,k}^\top\boldsymbol{\epsilon}_{t_2,k}\right]\right)
    \right\Vert_{M/2}\\
    &\leq \frac{C}{T^2_k}\sqrt{\sum_{t_1 = 1}^{T_k}\sum_{t_2 = 1}^{t_1 - B - 1}1^2} + 
    \frac{C}{T^2_k}\sqrt{\sum_{t_1 = 1}^{T_k}\sum_{t_2 = t_1+B+1}^{T_k}1^2} + \frac{CT_k^{1/2}}{B^\alpha} + \frac{C}{\sqrt{T_k}B^{\alpha - 2}}
    \leq \frac{C_1}{T_k}.
\end{align*}
From \eqref{eq.covariance},
\begin{align*}
    &\frac{1}{T_k^2}\sum_{t_1 = 1}^{T_k}\sum_{t_2 = 1}^{t_1 - B - 1}\left\vert \mathbf{E}\left[\boldsymbol{\epsilon}_{t_1,k}^\top\boldsymbol{\epsilon}_{t_2,k}\right]\right\vert +  \frac{1}{T_k^2}\sum_{t_1 = 1}^{T_k}\sum_{t_2 = t_1+B+1}^{T_k}\left\vert\mathbf{E}\left[\boldsymbol{\epsilon}_{t_1,k}^\top\boldsymbol{\epsilon}_{t_2,k}\right]\right\vert\\
    &\leq \frac{Cd}{T^2_k}\sum_{t_1 = 1}^{T_k}\sum_{t_2 = 1}^{t_1 - B - 1}\frac{1}{(1 + t_1 -  t_2)^\alpha} + \frac{Cd}{T^2_k}\sum_{t_1 = 1}^{T_k}\sum_{t_2 = t_1 + B + 1}^{T_k}\frac{1}{(1 + t_2 -  t_1)^\alpha}
    \leq \frac{C_1d}{T_kB^{\alpha - 1}}.
\end{align*}
From \eqref{eq.moment},
\begin{align*}
    \frac{1}{T_k^2}\left\Vert
    \sum_{t_1 = 1}^{T_k}\sum_{t_2 = (t_1 - B)\vee 1}^{(t_1 + B)\wedge T_k}\boldsymbol{\epsilon}_{t_1,k}^\top\boldsymbol{\epsilon}_{t_2,k}
    \right\Vert_{M/2} &\leq \frac{1}{T_k^2}\sum_{t_1 = 1}^{T_k}\sum_{t_2 = (t_1 - B)\vee 1}^{(t_1 + B)\wedge T_k}\left\Vert \boldsymbol{\epsilon}_{t_1,k}^\top\boldsymbol{\epsilon}_{t_2,k}\right\Vert_{M/2}\\
    &\leq \frac{1}{T_k^2}\sum_{t_1 = 1}^{T_k}\sum_{t_2 = (t_1 - B)\vee 1}^{(t_1 + B)\wedge T_k}\left(C\sqrt{d} + \frac{Cd}{(1 + \vert t_1 - t_2\vert)^\alpha }\right)\\
    &\leq \frac{C_1B\sqrt{d}}{T_k} + \frac{C_1d}{T_k}\leq C_2.
\end{align*}
Therefore, 
\begin{equation}
\begin{aligned}
\left\Vert\overline{\boldsymbol{\epsilon}}_k^\top\overline{\boldsymbol{\epsilon}}_k\right\Vert_{M/2} \leq \frac{C}{T_k} + \frac{Cd}{T_kB^{\alpha - 1}} + C\leq C_1.
\end{aligned}
\label{eq.delta_third_term}
\end{equation}
From \eqref{eq.delta_first_term}, \eqref{eq.delta_second_term}, and \eqref{eq.delta_third_term}, we have
\begin{equation}
    \begin{aligned}
        \left\Vert
    \widehat{\vartheta}_{t,k} - \vartheta_{t, k}
    \right\Vert_{M/2}\leq CB_1.
    \end{aligned}
    \label{eq.delta_vatthteta}
\end{equation}
From  \eqref{eq.size_vartheta} and \eqref{eq.delta_vatthteta},
\begin{align*}
    &\sum_{u = 0}^{T_k - B - 1}\mathcal{K}\left(\frac{u}{H}\right)\left\Vert\sum_{t_1 = B+1}^{T_k - u}\vartheta_{t_1,k}\left(\widehat{\vartheta}_{t_1 + u,k} - \vartheta_{t_1 + u,k}\right)\right\Vert_{M/4}\\
    &\leq \sum_{u = 0}^{T_k - B - 1}\mathcal{K}\left(\frac{u}{H}\right)\sum_{t_1 = B+1}^{T_k - u}\left\Vert \vartheta_{t_1,k}\right\Vert_{M/2}\left\Vert \widehat{\vartheta}_{t_1 + u,k} - \vartheta_{t_1 + u,k}\right\Vert_{M/2}\\
    &\leq C\sum_{u = 0}^{T_k - B - 1}\mathcal{K}\left(\frac{u}{H}\right)T_k\sqrt{\mathcal{T}_\circ(B_1 - B)}B_1
    \leq C_1H\mathcal{T}_\circ^{3/2}B_1^{3/2},
\end{align*}
and 
\begin{align*}
    &\sum_{u = 0}^{T_k - B - 1}\mathcal{K}\left(\frac{u}{H}\right)\left\Vert\sum_{t_1 = B+1}^{T_k - u}\vartheta_{t_1 + u,k}\left(\widehat{\vartheta}_{t_1,k} - \vartheta_{t_1,k}\right)\right\Vert_{M/4}\\
    &\leq \sum_{u = 0}^{T_k - B - 1}\sum_{t_1 = B+1}^{T_k - u}\mathcal{K}\left(\frac{u}{H}\right)\left\Vert\vartheta_{t_1 + u,k}\right\Vert_{M/2}\left\Vert \widehat{\vartheta}_{t_1,k} - \vartheta_{t_1,k}\right\Vert_{M/2}\\
    &\leq C\sum_{u = 0}^{T_k - B - 1}\mathcal{K}\left(\frac{u}{H}\right) T_k\sqrt{\mathcal{T}_\circ(B_1 - B)}B_1\leq C_1H\mathcal{T}_\circ^{3/2}B_1^{3/2}.
\end{align*}
Furthermore, we have 
\begin{align*}
    &\sum_{u = 0}^{T_k - B - 1}\mathcal{K}\left(\frac{u}{H}\right)\left\Vert\sum_{t_1 = B+1}^{T_k - u}\left(\widehat{\vartheta}_{t_1 + u,k} - \vartheta_{t_1 + u, k}\right)\left(\widehat{\vartheta}_{t_1,k} - \vartheta_{t_1,k}\right)\right\Vert_{M/4}\\
    &\leq \sum_{u = 0}^{T_k - B - 1} \sum_{t_1 = B+1}^{T_k - u}\mathcal{K}\left(\frac{u}{H}\right)\left\Vert \widehat{\vartheta}_{t_1 + u,k} - \vartheta_{t_1 + u, k}\right\Vert_{M/2}\left\Vert \widehat{\vartheta}_{t_1,k} - \vartheta_{t_1,k}\right\Vert_{M/2}\\
    &\leq C\sum_{u = 0}^{T_k - B - 1}\mathcal{K}\left(\frac{u}{H}\right) T_kB^2_1\leq C_1H\mathcal{T}_\circ B^2_1.
\end{align*}
From these observations, we have 
\begin{equation}
    \begin{aligned}
        &\frac{4\mathcal{T}_\circ(B_1 - B)}{V^2_k d}\left\Vert
    \sum_{t_1 = B+1}^{T_k}\sum_{t_2 = B+1}^{T_k}\left(\widehat{\vartheta}_{t_1,k}\widehat{\vartheta}_{t_2,k} - \vartheta_{t_1,k}\vartheta_{t_2,k}\right)\mathcal{K}\left(\frac{t_1 - t_2}{H}\right)
    \right\Vert_{M/4}\\
    &\leq \frac{C}{\mathcal{T}_\circ B_1 d} H\mathcal{T}_\circ^{3/2}B^{3/2}_1 + \frac{C}{\mathcal{T}_\circ B_1 d}H\mathcal{T}_\circ B^2_1\leq \frac{C_1H\sqrt{B_1}}{\sqrt{\mathcal{T}_\circ}},
    \end{aligned}
\end{equation}
and 
\begin{equation}
\begin{aligned}
    &\left\Vert
    \frac{4\mathcal{T}_\circ(B_1 - B)}{V^2_k d}\sum_{t_1 = B+1}^{T_k}\sum_{t_2 = B+1}^{T_k}\widehat{\vartheta}_{t_1,k}\widehat{\vartheta}_{t_2,k}\mathcal{K}\left(\frac{t_1 - t_2}{H}\right)\right.\\
&\left.- \mathrm{Var}\left(\frac{\sqrt{\mathcal{T}_\circ(B_1-  B)}}{V_k\sqrt{d}}\sum_{B\leq \vert t_1 - t_2\vert\leq B_1}^{T_k}\left(\boldsymbol{\epsilon}_{t_1,k}^\top\boldsymbol{\epsilon}_{t_2,k} - \mathbf{E}\left[\boldsymbol{\epsilon}_{t_1,k}^\top\boldsymbol{\epsilon}_{t_2,k}\right]\right)\right)
\right\Vert_{M/4}\\
&\leq \frac{C\log(\mathcal{T}_\circ)}{\mathcal{T}_\circ^{1/6 - \kappa_2}} + \frac{C}{\mathcal{T}_\circ^{\kappa_2 / 2}} + \frac{CH\sqrt{B_1}}{\sqrt{\mathcal{T}_\circ}}\\
&\leq \frac{C_1\log(\mathcal{T}_\circ)}{\mathcal{T}_\circ^{1/6 - \kappa_2}} + \frac{C_1}{\mathcal{T}_\circ^{\kappa_2 / 2}}.
\end{aligned}
\label{eq.var_S_b}
\end{equation}
From \eqref{eq.var_S_b}, we have 
\begin{align*}
    &\left\vert
    \mathrm{Var}^*\left(
    \widehat{S}^*_b
    \right) - \mathrm{Var}\left(\zeta\right)
    \right\vert\\
    &\leq \sum_{k = 2}^K\left\vert
    \frac{4\mathcal{T}_\circ (B_1 - B)}{V_k^2d}\sum_{t_1 = B+1}^{T_k}\sum_{t_2  = B+1}^{T_k}\mathcal{K}\left(\frac{t_1 - t_2}{H}\right)\widehat{\vartheta}_{t_1,k}\widehat{\vartheta}_{t_2,k}\right.\\
    &\left.- \mathrm{Var}\left(\frac{\sqrt{\mathcal{T}_\circ(B_1-  B)}}{V_k\sqrt{d}}\sum_{B\leq \vert t_1 - t_2\vert\leq B_1}^{T_k}\left(\boldsymbol{\epsilon}_{t_1,k}^\top\boldsymbol{\epsilon}_{t_2,k} - \mathbf{E}\left[\boldsymbol{\epsilon}_{t_1,k}^\top\boldsymbol{\epsilon}_{t_2,k}\right]\right)\right)
    \right\vert\\
    &+ (K - 1)\left\vert
    \frac{4\mathcal{T}_\circ (B_1 - B)}{V_1^2d}\sum_{t_1 = B+1}^{T_1}\sum_{t_2  = B+1}^{T_k}\mathcal{K}\left(\frac{t_1 - t_2}{H}\right)\widehat{\vartheta}_{t_1,1}\widehat{\vartheta}_{t_2,1}\right.\\
    &\left.- \mathrm{Var}\left(\frac{\sqrt{\mathcal{T}_\circ(B_1-  B)}}{V_1\sqrt{d}}\sum_{B\leq \vert t_1 - t_2\vert\leq B_1}^{T_k}\left(\boldsymbol{\epsilon}_{t_1,1}^\top\boldsymbol{\epsilon}_{t_2,1} - \mathbf{E}\left[\boldsymbol{\epsilon}_{t_1,1}^\top\boldsymbol{\epsilon}_{t_2,1}\right]\right)\right)
    \right\vert\\
    &= O_p\left(\frac{\log(\mathcal{T}_\circ)}{\mathcal{T}_\circ^{1/6 - \kappa_2}} + \frac{1}{\mathcal{T}_\circ^{\kappa_2 / 2}}\right).
\end{align*}

Choose $\psi = \sqrt{\log(\mathcal{T}_\circ)},$ then 
\begin{align*}
   \mathbf{Pr}^*\left(\widehat{S}^*_b\leq x\right) - \mathbf{Pr}\left(\zeta\leq x\right)
   &\leq\mathbf{E}^*\left[g_{\psi,x}\left(\widehat{S}^*_b\right)\right] - \mathbf{E}\left[g_{\psi,x - 1/\psi}\left(\zeta\right)\right]\\
   &\leq \sup_{x\in\mathbf{R}}\left\vert
   \mathbf{E}^*\left[g_{\psi,x}\left(\widehat{S}^*_b\right)\right] - \mathbf{E}\left[g_{\psi,x}\left(\zeta\right)\right]
   \right\vert + \mathbf{Pr}\left(x - \frac{1}{\psi} \leq \zeta\leq x + \frac{1}{\psi}\right)\\
   &\leq  \sup_{x\in\mathbf{R}}\left\vert
   \mathbf{E}^*\left[g_{\psi,x}\left(\widehat{S}^*_b\right)\right] - \mathbf{E}\left[g_{\psi,x}\left(\zeta\right)\right]
   \right\vert + \frac{C}{\psi},
\end{align*}
and 
\begin{align*}
    \mathbf{Pr}^*\left(\widehat{S}^*_b\leq x\right) - \mathbf{Pr}\left(\zeta\leq x\right)
   &\geq \mathbf{E}^*\left[g_{\psi,x - 1/\psi}\left(\widehat{S}^*_b\leq x\right)\right] - \mathbf{E}\left[g_{\psi,x}\left(\zeta\right)\right]\\
   &\geq -\sup_{x\in\mathbf{R}}\left\vert
   \mathbf{E}^*\left[g_{\psi,x}\left(\widehat{S}^*_b\right)\right] - \mathbf{E}\left[g_{\psi,x}\left(\zeta\right)\right]
   \right\vert - \mathbf{Pr}\left(x - \frac{1}{\psi} \leq \zeta\leq x + \frac{1}{\psi}\right)\\
   &\geq -\sup_{x\in\mathbf{R}}\left\vert
   \mathbf{E}^*\left[g_{\psi,x}\left(\widehat{S}^*_b\right)\right] - \mathbf{E}\left[g_{\psi,x}\left(\zeta\right)\right]
   \right\vert - \frac{C}{\psi}.
\end{align*}
Since $\widehat{S}^*_b$ has normal distribution conditional on observations, from Assumption 4,
\begin{align*}
    &\left\vert\mathbf{E}^*\left[g_{\psi,x}\left(\widehat{S}^*_b\right)\right] - \mathbf{E}\left[g_{\psi,x}(\zeta)\right]\right\vert\\
    &\leq C\psi\left\vert\sqrt{\mathrm{Var}^*\left(
    \widehat{S}^*_b
    \right)} - \sqrt{\mathrm{Var}(\zeta)}\right\vert\\
    &\leq C_1\psi\left\vert \mathrm{Var}^*\left(
    \widehat{S}^*_b\right) - \mathrm{Var}(\zeta)\right\vert = O_p\left(\frac{\log^{3/2}(\mathcal{T}_\circ)}{\mathcal{T}_\circ^{1/6 - \kappa_2}} + \frac{\sqrt{\log(\mathcal{T}_\circ)}}{\mathcal{T}_\circ^{\kappa_2 / 2}}\right),
\end{align*}
so
\begin{align*}
    \sup_{x\in\mathbf{R}}\left\vert
        \mathbf{Pr}^*\left(\widehat{S}^*_b\leq x\right) - \mathbf{Pr}\left(\zeta\leq x\right)
        \right\vert = O_p\left(\frac{1}{\sqrt{\log(\mathcal{T}_\circ)}}\right),
\end{align*}
which proves \eqref{eq.consistency_bootstrap}.
\end{proof}


\bibliographystyle{apalike} 
\bibliography{Arxiv}       

\end{document}